\documentclass[11pt,a4paper]{amsart}

\usepackage[T1]{fontenc}
\usepackage{lmodern}
\usepackage{amsmath,amssymb,amsthm,mathtools,mathrsfs}
\usepackage[margin=28mm,headheight=14pt]{geometry}
\usepackage{enumitem}
\usepackage{booktabs}
\usepackage{cite}
\usepackage{xcolor}
\usepackage{microtype}
\usepackage[colorlinks=true,
            linkcolor=blue!50!black,
            citecolor=blue!50!black,
            urlcolor=blue!50!black]{hyperref}

\allowdisplaybreaks[2]
\numberwithin{equation}{section}
\makeatletter
\renewcommand{\tocsection}[3]{\indentlabel{\@ifnotempty{#2}{\ignorespaces#1 #2\quad}}#3}

\def\l@section{\@tocline{1}{9pt plus 2pt}{0pt}{1.55pc}{\bfseries}}
\def\l@subsection{\@tocline{2}{0pt}{1.55pc}{2.35pc}{\mdseries}}
\def\l@subsubsection{\@tocline{3}{0pt}{3.90pc}{3.10pc}{\mdseries}}
\def\l@paragraph{\@tocline{4}{0pt}{8.2pc}{5.2pc}{\mdseries}}
\def\l@subparagraph{\@tocline{5}{0pt}{10.2pc}{5.8pc}{\mdseries}}

\def\@tocline#1#2#3#4#5#6#7{\relax
  \ifnum #1>\c@tocdepth
  \else
    \par\addpenalty\@secpenalty\addvspace{#2}\begingroup
      \hyphenpenalty=50\exhyphenpenalty=50
      \@tempdima#4\relax
      \parindent\z@
      \leftskip#3\relax\advance\leftskip\@tempdima\relax
      \rightskip\@pnumwidth plus 3em\relax
      \parfillskip-\@pnumwidth
      #5\leavevmode\hskip-\@tempdima #6\nobreak\relax
      \ifnum#1=1
        \hfil\nobreak
        \hbox to\@pnumwidth{\hss\bfseries #7}\else
        \nobreak\leaders\hbox{$\m@th\mkern4.5mu\hbox{.}\mkern4.5mu$}\hfill\nobreak
        \hbox to\@pnumwidth{\hss\mdseries #7}\fi
      \par\nobreak
    \endgroup
  \fi}
\makeatother

\newtheorem{theorem}{Theorem}[section]
\newtheorem{proposition}[theorem]{Proposition}
\newtheorem{lemma}[theorem]{Lemma}
\newtheorem{corollary}[theorem]{Corollary}
\theoremstyle{definition}

\theoremstyle{remark}
\newtheorem{remark}[theorem]{Remark}
\newtheorem*{remark*}{Remark}

\newcommand{\dd}{\,d}
\newcommand{\R}{\mathbb R}

\newcommand{\Q}{\bar Q}
\newcommand{\Ph}{\Phi_*}
\newcommand{\D}{r\partial_r}

\title[Nonradial stability of Keller--Segel blow-up profiles]
{Nonradial stability of an infinite family of self-similar blow-up profiles for the three-dimensional Keller--Segel equation}

\author{Te Li}
\address{State Key Laboratory of Mathematical Sciences, Academy of Mathematics and Systems Science, Chinese Academy of Sciences, Beijing 100190, China}
\email{teli@amss.ac.cn}

\author{Yuwei Sun}
\address{Academy of Mathematics and Systems Science, Chinese Academy of Sciences, Beijing 100190, China}
\email{sunyuwei@amss.ac.cn}

\author{Kaiqiang Zhang}
\address{School of Computer Science and Technology, Dongguan University of Technology, No.~1 Daxue Road, Songshan Lake District, Dongguan, Guangdong 523808, China}
\email{math.zhangkq@dgut.edu.cn}

\date{November 14, 2026}

\keywords{Keller--Segel system, self-similar blow-up, spectral gap, matched asymptotics, nonradial stability, stable manifold}

\hypersetup{
  pdftitle={Nonradial stability of an infinite family of self-similar blow-up profiles for the three-dimensional Keller\textendash{}Segel equation},
  pdfauthor={Te Li, Yuwei Sun, Kaiqiang Zhang}
}

\begin{document}
\pagestyle{plain}
\raggedbottom

\begin{abstract}
The three-dimensional parabolic-elliptic Keller-Segel equation admits infinitely many radial self-similar blow-up profiles. Earlier work established radial and nonradial stability for one explicit profile. We study the non-explicit profiles constructed by inner-outer matching and allow perturbations without any symmetry assumption. For each sufficiently large profile index, we prove that the corresponding blow-up dynamics are asymptotically stable along a Lipschitz manifold of initial data. When the blow-up center and scale are allowed to vary, the codimension of this manifold equals the number of intrinsic unstable radial modes.

We derive refined inner-outer asymptotics and two-scale estimates for the profiles and analyze the nonlocal, non-self-adjoint linearized operator in each spherical harmonic sector. In the radial sector, a reduced-mass transformation leads to a local Sturm-Liouville problem and gives the exact number of negative eigenvalues. In the $l=1$ sector, we identify the translation modes and exclude any additional unstable spectrum. For every $l\ge 2$, the sharp lower bound on the real part of the spectrum is $1/4$. We then combine the spectral decomposition with modulation equations, higher-order energy estimates, and a finite-dimensional Brouwer argument to obtain the nonlinear stability result.
\end{abstract}

\maketitle

\hypersetup{linkcolor=black}
\tableofcontents
\hypersetup{linkcolor=blue!50!black}

\section{Introduction}

\subsection{Background}

We consider the parabolic--elliptic Keller--Segel system 
\begin{equation}\label{eq:KS}
\left\{\begin{array}{l}
 u_t=\Delta u-\nabla \cdot\left(u \nabla \Phi_{u}\right),  \\
0=\Delta \Phi_{u}+u,
\end{array} (x,t)\in \mathbb{R}^{d}\times(0,T). \right.
\end{equation}

This system is a classical model for chemotactic
aggregation in biological populations
\cite{KellerSegel1970,Horstmann,Biler2020}. Here $u(x,t)$ denotes the
cell density and $\Phi_u$ the concentration of the chemoattractant,
with the evolution governed by the competition between diffusion and
chemotactic attraction. The same system has also been used
to describe the self-gravitating particles in stellar dynamics
\cite{Wolansky}. 
We refer to \cite{Horstmann,Biler2020} and the references therein for
further background.

Local well-posedness depends on the dimension and the function space.
For the three-dimensional $H^2$ setting used below, we give a proof in
Proposition~\ref{prop:local-wp}. We call a finite maximal time $T$ a
blow-up time if
\[
 \limsup_{t\to T}\|u(t)\|_{L^\infty(\mathbb R^d)}=\infty.
\]
For a nonnegative solution that blows up at $T$, the maximum principle yields
\[
 \|u(t)\|_{L^\infty(\mathbb R^d)}\ge(T-t)^{-1}.
\]
The blow-up is said to be of type I if 
\begin{equation*}
\limsup_{t\to T^-} \ (T-t) \| u(t)\|_{L^\infty(\mathbb R^d)}<\infty
\end{equation*}
and of type II otherwise. For integrable solutions with sufficient decay at infinity, the divergence structure of \eqref{eq:KS} gives mass conservation:
 \[
M:=\int_{\mathbb{R}^{d}}u_0(x)dx=\int_{\mathbb{R}^{d}}u(x,t)dx,\ \ 0\le t<T.
 \]
Equation~\eqref{eq:KS} is invariant under the scaling and spatial
translations. More precisely, if $u$ is a solution of \eqref{eq:KS},
then, for any $\lambda>0$ and $a\in\mathbb R^d$,
\[
u_{\lambda,a}(t,x)
:=
\frac{1}{\lambda^2}
u\left(\frac{t}{\lambda^2},\frac{x-a}{\lambda}\right)
\]
is also a solution. Under this transformation, the mass scales as
\[
\|u_{\lambda,a}(t)\|_{L^1(\mathbb R^d)}
=
\lambda^{d-2}
\left\|
u\left(\frac{t}{\lambda^2}\right)
\right\|_{L^1(\mathbb R^d)}.
\]
This scaling naturally leads to the classification according to mass
criticality: the case $d=2$ is mass-critical, since the $L^1$ norm is
invariant under the scaling, while $d=1$ and $d\ge3$ correspond to the
mass-subcritical and mass-supercritical regimes, respectively.

\emph{Singularity formation in dimension $d=2$.} In one space dimension, all solutions of \eqref{eq:KS} are global in
time \cite{Nagai-1995,Childress-Jerome}. In two dimensions, the problem admits finite-time blow-up solutions;
for the classical finite-mass blow-up regimes, these are of type~II
\cite{Naito-2008,Suzuki-2011}.
A distinguished feature of the two-dimensional Keller--Segel equation
is the critical mass threshold $8\pi$, which separates global existence
from finite-time blow-up, see \cite{Childress,Childress-Jerome,Diaz-1998,Blanchet-2006,Blanchet-2012}.
The classical radial type~II blow-up regime, its stability, and its
refined asymptotics have been extensively studied; see
\cite{Raphael-2014,Collot-Ghoul-2022,Collot-2022,Mizoguchi-2022,Davila-2024}
and the references therein. The recent works have revealed additional
multi-point and collision-type concentration mechanisms
\cite{Collot-Ghoul-Masmoudi-Nguyen-2024,Buseghin-Davila-2026}. More recently, a complete classification of radial
dynamics with finite second moment was given in \cite{Buseghin-Collot-202601,Buseghin-Collot-202602}.

\emph{Singularity formation in dimensions $d>2$.}
The situation changes substantially in dimensions $d>2$, where the
problem is mass-supercritical with respect to the natural scaling.
In contrast to the two-dimensional case, there is no positive critical
mass threshold separating global existence from finite-time blow-up.
Indeed, finite-time blow-up solutions exist with arbitrarily prescribed
positive mass \cite{Nagai-1995}.

For radial solutions, several classification results are known.
In dimensions $3\le d\le9$, radially decreasing finite-time blow-up
solutions are necessarily of type~I
\cite{Mizoguchi-2011}, and radial nonnegative type~I solutions are
asymptotically backward self-similar
\cite{Giga-Mizoguchi-Senba-2011}. Refined information on the final-time
profile has been obtained in
\cite{Souplet-Winkler,Bai-Zhou-2025,Chabi-Souplet-2026}.

Several distinct radial blow-up mechanisms have been constructed
in dimensions $d>2$, including backward  type~I self-similar blow-up \cite{Herrero-1998,BrennerEtAl1999,NWZ2026}, type~I blow-up on flatter spatial scales
\cite{Nguyen-Zaag-2023,Liu-wang-Zhou-2025}, type~II concentration of
a stationary state \cite{Mizo-Senba,Hou-Qin-Wang-2026}, and type~II concentration near a
shrinking sphere \cite{Collot-2023,Herrero-1997}. 
It is conjectured by Collot and the last author in \cite{CollotZhang2024} that there are no other radial blow-up mechanism. 

Nonradial singularity formation is also possible: in dimension three,
axisymmetric type~II blow-up solutions have recently been constructed in
\cite{Hou-Nguyen-Song-2025,Buseghin-Davila-Pino-Musso-2026}. 

Backward type~I self-similar blow-up solutions are of the form
\begin{equation*}
u(t,x)=\frac{1}{T-t}U\left(y\right),\ \ y=\frac{x}{\sqrt{T-t}},
\end{equation*}
where $U(y)$ is the self-similar profile satisfying
\begin{equation}\label{eq:steadystateU}
\Delta U-\frac{y\cdot\nabla U}{2}-U+\nabla U\cdot\nabla \Delta^{-1}U+U^2=0.
\end{equation}
Several classes of solutions to \eqref{eq:steadystateU} are known:
\begin{itemize}
    \item For $d\ge1$, there are the constant solutions
    \begin{equation*}
        U_{c_1}\equiv0,
        \qquad
         U_{c_2}\equiv1.
    \end{equation*}

    \item For $d\ge3$, there is the singular radial solution
    \begin{equation*}
         U_{s}(y)
        =
        \frac{2(d-2)}{y^2}.
    \end{equation*}

    \item For $d\ge3$, there is the explicit smooth positive radial
    solution \cite{BrennerEtAl1999}
    \begin{equation*}
        U_{e}(y)
        =
        \frac{4(d-2)(2d+y^2)}
        {(2(d-2)+y^2)^2}.
    \end{equation*}

    \item For $3\le d\le9$, there exist infinite many  smooth radial solutions
    $\{ U_n\}_{n\ge1}$
    \cite{Herrero-1998,BrennerEtAl1999,NWZ2026}, satisfying
    \begin{equation*}
         U_n(y)\asymp \frac{1}{|y|^2},
        \qquad |y|\to\infty.
    \end{equation*}
\end{itemize}

Besides the
explicit smooth profile $U_e$, infinitely many smooth radial
profiles $\{ U_n\}_{n\ge1}$ were constructed by shooting arguments first
 in dimension $d=3$ \cite{Herrero-1998}, and subsequently in
dimensions $3\le d\le9$
\cite{BrennerEtAl1999}. More recently, Nguyen, Wang and the
last author gave a new construction of infinitely many self-similar
profiles in dimensions $3\le d\le9$ by an inner--outer matching
argument \cite{NWZ2026}. A complete classification of backward self-similar profiles remains
open, even within the radial class. In particular, the precise relation
between the countable families obtained by earlier shooting arguments \cite{Herrero-1998,BrennerEtAl1999}
and those constructed more recently via inner--outer matching \cite{NWZ2026} is not yet understood.

It was shown in \cite{Giga-Mizoguchi-Senba-2011} that every radial,
nonnegative type~I blow-up solution is asymptotically backward
self-similar near the blow-up point as $t\to T$. This result highlights
the importance of self-similar blow-up dynamics for understanding the
formation and local structure of type~I singularities.

The stability theory for these profiles, however, is much less complete.
The radial nonlinear stability of the explicit profile $U_e$ was
first suggested by numerical evidence in \cite{BrennerEtAl1999} and later
proved rigorously in \cite{GlogicSchoerkhuber2024} by the semigroup theory methods. Collot and Zhang \cite{CollotZhang2024} subsequently established the
finite-codimensional nonlinear stability of the higher self-similar
profiles $\{ U_n\}_{n\ge1}$ under radial perturbations, using a modulation argument. More recently, Li and Zhou \cite{LiZhou2026}
established the stability of $U_e$ under general nonradial
perturbations through a detailed analysis of the non-self-adjoint
linearized operator in spherical harmonic sectors. They also pointed out that a rigorous nonradial
stability analysis for the non-explicit higher profiles was still
lacking.  

While the nonradial stability theory is now understood for the
explicit profile $U_e$, no corresponding rigorous result was
available for the non-explicit profiles.
The present work studies the sufficiently large-index profiles constructed in \cite{NWZ2026}. We establish their spectral structure and prove finite-codimensional nonlinear asymptotic stability under general nonradial perturbations.

\subsection{Main results}

In the rest of this paper, we set $d=3.$ For a radial self-similar profile $U=U(r)$ satisfying \eqref{eq:steadystateU}, it is convenient,
following \cite{BrennerEtAl1999,GlogicSchoerkhuber2024}, to introduce the reduced mass
\begin{equation*}
    \Phi(r)
    :=
    \frac{1}{2r^3}
    \int_0^r U(\rho)\rho^2\,d\rho .
\end{equation*}
Then $\Phi$ satisfies
\begin{equation}\label{eq:profile-main}
    \Phi''
    +\left(\frac4r-\frac r2\right)\Phi'
    -\Phi
    +6\Phi^2
    +2r\Phi\Phi'
    =0.
\end{equation}
Conversely, the density profile $U$ and the associated radial Newtonian
field can be recovered from $\Phi$ through
\begin{equation*}
    U=6\Phi+2r\Phi',
    \qquad
    P:=\partial_r\Delta^{-1}U=2r\Phi.
\end{equation*}

For radial functions, write $\Lambda f(r):=2f(r)+rf'(r)$.
 Here and below, ``large-index'' means that the index $n$ of $U_n$ is sufficiently large.
In the following proposition, we determine the number of zeros of \(\Lambda\Phi_n\) for the large-index self-similar profiles, which in turn determines the number of genuine radial unstable modes in Theorem \ref{thm:main-spectral} and the number of graph conditions \emph{within the fixed geometric gauge} in Theorem \ref{thm:nonradial-stability}.

\begin{proposition}[Large-index self-similar profiles]
\label{prop:matched-profile-family}
There exists an integer $n_*\gg1$ such that, for every $n\ge n_*$,
equation \eqref{eq:profile-main} admits a smooth radial solution
$\Phi_n$. 
Moreover, if
\[
    N_n
    :=
    \#\bigl\{
        r>0:\Lambda\Phi_n(r)=0
      \bigr\},
\]
then all the zeros of $\Lambda\Phi_n$ are simple and
\begin{equation}\label{eq:Nn-count-main}
    N_n=n+N_0
\end{equation}
for some integer $N_0$ independent of $n$. Consequently,
\[
    U_n:=6\Phi_n+2r\Phi_n'
\]
is a smooth radial solution of the self-similar profile equation
\eqref{eq:steadystateU} for every $n\ge n_*$.
\end{proposition}

We next describe the unstable spectrum of the linearized operator around $U_n$.
In the self-similar variables
\[
    y=\frac{x}{\sqrt{T-t}},
    \qquad
    s=-\log(T-t),
\]
the linearization at $U_n$ is given by
\begin{equation*}
    L_n f
    =
    -\Delta f+\frac12\Lambda f-2U_nf
    -\nabla\Delta^{-1}U_n\cdot\nabla f
    -\nabla U_n\cdot\nabla\Delta^{-1}f,
\end{equation*}
here
   $ \Lambda=2+y\cdot\nabla$ is the scaling generator.
We regard $L_n$ as an unbounded operator on $L^2(\mathbb R^3)$
with domain
\begin{equation*}
    \mathscr D
    :=
    \left\{
        f\in H^2(\mathbb R^3):
        y\cdot\nabla f\in L^2(\mathbb R^3)
    \right\}.
\end{equation*}
Our second main result gives a complete description of the unstable
spectrum of $L_n$.

\begin{theorem}[Unstable spectrum of the linearized operator $L_n$]
\label{thm:main-spectral}
Let $n$ be sufficiently large. There exist exactly $N_n$ simple
eigenvalues of $L_n$, denoted by $-\mu_{j,n}$ with
$\mu_{j,n}>0$ for $1\le j\le N_n$, which satisfy
\[
    -\mu_{1,n}
    <
    -\mu_{2,n}
    <
    \cdots
    <
    -\mu_{N_n,n}
    <-1,
\]
with corresponding radial eigenfunctions
\[
    \phi_{1,n},\ldots,\phi_{N_n,n}.
\]
Besides these eigenvalues, the only spectrum of $L_n$ in the open
left half-plane consists of the scaling eigenvalue $-1$ and the
translation eigenvalue $-1/2$. More precisely,
\[
    \ker(L_n+1)
    =
    \operatorname{span}\{\Lambda U_n\},\ \  \ker\left(L_n+\frac12\right)
    =
    \operatorname{span}
    \left\{
        \partial_1U_n,
        \partial_2U_n,
        \partial_3U_n
    \right\}.
\]
Thus
\begin{equation*}
    \sigma(L_n)\cap\{z\in\mathbb C:\Re z<0\}
    =
    \left\{
        -\mu_{1,n},\ldots,-\mu_{N_n,n},
        -1,-\frac12
    \right\}.
\end{equation*}
Consequently, the unstable spectral subspace is
\[
    E_n^{\rm u}
    =
    \operatorname{span}
    \left\{
        \phi_{1,n},\ldots,\phi_{N_n,n},
        \Lambda U_n,
        \partial_1U_n,
        \partial_2U_n,
        \partial_3U_n
    \right\},
\]
and hence
    $\dim E_n^{\rm u}=N_n+4$.
\end{theorem}

Theorem~\ref{thm:main-spectral} provides the starting point for the
nonlinear analysis. To state the nonlinear stability result, let
\[
    \zeta_{0,n},
    \qquad
    \zeta_{k,n}^{\rm tr}\ (1\le k\le3),
    \qquad
    \zeta_{j,n}^{\rm u}\ (1\le j\le N_n)
\]
be the corresponding dual modes defined in
Lemma~\ref{lem:adjointdualexistence}. We set
\[
 \mathcal Y
 :=
 \left\{
 f\in H^2(\mathbb R^3):
 \Lambda f\in L^2(\mathbb R^3)
 \right\},
 \qquad
 \|f\|_{\mathcal Y}
 :=
 \|f\|_{H^2}
 +
 \|\Lambda f\|_{L^2},
\]
and define the stable space
\[
 V_n
 :=
 \left\{
 v\in\mathcal Y:
 \begin{array}{l}
 \langle v,\zeta_{0,n}\rangle=0,\\[1mm]
 \langle v,\zeta_{k,n}^{\rm tr}\rangle=0,
 \quad 1\le k\le3,\\[1mm]
 \langle v,\zeta_{j,n}^{\rm u}\rangle=0,
 \quad 1\le j\le N_n
 \end{array}
 \right\}.
\]
We give the following nonlinear stability result.
\begin{theorem}[Finite-codimensional stability of $U_n$]
\label{thm:nonradial-stability}
Let $n$ be sufficiently large. There exist constants
$\delta_n>0$ and a Lipschitz map
\[
 c_n=(c_{1,n},\ldots,c_{N_n,n})
 :
 V_n\cap B_{\mathcal Y}(0,\delta_n)
 \longrightarrow
 \mathbb R^{N_n},
 \qquad
 c_n(0)=0,
\]
such that,
for every
$v_0\in V_n\cap B_{\mathcal Y}(0,\delta_n)$,
the corresponding solution of \eqref{eq:KS} with the initial data
\begin{equation}\label{eq:corrected-initial-data}
 u_0
 =
 U_n+v_0
 +
 \sum_{j=1}^{N_n}
 c_{j,n}(v_0)\phi_{j,n}
\end{equation}
blows up in finite
time $T<\infty$. More precisely, there exist a $C^1$ modulation center
$x(t)$ and a remainder $\widetilde u$ such that
\begin{equation*}
 u(t,x)
 =
 \frac1{T-t}
 \left[
 U_n\left(
 \frac{x-x(t)}{\sqrt{T-t}}
 \right)
 +
 \widetilde u\left(
 t,\frac{x-x(t)}{\sqrt{T-t}}
 \right)
 \right],
\end{equation*}
where
\begin{equation}\label{eq:nonlinear-convergence}
 \lim_{t\to T}
 \|\widetilde u(t)\|_{H^2(\mathbb R^3)}
 =0,
 \qquad
 \lim_{t\to T}
 x(t)
 =
 x_*.
\end{equation}
In particular, $x_*$ is a blow-up point and the blow-up is of type~I.
\end{theorem}

\begin{remark}[Codimension and the geometric gauge]\label{rem:codimension-gauge}
The meaning of ``finite codimension'' in the theorem is precise as
follows.  Define the closed gauge slice
\[
 X_n^{\rm geo}:=\{f\in\mathcal Y:
 \langle f,\zeta_{0,n}\rangle=0,\quad
 \langle f,\zeta_{k,n}^{\rm tr}\rangle=0\quad(1\le k\le3)\}.
\]
The dual normalization and the membership of all right modes in
\(\mathcal Y\) imply the complemented decompositions
\[
 \mathcal Y=X_n^{\rm geo}\oplus
   \operatorname{span}\{\Lambda U_n,\partial_1U_n,
       \partial_2U_n,\partial_3U_n\},
 \qquad
 X_n^{\rm geo}=V_n\oplus
   \operatorname{span}\{\phi_{j,n}:1\le j\le N_n\}.
\]
Indeed, each coefficient is extracted by its bounded dual pairing,
so both splittings are topological direct sums in \(\mathcal Y\).
Consequently, the set of perturbations
\[
 \mathcal M_n^{\rm fix}:=
 \left\{v_0+\sum_{j=1}^{N_n}c_{j,n}(v_0)\phi_{j,n}:
       v_0\in V_n\cap B_{\mathcal Y}(0,\delta_n)\right\}
\]
is a Lipschitz graph of codimension \(N_n\) {inside}
\(X_n^{\rm geo}\), and of codimension \(N_n+4\) in the full
space \(\mathcal Y\).  The four additional conditions fix the
initial scale and translation gauge; they are not four extra
intrinsic radial instabilities.  The corresponding symmetry-saturated set is shown to be a
Lipschitz graph of codimension \(N_n\) in the full space
\(\mathcal Y\) in Corollary~\ref{cor:full-geometry-stable}.
This stronger statement requires a separate argument: the group action
is strongly continuous on \(\mathcal Y\), but is not generally
Lipschitz in its geometric parameters with values in that norm.
\end{remark}

\begin{corollary}[Stability with free translation and scaling gauges]
\label{cor:full-geometry-stable}
For each sufficiently large fixed $n$, set
\[
 E_n^{\rm geo}:=\operatorname{span}\{\Lambda U_n,\partial_1U_n,
                         \partial_2U_n,\partial_3U_n\}.
\]
There are a radius $\rho_n>0$ and a Lipschitz map
\[
 \widehat c_n:\bigl(V_n\oplus E_n^{\rm geo}\bigr)
 \cap B_{\mathcal Y}(0,\rho_n)\longrightarrow\mathbb R^{N_n},
\]
with $\widehat c_n(0)=0$, such that every initial datum
\begin{equation}\label{eq:full-geometry-graph}
 u_0=U_n+z_0+\sum_{j=1}^{N_n}\widehat c_{j,n}(z_0)\phi_{j,n},
 \qquad z_0\in (V_n\oplus E_n^{\rm geo})\cap B_{\mathcal Y}(0,\rho_n),
\end{equation}
produces finite-time type~I blow-up with the asymptotic profile in
Theorem~\ref{thm:nonradial-stability}.  The corresponding initial-data set is a
local Lipschitz graph of codimension $N_n$ in the full space
$\mathcal Y$, and is precisely the local saturation of the
fixed-gauge graph $\mathcal M_n^{\rm fix}$ by translations and
positive dilations.  On a possibly smaller ball in $V_n$,
$\widehat c_n(v)=c_n(v)$.
\end{corollary}

The proof of Corollary~\ref{cor:full-geometry-stable} is given at the
end of the nonlinear analysis.  Its use of $L^2$-based gauge coordinates
is important: one must not assert parameter-Lipschitz continuity of
translations on all of $H^2$ or $\mathcal Y$.

\subsection{Comments on the results}

\emph{Quantitative information on the matched profiles.}
Smooth radial self-similar blow-up profiles were constructed by shooting
arguments in \cite{Herrero-1998,BrennerEtAl1999}, and a further infinite
family was obtained by inner--outer matching in \cite{NWZ2026}.
We work with the latter family. Building on its construction, we derive
the quantitative derivative and phase estimates required for the spectral
analysis. The matching phase, in particular, makes it possible
to determine the exact number of zeros of the scaling mode.

\medskip
\noindent\emph{The complete unstable spectrum.}
For the explicit profile $U_e$, Glogi\'c and Sch\"orkhuber
\cite{GlogicSchoerkhuber2024} proved the unstable
scaling mode, and Li and Zhou \cite{LiZhou2026} subsequently identified
the scaling and translation modes as the only unstable directions
under general nonradial perturbations. The higher-index profiles
$U_n$ have a different radial spectral structure: there are exactly $N_n=n+N_0$ additional radial
eigenvalues below $-1$.

Our spectral result shows that these radial eigenvalues, together
with the scaling and translation eigenvalues, exhaust the entire
unstable spectrum of $L_n$ on $L^2(\mathbb R^3)$. Thus general
three-dimensional perturbations introduce no additional intrinsic
unstable spectral directions. The full unstable subspace has
dimension $N_n+4$, and its intrinsic instability index grows
with the profile index. This precise classification is of
independent interest and provides the linear foundation for the
nonlinear analysis.

\medskip
\noindent\emph{Nonradial stability.}
Higher-profile stability was previously established under radial
symmetry in \cite{CollotZhang2024}. In the nonradial problem,
the reduced-mass scalar equation is no longer available, and
nonlinear interactions between angular components can excite
radial unstable modes. Our nonlinear result constructs, for every
sufficiently large member of the matched family, a fixed-gauge Lipschitz
graph of initial data whose solutions exhibit the prescribed
self-similar blow-up dynamics.  Corollary~\ref{cor:full-geometry-stable}
then restores the four geometric directions and gives a local
Lipschitz graph of codimension $N_n$ in the full graph space
$\mathcal Y$. Within the specified trapping regime,
the genuinely unstable coefficients are uniquely determined by the
stable initial component and depend Lipschitz continuously on it.
The result therefore realizes the exact spectral instability structure
in the nonlinear dynamics and describes the corresponding local
stable initial-data geometry.

\subsection{Strategy of the proof}
\label{subsec:strategy-proof}

We outline the proofs of Theorems~\ref{thm:main-spectral} and
\ref{thm:nonradial-stability}. The main steps form the chain
\begin{center}
\begin{tabular}{c}
quantitative analysis of the matched profiles\\[-1pt]
$\to$\\[-1pt]
complete classification of the unstable spectrum\\[-1pt]
$\to$\\[-1pt]
stable semigroup decay and an equivalent dissipative energy\\[-1pt]
$\to$\\[-1pt]
finite-codimensional nonlinear asymptotic stability.
\end{tabular}
\end{center}
Unless stated otherwise, the matching index $n$ is fixed and
sufficiently large.

\medskip
\noindent\textbf{Step 1-A. Inner--outer construction and uniform profile estimates.}
The starting point, developed in Section~\ref{sec:matched-profiles},
is the inner--outer construction of \cite{NWZ2026}. The corresponding
reduced-mass profile $\Phi_n$ has a shrinking core of scale
$\mu_n\to0$, where it is approximated by the universal stationary
profile $\bar Q$, and an exterior region where it approaches the
singular profile $\Phi_*(r)=r^{-2}$. More precisely, we obtain the
matched representation
\[
 \Phi_n(r)=
 \begin{cases}
  \mu_n^{-2}\bar Q(r/\mu_n)+R_{\mu_n}(r/\mu_n),&0\le r\le r_0,\\
  r^{-2}+\varepsilon_n\bigl(u_1(r)+w_{\varepsilon_n}(r)\bigr),
      &r\ge r_0,
 \end{cases}
\]
with quantitative estimates for the corrections and their derivatives.
In particular, writing $P_n=2r\Phi_n$ and
$U_n=6\Phi_n+2r\Phi_n'$, we establish the two-scale estimates
\eqref{eq:inner-scale-limits-main}, \eqref{eq:outer-scale-limits-main}.
These bounds will be used both for the radial eigenvalue count and
for the uniform positivity estimates in the nonradial sectors.

\medskip
\noindent\textbf{Step 1-B. Matching phase and enumeration of radial instabilities.}
The inner profile has a logarithmically oscillating tail. By retaining
its phase in the matching procedure, we select the scales $\mu_n$ so
that
\[
 -\omega_*\log\mu_n+\vartheta_0=n\pi+e_n,
 \qquad |e_n|<\frac\pi4.
\]
The inner oscillations, together with the stabilization of the zero
count in the exterior region, then give
\begin{equation}\label{eq:strategy-zero-count}
 N_n:=\#\{r>0:\Lambda\Phi_n(r)=0\}=n+N_0.
\end{equation}
The zeros are simple. This precise count is essential: through the
radial Sturm--Liouville problem, it determines the number of eigenvalues
strictly below the scaling eigenvalue $-1$.

\medskip
\noindent\textbf{Step 2-A. Classification of the unstable
spectrum by spherical harmonic decomposition.}
In Section~\ref{sec:unstable-spectrum}, we work with the closed
realization of $L_n$ on $L^2(\mathbb R^3)$ and decompose the spectral
problem into spherical harmonic sectors. The three ranges of angular
degrees require different arguments.

\begin{itemize}
\item \emph{Radial sector $l=0$.} The reduced-mass transformation
converts the nonlocal eigenvalue problem into a local, self-adjoint
Sturm--Liouville problem. The count \eqref{eq:strategy-zero-count}
yields exactly $N_n$ simple radial eigenvalues below $-1$. A separate
comparison of inner and outer spectral phases rules out eigenvalues
in $(-1,0]$.

\item \emph{The translation sector $l=1$.}
The translation invariance of the equation gives the
eigenvalue $-1/2$, with eigenfunction $U_n'$ in each fixed
degree-one radial component.
The main task is to exclude additional unstable eigenvalues,
despite the nonlocal structure of $L_{n,1}$.

Following the wave-operator approach of
\cite{LiWeiZhang2020}, we construct a transformation
$T_n$ whose kernel is precisely
$\operatorname{span}\{U_n'\}$ and which satisfies an
intertwining relation
\[
T_nL_{n,1}=\widetilde L_{n,1}T_n,
\]
where $\widetilde L_{n,1}$ is a local differential operator.
After a suitable gauge transformation, the spectral problem
is reduced to a self-adjoint Schr\"odinger operator
\[
H_n=-\partial_r^2+W_n(r).
\]
Using the quantitative bounds for the non-explicit profiles,
we prove that
\[
W_n(r)\ge c_w>0,\qquad r>0,
\]
uniformly for sufficiently large $n$.
This positivity excludes additional unstable eigenfunctions.
The corresponding domain and intertwining arguments also
exclude nontrivial Jordan chains.
Consequently, $-1/2$ is the only unstable spectral point in
the first angular sector, and it is semisimple with
multiplicity three in the full space.

\item \emph{The higher angular sectors $l\ge3$.}
For these modes, the angular contribution provides sufficient
coercivity to dominate the terms
arising from the profile and the nonlocal Newton interaction.

We derive an exact quadratic-form identity for $L_{n,l}$
and introduce the logarithmic radial variable $t=\log r$.
In these coordinates, the nonlocal terms can be represented
through one-dimensional integral operators whose norms are
controlled by Fourier multiplier estimates.
Combining these bounds with the uniform two-scale estimates
for $U_n$ and its derivatives, we obtain
\[
\operatorname{Re}
\langle L_{n,l}u,u\rangle_{L^2(r^2dr)}
\ge
\frac14\|u\|_{L^2(r^2dr)}^2
+c_3\|u\|_{L^2(dr)}^2,
\qquad l\ge3,
\]
where $c_3>0$ is independent of $l$ and all sufficiently
large $n$.
In particular, no eigenvalue can occur to the left of
the spectral threshold $1/4$.

\item \emph{The borderline sector $l=2$.}
This sector is substantially more delicate, since the
coarse angular coercivity estimate used for $l\ge3$
is insufficient when $l=2$.
We therefore introduce a partially localized operator
through a first-order integral transformation and a
weight
\[
m(r)=\frac{r}{\sqrt{1+r^2}},
\]
which behaves like $r$ near the origin and remains
bounded at infinity.

The resulting operator admits separate coercive
estimates in the inner and outer regions.
To combine them, we use a partition in logarithmic radius.
An essential difficulty is that a spatial cutoff of a
perturbation does not localize its Newton field.
We therefore retain the nonlocal interactions across
the cutoff and control them through explicit estimates
for the Newton tails and the localization commutators.

These estimates yield a global spectral exclusion
for the borderline sector, completing the analysis of
all higher angular modes.

\end{itemize}

The Fredholm reduction and the uniform high-mode resolvent bounds then
pass from the sectorwise eigenvalue exclusions to the full-space
spectral statement. In the notation of Theorem~\ref{thm:main-spectral},
with the negative radial eigenvalues written as $-\mu_{j,n}$, we obtain
\begin{equation}\label{eq:strategy-unstable-spectrum}
 \sigma(L_n)\cap\{\operatorname{Re}z<0\}
 =\bigl\{-\mu_{1,n},\ldots,-\mu_{N_n,n},-1,-\tfrac12\bigr\},
 \qquad \mu_{1,n}>\cdots>\mu_{N_n,n}>1.
\end{equation}
In particular, apart from scaling and translation, every unstable
direction is radial, and the full unstable spectral subspace has
dimension $N_n+4$.

\medskip
\noindent\textbf{Step 2-B. From spectral information to a stable evolution.}
The non-self-adjoint nature of $L_n$ means that spectral location alone
does not provide an $L^2$ coercivity inequality. We construct the
adjoint discrete modes, normalize them biorthogonally against the
unstable eigenfunctions, and define the associated projections
$P_{{\rm u},n}$ and $P_{{\rm s},n}=I-P_{{\rm u},n}$.
The compact-perturbation and semigroup arguments yield, for some
$M_n\ge1$ and $\omega_n>0$,
\begin{equation}\label{eq:strategy-semigroup}
 \|e^{-sL_n}P_{{\rm s},n}f\|_{L^2}
 \le M_ne^{-\omega_ns}\|P_{{\rm s},n}f\|_{L^2},
 \qquad s\ge0.
\end{equation}
To turn this decay into a dissipative estimate for the nonlinear
remainder, we use the stable-semigroup orbit integral, together with
the identity, to define a positive operator $B_n$ and the quadratic
energy $\mathcal E_n[f]=\langle B_nf,f\rangle$. For
$f\in D(L_n)\cap\operatorname{Ran}P_{{\rm s},n}$, it satisfies
\[
 \mathcal E_n[f]\asymp\|f\|_{L^2}^2,
 \qquad
 2\operatorname{Re}\langle B_nL_nf,f\rangle
 \ge\gamma_n\mathcal E_n[f]
\]
for some $\gamma_n>0$. The resulting energy is the principal linear
ingredient of the nonlinear bootstrap.

\medskip
\noindent\textbf{Step 3. Nonlinear modulation, bootstrap, and selection of the stable data.}
In Section~\ref{nl:sec:nonlinear-dynamics}, we decompose a solution
near $U_n$ as
\begin{equation}\nonumber
 u(t,x)=\frac1{\lambda^2(t)}
 \bigl[U_n(y)+\psi(s,y)+\varepsilon(s,y)\bigr],
 \qquad y=\frac{x-x(t)}{\lambda(t)},
 \qquad \frac{ds}{dt}=\lambda^{-2}(t),
\end{equation}
where
\[
 \psi(s,y)=\sum_{j=1}^{N_n}a_j(s)\phi_{j,n}(y),
 \qquad P_{{\rm s},n}\varepsilon=\varepsilon.
\]
The latter condition is imposed by orthogonality to all adjoint
unstable modes. Writing $v=\psi+\varepsilon$ and
$\mathcal N(v)=v^2+\nabla v\cdot\nabla\Delta^{-1}v$, the renormalized
remainder equation takes the form
\begin{equation}\nonumber
\begin{aligned}
 \partial_s\varepsilon+L_n\varepsilon
 ={}&\sum_{j=1}^{N_n}\bigl(\mu_{j,n}a_j-(a_j)_s\bigr)\phi_{j,n}
 +\left(\frac{\lambda_s}{\lambda}+\frac12\right)\Lambda(U_n+v)+\frac{x_s}{\lambda}\cdot\nabla(U_n+v)+\mathcal N(v).
\end{aligned}
\end{equation}
Pairing with the adjoint modes yields the modulation estimates
\[
 \left|\frac{\lambda_s}{\lambda}+\frac12\right|
 +\frac{|x_s|}{\lambda}
 +\sum_{j=1}^{N_n}|(a_j)_s-\mu_{j,n}a_j|
 \lesssim\|\varepsilon\|_{L^2}^2+|a|^2,\ \ |a|^2=\sum_{j=1}^{N_n}|a_j|^2.
\]

The modified energy from Step~2-B, together with higher-order
parabolic estimates, controls the nonlocal quadratic terms and
closes an $H^2$ bootstrap for the stable remainder. We also estimate
$\Lambda\varepsilon$ in $L^2$ to handle the unbounded scaling
generator in the difference of two modulated equations. In the
bootstrap regime one has, for a sufficiently small $\mu>0$,
\begin{equation}\label{eq:strategy-bootstrap}
 |a(s)|+\|\varepsilon(s)\|_{H^2}
 +\|\Lambda\varepsilon(s)\|_{L^2}
 \lesssim e^{-\mu s},\qquad s\ge s_0.
\end{equation}
The $N_n$ genuinely unstable coefficients are selected by a
finite-dimensional topological argument. If all choices of their
initial values exited the bootstrap regime, the strict outgoing
condition at the unstable boundary would produce a continuous
retraction of a closed ball onto its boundary, contradicting
Brouwer's no-retraction theorem. Hence at least one choice remains
trapped for all forward renormalized times. Difference estimates
then prove uniqueness of this choice within the prescribed trapping
regime and Lipschitz dependence on the stable initial remainder.

Finally, the modulation bounds imply
$\lambda^2(t)\sim T-t$ and convergence of the center $x(t)$, while
\eqref{eq:strategy-bootstrap} yields convergence to $U_n$ in the
rescaled $H^2$ norm. This proves the finite-time type~I blow-up
and the finite-codimensional Lipschitz family of initial data
claimed in Theorem~\ref{thm:nonradial-stability}.

\emph{Organization of the article.}
In Section~\ref{sec:matched-profiles}, we establish the
quantitative inner--outer profile estimates and determine the
zero count of the scaling mode, thereby proving
Proposition~\ref{prop:matched-profile-family}.
In Section~\ref{sec:unstable-spectrum}, we analyze the closed
linearized operator through spherical harmonic decomposition
and determine its full unstable spectrum, proving
Theorem~\ref{thm:main-spectral}.
We also construct the adjoint modes, establish stable semigroup
decay, and obtain an equivalent energy.
In Section~\ref{nl:sec:nonlinear-dynamics}, we derive the
modulation equations and nonlinear estimates, carry out the
finite-dimensional selection argument, and establish the
Lipschitz dependence of the selected coefficients, thereby
completing the proof of Theorem~\ref{thm:nonradial-stability}.
Appendix~\ref{app:auxiliary} contains auxiliary quantitative
core bounds and the operator-theoretic and semigroup arguments
used in the main text.

\subsection{Notations}

\emph{General notations.} Set \[\Phi_*(r)=r^{-2},\quad \omega_*:=\frac{\sqrt7}{2},\quad \langle r\rangle=\sqrt{1+r^2}.\] 
We denote by
\[
\langle f,g\rangle
:=
\int_{\mathbb R^3}f(y)\overline{g(y)}\,dy
\]
the standard inner product on $L^2(\mathbb R^3)$. For a radial function $f=f(r)$, we denote by $f'$
its derivative with respect to the radial variable $r$.
 For nonnegative quantities $A$ and $B$, we write
$A\lesssim B$ if $A\le CB$ for some constant $C>0$.
We write $A\asymp B$ if both $A\lesssim B$
and $B\lesssim A$ hold, and use $A\sim B$
to mean $A/B\to1$ in the relevant limit.
For $k\ge0$, we write $g=O_k(r^{-\alpha})$
if $\sum_{j=0}^k r^j|\partial_r^jg(r)|
\lesssim r^{-\alpha}$
in the specified asymptotic regime and $g=o_k(r^{-\alpha})$ if $r^\alpha \sum_{j=0}^{k} r^j |\partial_r^j g(r)| \longrightarrow 0$ in the specified limit. 
In addition, $\delta_{ij}$ denotes the Kronecker delta. We also write
\begin{equation}\label{two-notations}
   P_n:=2r\Phi_n,
\qquad
U_n:=6\Phi_n+2r\Phi_n'. 
\end{equation}
\emph{Linearized operators.}
For a closed operator $A$, we denote its spectrum by
$\sigma(A)$, and write $\operatorname{Ran}P$ for the
range of a projection $P$. For a bounded linear operator $\mathsf A:L^2\to L^2$, we denote its
operator norm by
\[
\|\mathsf A\|_{2\to2}
:=
\sup_{f\in L^2\setminus\{0\}}
\frac{\|\mathsf A f\|_{L^2}}{\|f\|_{L^2}}.
\]
The scaling generator is defined by
\begin{equation*}
\Lambda f:=2f+x\cdot\nabla f.
\end{equation*}
The linearized operator around $U_n$ in self-similar
variables is given by
\begin{equation}\label{us:linearized}
\begin{aligned}
L_n f
=-\Delta f+\frac12\Lambda f-2U_nf-\nabla\Delta^{-1}U_n\cdot\nabla f
-\nabla U_n\cdot\nabla\Delta^{-1}f,
\end{aligned}
\end{equation}
acting on $L^2(\mathbb R^3)$ with domain
\begin{equation}\label{us:domain}
D(L_n)=\mathscr D
:=
\left\{
f\in H^2(\mathbb R^3):
x\cdot\nabla f\in L^2(\mathbb R^3)
\right\}.
\end{equation}
For radial functions, we introduce the reduced-mass
transformation
\begin{equation}\label{reduced mass}
    \mathcal T^{-1}f(r)
:=
\frac{1}{2r^3}\int_0^r f(\rho)\rho^2\,d\rho,
\qquad
\mathcal T h=2(rh'+3h).
\end{equation}
Under this transformation, the radial part of $L_n$
is formally intertwined with the local operator
\begin{equation}\label{local-operator}
\mathcal A_n h
=
-h''
+\left(\frac r2-\frac4r-2r\Phi_n\right)h'
+(1-12\Phi_n-2r\Phi_n')h.
\end{equation}

\medskip
\noindent
\emph{The universal inner profile.}
Let $Q$ denote the positive radial stationary solution
of \eqref{eq:KS}, which satisfies
\begin{equation}\label{eq:Q}
Q''+\frac2rQ'+Q^2
+\frac{Q'}{r^2}\int_0^r Q(s)s^2\,ds=0,
\qquad Q(0)=1,\quad Q'(0)=0.
\end{equation}
Its reduced-mass profile is defined by
\begin{equation}\label{eq:PbarQ}
\bar Q(r)
:=
\frac{1}{2r^3}\int_0^r Q(s)s^2\,ds.
\end{equation}
Then $\bar Q$ satisfies
\begin{equation}\label{app:profile}
\bar Q''+\frac4r\bar Q'
+6\bar Q^2+r(\bar Q^2)'=0,
\qquad
\bar Q(0)=\frac16,\quad
\bar Q'(0)=0.
\end{equation}

  \section{Asymptotics of the matched profiles}
\label{sec:matched-profiles}

Following the inner--outer construction of
\cite{NWZ2026}, we establish the matching phase
and uniform derivative estimates needed for the
spectral analysis in Section~\ref{sec:unstable-spectrum}.
These results also yield the zero count of
$\Lambda\Phi_n$.
Proposition~\ref{prop:matched-profile-family} then follows from Proposition \ref{thm:matched-family-main} and the relationship
$U_n=6\Phi_n+2r\Phi_n'$.

\begin{proposition}[Matched self-similar profiles]
\label{thm:matched-family-main}
There exist $r_0>0$, two sequences
\[
\mu_n\to0,
\qquad
\varepsilon_n\to0,
\]
and distinct smooth radial solutions $\Phi_n$ of
\eqref{eq:profile-main} such that
\begin{equation*}
\Phi_n(r)=
\begin{cases}
\displaystyle
\mu_n^{-2}
\bigl[\Q+\mu_n^2R_n\bigr](r/\mu_n),
& 0\le r\le r_0,\\[2mm]
\displaystyle
r^{-2}+\varepsilon_n(u_1+w_n)(r),
& r\ge r_0.
\end{cases}
\end{equation*}
And for all sufficiently large $n$,
\[
 \Phi_n(r)>0\quad(r\ge0),\qquad
 \Phi_n'(r)<0\quad(r>0).
\]

 More precisely, the following
properties hold.
\medskip
\noindent
\textnormal{(i) Inner expansion.}
The universal core $\Q$ solving \eqref{app:profile} satisfies
\begin{equation*}
\Q(s)
=
\frac16-\frac{s^2}{60}+\frac{s^4}{630}+O(s^6),
\qquad s\to0,
\end{equation*}
and
\begin{equation}\label{eq:core-tail-main}
\Q(s)
=
s^{-2}
+c_Qs^{-5/2}
\sin(\omega_*\log s+\vartheta_Q)
+O_1(s^{-3}),
\qquad s\to\infty,
\end{equation}
for some $c_Q\ne0$. Moreover,
\begin{equation*}
\Q'(s)<0,
\qquad
-\Q'(s)\gtrsim \frac{s}{(1+s)^4}.
\end{equation*}
The inner correction $R_n$ satisfies
\begin{equation}\label{eq:Rn-bound-main}
|R_n(s)|
+|sR_n'(s)|
+|s^2R_n''(s)|
\lesssim
(1+s)^{-1/2},
\qquad
0\le s\le \frac{r_0}{\mu_n},
\end{equation}
together with
$R_n(s)=O(s^2)$,
$R_n'(s)=O(s)$ for
$0\le s\le1$.
\medskip
\noindent
\textnormal{(ii) Outer expansion.}
The outer profile satisfies
\begin{equation}\label{eq:outer-main}
u_1(r)
=
r^{-2}-2r^{-4}+O(r^{-6}),
\qquad r\to\infty,
\end{equation}
while
\begin{equation*}
|\partial_r^j w_n(r)|
\lesssim
|\varepsilon_n|\,r^{-4-j},
\qquad
r\ge1,\quad 0\le j\le2.
\end{equation*}
\medskip
\noindent
\textnormal{(iii) Matching law.}
The matching scales can be numbered so that
\begin{equation}\label{eq:matching-phase-main}
-\omega_*\log\mu_n+\vartheta_0
=
n\pi+e_n,
\qquad
|e_n|<d_0<\frac{\pi}{4}.
\end{equation}
\medskip
\noindent
\textnormal{(iv) Two-scale estimates.}
By the definitions of $P_n$ and $U_n$ in \eqref{two-notations},
for every fixed $L>0$,
\begin{equation}\label{eq:inner-scale-limits-main}
(\mu_nP_n,\mu_n^2U_n,\mu_n^3U_n',\mu_n^4U_n'')
(\mu_ns)
\longrightarrow
(P,Q,Q',Q'')(s)
\end{equation}
uniformly for $0\le s\le L$. On the complementary scale,
\begin{equation}\label{eq:outer-scale-limits-main}
\lim_{R\to\infty}
\limsup_{n\to\infty}
\sup_{r\ge R\mu_n}
\Bigl(
|rP_n-2|
+|r^2U_n-2|
+|r^3U_n'+4|
+|r^4U_n''-12|
\Bigr)
=0.
\end{equation}
In particular, for some $C>0$ and $n_0\ge1$,
\[
\sup_{n\ge n_0}\sup_{r\ge0}
\bigl(|rP_n(r)|+|r^2U_n(r)|+|r^3U_n'(r)|
      +|r^4U_n''(r)|\bigr)\le C.
\]

\medskip
\noindent
\textnormal{(v) Zero enumeration.}
There exists an integer $N_0$, independent of $n$, such that,
for all sufficiently large $n$,
\begin{equation*}
 N_n:=\#\{r>0:\Lambda\Phi_n(r)=0\}=n+N_0.
\end{equation*}
\end{proposition}
\medskip
The rest of this section aims to prove Proposition \ref{thm:matched-family-main}. 
\subsection{The universal inner profile and its logarithmic tail}

Recall $\Ph=r^{-2}$ and define
\begin{equation*}
 H_\infty h=-h''-\frac4r h'-12\Ph h-2r(\Ph h)'.
\end{equation*}
The fundamental solutions and Wronskian of $H_\infty$, as in
\cite[Lemma~2.4]{NWZ2026}, are
\begin{equation*}
 \phi_1=r^{-5/2}\sin(\omega_*\log r),\qquad
 \phi_2=r^{-5/2}\cos(\omega_*\log r),\qquad
 \phi_1'\phi_2-\phi_2'\phi_1=\omega_* r^{-6}.
\end{equation*}
For $f$ such that the following integrals converge, set
\begin{equation*}
 \psi f(r)=\phi_1(r)\int_r^\infty \frac{s^6f(s)\phi_2(s)}{\omega_*}\,ds
 -\phi_2(r)\int_r^\infty \frac{s^6f(s)\phi_1(s)}{\omega_*}\,ds.
\end{equation*}
Then $H_\infty(\psi f)=f$. In particular,
for $p>9/2$ and $r\geq R>0$, define $F_p:=\sup_{s\geq R}s^p|f(s)|$, then
\begin{equation}\label{eq:Greenbd}
 |\psi f(r)|+|r(\psi f)'(r)|\leq C_pF_p r^{2-p}.
\end{equation}
Indeed, the tail integrals are bounded by
$C F_p\int_r^\infty s^{7/2-p}ds$. The restriction $p>9/2$
ensures convergence.

\begin{lemma}[Asymptotics of the universal inner profile]\label{lem:25}
For the solution $\bar{Q}$ of \eqref{app:profile}, there exist two
real constants $A,B$, with $A^2+B^2>0$, such that
\begin{equation}\label{eq:Qtail}
 \Q(r)=\frac1{r^2}+\frac{A\sin(\omega_*\log r)+B\cos(\omega_*\log r)}{r^{5/2}}
 +O_1(r^{-3})\qquad(r\to\infty).
\end{equation}
Moreover, 
\begin{align}
 \Q(r)&=\frac16-\frac{r^2}{60}+\frac{r^4}{630}+O(r^6)
 &&(r\to0),\label{eq:originQ}\\
 \Q'(r)&=-2r^{-3}+O(r^{-7/2}),\qquad
 \Q''(r)=6r^{-4}+O(r^{-9/2})
 &&(r\to\infty),\notag
\\
 \Q'(r)&<0,\qquad
 -\Q'(r)\geq c\frac{r}{(1+r)^4}
 &&(r>0).\label{eq:Qmonotone}
\end{align}
where $c>0$ is a universal constant.
\end{lemma}

\begin{proof}
\emph{Step 1. The origin-regular orbit.} Put $t=\log r$ and set
\begin{equation*}
 y(t)=r^2\Q(r)=\frac{rP(r)}2,
 \qquad \beta(t)=\frac{Q(r)}{2\Q(r)}=\frac{rQ(r)}{P(r)}.
\end{equation*}
Both are positive. We know from \eqref{eq:Q} and \eqref{eq:PbarQ} that
\begin{equation}\label{eq:PU}
 P'+\frac2rP=Q,\qquad Q'=-PQ,
 \qquad P(0)=0,\quad Q(0)=1.
\end{equation}
Directly from \eqref{eq:PU}, with dots denoting
$d/dt$,
\begin{equation}\label{eq:system}
 \dot y=y(\beta-1),\qquad
 \dot\beta=\beta(3-\beta-2y).
\end{equation}
Define
\begin{equation*}
 E(y,\beta)=2(y-1-\log y)+(\beta-1-\log\beta).
\end{equation*}
Substitution into \eqref{eq:system} gives exactly
\begin{equation*}
 \dot E=-(\beta-1)^2.
\end{equation*}
For any fixed $t_0\in\mathbb R$, the forward orbit is
contained in the sublevel set
\[
\left\{(y,\beta)\in(0,\infty)^2:
E(y,\beta)\le E(y(t_0),\beta(t_0))\right\},
\]
which is compact in $(0,\infty)^2$. Hence the orbit is
forward global and precompact. By LaSalle's invariance
principle, its $\omega$-limit set is contained in the
largest invariant subset of
\[
\{\dot E=0\}=\{\beta=1\}.
\]
Since $\dot\beta=2(1-y)$ on $\{\beta=1\}$, this
largest invariant subset consists only of $(1,1)$.
Consequently,
\begin{equation*}
 (y(t),\beta(t))\longrightarrow(1,1)
 \qquad (t\to\infty),
\end{equation*}
and hence
\begin{equation}\label{eq:conv}
\Q(r)=r^{-2}(1+o(1)),\qquad
\D\Q=-2r^{-2}+o(r^{-2})
\qquad (r\to\infty).
\end{equation}

\emph{Step 2. Quantitative approach to the singular state.} Set
$X=(y-1,\beta-1)^{\mathsf T}$. Equations \eqref{eq:system} yield
\begin{equation*}
 \dot X=MX+N(X),\qquad
 M=\begin{pmatrix}0&1\\-2&-1\end{pmatrix},
 \quad |N(X)|\leq C|X|^2.
\end{equation*}
The eigenvalues of $M$ are $-1/2\pm i\omega_*$, and
$\|e^{Mt}\|\leq C_0e^{-t/2}$ for $t\geq0$. Choose
$0<\delta<1/4$. By \eqref{eq:conv} choose $T$ so large that
$C|X(t)|\leq\delta/(2C_0)$ for $t\geq T$. Variation of
constants and Gronwall then give
\begin{equation}\label{eq:weak}
 |X(t)|\leq C_T e^{-(1/2-\delta)t}\qquad(t\geq T).
\end{equation}
Set \[v=\Q-\Phi_*.\] Since $d/dt=\D$ and $\Q=r^{-2}y$,
the identity $r\Q'=r^{-2}(\dot y-2y)$ and
\eqref{eq:weak} imply
\begin{equation}\label{eq:weakv}
 |v(r)|+|rv'(r)|\leq C r^{-5/2+\delta}.
\end{equation}

\emph{Step 3. Apply the exterior Green operator in \cite{NWZ2026} to the actual
profile.} Subtract the equation for $\Ph$ from
\eqref{app:profile}. Exactly as in \cite{NWZ2026},
\begin{equation}\label{eq:nlf}
 H_\infty v=F(v),\qquad F(v)=6v^2+r(v^2)'.
\end{equation}
By \eqref{eq:weakv}, $F(v)=O(r^{-5+2\delta})$.
Since $5-2\delta>9/2$, \eqref{eq:Greenbd} applies, giving
\begin{equation}\label{eq:wweak}
 w:=\psi F(v)=O_1(r^{-3+2\delta})=o_1(r^{-5/2}).
\end{equation}
The function $v-w$ solves $H_\infty(v-w)=0$ on a tail,
hence for some constants $A,B$,
\begin{equation}\label{eq:ab}
 v=A\phi_1+B\phi_2+w.
\end{equation}
In particular, this identifies the asymptotics of the
origin-regular solution and gives $v=O_1(r^{-5/2})$. Returning to
\eqref{eq:nlf}, now $F(v)=O(r^{-5})$ and \eqref{eq:Greenbd}
improves $w=O_1(r^{-3})$. The equation
\eqref{eq:nlf} and $H_\infty v=F(v)$ also yield
$r^2v''=O(r^{-5/2})$ and, after subtracting the
explicit homogeneous terms in \eqref{eq:ab},
$r^2w''=O(r^{-3})$. Thus the remainder is $O_2(r^{-3})$.

\emph{Step 4. Nonvanishing of the oscillatory coefficient.} Suppose $A=B=0$.
Then $v=\psi F(v)$ and, by \eqref{eq:wweak},
$K_R:=\sup_{r\geq R}r^{5/2}(|v(r)|+|rv'(r)|)$ is finite
and even tends to zero as $R\to\infty$. On $r\geq R$,
$|F(v)(r)|\leq C K_R^2r^{-5}$. Estimate
\eqref{eq:Greenbd}, applied at each $r\geq R$, gives
\begin{equation*}
 K_R\leq C R^{-1/2}K_R^2.
\end{equation*}
Since $K_R$ is bounded as $R\to\infty$, this forces $K_R=0$
for sufficiently large $R$. Thus $\Q=\Ph$ on a tail; uniqueness
for the ODE extends this equality to all $r>0$, contradicting
$\Q(0)=1/6$. Consequently $(A,B)\ne(0,0)$.

\emph{Step 5. Derivative asymptotics and monotonicity.}
Write $\Q=q_0+q_2r^2+q_4r^4+O(r^6)$ at the regular
origin. Substituting into \eqref{app:profile} gives
\[q_0=1/6, \ 10q_2+6q_0^2=0, \  
28q_4+16q_0q_2=0.\] Hence $q_2=-1/60$ and $q_4=1/630$, proving
\eqref{eq:originQ}.
The derivative estimate from \eqref{eq:Qtail} gives
$\Q'=-2r^{-3}+O(r^{-7/2})$. Substitution into
\eqref{app:profile}, using $\Q=r^{-2}+O_1(r^{-5/2})$,
then gives $\Q''=6r^{-4}+O(r^{-9/2})$.

We next prove \eqref{eq:Qmonotone}. We know from \eqref{eq:PU} that
$Q'=-PQ<0$. The integral definition \eqref{eq:PbarQ}
therefore implies
\[
 6\Q(r)=\frac3{r^3}\int_0^r Q(s)s^2\,ds>Q(r).
\]
Since $Q=6\Q+2r\Q'$, this gives $\Q'(r)<0$.
Near zero, we have $-\Q'(r)=r/30+O(r^3)$, and near infinity, there holds
$-\Q'(r)=2r^{-3}+O(r^{-7/2})$.
On the remaining compact annulus the continuous positive
function $-(1+r)^4\Q'(r)/r$ has a positive minimum.
This proves \eqref{eq:Qmonotone}.
\end{proof}

\subsection{Homogeneous linear theory and a weighted right inverse}

For $R\ge1$, let $F_R=C([0,R])$ and let $Y_R,Z_R$
consist of functions in $C([0,R])\cap C^1((0,R])$ for which
$rh'$ extends continuously to $r=0$. Equip these spaces
with the following norms:
\begin{equation*}
 \|h\|_{Y_R}:=\sup_{0\le r\le R}(1+r)^{-1/2}
 \bigl(|h(r)|+r|h'(r)|\bigr),
\end{equation*}
and
\begin{equation}\label{eq:fnorm}
 \begin{aligned}
 \|f\|_{F_R}
 :=\sup_{0\le r\le R}(1+r)^{5/2}|f(r)|,\ 
 \|h\|_{Z_R}
:=\sup_{0\le r\le R}(1+r)^{1/2}
 \bigl(|h(r)|+r|h'(r)|\bigr).
 \end{aligned}
\end{equation}

\begin{lemma}[Homogeneous modes and a weighted right inverse]\label{lem:26}
Let 
\begin{equation}\label{eq:H}
 Hh=-h''-\frac4r h'-12\Q h-2r(\Q h)'.
\end{equation}
Then $H(\Lambda\Q)=0$, and there exists a linearly independent
solution $\rho$ of $H\rho=0$ such that
\begin{equation}\label{eq:mdef}
 (\Lambda\Q)'\rho-(\Lambda\Q)\rho'=M_*^{-1},
 \qquad
 M_*(r)=r^4\exp\!\left(\int_0^r2s\Q(s)\,ds\right).
\end{equation}
For some $c_1c_2\ne0$ and $\theta\in\mathbb R$,
\begin{equation}\label{eq:urho}
 \begin{aligned}
 \Lambda\Q(r)&=c_1r^{-5/2}\sin(\omega_*\log r+\theta)
                  +O_1(r^{-3}),\\
 \rho(r)&=c_2r^{-5/2}\cos(\omega_*\log r+\theta)
                  +O_1(r^{-3})
 \end{aligned}
 \qquad (r\to\infty),
\end{equation}
while $|\rho(r)|+r|\rho'(r)|\le Cr^{-3}$ for $0<r\le1$.

For $R\ge1$ and $f\in C([0,R])$, define
\begin{equation}\label{eq:S}
 Sf(r)=\rho(r)\int_0^r f(s)(\Lambda\Q)(s)M_*(s)\,ds
  -(\Lambda\Q)(r)\int_0^r f(s)\rho(s)M_*(s)\,ds.
\end{equation}
Then $H(Sf)=f$ on $(0,R]$. In addition,
the right inverse satisfies
\begin{align}
 \|Sf\|_{Y_R}
 &\le C\sup_{0\le r\le R}(1+r)^2|f(r)|,
 \label{eq:Sbd}\\
 \|Sf\|_{Z_R}
 &\le C\|f\|_{F_R},\qquad
 |Sf(r)|+r|(Sf)'(r)|\le C\|f\|_{F_R}r^2
 \quad(0\le r\le1).
 \label{eq:strongS}
\end{align}
The constant $C$ is independent of $R$.
\end{lemma}

\begin{proof}
\emph{Step 1. Homogeneous solutions.}
Scaling invariance of \eqref{app:profile} gives
$H(\Lambda\Q)=0$. Writing
$g=A\sin(\omega_*\log r)+B\cos(\omega_*\log r)$ with $A,B$ as in \eqref{eq:Qtail},
Lemma~\ref{lem:25} yields
\[
 \Lambda\Q
 =r^{-5/2}\bigl(-\tfrac12 g+rg'\bigr)+O(r^{-3}).
\]
Its leading coefficient vector is nonzero because $A^2+B^2>0$.

For any homogeneous solution $h$, put
$h=r^{-5/2}x(t)$, $t=\log r$. Since
$\Q-r^{-2}=O_1(r^{-5/2})$, its equation takes the form
\[
 V=\binom{x}{\dot x},\qquad
J=\begin{pmatrix}0&1\\-\omega_*^2&0\end{pmatrix},\quad
  \dot V=(J+B(t))V,
 \quad \|B(t)\|\le Ce^{-t/2}.
\]
As $e^{\pm Jt}$ are bounded,
$Z=e^{-Jt}V$ satisfies $\dot Z=\widetilde B(t)Z$, where
$\widetilde B(t)=e^{-Jt}B(t)e^{Jt}$ and
$\|\widetilde B(t)\|\le Ce^{-t/2}$.
Gronwall's inequality implies
$Z(t)=Z_\infty+O(e^{-t/2})$. Moreover, $Z_\infty\ne0$
unless $h\equiv0$. Otherwise,
\[
 Z(t)=-\int_t^\infty\widetilde B(s)Z(s)\,ds
\]
and $\int_T^\infty\|\widetilde B(s)\|\,ds<1$ for large
$T$ force $Z\equiv0$. Thus every nontrivial homogeneous
solution has a nonzero sine--cosine leading term with
$O_1(r^{-3})$ remainder.

Since $(\Lambda\Q)(0)=1/3$, for some $R_0>0$ the
function $\Lambda\Q$ is nonzero on $(0,R_0]$. Define
\begin{equation*}
 \rho_0(r)=(\Lambda\Q)(r)
 \int_r^{R_0}\frac{M_*(s)^{-1}}{(\Lambda\Q(s))^2}\,ds.
\end{equation*}
Reduction of order gives $H\rho_0=0$ and
$(\Lambda\Q)'\rho_0-(\Lambda\Q)\rho_0'=M_*^{-1}$.
Since $M_*(r)\asymp r^4$ near zero,
$\rho_0=O(r^{-3})$ and $\rho_0'=O(r^{-4})$.
Extend $\rho_0$ to $(0,\infty)$ by the ODE.
At infinity,
\[
 M_*^{-1}(r)=c_*r^{-6}(1+O(r^{-1/2})),\qquad c_*>0.
\]
The Wronskian therefore forces a nonzero cosine
coefficient of $\rho_0$ relative to the phase of
$\Lambda\Q$. Subtracting a suitable multiple of
$\Lambda\Q$ gives $\rho$ and \eqref{eq:urho},
without altering \eqref{eq:mdef} or the origin bounds.

\emph{Step 2. Weighted right inverse.}
Write
$I_1(r)=\int_0^r f(\Lambda\Q)M_*$ and
$I_2(r)=\int_0^r f\rho M_*$. Both integrals converge
at zero. Differentiating \eqref{eq:S} gives
\begin{equation*}
 (Sf)'=\rho'I_1-(\Lambda\Q)'I_2.
\end{equation*}
Using $H(\Lambda\Q)=H\rho=0$ and \eqref{eq:mdef}
in the second derivative yields $H(Sf)=f$.

To prove both weighted bounds at once, let
$2\le q\le5/2$ and
$F_q=\sup_{0\le s\le R}(1+s)^q|f(s)|$.
The origin estimates imply
\[
 |I_1(r)|\le CF_qr^5,\qquad
 |I_2(r)|\le CF_qr^2\qquad(0<r\le1),
\]
hence $|Sf|+r|(Sf)'|\le CF_qr^2$ on this interval.
For $1\le r\le R$,
\[
 |\Lambda\Q|+|r(\Lambda\Q)'|+|\rho|+|r\rho'|
 \le Cr^{-5/2},\qquad M_*(r)\asymp r^6.
\]
Consequently,
\[
 |I_1(r)|+|I_2(r)|
 \le CF_q\left(1+\int_1^r s^{7/2-q}\,ds\right)
 \le CF_qr^{9/2-q},
\]
and therefore
\[
 |Sf(r)|+r|(Sf)'(r)|\le CF_qr^{2-q},
 \qquad 1\le r\le R.
\]
The cases $q=2$ and $q=5/2$ give
\eqref{eq:Sbd} and \eqref{eq:strongS}, respectively,
with constants independent of $R$.
\end{proof}

\subsection{The inner construction}

\begin{proposition}[Construction of the inner profile]\label{prop:27}
There exist $r_*>0$ and $C>0$ with the following property.
For any $0<r_0\leq r_*$ and $0<\lambda\leq r_0$, put
$S=r_0/\lambda\geq1$. The  equation \eqref{eq:profile-main}
has an origin-regular radial solution on $[0,r_0]$ of the form
\begin{equation}\label{eq:Phiin}
 \Phi_{\mathrm{in},\lambda}(r)
 =\lambda^{-2}\bigl[\Q+\lambda^2R_\lambda\bigr](r/\lambda).
\end{equation}
In addition,
\begin{equation}\label{eq:Rbounds}
 \sup_{0\leq s\leq S}(1+s)^{1/2}
       (|R_\lambda(s)|+|sR_\lambda'(s)|)\leq C,
 \qquad
 |R_\lambda(s)|+|sR_\lambda'(s)|\leq Cs^2
 \quad(0\leq s\leq1).
\end{equation}
In particular, $R_\lambda(0)=R_\lambda'(0)=0$ and
$\Phi_{\mathrm{in},\lambda}(0)=1/(6\lambda^2)$.
\end{proposition}

\begin{proof}
\emph{Step 1. Exact rescaling.}
For $s=r/\lambda$ set $q(s)=\lambda^2\Phi(\lambda s)$.
Multiplying \eqref{eq:profile-main} by $\lambda^4$ yields
\begin{equation}\label{eq:qeq}
 q''+\frac4s q'+6q^2+s(q^2)'
    -\frac{\lambda^2}{2}(2q+sq')=0.
\end{equation}
Let $q=\Q+\lambda^2R$. Since $\Q$ solves
\eqref{app:profile}, the terms linear in $R$ form $-HR$,
where $H$ is given in \eqref{eq:H}. After division by
$\lambda^2$, \eqref{eq:qeq} becomes exactly
\begin{equation}\label{eq:Req}
 HR=-\frac12\Lambda\Q-\frac{\lambda^2}{2}\Lambda R
       +\lambda^2\{6R^2+s(R^2)'\}.
\end{equation}

\emph{Step 2. Fixed-point construction.}
Use $F_S,Z_S$ from \eqref{eq:fnorm}. More precisely, take
$F_S=C([0,S])$ with its indicated weighted norm and take
$Z_S$ to be the Banach space of functions continuous on
$[0,S]$ and $C^1$ on $(0,S]$ for which $sR'(s)$ extends
continuously to zero, equipped with the indicated norm.
Write $\mathcal Gf:=Sf$ for the right inverse in
\eqref{eq:S}, to distinguish it from the endpoint $S=r_0/\lambda$.
Define
\begin{equation*}
 \mathcal T_\lambda(R)=\mathcal G\left[-\frac12\Lambda\Q
  -\frac{\lambda^2}{2}\Lambda R
  +\lambda^2\{6R^2+s(R^2)'\}\right].
\end{equation*}
The source $\Lambda\Q=2\Q+s\Q'$ is bounded near zero.
By Lemma~\ref{lem:25}, the leading $s^{-2}$ terms cancel
and $\Lambda\Q=O(s^{-5/2})$ at infinity. Thus
\begin{equation}\label{eq:source}
 \left\|-\frac12\Lambda\Q\right\|_{F_S}\leq C_0,
\end{equation}
uniformly for $S\geq1$.

For $R\in Z_S$,
$|\Lambda R|\leq3(1+s)^{-1/2}\|R\|_{Z_S}$;
hence
\begin{equation*}
 \left\|\frac{\lambda^2}{2}\Lambda R\right\|_{F_S}
 \leq C\lambda^2(1+S)^2\|R\|_{Z_S}
 =C(\lambda+r_0)^2\|R\|_{Z_S}
 \leq Cr_0^2\|R\|_{Z_S}.
\end{equation*}
Writing $\mathcal N(R)=6R^2+s(R^2)'=6R^2+2R(sR')$,
we obtain
\begin{align}\label{eq:nonlinest}
 \|\lambda^2\mathcal N(R)\|_{F_S}
 &\leq C\lambda^2(1+S)^{3/2}\|R\|_{Z_S}^2
 \leq Cr_0^2\|R\|_{Z_S}^2.
\end{align}
Here the last inequality follows explicitly from
$\lambda^2(1+S)^{3/2}
 =\lambda^{1/2}(\lambda+r_0)^{3/2}
 \leq2^{3/2}r_0^2$.

Let $\|R_j\|_{Z_S}\leq M$ and $D=R_1-R_2$.
Since
$R_1^2-R_2^2=(R_1+R_2)D$,
\[
 s(R_1^2-R_2^2)'=s(R_1+R_2)'D+(R_1+R_2)sD',
\]
the preceding calculation gives
\begin{equation}\label{eq:Lipschitz}
 \|\lambda^2[\mathcal N(R_1)-\mathcal N(R_2)]\|_{F_S}
 \leq Cr_0^2M\|D\|_{Z_S}.
\end{equation}
Combining \eqref{eq:strongS} and
\eqref{eq:source}--\eqref{eq:Lipschitz},
\begin{align*}
 \|\mathcal T_\lambda(R)\|_{Z_S}
 &\leq CC_0+Cr_0^2(\|R\|_{Z_S}+\|R\|_{Z_S}^2),\\
 \|\mathcal T_\lambda(R_1)-\mathcal T_\lambda(R_2)\|_{Z_S}
 &\leq Cr_0^2(1+M)\|R_1-R_2\|_{Z_S}.
\end{align*}
Choose $M\geq2CC_0+1$ first and then choose $r_*>0$
so small that $Cr_*^2(1+M)\leq1/2$ and
$Cr_*^2(M+M^2)\leq M-CC_0$. These choices do not depend
on $S$ or $\lambda$. The Banach fixed-point theorem gives
a unique $R_\lambda$ in the closed radius-$M$ ball of
$Z_S$, with $\|R_\lambda\|_{Z_S}\leq M$.

\emph{Step 3. Regularity at the origin and verification of the equation.}
The forcing on the right of \eqref{eq:Req} is uniformly
bounded in $F_S$, by \eqref{eq:source}--\eqref{eq:nonlinest}.
The near-zero part of \eqref{eq:strongS} therefore gives
\begin{equation*}
 |R_\lambda(s)|+|sR_\lambda'(s)|\leq Cs^2
 \quad(0\leq s\leq1),
\end{equation*}
with $C$ independent of $\lambda$ and $S$.
In particular $R_\lambda'(s)=O(s)$ and
$R_\lambda(0)=R_\lambda'(0)=0$.
The integral representation gives $C^1$ regularity at $0$.
For rigor at the regular singular endpoint, write $HR_\lambda=f_\lambda$
with $f_\lambda$ denoting the right-hand side of \eqref{eq:Req}.
Using the integrating factor $M_*$ in \eqref{eq:mdef},
\[
 R_\lambda'(s)=-\frac1{M_*(s)}\int_0^s M_*(t)
 \bigl[f_\lambda(t)+(12\Q(t)+2t\Q'(t))R_\lambda(t)\bigr]\,dt.
\]
Since $f_\lambda(0)=-\Lambda\Q(0)/2=-1/6$ and
$M_*(s)=s^4(1+O(s^2))$, we get
$R_\lambda'(s)=s/30+o(s)$. Substitution in
\eqref{eq:Req} gives $R_\lambda''(s)\to1/30$.
In particular,
\begin{equation*}
 R_\lambda(s)=\frac{s^2}{60}+o(s^2)
 \qquad(s\to0).
\end{equation*}
The integral equation and smooth even coefficients then yield higher
origin regularity by differentiation in $s^2$.

Finally $R_\lambda=\mathcal T_\lambda(R_\lambda)$
and $H\mathcal Gf=f$ imply \eqref{eq:Req}.
Reversing the exact rescaling proves that
\eqref{eq:Phiin} solves \eqref{eq:profile-main} on $[0,r_0]$.
Since $\Q(0)=1/6$ and $R_\lambda(0)=0$,
$\Phi_{\mathrm{in},\lambda}(0)=1/(6\lambda^2)$.
This completes the proof.
\end{proof}

\begin{corollary}[Second-derivative bounds and monotonicity]
\label{cor:secondmono}
Under the hypotheses of Proposition~\ref{prop:27}, after decreasing
$r_*>0$ if necessary, one has, uniformly for
$0<\lambda\le r_0\le r_*$ and $S=r_0/\lambda$,
\begin{equation}\label{eq:Rsecond}
 \sup_{0\le s\le S}(1+s)^{5/2}|R_\lambda''(s)|\le C.
\end{equation}
Moreover,
\begin{equation}\label{eq:innerdecreasing}
 (\Q+\lambda^2R_\lambda)'(s)<0,
 \qquad 0<s\le S.
\end{equation}
\end{corollary}

\begin{proof}
By \eqref{eq:Req},
\begin{equation}\label{eq:solvedR}
\begin{aligned}
 R_\lambda''={}-\left(\frac4s+2s\Q\right)R_\lambda'
 -(12\Q+2s\Q')R_\lambda
 +\frac12\Lambda\Q+\frac{\lambda^2}{2}\Lambda R_\lambda-\lambda^2\bigl[6R_\lambda^2+s(R_\lambda^2)'\bigr].
\end{aligned}
\end{equation}
For $0<s\le1$, \eqref{eq:Rbounds} gives
$R_\lambda=O(s^2)$ and $R_\lambda'=O(s)$ uniformly,
so \eqref{eq:solvedR} yields $|R_\lambda''|\le C$.
For $1\le s\le S$, using
\eqref{eq:Rbounds}, Lemma~\ref{lem:25}, and
$\lambda^2s^2\le r_0^2$, $\lambda^2s^{3/2}\le r_0^2$,
we obtain $|R_\lambda''(s)|\le Cs^{-5/2}$.
Together these estimates prove \eqref{eq:Rsecond}.

Furthermore, \eqref{eq:Qmonotone} and \eqref{eq:Rbounds}
imply
\[
 \sup_{0<s\le S}
 \frac{\lambda^2|R_\lambda'(s)|}{-\Q'(s)}
 \le C\lambda^2\max\{1,S^{3/2}\}
 \le Cr_0^2.
\]
Taking $r_*$ sufficiently small proves
\eqref{eq:innerdecreasing}.
\end{proof}

Define
\begin{equation*}
 P_\lambda(r)=2r\Phi_{\mathrm{in},\lambda}(r),\qquad
 U_\lambda(r)=6\Phi_{\mathrm{in},\lambda}(r)
                   +2r\Phi_{\mathrm{in},\lambda}'(r).
\end{equation*}
and the rescaled functions
\[
 p_\lambda(s)=\lambda P_\lambda(\lambda s),\qquad
 h_\lambda(s)=\lambda^2U_\lambda(\lambda s).
\]
\begin{corollary}[Inner-scale convergence]
\label{cor:innerlimits}
Let
\begin{equation}\label{eq:R0}
 R_0:=\mathcal G\!\left(-\frac12\Lambda\Q\right),
 \qquad HR_0=-\frac12\Lambda\Q,
\end{equation}
where $\mathcal G$ is the right inverse in \eqref{eq:S}.
Then $R_0$ is regular at the origin, $R_0(0)=0$, and
\begin{equation}\label{eq:R0origin}
 R_0(s)=\frac{s^2}{60}+O(s^4),\qquad
 6R_0(s)+2sR_0'(s)=\frac{s^2}{6}+O(s^4)
 \qquad (s\to0).
\end{equation}
For every fixed $L>0$ and $0<\lambda\le r_0\le r_*$
with $L\le S=r_0/\lambda$,
\begin{equation}\label{eq:C2convergence}
 \|R_\lambda-R_0\|_{C^2([0,L])}\le C_L\lambda^2.
\end{equation}
Moreover, 
\begin{equation}\label{eq:scaledlimits}
 \|p_\lambda-P\|_{C^1([0,L])}
 +\|h_\lambda-Q\|_{C^2([0,L])}
 \le C_L\lambda^2.
\end{equation}
\end{corollary}

\begin{proof}
By Lemma~\ref{lem:26} and
$\Lambda\Q=O((1+s)^{-5/2})$, formula \eqref{eq:R0}
defines a compatible origin-normalized solution on $[0,\infty)$,
with $R_0=O(s^2)$ and $R_0'=O(s)$ near zero.
Since $\Q$ is smooth and even at the origin,
$-\frac12\Lambda\Q=-\frac16+O(s^2)$ and
$M_*(s)=s^4(1+O(s^2))$. The integrating-factor form
of $HR_0=-\frac12\Lambda\Q$ gives
$R_0'(s)=s/30+O(s^3)$, proving \eqref{eq:R0origin}.

Set $D_\lambda=R_\lambda-R_0$. Subtracting the
{fixed-point representations} of $R_\lambda$ and
$R_0$ gives
\begin{equation*}
 D_\lambda=\mathcal Gf_\lambda,\qquad
 f_\lambda=\lambda^2\left[-\frac12\Lambda R_\lambda
 +6R_\lambda^2+s(R_\lambda^2)'\right],
 \qquad HD_\lambda=f_\lambda.
\end{equation*}
Put $L_1=\max\{1,L\}$. For fixed $L$ with $L\le S$,
the bound \eqref{eq:Rbounds} and the right-inverse estimate
\eqref{eq:strongS} %
applied on $[0,L_1]$, give
$\|f_\lambda\|_{F_{L_1}}+\|D_\lambda\|_{Z_{L_1}}
\le C_L\lambda^2$ and, upon restriction to $[0,L]$,

\begin{equation}\label{eq:Dnear}
 |D_\lambda(s)|+s|D_\lambda'(s)|
 \le C_L\lambda^2s^2,
 \qquad 0<s\le\min\{1,L\}.
\end{equation}
Since
\begin{equation*}
 D_\lambda''=-\left(\frac4s+2s\Q\right)D_\lambda'
 -(12\Q+2s\Q')D_\lambda-f_\lambda,
\end{equation*}
\eqref{eq:Dnear} controls $D_\lambda'/s$ near zero;
away from zero the $Z_L$ bound applies. The origin
regularity of both solutions then yields
$\|D_\lambda\|_{C^2([0,L])}\le C_L\lambda^2$.

Finally, \eqref{eq:Phiin} gives
\begin{equation*}
 p_\lambda=P+2\lambda^2sR_\lambda,
 \qquad
 h_\lambda=Q+\lambda^2(6R_\lambda+2sR_\lambda').
\end{equation*}
The local $C^2$ bounds on $R_\lambda$ imply
\[
 \|p_\lambda-P\|_{C^1([0,L])}
 +\|h_\lambda-Q\|_{C^1([0,L])}
 \le C_L\lambda^2.
\]
To estimate the second derivative of $h_\lambda$, use
the exact scaled first-order system
\begin{equation*}
 p_\lambda'=h_\lambda-\frac2s p_\lambda,
 \qquad
 h_\lambda'=-h_\lambda p_\lambda
   +\frac{\lambda^2}{2}(s h_\lambda-p_\lambda).
\end{equation*}
Differentiating the second identity and using
$Q''=-Q'P-QP'$ yields
\[
 h_\lambda''-Q''
 =-(h_\lambda'-Q')p_\lambda
   -Q'(p_\lambda-P)
   -(h_\lambda-Q)p_\lambda'
   -Q(p_\lambda'-P')
   +\frac{\lambda^2}{2}
       (h_\lambda+s h_\lambda'-p_\lambda').
\]
All terms on the right are $O_L(\lambda^2)$
uniformly on $[0,L]$, proving \eqref{eq:scaledlimits}.
\end{proof}

\begin{remark*}[The unscaled inner correction and profile indexing]
Corollary~\ref{cor:innerlimits} distinguishes scaled convergence from
absolute convergence of the correction. For every fixed $s>0$,
\[
 \Phi_{\mathrm{in},\mu}(\mu s)-\mu^{-2}\Q(s)=R_\mu(s)
 \longrightarrow R_0(s),\qquad R_0(s)=s^2/60+O(s^4)\ \text{near }0.
\]
Thus the correction does not tend to zero uniformly in an interval
containing the inner scale. The absolute uniform convergence assertion
in \cite[Proposition~2.1(1), equation~(2.6)]{NWZ2026} is not used here.
The valid scaled statement is
$\mu^2\Phi_{\mathrm{in},\mu}(\mu\,\cdot)\to\Q$ in
$C^2([0,L])$ for each fixed $L$, as proved above.
We use the inner--outer construction of that reference, with the
origin normalization and the matching roots specified in
Proposition~\ref{prop:21}. Our index enumerates those matching roots;
it is not identified with an absolute shooting index from another
normalization. Any fixed indexing offset is retained explicitly as
$N_0$ in \eqref{eq:Nn-count-main}.
\end{remark*}

\subsection{Interior--exterior matching}

\begin{proposition}[Interior--exterior matching]\label{prop:21}
There exist $r_0>0$, sequences $\mu_n\to0$ and
$\varepsilon_n\to0$, and distinct smooth radial solutions
$\Phi_n$ of \eqref{eq:profile-main} on $[0,\infty)$ such that
\begin{equation}\label{eq:globalprofile}
 \Phi_n(r)=
 \begin{cases}
  \mu_n^{-2}\Q(r/\mu_n)+R_{\mu_n}(r/\mu_n),&0\le r\le r_0,\\
  r^{-2}+\varepsilon_n[u_1(r)+w_{\varepsilon_n}(r)],&r\ge r_0.
 \end{cases}
\end{equation}
The inner corrections satisfy, uniformly in $n$,
\begin{equation}\label{eq:globalinner}
 \sup_{0\le s\le r_0/\mu_n}(1+s)^{1/2}
 \bigl(|R_{\mu_n}(s)|+|sR_{\mu_n}'(s)|
       +|s^2R_{\mu_n}''(s)|\bigr)\le C,
\end{equation}
with $|R_{\mu_n}(s)|\le Cs^2$ for $0\le s\le1$.
The inner-scale convergence \eqref{eq:scaledlimits} holds with
$\lambda=\mu_n$ on every fixed compact interval.
For some $\gamma_0\in\mathbb R$ and $0<\delta<\pi/4$,
the matching scales can be indexed so that
\begin{equation}\label{eq:phaseinterval}
 n\pi-\delta< -\omega_*\log\mu_n+\gamma_0<n\pi+\delta
\end{equation}
for all sufficiently large $n$. In the zero-counting statement
\eqref{eq:matching-phase-main}, we use the same phase constant
$\vartheta_0=\gamma_0$ and set $d_0=\delta$, after the same
consecutive matching-interval numbering. In particular, only
$|e_n|<d_0$ is required; no convergence $e_n\to0$ is used.
Moreover, for every fixed sufficiently small $s_*>0$,
\begin{equation*}
 \Phi_n(\mu_n s_*)-\mu_n^{-2}\Q(s_*)
 \longrightarrow R_0(s_*)>0.
\end{equation*}
\end{proposition}

\begin{proof}
\emph{Step 1. Exterior and interior families.}
We use the exterior construction of \cite[Section~2.2]{NWZ2026}.
Its contraction argument, with the Green operator normalized by
$L\tau f=f$, extends to signed sufficiently small $\varepsilon$:
the same ball is invariant and the Lipschitz constant depends on
$|\varepsilon|$ alone. Differentiating its polynomial fixed-point
map and inverting $I-\varepsilon\tau DG(w_\varepsilon)$ gives
parameter differentiability (in fact twice differentiability) on
this signed interval.
It yields
\begin{equation}\label{eq:outerrep}
 \Phi_{\mathrm{out},\varepsilon}=\Phi_*+
 \varepsilon(u_1+w_\varepsilon),\qquad
 \|w_\varepsilon\|_{X_{r_0}}\le C|\varepsilon|r_0^{-1/2},
 \quad \|\partial_\varepsilon w_\varepsilon\|_{X_{r_0}}
 \le Cr_0^{-1/2},
\end{equation}
where $\Phi_*=r^{-2}$ and
\begin{equation}\label{eq:Xouter}
 \|w\|_{X_{r_0}}=
 \sup_{r_0\le r\le1}(r^{5/2}|w|+r^{7/2}|w'|)
 +\sup_{r\ge1}(r^4|w|+r^5|w'|).
\end{equation}
In particular,
\begin{equation}\label{eq:wmatch}
 |w_\varepsilon(r_0)|+r_0|w_\varepsilon'(r_0)|
 \le C|\varepsilon|r_0^{-3}.
\end{equation}
The outer homogeneous mode satisfies
\begin{align*}
 u_1(r)&=c_1r^{-5/2}\sin(\omega_*\log r+\theta_1)
          +O_1(r^{-1/2}) &&(r\to0),\\
 u_1(r)&=r^{-2}(1+O_1(r^{-2})) &&(r\to\infty).
\end{align*}
where $c_1\ne0$. Choose a sufficiently small $r_0$ along
sections where the leading sine stays away from zero. Then
\begin{equation}\label{eq:r0section}
 |u_1(r_0)|\ge cr_0^{-5/2},\qquad
 |u_1'(r_0)|\le Cr_0^{-7/2}.
\end{equation}
The interior family, from Proposition~\ref{prop:27}, is
$\Phi_{\mathrm{in},\mu}(r)=\mu^{-2}\Q(r/\mu)+R_\mu(r/\mu)$.

\emph{Step 2. Value matching and derivative mismatch.}
Set $S=r_0/\mu$, $A(s)=\Q(s)-s^{-2}$, and
\begin{equation*}
 I_\mu:=\Phi_{\mathrm{in},\mu}(r_0)-\Phi_*(r_0)
       =\mu^{-2}A(S)+R_\mu(S).
\end{equation*}
By Lemma~\ref{lem:25} and \eqref{eq:Rbounds},
\begin{equation*}
 |\mu^{-2}A(S)|\le C\mu^{1/2}r_0^{-5/2},\qquad
 |R_\mu(S)|\le C\mu^{1/2}r_0^{-1/2},\qquad
 |I_\mu|\le C\mu^{1/2}r_0^{-5/2}.
\end{equation*}
The equation
\begin{equation}\label{eq:valuematch}
 \varepsilon[u_1(r_0)+w_\varepsilon(r_0)]=I_\mu
\end{equation}
has, by \eqref{eq:outerrep} and \eqref{eq:r0section},
a unique small solution $\varepsilon=\varepsilon(\mu)$ for
all sufficiently small $\mu>0$. It is continuous in $\mu$ by
continuous dependence of the interior initial-value problem and
satisfies
\begin{equation}\label{eq:epssize}
 |\varepsilon(\mu)|\le C\mu^{1/2}.
\end{equation}
Define the derivative mismatch
$D(\mu)=\Phi_{\mathrm{out},\varepsilon(\mu)}'(r_0)
-\Phi_{\mathrm{in},\mu}'(r_0)$. Using \eqref{eq:valuematch}
gives the exact identity
\begin{equation}\label{eq:exactD}
 D(\mu)=I_\mu
 \frac{u_1'(r_0)+w_{\varepsilon(\mu)}'(r_0)}
      {u_1(r_0)+w_{\varepsilon(\mu)}(r_0)}
 -\mu^{-3}A'(S)-\mu^{-1}R_\mu'(S).
\end{equation}
Let $\kappa_0=u_1'(r_0)/u_1(r_0)$, so that
$|\kappa_0|\le Cr_0^{-1}$. From
\eqref{eq:wmatch}, \eqref{eq:r0section}, and
\eqref{eq:epssize},
\begin{equation}\label{eq:kapperror}
 \left|\frac{u_1'+w_\varepsilon'}{u_1+w_\varepsilon}(r_0)
       -\kappa_0\right|
 \le C|\varepsilon|r_0^{-3/2}
 \le C\mu^{1/2}r_0^{-3/2}.
\end{equation}
Furthermore, writing the tail of Lemma~\ref{lem:25} as
\begin{align}\label{eq:Aphase}
 A(s)&=c_Qs^{-5/2}\sin(\omega_*\log s+\theta_Q)+O(s^{-3}),\\
 A'(s)&=c_Qs^{-7/2}\bigl[-\tfrac52\sin(\omega_*\log s+\theta_Q)
       +\omega_*\cos(\omega_*\log s+\theta_Q)\bigr]+O(s^{-4}),\notag
\end{align}
where $c_Q\ne0$, and setting
$\vartheta=\omega_*\log S+\theta_Q$, we find that
the leading part of \eqref{eq:exactD} is
\begin{equation*}
 D_0(\mu)=c_Q\mu^{1/2}r_0^{-7/2}
 \bigl[(r_0\kappa_0+\tfrac52)\sin\vartheta
       -\omega_*\cos\vartheta\bigr].
\end{equation*}
Its oscillatory amplitude is at least $|c_Q|\omega_*>0$.
The errors in \eqref{eq:Aphase} and
\eqref{eq:kapperror} contribute $O(\mu r_0^{-4})$;
the remaining interior terms satisfy
\begin{equation*}
 |\kappa_0R_\mu(S)|+|\mu^{-1}R_\mu'(S)|
 \le C\mu^{1/2}r_0^{-3/2}.
\end{equation*}
Consequently, for some $a_0\ne0$ and $\gamma_0\in\mathbb R$,
\begin{equation}\label{eq:Dphase}
 D(\mu)=a_0\mu^{1/2}r_0^{-7/2}
 \bigl[\sin(-\omega_*\log\mu+\gamma_0)+E(\mu,r_0)\bigr],
 \qquad |E(\mu,r_0)|\le C\biggl(r_0^2+\sqrt{\frac\mu{r_0}}\biggr).
\end{equation}
Here the relative error constant can be chosen uniformly along
the selected small matching sections.

\emph{Step 3. Selection of the matching scales.}
Fix $0<\delta<\pi/4$ and choose $r_0$ sufficiently small that
$Cr_0^2<\frac14\sin\delta$. For large integers $k$, define
$\mu_{k,\pm}>0$ by
\[
 -\omega_*\log\mu_{k,\pm}+\gamma_0=k\pi\pm\delta.
\]
Then \eqref{eq:Dphase} shows that
$D(\mu_{k,+})D(\mu_{k,-})<0$ for all sufficiently large
$k$. Continuity provides a root $\mu_k$ between them.
After reindexing, $\mu_n\to0$ and the intervals are disjoint;
put $\varepsilon_n=\varepsilon(\mu_n)$.
The corresponding interior and exterior solutions agree in
value and first derivative at $r_0$, hence glue smoothly by
uniqueness for the profile ODE. Their central values
$\Phi_n(0)=1/(6\mu_n^2)$ are distinct.
This proves \eqref{eq:globalprofile} and \eqref{eq:phaseinterval}.
The bounds \eqref{eq:globalinner} and the inner-scale limits
follow from Proposition~\ref{prop:27} and
Corollaries~\ref{cor:secondmono}--\ref{cor:innerlimits}.
Finally, for fixed sufficiently small $s_*>0$,
\[
 \Phi_n(\mu_ns_*)-\mu_n^{-2}\Q(s_*)
 =R_{\mu_n}(s_*)\longrightarrow R_0(s_*)>0
\]
by \eqref{eq:R0origin} and \eqref{eq:C2convergence}.
\end{proof}

\begin{lemma}[Uniform exterior derivative control]\label{lem:outer-derivatives}
Let $\Phi_{\mathrm{out},\varepsilon}$ be the exterior family in
\eqref{eq:outerrep} for sufficiently small signed $\varepsilon$, and set
$y_\varepsilon(r)=r^2\Phi_{\mathrm{out},\varepsilon}(r)$ and
$a=1+\varepsilon$. Then, uniformly for $r\ge1$,
\begin{equation}\label{eq:outer-y-derivatives}
 |y_\varepsilon(r)-a|+r|y_\varepsilon'(r)|
 +r^2|y_\varepsilon''(r)|+r^3|y_\varepsilon'''(r)|
 \le C|\varepsilon|r^{-2}.
\end{equation}
In particular, the normalized exterior correction obeys
\begin{equation}\label{eq:outer-wsecond}
 |\partial_r^j w_\varepsilon(r)|\le
 C|\varepsilon|r^{-4-j},\qquad r\ge1,\quad j=0,1,2.
\end{equation}
Moreover, as $r\to\infty$,
\[
 u_1(r)=r^{-2}-2r^{-4}+O(r^{-6}),\qquad
 u_1'(r)=-2r^{-3}+8r^{-5}+O(r^{-7}),
\]
and, for each fixed sufficiently large $n$ and every integer $k\ge0$,
\[
 \partial_r^k\!\left[\Phi_n(r)-(1+\varepsilon_n)r^{-2}\right]
 =O_{n,k}(r^{-4-k}), \qquad r\to\infty.
\]

\end{lemma}

\begin{proof}
The exterior estimates \eqref{eq:outerrep}--\eqref{eq:Xouter}
and the expansion of $u_1$ yield
$y_\varepsilon=a+O_1(|\varepsilon|r^{-2})$. Inserting
$\Phi=r^{-2}y$ into \eqref{eq:profile-main} gives
\begin{equation}\label{eq:youtereq}
 y''+\left(\frac{2y}{r}-\frac r2\right)y'
       +\frac{2y(y-1)}{r^2}=0.
\end{equation}
With
\[
 m_\varepsilon(r)=\exp\!\left(
 \int_1^r\left(\frac{2y_\varepsilon(t)}t-\frac t2\right)dt
 \right),
\]
we have $(m_\varepsilon y_\varepsilon')'=
-2m_\varepsilon r^{-2}y_\varepsilon(y_\varepsilon-1)$.
The boundary term at infinity vanishes, hence
\begin{equation}\label{eq:youterintegral}
 y_\varepsilon'(r)=\frac2{m_\varepsilon(r)}
 \int_r^\infty m_\varepsilon(t)t^{-2}
            y_\varepsilon(t)(y_\varepsilon(t)-1)\,dt.
\end{equation}
Here
$m_\varepsilon(r)=C_\varepsilon e^{-r^2/4}r^{2a}
 (1+O(|\varepsilon|r^{-2}))$.
Successive Gaussian-tail integrations by parts, followed by
integration from infinity, give
\begin{equation*}
\begin{aligned}
 y_\varepsilon(r)
 &=a-2a(a-1)r^{-2}+4a(a-1)r^{-4}
    +O(|\varepsilon|r^{-6}),\\
 y_\varepsilon'(r)
 &=4a(a-1)r^{-3}-16a(a-1)r^{-5}
    +O(|\varepsilon|r^{-7}).
\end{aligned}
\end{equation*}
Substitution into \eqref{eq:youtereq} gives
$y_\varepsilon''=-12a(a-1)r^{-4}+O(|\varepsilon|r^{-6})$;
differentiating the same equation, with cancellation of its
leading $r^{-3}$ terms, yields
$y_\varepsilon'''=O(|\varepsilon|r^{-5})$.
This proves \eqref{eq:outer-y-derivatives}.

For completeness, the same estimates hold after up to two
$\varepsilon$-derivatives of $y_\varepsilon-a$, with the factor
$|\varepsilon|$ removed. This follows by differentiating the
Volterra equation \eqref{eq:youterintegral}: its Gaussian kernel
localizes $t-r$ to $O(r^{-1})$ for $r\ge1$, writing $K_\varepsilon(r,t)=m_\varepsilon(t)/m_\varepsilon(r)$,
the differentiated kernel satisfies
\[
 |\partial_\varepsilon K_\varepsilon(r,t)|
 \le C\log(t/r)K_\varepsilon(r,t),\qquad
 |\partial_\varepsilon^2 K_\varepsilon(r,t)|
 \le C\{\log(t/r)+\log^2(t/r)\}K_\varepsilon(r,t).
\]
The squared logarithm comes from differentiating the exponential twice.
Both logarithmic moments are integrable against the Gaussian tail.
Differentiating the Volterra equation successively (first on a large
exterior interval, where the linearized integral operator has norm
$O(R^{-2})$, and then continuing to $r=1$) gives the claimed weighted
estimates without dropping the second-derivative term. The required parameter
smoothness on $[r_0,1]$ follows from the differentiable
contraction in \eqref{eq:outerrep} and ODE regularity.
Since $u_1=\partial_\varepsilon
\Phi_{\mathrm{out},\varepsilon}|_{\varepsilon=0}$ and
$w_\varepsilon=\varepsilon^{-1}
(\Phi_{\mathrm{out},\varepsilon}-\Phi_* -\varepsilon u_1)$,
Taylor's formula in $\varepsilon$ gives, for $j\le2$,
\[
 \partial_r^jw_\varepsilon
 =\frac1\varepsilon\int_0^\varepsilon
  (\varepsilon-\eta)\partial_r^j\partial_\eta^2
  \Phi_{\mathrm{out},\eta}\,d\eta
 =O(|\varepsilon|r^{-4-j}),
\]
which is \eqref{eq:outer-wsecond}. %
For $g=r^2u_1$, the linearized equation at $y=1$ reads
\[
 (r^2e^{-r^2/4}g')'=-2e^{-r^2/4}g.
\]
By the large-$r$ asymptotics of $u_1$ stated in the proof of
Proposition~\ref{prop:21} and the Gaussian-tail estimate,
$g'=4r^{-3}+O(r^{-5})$.
Since $g(\infty)=1$, integration gives
$g=1-2r^{-2}+O(r^{-4})$.
This proves \eqref{eq:outer-main} and the expansion of $u_1'$.
For fixed $\varepsilon$, differentiating \eqref{eq:youtereq}
and applying the same integrating-factor argument inductively yields
\[
 \partial_r^k(y_\varepsilon-a)
 =O_{\varepsilon,k}(r^{-2-k}),\qquad k\ge0.
\]
At each step, the exponentially growing homogeneous term is excluded
by the preceding derivative bound.
Multiplication by $r^{-2}$ gives the asserted estimates for $\Phi_n$.

\end{proof}

\begin{corollary}[Global two-scale bounds]\label{cor:twoscale}
For the global family in Proposition~\ref{prop:21}, by the definitions of $P_n$ and $U_n$ in \eqref{two-notations}, there holds
\begin{equation}\label{eq:overlaplimit}
 \lim_{R\to\infty}\limsup_{n\to\infty}\sup_{r\ge R\mu_n}
 \bigl(|rP_n(r)-2|+|r^2U_n(r)-2|
       +|r^3U_n'(r)+4|+|r^4U_n''(r)-12|\bigr)=0.
\end{equation}
The four weighted quantities
$rP_n$, $r^2U_n$, $r^3U_n'$ and $r^4U_n''$
are uniformly bounded on $[0,\infty)$ for all sufficiently large $n$.
\end{corollary}

\begin{proof}
In $\mu_n\le r\le r_0$, put $s=r/\mu_n$. The interior
representation gives
\[
 rP_n=2s^2\Q(s)+2\mu_n^2s^2R_{\mu_n}(s),\qquad
 r^2U_n=s^2(6\Q+2s\Q')
       +\mu_n^2s^2(6R_{\mu_n}+2sR_{\mu_n}').
\]
Lemma~\ref{lem:25} and \eqref{eq:Rbounds} yield
\begin{equation}\label{eq:overlapph}
 |rP_n-2|+|r^2U_n-2|
 \le C\bigl(s^{-1/2}+\mu_n^2s^{3/2}\bigr).
\end{equation}
Thus the right-hand side tends to zero uniformly on
$R\mu_n\le r\le r_0$ in the stated iterated limit.

Write $p_n=rP_n$, $h_n=r^2U_n$, and $a_n=r^3U_n'$.
The profile equation gives
$P_n'=U_n-2P_n/r$ and
$U_n'=rU_n/2-P_n/2-U_nP_n$; consequently,
\begin{equation}\label{eq:abidentities}
 a_n=\frac{r^2}{2}(h_n-p_n)-h_np_n,\qquad
 r^4U_n''=\frac{r^2}{2}(a_n+2p_n)
              -a_np_n-h_n^2+2h_np_n.
\end{equation}
Since $r\le r_0$ in the overlap region,
\eqref{eq:overlapph} and \eqref{eq:abidentities} yield
$a_n\to-4$ and $r^4U_n''\to12$ in the same
iterated-limit sense.

For $r\ge r_0$, use
$\Phi_n=r^{-2}+\varepsilon_n(u_1+w_{\varepsilon_n})$.
By \eqref{eq:Xouter}, \eqref{eq:outerrep}, and
$u_1=r^{-2}(1+O_1(r^{-2}))$ as $r\to\infty$,
\[
 \sup_{r\ge r_0}
 \bigl(|p_n-2|+|h_n-2|+|r^2(h_n-p_n)|\bigr)
 \le C(r_0)|\varepsilon_n|\longrightarrow0.
\]
Here $r^2(h_n-p_n)=2r^4\Lambda\Phi_n$ and
$\Lambda(r^{-2})=0$. The first identity in
\eqref{eq:abidentities} therefore shows that
$a_n\to-4$ uniformly on the exterior. To estimate the
second derivative uniformly on the unbounded interval, use
Lemma~\ref{lem:outer-derivatives}, which gives, for $r\ge1$,
$y_n:=r^2\Phi_n=1+\varepsilon_n+O_3(|\varepsilon_n|r^{-2})$,
where $O_3$ controls derivatives through order three.
Since
\[
 U_n=2r^{-2}y_n+2r^{-1}y_n',\quad
 r^4U_n''=12y_n-4ry_n'-2r^2y_n''+2r^3y_n''',
\]
we obtain
\[
 \sup_{r\ge1}|r^4U_n''-12|\le C|\varepsilon_n|.
\]
On the fixed interval $[r_0,1]$, parameter-smooth ODE
dependence gives the same estimate with a constant depending
on $r_0$. This proves \eqref{eq:overlaplimit}.

Finally, on $0\le r\le R\mu_n$ for fixed $R$, the
inner-scale convergence \eqref{eq:scaledlimits} gives
uniform bounds for all four dimensionless quantities.
The preceding overlap and exterior estimates provide
uniform bounds on the remainder of the half-line.
\end{proof}

\begin{lemma}[Quantitative inner and outer coefficient transfer]
\label{lem:quantitative-coefficient-transfer}
For the matched profiles of Proposition~\ref{prop:21}, write
\[
\begin{gathered}
 y_n(r)=r^2\Phi_n(r),\qquad
 \zeta_n(r)=-\frac{\Phi_n'(r)}{r\Phi_n(r)^2},\\
 p_n(r)=rP_n(r),\quad h_n(r)=r^2U_n(r),\quad
 a_n(r)=r^3U_n'(r),\quad b_n(r)=r^4U_n''(r),
\end{gathered}
\]
and define the corresponding core quantities
\[
 (y_{\rm c},\zeta_{\rm c},p_{\rm c},h_{\rm c},a_{\rm c},b_{\rm c})(s)
 =\left(s^2\Q,-\frac{\Q'}{s\Q^2},sP,s^2Q,s^3Q',s^4Q''\right)(s).
\]
There exists $C$, independent of sufficiently large $n$, such that
\begin{equation}\label{eq:quantitative-inner-transfer}
 \sup_{0<s\le r_0/\mu_n}\!\left(
 |y_n(\mu_ns)-y_{\rm c}(s)|
 +|\zeta_n(\mu_ns)-\zeta_{\rm c}(s)|
 +\sum_{j\in\{p,h,a,b\}}|j_n(\mu_ns)-j_{\rm c}(s)|
 \right)
 \le C\mu_n^{1/2}r_0^{3/2}.
\end{equation}
For $r\ge r_0$, all four quantities $p_n,h_n,a_n,b_n$
are within $C(r_0)|\varepsilon_n|$ of $(2,2,-4,12)$,
respectively. Consequently, for each $R\ge1$ with
$R\mu_n\le r_0$,
\begin{equation}\label{eq:quantitative-overlap-transfer}
 \sup_{r\ge R\mu_n}
 \bigl(|p_n-2|+|h_n-2|+|a_n+4|+|b_n-12|\bigr)
 \le C R^{-1/2}+C\mu_n^{1/2}r_0^{3/2}
           +C(r_0)|\varepsilon_n|.
\end{equation}
\end{lemma}

\begin{proof}
Put $q_n(s)=\Q(s)+\mu_n^2R_n(s)$, so that
$\Phi_n(\mu_ns)=\mu_n^{-2}q_n(s)$. From
Lemma~\ref{lem:25}, Corollary~\ref{cor:secondmono}, and
\eqref{eq:globalinner},
\[
 \Q(s)\ge c(1+s)^{-2},\qquad
 -\Q'(s)\ge\frac{cs}{(1+s)^4},
\]
while
\[
 |R_n|+|sR_n'|\le C(1+s)^{-1/2},\qquad
 |R_n''|\le C(1+s)^{-5/2},
\]
with $R_n=O(s^2)$ and $R_n'=O(s)$ for $s\le1$.
Thus, uniformly for $0<s\le r_0/\mu_n$,
\begin{equation}\label{eq:relative-inner-error}
 \left|\frac{\mu_n^2R_n}{\Q}\right|
 +\left|\frac{\mu_n^2R_n'}{\Q'}\right|
 \le C\mu_n^2(1+s)^{3/2}
 \le C\mu_n^{1/2}r_0^{3/2}.
\end{equation}
Here the quotient involving $\Q'$ is defined for $s>0$;
its limit at the regular origin is bounded. In particular, $q_n$
remains uniformly comparable to $\Q$. The identity
\[
 \zeta_n(\mu_ns)
 =-\frac{\Q'(s)+\mu_n^2R_n'(s)}{s q_n(s)^2}
\]
and the boundedness of $\zeta_{\rm c}$ now give the required
estimate for $\zeta_n$. The value and first-derivative identities
\[
\begin{aligned}
 p_n(\mu_ns)-p_{\rm c}(s)&=2\mu_n^2s^2R_n(s),\\
 h_n(\mu_ns)-h_{\rm c}(s)&=\mu_n^2s^2(6R_n+2sR_n')(s),\\
 a_n(\mu_ns)-a_{\rm c}(s)&=\mu_n^2s^3(8R_n'+2sR_n'')(s)
\end{aligned}
\]
are bounded by the right-hand side of
\eqref{eq:relative-inner-error}. The same holds for $y_n-y_{\rm c}$.

For the second density derivative use the exact profile identities
\eqref{eq:abidentities} rather than differentiating $R_n$ a third
time. The core equations $Q'=-PQ$ and $P'+2P/s=Q$ give
\[
 a_{\rm c}=-h_{\rm c}p_{\rm c},\qquad
 b_{\rm c}=-a_{\rm c}p_{\rm c}-h_{\rm c}^2
                       +2h_{\rm c}p_{\rm c}.
\]
In particular,
$|a_{\rm c}+2p_{\rm c}|\le C(1+s)^{-1/2}$ for $s\ge1$
by \eqref{eq:core-tail-main}, while it is $O(s^2)$ for $s\le1$.
Since $r=\mu_ns\le r_0$, the term
$r^2(a_n+2p_n)/2$ in \eqref{eq:abidentities} is bounded by
$C\mu_n^2(1+s)^{3/2}$; the remaining differences are products
of uniformly bounded quantities with the preceding controlled
errors. This proves \eqref{eq:quantitative-inner-transfer}.

On $r\ge r_0$, the exterior representation and
Lemma~\ref{lem:outer-derivatives}, including smooth ODE dependence on
$[r_0,1]$, give the stated $C(r_0)|\varepsilon_n|$ bound.
Finally, for $s\ge R\ge1$ the four core quantities satisfy
\[
 |p_{\rm c}-2|+|h_{\rm c}-2|
 +|a_{\rm c}+4|+|b_{\rm c}-12|\le CR^{-1/2}
\]
by Lemma~\ref{lem:25}. Splitting the half-line at $r_0$
therefore yields \eqref{eq:quantitative-overlap-transfer}.
\end{proof}

\subsection{Zeros of the Scaling Mode}

\begin{lemma}\label{us:radial-zero-enumeration}
For the consecutive matching-interval numbering, there exists an
integer $N_0$, independent of $n$, such that
\[
    N_n
    :=
    \#\{r>0:\Lambda\Phi_n(r)=0\}
    =
    n+N_0
\]
for all sufficiently large $n$.
\end{lemma}

\begin{proof}
 Lemma~\ref{lem:26} gives
constants $c_{\Lambda\Q}\ne0$ and $\vartheta_{\Lambda\Q}$ such that
\[
 \Lambda\Q(s)=c_{\Lambda\Q}s^{-5/2}
 \sin(\omega_*\log s+\vartheta_{\Lambda\Q})+O_1(s^{-3}).
\]
In particular, since the sine and its derivative cannot vanish
simultaneously, there exist $T>1$ and $c>0$ with
\begin{equation*}
 |\Lambda\Q(s)|+|s(\Lambda\Q)'(s)|\ge cs^{-5/2},
 \qquad s\ge T.
\end{equation*}
Every zero of $\Lambda\Q$ is simple, by
$H(\Lambda\Q)=0$ and uniqueness for the homogeneous ODE.
Fix $0<\eta<d_0$. Choose $T$ sufficiently large, along a sequence
for which the limiting phase is close to $\pi/2$ modulo $\pi$,
so that $\Lambda\Q(T)\ne0$ and
$\operatorname{dist}(\omega_*\log T+\vartheta_{\Lambda\Q},\pi\mathbb Z)>\eta$.
As $\Lambda\Q(0)=1/3$, the number of zeros on $[0,T]$
is finite. The local convergence \eqref{eq:C2convergence}, together
with the inner representation, preserves this finite count for large $n$.

Fix $0<\rho<r_0$, to be chosen below. Since
\[
 \mu_n^2\Lambda\Phi_n(\mu_ns)
 =\Lambda\Q(s)+\mu_n^2\Lambda R_n(s),
\]
\eqref{eq:Rn-bound-main} gives
\begin{equation*}
 \frac{\mu_n^2\bigl(|\Lambda R_n(s)|
          +|s(\Lambda R_n)'(s)|\bigr)}
      {|\Lambda\Q(s)|+|s(\Lambda\Q)'(s)|}
 \le C\mu_n^2s^2\le C\rho^2,
 \qquad T\le s\le \rho/\mu_n.
\end{equation*}
To make zero counting precise, put $t=\log s$ and
\[
 F_n(t)=c_{\Lambda\Q}^{-1}e^{5t/2}
 \mu_n^2\Lambda\Phi_n(\mu_ne^t).
\]
By the differentiated asymptotics and \eqref{eq:Rn-bound-main},
\[
 \|F_n-\sin(\omega_*t+\vartheta_{\Lambda\Q})\|_{C^1(
 [\log T,\log(\rho/\mu_n)])}
 \le C(T^{-1/2}+\rho^2).
\]
We make the zero-stability threshold explicit using the already fixed
$0<\eta<d_0$ and put
\[
 I_k=\{t:|\omega_*t+\vartheta_{\Lambda\Q}-k\pi|<\eta\},
 \qquad k\in\mathbb Z.
\]
On each $I_k$, the derivative of the limiting sine has constant sign
and absolute value at least $\omega_*\cos\eta$; outside their union,
its absolute value is at least $\sin\eta$. If
\begin{equation}\label{eq:quantitative-inner-zero-stability}
 C(T^{-1/2}+\rho^2)
 <\frac12\min\{\sin\eta,\omega_*\cos\eta\},
\end{equation}
then $F_n$ has no zeros outside the $I_k$, and has exactly one simple
zero in each $I_k$ contained in the interval under consideration.
Indeed, the $C^0$ estimate preserves opposite signs at both endpoints
of each such $I_k$, while the $C^1$ estimate preserves strict
monotonicity inside.

To fix the upper-endpoint count, use
\eqref{eq:matching-phase-main} and select $\rho$ (arbitrarily
small) so that
\[
 \operatorname{dist}\bigl(
 \omega_*\log\rho+\vartheta_{\Lambda\Q}-\vartheta_0,
 \pi\mathbb Z\bigr)>2d_0,
\]
while also requiring $\Lambda u_1(\rho)\ne0$.
Such choices exist because $2d_0<\pi/2$ and the excluded
zeros of $\Lambda u_1$ are isolated. The lower endpoint $T$ was chosen
away from zeros of the limiting sine. Since $|e_n|<d_0$, the upper
endpoint phase stays at distance greater than $d_0$ from $\pi\mathbb Z$.
We may therefore first choose $T$ sufficiently large and then choose
such a sufficiently small $\rho$ that
\eqref{eq:quantitative-inner-zero-stability} holds. Indeed,
\[
 \omega_*\log(\rho/\mu_n)+\vartheta_{\Lambda\Q}
 =n\pi+\bigl(\omega_*\log\rho
 +\vartheta_{\Lambda\Q}-\vartheta_0\bigr)+e_n,
 \qquad |e_n|<d_0.
\]
Let $m_T,m_\rho\in\mathbb Z$ be the unique integers with
\[
 m_T\pi<\omega_*\log T+\vartheta_{\Lambda\Q}<(m_T+1)\pi,
 \qquad
 m_\rho\pi<\omega_*\log\rho+\vartheta_{\Lambda\Q}
 -\vartheta_0<(m_\rho+1)\pi.
\]
Then the number of zeros of the limiting sine between the two
endpoints is exactly $n+m_\rho-m_T$ for all sufficiently large $n$.
The quantitative zero-stability estimate above preserves this count.
Including the fixed finite number on $[0,T]$ yields
\begin{equation}\label{us:inner-zero-count}
 \#\{0<r<\rho:\Lambda\Phi_n(r)=0\}=n+C_{\mathrm{in}}
\end{equation}
for an integer $C_{\mathrm{in}}$ independent of large $n$.

It remains to count zeros on $[\rho,\infty)$. First,
$\varepsilon_n\ne0$: otherwise the matched profile would
have the same Cauchy data as $\Phi_*=r^{-2}$ at $r_0$,
contradicting origin regularity by uniqueness of the profile ODE.
Define
\[
 V_n(r)=\varepsilon_n^{-1}(\Phi_n(r)-r^{-2}),
 \qquad \rho\le r\le r_0.
\]
At $r_0$, the exact matching equations give
\[
 (V_n(r_0),V_n'(r_0))
 =(u_1(r_0)+w_{\varepsilon_n}(r_0),
   u_1'(r_0)+w_{\varepsilon_n}'(r_0)).
\]
Substitution of $\Phi_n=r^{-2}+\varepsilon_nV_n$ into
\eqref{eq:profile-main} gives, on $[\rho,r_0]$,
\[
 V_n''+\left(\frac6r-\frac r2\right)V_n'
 +\left(\frac8{r^2}-1\right)V_n
 =-\varepsilon_n\bigl[6V_n^2+r(V_n^2)'\bigr].
\]
This is a regular perturbation of the equation for $u_1$ on
the fixed compact interval. As $\varepsilon_n\to0$, continuous
ODE dependence gives
\[
 V_n\longrightarrow u_1\quad\text{in }C^2([\rho,r_0]),
 \qquad
 \varepsilon_n^{-1}\Lambda\Phi_n=\Lambda V_n
 \longrightarrow\Lambda u_1\quad\text{in }C^1([\rho,r_0]).
\]
On $[r_0,\infty)$ the exact identity
$\varepsilon_n^{-1}\Lambda\Phi_n
=\Lambda u_1+\Lambda w_{\varepsilon_n}$ holds.
Thus the convergence is $C^1$ on each compact subset of
$[\rho,\infty)$.

The function $v=\Lambda u_1$ satisfies
\[
 v''+\left(\frac6r-\frac r2\right)v'
 +\left(\frac8{r^2}-2\right)v=0.
\]
It is not identically zero, since %
\eqref{eq:outer-main} and Lemma~\ref{lem:outer-derivatives} give

\[
 \Lambda u_1(r)=4r^{-4}+O(r^{-6}),\qquad r\to\infty.
\]
Consequently every zero of $\Lambda u_1$ on $(0,\infty)$
is simple. Moreover, by \eqref{eq:Xouter},
\[
 \Lambda w_{\varepsilon_n}(r)=O(|\varepsilon_n|r^{-4}),
 \qquad r\ge1.
\]
Choose a fixed large $R>r_0$ such that
$\Lambda u_1(r)\ge2r^{-4}$ for $r\ge R$. For all large
$n$, the preceding bounds show that
$\Lambda u_1+\Lambda w_{\varepsilon_n}$ has no zero
on $[R,\infty)$. Since $\Lambda u_1(\rho)\ne0$ and
all its zeros on $[\rho,R]$ are simple, compact $C^1$
convergence preserves their number. Hence
\begin{equation}\label{us:outer-zero-count}
 \#\{r\ge\rho:\Lambda\Phi_n(r)=0\}=C_{\mathrm{out}}
\end{equation}
for an integer $C_{\mathrm{out}}$ independent of large $n$.
Combining \eqref{us:inner-zero-count} and
\eqref{us:outer-zero-count} gives
\[
 N_n=n+C_{\mathrm{in}}+C_{\mathrm{out}}=n+N_0,
 \qquad N_0:=C_{\mathrm{in}}+C_{\mathrm{out}}.
\]
\end{proof}

\begin{proof}[Proof of Proposition~\ref{thm:matched-family-main}]
The core asymptotics are proved in Lemma~\ref{lem:25},
the interior correction and its derivative bounds follow from
Proposition~\ref{prop:27} and Corollary~\ref{cor:secondmono}.
Proposition~\ref{prop:21} supplies the global matching and phase law,
Corollaries~\ref{cor:innerlimits} and \ref{cor:twoscale} yield the
two-scale estimates, and Lemma~\ref{us:radial-zero-enumeration}
proves the zero enumeration. The exterior assertions in (ii) follow
from the exterior homogeneous mode and Lemma~\ref{lem:outer-derivatives}. Finally, Corollary~\ref{cor:secondmono} gives
$\Phi_n'(r)<0$ for $0<r\le r_0$.
On $[r_0,\infty)$, \eqref{eq:outerrep}--\eqref{eq:Xouter}
and the stated asymptotics of $u_1$ imply
\[
\sup_{r\ge r_0} r^3
\bigl|\Phi_n'(r)+2r^{-3}\bigr|
\le C(r_0)|\varepsilon_n|\longrightarrow0.
\]
Thus, for all sufficiently large $n$, we have
$\Phi_n'(r)<0$ for every $r>0$.
Since $\Phi_n(r)\to0$ as $r\to\infty$, strict monotonicity
also yields $\Phi_n(r)>0$ for every $r\ge0$.
\end{proof}

 \section{Spectral analysis and linearized flow}
\label{sec:unstable-spectrum}

The aim of this section is to prove Theorem \ref{thm:main}, from which  Theorem \ref{thm:main-spectral} follows.
Fix a sufficiently large matching index $n$.  The profile equation and the identities
\begin{align}
 &\Phi_n''+\left(\frac4r-\frac r2\right)\Phi_n'
       -\Phi_n+6\Phi_n^2+2r\Phi_n\Phi_n'=0,\label{us:profile-ode}\\
 &P_n'+\frac2rP_n=U_n,\qquad
 U_n'=\frac r2U_n-\frac12P_n-U_nP_n\label{us:profile-system}
\end{align}
will be used repeatedly in this section. 
 All multiplicities below are
algebraic multiplicities in the complex Hilbert space $L^2(\mathbb R^3)$.
The only profile information used in this section is the matched expansion and the two-scale bounds in Proposition~\ref{thm:matched-family-main}.

\subsection{Closed realization and spherical harmonic reduction}
\label{us:profile-inputs}

Write $\mathbf L_0=-\Delta+1+\tfrac12x\cdot\nabla$ and recall
\begin{equation*}
 \mathscr D=\{f\in H^2(\mathbb R^3):x\cdot\nabla f\in L^2(\mathbb R^3)\}.
\end{equation*}
We realize $ L_n$ on $\mathscr D$.

\begin{lemma}[Fredholm reduction]\label{us:fredholm}
The operator $ L_n$ is closed on $\mathscr D=D(\mathbf L_0)$, and,
for every $z$ with $\operatorname{Re}z<\tfrac14$,
\[
  L_n-z:\mathscr D\longrightarrow L^2(\mathbb R^3)
\]
is Fredholm of index zero. Consequently,
\[
 \sigma( L_n)\cap\{\operatorname{Re}z<\tfrac14\}
 =\sigma_p( L_n)\cap\{\operatorname{Re}z<\tfrac14\},
\]
and every spectral point in this half-plane is an isolated eigenvalue of
finite algebraic multiplicity.
\end{lemma}

\begin{proof}
By the classical theory of Ornstein--Uhlenbeck operators,
$\mathbf L_0$ is closed on $\mathscr D$.
For every $z\in\mathbb C$ with $\operatorname{Re}z<1/4$,
$\mathbf L_0-z:\mathscr D\to L^2(\mathbb R^3)$
is bijective with a bounded inverse on $L^2(\mathbb R^3)$.
Moreover, there is a constant $C>0$ such that
\begin{equation}\label{us:free-graph}
 \|f\|_{H^2}+\|x\cdot\nabla f\|_2
 \le C\bigl(\|\mathbf L_0f\|_2+\|f\|_2\bigr),
 \qquad f\in\mathscr D.
\end{equation}

For fixed $n$, the profile is smooth and
$U_n=O(\langle r\rangle^{-2})$,
$\nabla U_n=O(\langle r\rangle^{-3})$, and
$P_n=O(\langle r\rangle^{-1})$ at infinity. Thus multiplication by
$U_n$ and the first-order term $P_n\partial_r$ are compact from
$\mathscr D$, with its graph norm, into $L^2$: use Rellich compactness
on a ball, \eqref{us:free-graph}, and smallness of the coefficients
outside a large ball. The Newton term is already compact from $L^2$
into $L^2$. Indeed, $\nabla\Delta^{-1}$ maps a bounded $L^2$ set to a
bounded subset of $H^1$ locally, and
\[
 \|\mathbf1_{\{|x|>R\}}\nabla U_n\cdot\nabla\Delta^{-1}f\|_2
 \lesssim\|\mathbf1_{\{|x|>R\}}\nabla U_n\|_3\|f\|_2\longrightarrow0
\]
uniformly on that set. Consequently $\mathcal K_n= L_n-\mathbf L_0$
is relatively compact. The interpolation estimate
\[
 \|\mathcal K_nf\|_2\le\varepsilon\|\mathbf L_0f\|_2
                         +C_{n,\varepsilon}\|f\|_2
\]
also proves closedness %
of $L_n=\mathbf L_0+\mathcal K_n$ on the same domain. %
Integration by parts gives
$\operatorname{Re}\langle\mathbf L_0f,f\rangle
=\|\nabla f\|_2^2+\frac14\|f\|_2^2$.
Hence, for $a>0$,
$\|(\mathbf L_0+a)^{-1}\|_{2\to2}\le(a+1/4)^{-1}$.
Using
$\mathbf L_0(\mathbf L_0+a)^{-1}
=I-a(\mathbf L_0+a)^{-1}$
and the preceding relative-bound estimate, we obtain
$\|\mathcal K_n(\mathbf L_0+a)^{-1}\|_{2\to2}
\le 2\varepsilon+C_{n,\varepsilon}/(a+1/4)$.
Choosing $\varepsilon=1/4$ and then $a$ sufficiently large
gives
$\|\mathcal K_n(\mathbf L_0+a)^{-1}\|_{2\to2}<1$.
 The analytic
Fredholm theorem, applied to
\[
  L_n-z=
 [I+\mathcal K_n(\mathbf L_0-z)^{-1}](\mathbf L_0-z),
\]
proves the assertions.
\end{proof}

We use the same spherical harmonic decomposition as in
Li--Zhou \cite{LiZhou2026}.
Let $\{Y_{l,m}\}_{l\ge0,\,|m|\le l}$ be an orthonormal basis
of $L^2(\mathbb S^2)$ with $-\Delta_{\mathbb S^2}Y_{l,m}=J_lY_{l,m}$, $J_l:=l(l+1)$. We write $x=r\omega$, where $r=|x|>0$ and
$\omega=x/|x|\in\mathbb S^2$. Using the above spherical harmonics, every $f\in L^2(\mathbb R^3)$ admits the expansion
\[
 f(x)
 =\sum_{l=0}^{\infty}\sum_{m=-l}^{l}
 f_{l,m}(r)Y_{l,m}(\omega),
 \qquad
 f_{l,m}(r)
 :=\int_{\mathbb S^2}
 f(r\omega)\overline{Y_{l,m}(\omega)}\,d\omega,
\]
with convergence in $L^2(\mathbb R^3)$. Moreover,
\[
 \|f\|_{L^2(\mathbb R^3)}^2
 =\sum_{l=0}^{\infty}\sum_{m=-l}^{l}
 \int_0^\infty |f_{l,m}(r)|^2r^2\,dr.
\]
We write 
\begin{equation*}
 \mathscr X_{l,m}:=\{u(r)Y_{l,m}(\omega):
 u\in L^2((0,\infty),r^2\,dr)\}\subset L^2(\mathbb R^3)
\end{equation*}
and $\Pi_{l,m}$ for the orthogonal projection onto
$\mathscr X_{l,m}$, given by
\[
 (\Pi_{l,m}f)(x)
 =f_{l,m}(r)Y_{l,m}(\omega).
\]
The map
$u(r)\mapsto u(r)Y_{l,m}(\omega)$ is unitary from
$L^2((0,\infty),r^2\,dr)$ onto $\mathscr X_{l,m}$. Denote by $\Delta_l$ the radial Laplacian in the degree-$l$
spherical harmonic sector:
\begin{equation}
 \Delta_l=\partial_r^2+\frac2r\partial_r-\frac{J_l}{r^2}.
 \label{eq:Deltal}
\end{equation}
Under the above unitary identification, the restriction of $L_n$
to $\mathscr X_{l,m}$ is represented by the radial operator $L_{n,l}$
given below; see Lemma~\ref{us:fixed-mode-fredholm} for more details.
\begin{equation}\label{us:mode-operator}
 L_{n,l}u=-\Delta_lu+\frac12(2u+ru')-2U_nu-P_nu'-U_n'\partial_r\Delta_l^{-1}u.
\end{equation}
The Newton inverse of $\Delta_l$, initially for compactly supported functions, is
\begin{equation}\label{us:newton}
 \Delta_l^{-1}u(r)=-\frac1{2l+1}
 \left(r^{-l-1}\int_0^r u(s)s^{l+2}\,ds
       +r^l\int_r^\infty u(s)s^{1-l}\,ds\right).
\end{equation}
For a general $L^2$ input the Newton field is understood by the extension
specified after \eqref{us:linearized}. In the radial mode only its
derivative is needed, and
$\partial_r\Delta_0^{-1}u=r^{-2}\int_0^r u(s)s^2\,ds$.
Local elliptic regularity, applied also to the Newton potential, shows
that eigenfunctions and generalized eigenfunctions are smooth. Their
degree-$l$ coefficients have the regular expansion $r^l$ times a smooth
even function near zero.

For the operator-theoretic arguments below, we record at this point the
closed realization of each fixed spherical harmonic mode.  Define
\begin{equation*}
 \mathcal U_{l,m}:L^2((0,\infty),r^2\dd r)\longrightarrow\mathscr X_{l,m},
 \qquad
 (\mathcal U_{l,m}f)(r,\omega)=f(r)Y_{l,m}(\omega).        
\end{equation*}

\begin{lemma}[Fixed-mode realization and Fredholm reduction]
\label{us:fixed-mode-fredholm}
For each $l\ge0$ and $|m|\le l$, the operators $L_n$ and $\mathbf L_0$ preserve $\mathscr X_{l,m}$. The map $\mathcal U_{l,m}$ is unitary, and
\begin{align}
 \mathcal U_{l,m}^{-1}
 \bigl( L_n|_{\mathscr X_{l,m}}\bigr)\mathcal U_{l,m}
 &=L_{n,l},                                                \label{eq:fixedmodeidentification}\\
 \mathcal U_{l,m}^{-1}
 \bigl(\mathbf L_0|_{\mathscr X_{l,m}}\bigr)\mathcal U_{l,m}
 &=L_{0,l},                                                \label{eq:fixedmodefreeidentification}
\end{align}
where
\begin{equation*}
 L_{0,l}f=-f''-\frac2r f'+\frac{J_l}{r^2}f+f+\frac r2f'.  
\end{equation*}
The identities include the closed domains
\begin{equation*}
 \mathcal U_{l,m}^{-1}
 \bigl(D( L_n)\cap\mathscr X_{l,m}\bigr)=D(L_{n,l}),
 \qquad
 \mathcal U_{l,m}^{-1}
 \bigl(D(\mathbf L_0)\cap\mathscr X_{l,m}\bigr)=D(L_{0,l}).
\end{equation*}
Moreover, for every $z$ with $\operatorname{Re}z<1/4$,
\begin{equation*}
 L_{n,l}-z:D(L_{n,l})\longrightarrow L^2((0,\infty),r^2\dd r)
 \quad\hbox{is Fredholm},\qquad
 \operatorname{ind}(L_{n,l}-z)=0,                         
\end{equation*}
and
\begin{equation}
 \sigma(L_{n,l})\cap\{\operatorname{Re}z<1/4\}
 =\sigma_p(L_{n,l})\cap\{\operatorname{Re}z<1/4\}.       \label{eq:fixedmodespectrumpoint}
\end{equation}
Thus, in the half-plane $\operatorname{Re}z<1/4$,
excluding eigenvalues of $L_{n,l}$ suffices to exclude its spectrum.
\end{lemma}

\begin{proof}
The proof is given in Appendix~\ref{app:proof-fixedmodefredholm}.
\end{proof}

We next state the main spectral theorem of this section.

\begin{theorem}[Complete unstable spectrum and optimal high-mode bound]
\label{thm:main}
There exists $n_*\ge1$ such that, for every $n\ge n_*$, the following
statements hold.

\medskip
\noindent\textnormal{\rm (i) Radial sector $l=0$.}
Set
\[
 N_n:=\#\{r>0:\Lambda\Phi_n(r)=0\}.
\]
The radial operator $L_{n,0}$ has exactly $N_n$ algebraically simple
eigenvalues below $-1$,
\[
 -\mu_{1,n}<\cdots<-\mu_{N_n,n}<-1.
\]
The eigenvalues
$-\mu_{1,n},\ldots,-\mu_{N_n,n},-1$ satisfy
\[
 \ker(L_{n,0}+\mu_{j,n})
 =\operatorname{span}\{\phi_{j,n}\},
 \qquad 1\le j\le N_n,
 \qquad
 \ker(L_{n,0}+1)
 =\operatorname{span}\{\Lambda U_n\},
\]
where $\phi_{j,n}$ denotes the radial eigenfunction associated with
$-\mu_{j,n}$. There is no radial spectrum in $(-1,0]$. Moreover, for the
consecutive matching-interval numbering,
\begin{equation*}
 N_n=n+N_0,
\end{equation*}
where $N_0\in\mathbb Z$ is independent of $n$.

\medskip
\noindent\textnormal{\rm (ii) Translation sector $l=1$.}
$L_{n,1}$ has $-1/2$ as its only negative spectral point, and this eigenvalue is algebraically simple with eigenfunction $U_n'$. For the full operator
$ L_n $, $-1/2$ is semi-simple of multiplicity three, and
\begin{equation*}
 \ker( L_n+1/2)
 =
 \operatorname{span}
 \{\partial_1U_n,\partial_2U_n,\partial_3U_n\}.
\end{equation*}
\medskip
\noindent\textnormal{\rm (iii) Higher angular sectors $l\ge2$.}
For every $l\ge2$,
\begin{equation*}
 \sigma(L_{n,l})\cap
 \{\operatorname{Re}z<\tfrac14\}
 =\varnothing.
\end{equation*}
Moreover,
\[
 \frac14+i\mathbb R\subset\sigma(L_{n,l}),
\]
so that $1/4$ is the exact spectral left endpoint in every angular
sector $l\ge2$.

\medskip
\noindent\textnormal{\rm (iv) Complete unstable spectrum.}
There exists $\delta_n^{\rm sp}>0$ such that
\begin{equation*}
 \sigma( L_n)\cap\{\operatorname{Re}z<\delta_n^{\rm sp}\}
 =
 \{-\mu_{1,n},\ldots,-\mu_{N_n,n},-1,-\tfrac12\}.
\end{equation*}
If $P_{{\rm u},n}$ denotes the Riesz projection onto the negative
spectrum of $ L_n$, then
\[
 \operatorname{Ran}P_{{\rm u},n}
 =
 \operatorname{span}
 \left\{
   \phi_{1,n},\ldots,\phi_{N_n,n},
   \Lambda U_n,
   \partial_1U_n,\partial_2U_n,\partial_3U_n
 \right\},
\]
and hence
\[
 \operatorname{rank}P_{{\rm u},n}=N_n+4.
\]
\end{theorem}

Throughout this section we use the profile estimates in
Proposition~\ref{thm:matched-family-main} and the notation fixed above.
The inner remainder, logarithmic tail, matching phase, and two-scale limits
referred to below are precisely
\eqref{eq:Rn-bound-main}, \eqref{eq:core-tail-main},
\eqref{eq:matching-phase-main}, and
\eqref{eq:inner-scale-limits-main}--\eqref{eq:outer-scale-limits-main}.

\subsection{The radial spectrum}
\label{us:radial}
We analyze the radial operator $L_{n,0}$ through the reduced-mass
transformation $\mathcal T$ introduced in \eqref{reduced mass}.
The identity $ L_{n,0}\mathcal T=\mathcal T\mathcal A_n $ reduces the nonlocal expression to the local expression
$\mathcal A_n$ in \eqref{local-operator}.
The correspondence between their eigenfunctions and operator
domains is established in Lemma~\ref{us:radial-transfer}.

To specify the self-adjoint realization of $\mathcal A_n$, set
\begin{equation}\label{us:mass-weight}
 \rho_n(r)=r^4\exp\left(\int_0^r2s\Phi_n(s)\,ds-\frac{r^2}4\right).
\end{equation}
Then $\mathcal A_n=-\rho_n^{-1}\partial_r(\rho_n\partial_r)
+1-12\Phi_n-2r\Phi_n'$. Under the unitary map $h\mapsto\rho_n^{1/2}h$, this expression
becomes a Schr\"odinger operator with potential
$2r^{-2}+O(1)$ at zero and $r^2/16+O(1)$ at infinity.
Both endpoints are limit-point.
Henceforth, $\mathcal A_n$ denotes its unique self-adjoint
realization in $L^2((0,\infty),\rho_n\,dr)$.
It has compact resolvent and simple eigenvalues, and the
Sturm oscillation theorem applies \cite{Zettl2005}.

\begin{lemma}[Radial count and gap]\label{us:radial-mass-result}
For sufficiently large $n$, the operator $\mathcal A_n$ has exactly
$N_n+1$ negative eigenvalues. Exactly $N_n$ lie below $-1$, and $-1$
is simple with eigenfunction $\Lambda\Phi_n$. In addition, there is no eigenvalue
in $(-1,0]$.
\end{lemma}
\begin{proof}
If $\Phi_\lambda(r)=\lambda^2\Phi_n(\lambda r)$, substitution in
\eqref{us:profile-ode}, followed by differentiation at $\lambda=1$,
gives
\begin{equation}\label{us:mass-scaling}
 \mathcal A_n(\Lambda\Phi_n)=-\Lambda\Phi_n.
\end{equation}
Since $\Lambda\Phi_n$ is regular at the origin and is $O(r^{-2})$
at infinity, it belongs to the domain of the limit-point
realization and is an eigenfunction at $-1$.
All its positive zeros are simple by uniqueness for the second-order
equation. Sturm oscillation therefore places $-1$ at index $N_n+1$.

It remains to exclude $(-1,0]$. We give the phase argument, including
the parameter estimates needed on the long inner interval. For a real
solution $h$, introduce the adapted Cauchy vector and its continuous
projective angle
\begin{equation*}
 \mathcal C[h](r)=\left(h(r),\frac{rh'(r)+\frac52h(r)}{\omega_*}\right),
 \qquad \vartheta[h]=\arg\left(h+i\frac{rh'+\frac52h}{\omega_*}\right).
\end{equation*}
Two nonzero Cauchy vectors are collinear precisely when their adapted
angles agree modulo $\pi$.

For the singular outer profile $\Phi_*=r^{-2}$, the mass operator becomes
\[
 \mathcal A_\infty=-\partial_r^2+\left(\frac r2-\frac6r\right)\partial_r +1-\frac8{r^2}.
\]
For $\lambda\in[-1,0]$, the substitution $h=r^{-5/2+i\omega_*}v(r^2/4)$ reduces
$(\mathcal A_\infty-\lambda)h=0$ to Kummer's equation with
parameters $a=-\frac14-\lambda+i\omega_*/2$ and $b=1+i\omega_*$.
We select the solution with polynomial decay at infinity,
obtained from the Kummer function $U(a_\lambda,b,\cdot)$.
Its connection formula
\cite[Sections~13.2 and~13.7]{DLMF}
gives a real normalization $\psi_\lambda$ for which
\[
 \psi_\lambda(r)=r^{-5/2}
 \left\{
 \cos(\omega_*\log r-\Theta(\lambda))
 +O_{C^1_\lambda}(r^2)
 \right\},
 \qquad 0<r\le r_*,
\]
where $r_*>0$ and the remainder bound are uniform for
$\lambda\in[-1,0]$, and
\[
 \Theta(\lambda)
 :=\arg\frac{2^{i\omega_*}\Gamma(i\omega_*)}
 {\Gamma(-\frac14-\lambda+i\omega_*/2)}
\]
is chosen continuously on $[-1,0]$.
We require the matching radius $r_0$ to satisfy $r_0\le r_*$,
with further smallness conditions imposed below.
For clarity, a real normalization with this phase is
\begin{equation}\label{us:weyl-exact-kummer}
 \begin{gathered}
 a_\lambda=-\tfrac14-\lambda+i\omega_*/2,
 \qquad B_\lambda=
      2^{i\omega_*}\frac{\Gamma(i\omega_*)}{\Gamma(a_\lambda)},\\
 \psi_\lambda(r)=\frac{1}{2|B_\lambda|}
 \operatorname{Re}\!\left[
  2^{-i\omega_*}r^{-5/2+i\omega_*}
  U(a_\lambda,1+i\omega_*,r^2/4)\right].
 \end{gathered}
\end{equation}
Here $B_\lambda\ne0$ throughout the compact parameter interval;
the Kummer connection formula gives the asserted unit cosine
amplitude.  At infinity its polynomial coefficient is
$4^{-1/4-\lambda}/(2|B_\lambda|)>0$, uniformly in $\lambda$.
The error has the same bound after one $r\partial_r$ derivative,
uniformly for $-1\le\lambda\le0$. Consequently,
\begin{equation}\label{us:outer-phase}
 \vartheta[\psi_\lambda](r_0)=\Theta(\lambda)-\omega_*\log r_0+O_{C^1_\lambda}(r_0^2).
\end{equation}
The digamma series yields
\begin{equation}\label{us:gamma-derivative}
 \Theta'(\lambda)=\sum_{j=0}^\infty
 \frac{\omega_*/2}{(j-\frac14-\lambda)^2+(\omega_*/2)^2}
 \ge\frac{\sqrt7}{4}.
\end{equation}
Integrating this series over $[-1,0]$ telescopes the arctangent terms
and fixes the branch of the phase increment:
\begin{equation*}
 \Theta(0)-\Theta(-1)=\pi-\arctan\sqrt7<\pi.
\end{equation*}

We next compare the model solution $\psi_\lambda$ with the
outer solution for the actual profile $\Phi_n$.
For $\lambda\in[-1,0]$, we construct a real solution $v_{n,\lambda}$ satisfying
\[
 (\mathcal A_n-\lambda)v_{n,\lambda}=0
 \quad\text{on }[r_0,\infty),
 \qquad
 \lim_{r\to\infty}\frac{v_{n,\lambda}(r)}{\psi_\lambda(r)}=1.
\]
This solution will be constructed below by a Volterra equation.
Define its phase at the matching radius by $\vartheta_{\mathrm{out},n}(\lambda)
 :=\vartheta[v_{n,\lambda}](r_0)$, 
with a continuous choice of the projective angle. Let
\[
 d_n=\sup_{r_0\le r\le1}r^{5/2}(|\Phi_n-\Phi_*|+|r(\Phi_n-\Phi_*)'|)
 +\sup_{r\ge1}r^2(|\Phi_n-\Phi_*|+|r(\Phi_n-\Phi_*)'|).
\]
Since
\[
 r^2(\Phi_n-\Phi_*)=\frac{rP_n-2}{2},
 \qquad
 r^3(\Phi_n'-\Phi_*')
 =\frac{r^2U_n-3rP_n+4}{2},
\]
it follows from \eqref{eq:outer-scale-limits-main},
with $r_0>0$ fixed, that $d_n\to0$.
To fix the growing solution without ambiguity, choose $R>1$
so large that $\psi_\lambda(r)>0$ for $r\ge R$ and all
$\lambda\in[-1,0]$. Define there
\begin{equation*}
 \chi_\lambda(r)=\psi_\lambda(r)
  \int_R^r\frac{s^{-6}e^{s^2/4}}{\psi_\lambda(s)^2}\,ds,
\end{equation*}
and continue it to $[r_0,R]$ by the ODE.  This gives a second
real solution differentiable in $\lambda$.  In the following
asymptotic normalization the Wronskian constant is in fact
equal to 1.
For $r\ge1$, write
\begin{equation}\label{us:weyl-two-branches}
 \begin{split}
  \psi_\lambda(r)&=r^{2\lambda-2}p_\lambda(r),\\
  \chi_\lambda(r)&=e^{r^2/4}r^{-5-2\lambda}q_\lambda(r),\\
  W_\lambda(r)&=\psi_\lambda\chi_\lambda'
            -\psi_\lambda'\chi_\lambda
            =r^{-6}e^{r^2/4}.
 \end{split}
\end{equation}
The Kummer asymptotics, together with the preceding
reduction-of-order representation, imply continuous differentiability
in $\lambda\in[-1,0]$, and
\begin{equation}\label{us:weyl-prefactors}
\sup_{\substack{r\ge1\\-1\le\lambda\le0}}
 \Bigl(|p_\lambda|+|rp_\lambda'|+|q_\lambda|+|rq_\lambda'|
 +|\partial_\lambda p_\lambda|+|r\partial_r\partial_\lambda p_\lambda|
 +|\partial_\lambda q_\lambda|+|r\partial_r\partial_\lambda q_\lambda|\Bigr)
 \le C.
\end{equation}
Here $p_\lambda$ tends to the strictly positive constant
$4^{-1/4-\lambda}/(2|B_\lambda|)$, while
$q_\lambda$ tends to twice its reciprocal; the stated uniform
$C^1_\lambda$ bounds follow from the Gaussian-tail expansion of the
reduction-of-order integral.  On the fixed
interval $[r_0,1]$, both solutions and their first $r$- and
$\lambda$-derivatives are bounded by $C(r_0)$.
The Green operator with terminal integration at infinity is
\begin{equation}\label{us:weyl-green}
 \mathcal G_\lambda F(r)=
 -\psi_\lambda(r)\int_r^\infty\frac{\chi_\lambda F}{W_\lambda}(s)\,ds
 +\chi_\lambda(r)\int_r^\infty\frac{\psi_\lambda F}{W_\lambda}(s)\,ds.
\end{equation}
Since $\mathcal A_n-\mathcal A_\infty=\mathcal R_n$ with
\begin{equation}\label{us:weyl-perturbation}
 \mathcal R_n h=-2r(\Phi_n-\Phi_*)h'
       -\bigl[12(\Phi_n-\Phi_*)+2r(\Phi_n-\Phi_*)'\bigr]h,
\end{equation}
the equation $(\mathcal A_n-\lambda)v_{n,\lambda}=0$ becomes $(\mathcal A_\infty-\lambda)v_{n,\lambda}
 =-\mathcal R_n v_{n,\lambda}$. 
We seek the normalized outer solution through the Volterra equation $v_{n,\lambda}
 =\psi_\lambda-\mathcal G_\lambda\mathcal R_n v_{n,\lambda}$ and compare the Cauchy data of $v_{n,\lambda}$ and
$\psi_\lambda$ at $r_0$, together with their
$\lambda$-derivatives. 

The moving polynomial weight requires care when differentiating in
$\lambda$: in general $\partial_\lambda\psi_\lambda$ does
\emph{not} belong to $X_\lambda$, because it contains
$2(\log r)\psi_\lambda$ at infinity.  We therefore work in a
fixed space.  For $-1\le\lambda\le0$ set
\begin{equation*}
 (T_\lambda h)(r)=
 \begin{cases}
   h(r),&r_0\le r\le1,\\
   r^{2-2\lambda}h(r),&r\ge1,
 \end{cases}
 \qquad
 \|g\|_{Y}:=\sup_{r\in[r_0,1]\cup[1,\infty)}
                 (|g(r)|+|rg'(r)|).
\end{equation*}
The derivative in this norm is understood piecewise at $r=1$;
functions are continuous there.  The map $T_\lambda:X_\lambda\to Y$
is an isomorphism with uniformly equivalent norms, where
\begin{equation*}
 \|h\|_{X_\lambda}
 =\sup_{r_0\le r\le1}(|h|+|rh'|)
  +\sup_{r\ge1}r^{2-2\lambda}(|h|+|rh'|).
\end{equation*}
Put $\widetilde K_{n,\lambda}:=
T_\lambda\mathcal G_\lambda\mathcal R_nT_\lambda^{-1}$.
We claim the genuinely fixed-space estimate
\begin{equation}\label{us:weyl-c1-fixed}
 \sup_{-1\le\lambda\le0}
 \left(\|\widetilde K_{n,\lambda}\|_{Y\to Y}
       +\|\partial_\lambda\widetilde K_{n,\lambda}\|_{Y\to Y}\right)
 \le C(r_0)d_n.
\end{equation}
To see why parameter differentiation causes no loss, write
$g=T_\lambda h$ and, for $s\ge1$, use
\begin{equation}\label{us:weyl-F}
 \mathcal R_nh(s)=s^{2\lambda-4}\mathcal F_{n,\lambda}g(s),
\end{equation}
where, suppressing the argument $s$ on the right,
\begin{equation*}
 \begin{aligned}
 \mathcal F_{n,\lambda}g={}&
 -2s^2(\Phi_n-\Phi_*)\bigl((2\lambda-2)g+sg'\bigr)\\
 &-\bigl[12s^2(\Phi_n-\Phi_*)
          +2s^3(\Phi_n-\Phi_*)'\bigr]g.
 \end{aligned}
\end{equation*}
The definition of $d_n$ implies
$\|\mathcal F_{n,\lambda}g\|_\infty+
\|\partial_\lambda\mathcal F_{n,\lambda}g\|_\infty
\le Cd_n\|g\|_Y$, with $g$ fixed during differentiation.
For $r\ge1$, the two terms of
$\widetilde K_{n,\lambda}g$ are, respectively,
\begin{align}
 -p_\lambda(r)
     \int_r^\infty s^{-3}q_\lambda(s)
               \mathcal F_{n,\lambda}g(s)\,ds,
 \label{us:weyl-conj-first}\\
 q_\lambda(r)r^{-3}e^{r^2/4}
     \int_r^\infty e^{-s^2/4}
       \left(\frac{s}{r}\right)^{4\lambda}
       p_\lambda(s)\mathcal F_{n,\lambda}g(s)\,ds.
 \label{us:weyl-conj-second}
\end{align}
All explicit powers $s^{2\lambda}$ have canceled in the first
integral.  In the second, differentiation in $\lambda$ produces
$4\log(s/r)$, rather than an unbounded $\log r$ term.  Uniformly
for $r\ge1$, $-1\le\lambda\le0$ and $k=0,1$,
\begin{equation}\label{us:weyl-gaussian-log}
 e^{r^2/4}\int_r^\infty e^{-s^2/4}
 \left(\frac{s}{r}\right)^{4\lambda}
 \bigl(1+|\log(s/r)|^k\bigr)\,ds\le\frac Cr.
\end{equation}
Together with \eqref{us:weyl-prefactors}, this bounds the two
integrals and their $\lambda$-derivatives by
$Cd_n\|g\|_Y$.  One $r\partial_r$ derivative is also bounded:
Leibniz' rule produces only the lower-endpoint terms and a factor
$O(r^2)$ multiplying the Gaussian tail, which is still bounded
uniformly after the prefactor $r^{-3}$.  The same argument applies
to the $\lambda$-derivative, since $\log(s/r)=0$ at $s=r$.
On $[r_0,1]$, split the Green integrals at $s=1$;
parameter-smooth ODE dependence on that compact interval and the
same Gaussian bounds for the tails prove \eqref{us:weyl-c1-fixed}.

As proved above, $d_n\to0$ for fixed $r_0$.
Write $\widetilde v_{n,\lambda}=T_\lambda v_{n,\lambda}$ and
$\widetilde\psi_\lambda=T_\lambda\psi_\lambda$.  The Volterra
problem becomes
\[
 \widetilde v_{n,\lambda}
 =\widetilde\psi_\lambda
  -\widetilde K_{n,\lambda}\widetilde v_{n,\lambda}.
\]
By \eqref{us:weyl-prefactors},
$\widetilde\psi_\lambda\in C^1([-1,0];Y)$.
A differentiable Neumann series and \eqref{us:weyl-c1-fixed}
therefore imply
\begin{equation}\label{us:weyl-c1-convergence}
 \sup_{-1\le\lambda\le0}\left(
 \|\widetilde v_{n,\lambda}-\widetilde\psi_\lambda\|_Y
 +\|\partial_\lambda(\widetilde v_{n,\lambda}
                              -\widetilde\psi_\lambda)\|_Y
 \right)\le C(r_0)d_n.
\end{equation}
The bounds in \eqref{us:weyl-conj-first}--%
\eqref{us:weyl-conj-second} also give
$T_\lambda(v_{n,\lambda}-\psi_\lambda)=O(r^{-2})$
as $r\to\infty$, and hence
$v_{n,\lambda}/\psi_\lambda\to1$. Since $v_{n,\lambda}=O(r^{2\lambda-2})$ and
$\Phi_n=O(r^{-2})$ at infinity, \eqref{us:mass-weight}
gives $v_{n,\lambda}\in L^2((r_0,\infty),\rho_n\,dr)$.
At $r=r_0$ the conjugating factor is the identity.  The adapted
Cauchy vector of $\psi_\lambda$ is uniformly nonzero on the compact
parameter interval, so its projective phase is a smooth local
function of the Cauchy data.  We have proved, with no derivative
claim in the moving space $X_\lambda$,
\begin{equation}\label{us:actual-outer-phase}
 \vartheta_{\mathrm{out},n}(\lambda)
 =\vartheta[\psi_\lambda](r_0)
  +o_{C^1_\lambda([-1,0])}(1)
 \quad \text{as } n\to\infty.
\end{equation}

Let $u_{n,\lambda}$ be the origin-regular solution of
\[
 (\mathcal A_n-\lambda)u_{n,\lambda}=0,
 \quad
 u_{n,\lambda}(r)=1+O(r^2)
 \quad\text{as }r\to0.
\]
Define $\vartheta_{\mathrm{in},n}(\lambda)
 :=\vartheta[u_{n,\lambda}](r_0)$,
again with a continuous choice of the projective angle.
To estimate its dependence on $\lambda$, write
$r=\mu_ns$, $h(s)=u_{n,\lambda}(\mu_ns)$, and
$q_n=\bar Q+\mu_n^2R_n$. Direct rescaling of the spectral equation gives
\begin{align}
 \mathcal H_{\Q} h
 &=\mu_n^2\{(1-\lambda)h+\tfrac s2h'-2sR_nh'
-(12R_n+2sR_n')h\},\label{us:inner-spectral}\\
 \mathcal H_{\Q}h&:=h''+(4/s+2s\Q)h'+(12\Q+2s\Q')h.\nonumber
\end{align}
For $s\ge1$, the fundamental matrix for first-order system in $(f,sf')$ associated
with $\mathcal H_{\Q}f=0$ has the form $G(s)=s^{-5/2}B(\log s)$, with $B$ and $B^{-1}$
bounded. This follows by inserting \eqref{eq:core-tail-main}: in
logarithmic radius the coefficient matrix is a constant matrix with
eigenvalues $-5/2\pm i\omega_*$ plus an integrable perturbation.
For $1\le s\le r_0/\mu_n$, \eqref{eq:Rn-bound-main} gives
$|R_n(s)|+|sR_n'(s)|\le Cs^{-1/2}$.
Writing $(h,sh')^T=G(s)a(s)$ with $h$ satisfying
\eqref{us:inner-spectral} and using this bound, we obtain
\[
 |a'(s)|\le C\mu_n^2(s+s^{1/2})|a(s)|,
 \qquad
 \int_1^{r_0/\mu_n}\mu_n^2(s+s^{1/2})\,ds
 \le C(r_0^2+\mu_n^{1/2}r_0^{3/2}).
\]
The origin-regular normalization has a nonzero limiting coefficient
vector at $s=1$, so $|a(1)|$ is bounded away from zero uniformly
in $\lambda$ for large $n$. Gronwall's inequality gives $|a(r_0/\mu_n)|\asymp |a(1)|$.
Differentiating \eqref{us:inner-spectral} in $\lambda$ gives
the source term $-\mu_n^2h$. By \eqref{eq:Rn-bound-main} and the
smoothness of $\Q$, $|q_n|+|sq_n'|$ is uniformly bounded on $[0,1]$. Applying Gronwall's inequality first to the differentiated equation
in integrating-factor form on $[0,1]$ and then to the equation for
$\partial_\lambda a$ on $[1,r_0/\mu_n]$, we obtain
$|\partial_\lambda a(1)|\le C\mu_n^2$ and
$|\partial_\lambda a(r_0/\mu_n)|\le Cr_0^2$. After removing the positive factor $(r_0/\mu_n)^{-5/2}$,
the adapted Cauchy vector at $r_0$ is a $\lambda$-independent,
uniformly invertible linear image of $a(r_0/\mu_n)$. The angle derivative formula therefore gives
\begin{equation}\label{us:inner-phase-derivative}
\sup_{[-1,0]}|\partial_\lambda\vartheta_{\mathrm{in},n}|\le Cr_0^2.
\end{equation}

Set $D_n(\lambda)
 :=\vartheta_{\mathrm{out},n}(\lambda)
   -\vartheta_{\mathrm{in},n}(\lambda)$. 
The endpoint conditions imply that $\lambda$ is an eigenvalue
of $\mathcal A_n$ if and only if
$D_n(\lambda)\in\pi\mathbb Z$.
Since $-1$ is an eigenvalue by \eqref{us:mass-scaling},
we choose the continuous lifts so that $D_n(-1)=0$. Equations
\eqref{us:outer-phase}, \eqref{us:gamma-derivative},
\eqref{us:actual-outer-phase}, and \eqref{us:inner-phase-derivative}
give, uniformly for $-1\le\lambda\le0$,
\[
 D_n'(\lambda)=\Theta'(\lambda)+O(r_0^2)+o_n(1)
 -\vartheta_{\mathrm{in},n}'(\lambda)
 \ge\frac{\sqrt7}{4}-o_n(1)-Cr_0^2.
\]
Thus $D_n'>0$ after first choosing $r_0$ small and then $n$ large.
Integrating the same phase estimates from $-1$ to $0$ gives
\[
 D_n(0)=\pi-\arctan\sqrt7+O(r_0^2)+o_n(1)<\pi.
\]
Thus $0<D_n(\lambda)<\pi$ for $-1<\lambda\le0$. An eigenvalue would
require $D_n(\lambda)\in\pi\mathbb Z$, a contradiction.
\end{proof}

\begin{lemma}[Transfer to the density variable]
\label{us:radial-transfer}
In the half-plane $\operatorname{Re} z<\tfrac14$, the following assertions hold.

\begin{enumerate}
\item Every eigenvalue $z$ of $L_{n,0}$ is real and algebraically
simple. If $f$ is a corresponding eigenfunction, then
$h=\mathcal T^{-1}f$ belongs to the operator domain of $\mathcal A_n$ and satisfies
\[
\mathcal A_n h=zh.
\]

\item Conversely, if
\[
\mathcal A_n h=\lambda h,
\qquad \lambda<\tfrac14,
\]
then $f=\mathcal T h$ belongs to $D(L_{n,0})$ and satisfies
\[
L_{n,0}f=\lambda f.
\]

\item The maps $\mathcal T$ and $\mathcal T^{-1}$ identify the
corresponding eigenspaces and preserve algebraic multiplicity.
\end{enumerate}
\end{lemma}

\begin{proof}
The proof is given in
Appendix~\ref{app:proof-radialunweightedspectrum}.
\end{proof}

\begin{proposition}[Radial unstable modes]
\label{us:radial-result}
For all sufficiently large $n$, there are exactly $N_n$
algebraically simple radial eigenvalues
\[
    -\mu_{1,n}
    <
    \cdots
    <
    -\mu_{N_n,n}
    <
    -1.
\]
Together with the simple scaling eigenvalue $-1$, they exhaust the
radial spectrum in $\{\operatorname{Re}z\leq0\}$. Moreover,
\[
    \ker(L_{n,0}+1)
    =
    \operatorname{span}\{\Lambda U_n\},
\]
and, with the consecutive matching-interval numbering,
\[
    N_n=n+N_0.
\]
\end{proposition}

\begin{proof}
By Lemma~\ref{us:radial-mass-result}, the mass operator
$\mathcal A_n$ has exactly $N_n$ simple eigenvalues below $-1$,
the eigenvalue $-1$ is simple with eigenfunction $\Lambda\Phi_n$,
and there is no eigenvalue in $(-1,0]$.
Lemma~\ref{us:radial-transfer} transfers these eigenvalues,
including their algebraic multiplicities, to the radial density
operator $L_{n,0}$. Since
\[
    \mathcal T\Lambda\Phi_n=\Lambda U_n,
\]
the eigenvalue $-1$ has eigenspace
$\operatorname{span}\{\Lambda U_n\}$.
Finally, Lemma~\ref{us:fixed-mode-fredholm} excludes any additional radial
spectrum in $\{\operatorname{Re}z\leq0\}$, while
Lemma~\ref{us:radial-zero-enumeration} gives
$N_n=n+N_0$.
\end{proof}

\subsection{The translation sector}
\label{us:translation}

The degree-one argument uses the wave operator method introduced for the
Oseen problem in \cite{LiWeiZhang2020} and firstly adapted to the Keller--Segel equation in
\cite{LiZhou2026}. Using the profile equation and the estimates from
Section~\ref{sec:matched-profiles}, we verify the reduction for the present profiles and prove
the required potential bound directly.

\emph{Elimination of the Newton term.}
For $j>0$ set
\begin{equation*}
 D_j=\partial_r+\frac jr,\qquad
 D_j^{-1}u=r^{-j}\int_0^r u(s)s^j\,ds.
\end{equation*}
When $j\le0$, the inverse used below is
$D_j^{-1}u=-r^{-j}\int_r^\infty u(s)s^j\,ds$.
Integration by parts gives
\begin{equation}\label{us:denominator}
 K_n:=D_3^{-1}U_n'=U_n-\frac3rP_n=2r\Phi_n'.
\end{equation}
By Proposition~\ref{thm:matched-family-main}, $\Phi_n'(r)<0$
for all $r>0$, when $n$ is large, so $K_n(r)<0$
on $(0,\infty)$.

Define
\begin{equation*}
 T_nu=u-\frac{U_n'}{K_n}D_3^{-1}u.
\end{equation*}
If $F=D_3^{-1}u$, then $T_nu=0$ is equivalent to
$(F/K_n)'=0$. Thus
\begin{equation}\label{us:wave-kernel}
 \ker T_n=\operatorname{span}\{U_n'\}.
\end{equation}

\begin{lemma}[Exact wave reduction]\label{us:wave-identity}
For regular degree-one functions $u\in D(L_{n,1})$, we have
\begin{equation}\label{us:intertwining}
 T_nL_{n,1}=\widetilde L_{n,1}T_n,
 \qquad
 \widetilde L_{n,1}=-\partial_r^2+A_n\partial_r+B_n^{\rm w},
\end{equation}
where
\begin{equation*}
 A_n=-\frac2r+\frac r2-P_n,\qquad
 B_n^{\rm w}=\frac2{r^2}+1-2U_n-2\left(\frac{U_n'}{K_n}\right)'.
\end{equation*}
\end{lemma}

\begin{proof}
Suppress the index $n$ and write
$F=D_3^{-1}u$, $G=D_0^{-1}u$, $R=U'/K$, and
$C=2/r^2+1-2U$. Differentiating \eqref{us:newton} gives
$\partial_r\Delta_1^{-1}u=(2F+G)/3$. If
$b=1/r+r/2-P$, direct application of $D_3$, using
$P'+2P/r=U$ and $K=U-3P/r$, gives
\begin{equation*}
 D_3^{-1}L_1u=-u'+bu-(1+2U/3)F-KG/3.
\end{equation*}
The integration constant vanishes by regularity at zero. Substitution
in $T L_1u$ cancels $G$ and yields
\begin{equation}\label{us:wave-left}
 T L_1u=-u''+(A+R)u'+(C-Rb)u+R(1+2P/r)F.
\end{equation}
On the other hand, writing $S=3R/r-R'$ and differentiating
$Tu=u-RF$ twice gives
\[
 \widetilde L_1Tu=-u''+(A+R)u'+(C-Rb)u+EF,
 \quad E=-S'+3S/r+AS-(C-2R')R.
\]
To identify the last coefficient, differentiate the stationary
density equation under translation. This gives
$L_1U'=-U'/2$. Equivalently, it follows directly from
\eqref{us:profile-system} and
$\partial_r\Delta_1^{-1}U'=P'$. Set $u=U'$ in the difference of the
last display and \eqref{us:wave-left}. Since $F=K$, $TU'=0$, and
$K\ne0$, one obtains $E=R(1+2P/r)$. This proves the identity.
\end{proof}

Set
\begin{equation*}
 \mathcal M_n(r)=\frac1r\exp\left(\frac{r^2}8-\int_0^rs\Phi_n(s)\,ds\right).
\end{equation*}
Since $\mathcal M_n'/\mathcal M_n=A_n/2$, we remove the first-order term by conjugation as in Section~\ref{us:radial}:
\begin{equation}\label{us:schrodinger}
 H_n=\mathcal M_n^{-1}\widetilde L_{n,1}\mathcal M_n=-\partial_r^2+W_n,
 \qquad W_n=B_n^{\rm w}-\frac12A_n'+\frac14A_n^2.
\end{equation}
In the dimensionless profile variables $y_n=r^2\Phi_n$ and
$\alpha_n=y_n\zeta_n=-r\Phi_n'/\Phi_n$
(see Lemmas~\ref{lem:quantitative-coefficient-transfer}
and~\ref{us:phase-bounds}), the potential decomposes as
\begin{align}
W_n&=r^2A_0(\alpha_n)+B_0(y_n,\alpha_n)+r^{-2}C_0(y_n,\alpha_n),\label{us:potential-split}\\
 A_0(\alpha)&=\frac{\alpha^2-16\alpha+32}{16\alpha^2},\nonumber\\
 B_0(y,\alpha)&=-\frac14\left(11-\frac{32}{\alpha}+2y-\frac{40y}{\alpha}+\frac{96y}{\alpha^2}\right).\label{us:ab-coefficients}
\end{align}
Here $C_0$ is \eqref{us:core-potential-coefficient}. To verify the
decomposition without differentiating uncontrolled remainders, use
\begin{align*}
 r y_n'&=(2-\alpha_n)y_n,\\
 r\alpha_n'&=(3-\alpha_n)(2y_n-\alpha_n)+\tfrac12r^2(\alpha_n-2),\\
 \frac{U_n'}{K_n}
 &=r\left(\frac12-\frac1{\alpha_n}\right)
                      +\frac1r\left(\frac{6y_n}{\alpha_n}-2y_n\right).
\end{align*}
Differentiating the last expression with the first two identities and
collecting the powers $r^2,1,r^{-2}$ in \eqref{us:schrodinger} gives
\eqref{us:potential-split}--\eqref{us:ab-coefficients}.

\begin{lemma}[Uniform positivity of the eliminated operator]\label{us:positive-H}
There is $c_w>0$, independent of sufficiently large $n$, such that
$W_n(r)\ge c_w$ for all $r>0$. The Friedrichs realization of $H_n$
on $L^2(0,\infty)$ is limit-point at both endpoints and satisfies
$H_n\ge c_wI$.
\end{lemma}

\begin{proof}
In the inner region, Lemma~\ref{lem:quantitative-coefficient-transfer}
shows that $(y_n,\alpha_n/y_n)$ is a uniformly small perturbation
of $(y,\zeta)$, including at $s=0$ by continuous extension. This pair stays in a compact set with $\zeta$ bounded
away from zero. Formula \eqref{us:c-zeta} and
Lemma~\ref{us:phase-bounds} therefore give $C_0(y_n,\alpha_n)\ge1/40$.
Furthermore
\[
 A_0(\alpha)+\frac1{16}=\frac{(\alpha-4)^2}{8\alpha^2}\ge0,
 \qquad B_0(y_n,\alpha_n)\ge-C(1+s^{-2}),\quad s=r/\mu_n.
\]
The latter follows from $\zeta\ge c>0$ and
$y^{-1}\le C(1+s^{-2})$, which follow from the phase bounds and
the endpoint expansions. Consequently,
\begin{equation*}
 W_n(r)\ge\frac1{40r^2}-C-C\frac{\mu_n^2}{r^2}-\frac{r^2}{16}
    \quad(0<r\le r_0).
\end{equation*}
Choose $r_0$ small and then $n$ large so that this has a fixed
positive lower bound.

For $r\ge r_0$, the limiting potential corresponding to the singular profile
$\Phi_*(r)=r^{-2}$ is
\begin{equation*}
 W_*(r)=\frac{r^2}{16}+\frac2{r^2}-\frac14,
 \qquad \inf_{r>0}W_* =\frac1{\sqrt2}-\frac14>0.
\end{equation*}
On $[r_0,1]$, \eqref{eq:outer-scale-limits-main} yields
$\Phi_n\to\Phi_*$ in $C^1([r_0,1])$, and hence
$W_n\to W_*$ uniformly.
For $r\ge1$,
\eqref{eq:outer-main} and Lemma~\ref{lem:outer-derivatives}  give
$y_n=1+O(\varepsilon_n)$ and
$\alpha_n=2+O(\varepsilon_nr^{-2})$.In particular, $r^2[A_0(\alpha_n)-A_0(2)]=O(\varepsilon_n)$,
and the smoothness of $B_0,C_0$ near $(1,2)$ gives
$\sup_{r\ge r_0}|W_n-W_*|\to0$ by \eqref{us:potential-split}. This proves the uniform positive
bound on the whole half-line.

For fixed $n$, write
$\Phi_n(r)=\varphi_0+\varphi_2r^2+O(r^4)$ near zero, with
$\varphi_2<0$. Then $U_n'/K_n=5/r+O(r)$, and
\eqref{us:schrodinger} gives $W_n=12/r^2+O(1)$.
Its indicial roots are $4$ and $-3$, so zero is limit-point. At
infinity $W_n=r^2/16+O(1)$ and the other endpoint is limit-point
as well. The Friedrichs operator is the maximal $L^2$ realization
of this expression and inherits the lower bound.
\end{proof}
\emph{Domain transfer and algebraic multiplicity.}
The following Hardy estimate provides the $L^2$ bound needed
for the domain transfer and will also be used for the
degree-two transformation.

\begin{lemma}[A Hardy estimate for the integral transforms]\label{us:hardy}
For $j>1/2$,
\begin{equation}\label{us:hardy-bound}
 \|D_j^{-1}u\|_{L^2(dr)}\le\frac2{2j-1}\|u\|_{L^2(r^2\,dr)}.
\end{equation}
In particular, $\mathcal W_n=\mathcal M_n^{-1}T_n$ extends to a bounded map
from $L^2(r^2\,dr)$ to $L^2(dr)$. For every generalized
eigenfunction chain of $L_{n,1}$, its image under $\mathcal W_n$
is a chain in the operator domain of $H_n$.
\end{lemma}

\begin{proof}
Take $u\in C_c^\infty(0,\infty)$ and set $g=D_j^{-1}u$.
Then $g$ vanishes near zero and $g=O(r^{-j})$ at infinity.
When integrating $(r|g|^2)'$ over $(0,\infty)$, the boundary
terms vanish for $j>1/2$. Using $g'=u-jg/r$ gives
\[
 (2j-1)\|g\|_{L^2(dr)}^2
   =2\operatorname{Re}\int_0^\infty ru\bar g\,dr.
\]
Cauchy--Schwarz and density prove \eqref{us:hardy-bound}.

The profile expansions at zero and infinity give
$|U_n'/K_n|\le C_n/r$ and $|\mathcal M_n^{-1}|\le C_nr$ globally.
Thus $T_n$ is bounded on $L^2(r^2\,dr)$ by
\eqref{us:hardy-bound} with $j=3$, and $\mathcal W_n$ has the
asserted bound. If
$(L_{n,1}-z)u_k=u_{k-1}$, with $u_{-1}=0$, the local identity
\eqref{us:intertwining} gives, in distributions on $(0,\infty)$,
\[
 (H_n-z)\mathcal W_nu_k=\mathcal W_nu_{k-1}.
\]
Both functions are in $L^2(dr)$. The maximal-domain characterization
in Lemma~\ref{us:positive-H} therefore places $\mathcal W_nu_k$
in $D(H_n)$. Induction proves the assertion for the entire chain.
\end{proof}

\begin{proposition}[The degree-one spectrum]\label{us:translation-result}
There is $\delta_1>0$, independent of sufficiently large $n$, such
that
\begin{equation}\label{us:l1-spectrum}
 \sigma(L_{n,1})\cap\{\operatorname{Re}z<\delta_1\}=\{-\tfrac12\}.
\end{equation}
On each fixed $m$ component, $-\tfrac12$ is algebraically simple
with eigenfunction $U_n'$.
\end{proposition}

\begin{proof}
Take $\delta_1=\tfrac12\min\{c_w,1/4\}$.
If $L_{n,1}u=zu$ and $\operatorname{Re}z<\delta_1$, then either
$T_nu=0$, in which case \eqref{us:wave-kernel} gives $z=-1/2$,
or the nonzero function $p=\mathcal W_nu$ belongs to $D(H_n)$ and
satisfies $H_np=zp$. The latter is impossible because $H_n\ge c_w$.
This also proves $\ker(L_{n,1}+1/2)=\operatorname{span}\{U_n'\}$.
If a Jordan vector existed, it could be normalized to satisfy
$(L_{n,1}+1/2)u_1=U_n'$. Applying $\mathcal W_n$ would give
$(H_n+1/2)\mathcal W_nu_1=0$, hence $T_nu_1=0$ and
$u_1\in\operatorname{span}\{U_n'\}$, a contradiction. The Fredholm
reduction upgrades this point-spectrum statement to
\eqref{us:l1-spectrum} and gives algebraic simplicity.
\end{proof}

\subsection{Higher angular modes}
\label{us:high-modes}

For $l\ge2$ set $w=\Delta_l^{-1}u$. Integration by parts in
\eqref{us:mode-operator} gives the exact real-part identity
\begin{align}
 \mathfrak q_{n,l}[u]
 :=\operatorname{Re}\langle L_{n,l}u,u\rangle_{L^2(r^2\,dr)}
 &=\int_0^\infty
 \left(r^2|u'|^2+J_l|u|^2+\frac14r^2|u|^2
                          -\frac32r^2U_n|u|^2\right)\,dr\nonumber\\
 &\quad+\frac12\int_0^\infty r^2
             \left(U_n''-\frac2rU_n'\right)|w'|^2\,dr
       -\frac{J_l}{2}\int_0^\infty U_n''|w|^2\,dr.
                                                    \label{us:exact-high-form}
\end{align}
For the nonlocal term, replace $u$ by $\Delta_lw$ before integrating:
\[
 -\operatorname{Re}\int_0^\infty U_n'w'\overline{\Delta_lw}\,r^2\,dr
 =\frac12\int_0^\infty r^2(U_n''-2U_n'/r)|w'|^2\,dr
       -\frac{J_l}{2}\int_0^\infty U_n''|w|^2\,dr.
\]
The two local drift contributions are $\frac14\|u\|_{L^2(r^2dr)}^2$
and $\frac12\int r^2U_n|u|^2$, respectively, which explains the
coefficient $-3/2$ in \eqref{us:exact-high-form}.

Set
\begin{equation*}
 t=\log r,\qquad F(t)=r^{1/2}u(r),\qquad v(t)=r^{-3/2}w(r),
 \qquad D=\partial_t.
\end{equation*}
Then
\begin{equation*}
 F=\mathcal M_lv,\qquad
 \mathcal M_l=D^2+4D+\frac{15}{4}-J_l
             =(D-l+\tfrac32)(D+l+\tfrac52).
\end{equation*}
Its inverse on $L^2(\mathbb R,dt)$ has the exponentially decaying
Green kernel
\[
 G_l(t)=-\frac1{2l+1}
 \begin{cases}e^{(l-3/2)t},&t<0,\\e^{-(l+5/2)t},&t>0.
 \end{cases}
\]
This is exactly the inverse selected by \eqref{us:newton}. If
$c_l=J_l-15/4>0$, Fourier transformation gives
\begin{equation}\label{us:mellin-bounds}
 \|v\|_2^2\le c_l^{-2}\|F\|_2^2,\qquad
 \|(D+\tfrac32)v\|_2^2\le c_l^{-1}\|F\|_2^2.
\end{equation}
Indeed, the denominator of the squared multiplier is
$(\xi^2+c_l)^2+16\xi^2$, and
$(\xi^2+9/4)/[(\xi^2+c_l)^2+16\xi^2]\le1/c_l$ for $l\ge2$.
In this and the following Mellin calculations, unweighted norms refer
to $L^2(dt)$.

Write
\begin{equation*}
 h_n=r^2U_n,\qquad
 \mathfrak a_n=r^4(U_n''-2U_n'/r),\qquad
 \mathfrak b_n=r^4U_n''.
\end{equation*}
Equation \eqref{us:exact-high-form} becomes
\begin{align}
 \mathfrak q_{n,l}[u]
 &=\|(D-\tfrac12)F\|_2^2
   +\int(J_l-\tfrac32h_n)|F|^2\,dt
   +\frac14\|e^tF\|_2^2\nonumber\\
 &\quad+\frac12\int\mathfrak a_n|(D+\tfrac32)v|^2\,dt
   -\frac{J_l}{2}\int\mathfrak b_n|v|^2\,dt.
                                                     \label{us:mellin-form}
\end{align}
The universal limits of these coefficients can be estimated on the
coarse region \eqref{us:coarse-region}. The identities
$Q'=-PQ$ and \eqref{us:autonomous} give
\begin{equation*}
 \bar h=2y\beta\le\frac{15}{4},\qquad
 \bar{\mathfrak a}=4y^2\beta(4+2y-\beta)\ge0,\qquad
 \bar{\mathfrak b}=4y^2\beta(2+2y-\beta)\le\frac{1300}{27}.
\end{equation*}
For the last bound, maximize the concave quadratic in $\beta$.
For $0\le y\le10/11$ its vertex gives
$4y^2(1+y)^2\le4(10/11)^2(21/11)^2<1300/27$.
For $10/11\le y\le5/3$ its maximum is
$4y^2(3-6y/5)(16y/5-1)$; its derivative is
$\frac{24y}{25}(-64y^2+135y-25)>0$ on this interval.
Its endpoint value is $1300/27$.

The quantitative comparison in
Lemma~\ref{lem:quantitative-coefficient-transfer}, combined with the
core phase-plane bounds, now implies that, for each $\eta>0$ and
all sufficiently large $n$,
\begin{equation}\label{us:global-high-coefficients}
 h_n\le15/4+\eta,\qquad \mathfrak a_n\ge-\eta,\qquad
 \mathfrak b_n\le1300/27+\eta
 \quad\text{on }(0,\infty).
\end{equation}
The singular outer coefficients are $2,20,12$, respectively, and
satisfy the same bounds.

\begin{proposition}[Exclusion in the higher angular degrees]\label{us:high-result}
For sufficiently large $n$, there exists a constant $c_3>0$, independent of $l\ge3$
and $n$, such that
\begin{equation}\label{us:high-coercivity}
 \mathfrak q_{n,l}[u]\ge c_3\|u\|_{L^2(dr)}^2
                  +\frac14\|u\|_{L^2(r^2\,dr)}^2.
\end{equation}
Consequently, $\sigma(L_{n,l})\cap\{\operatorname{Re}z<1/4\}=\varnothing$
for every $l\ge3$.
\end{proposition}

\begin{proof}
Insert \eqref{us:global-high-coefficients} and
\eqref{us:mellin-bounds} into \eqref{us:mellin-form}. After discarding
the first nonnegative term, the coefficient of $\|F\|_2^2$ is at least
\[
 m_\eta(J)=J-\frac32(15/4+\eta)
       -\frac{\eta}{2(J-15/4)}
       -\frac{J(1300/27+\eta)}{2(J-15/4)^2}.
\]
This function increases for $J\ge12$. At $J=12$,
\[
 m_0(12)=12-\frac{45}{8}-\frac{6}{(33/4)^2}\frac{1300}{27}
        =\frac{167051}{78408}>0,
 \quad
 0\le m_0(J)-m_\eta(J)\le\frac53\eta.
\]
Choose $\eta$ so that the loss is at most $m_0(12)/2$, and set
$c_3=m_0(12)/2$. This proves \eqref{us:high-coercivity} on a smooth
core and then on the operator domain by approximation. The Newton
terms pass to the limit by \eqref{us:mellin-bounds}; the required
$L^2(dr)$ control of $u$ follows from the angular part of its $H^1$
norm. Applying the inequality to an eigenfunction excludes
$\operatorname{Re}z<1/4$. Lemma~\ref{us:fixed-mode-fredholm} gives the full
spectral conclusion.
\end{proof}

\emph{The borderline mode $l=2$.}
\label{us:degree-two}
The preceding coarse bound is insufficient at $J_2=6$. We use a
partially localized operator whose weight agrees with $r$ at the
core and is bounded at infinity. The transition of the weight is
kept inside the exact quadratic form; the Newton tails are retained
when the inner and outer estimates are joined.

The factorization $\Delta_l=D_{l+2}D_{-l}$ and the Newton inverse
give the identity
\begin{align}
 D_{l+2}^{-1}L_{n,l}D_{l+2}
 &=-\partial_r^2-\frac2r\partial_r+\frac{(l+1)(l+2)}{r^2}
       +\frac12(r\partial_r+1)-P_nD_2-U_n\nonumber\\
 &\quad+lD_{l+2}^{-1}(V_{1,n}+V_{2,n}D_{-l}^{-1}),\label{us:partial-localization}\\
 V_{1,n}&=-(P_n/r)',\qquad V_{2,n}=-U_n'/r.\nonumber
\end{align}
For verification, commute the first-order inverse through a
multiplication by using
$D_j^{-1}VD_j=V-D_j^{-1}V'$ and use
$[r\partial_r,D_j]=-D_j$ and $D_2P_n=U_n$.
These identities give all the terms in
\eqref{us:partial-localization}, including the factor $l$.
This is the profile-independent identity in \cite[arXiv v2, Lemma~2.7]{LiZhou2026}.

For $l=2$ fix
\begin{equation}\label{us:fixed-weight}
 m(r)=\frac r{\sqrt{1+r^2}},\qquad
 \gamma=\frac{rm'}m=\frac1{1+r^2},\qquad
 \mathcal B_n=mD_4^{-1}L_{n,2}D_4m^{-1}.
\end{equation}
The last formula first denotes a differential--integral expression;
no assertion of bounded similarity is made or needed. If
$K_+=mD_4^{-1}m^{-1}$ and $K_-=mD_{-2}^{-1}m^{-1}$, direct
conjugation gives
\begin{align}
 \mathcal B_nf
 &=-f''+\frac{2(\gamma-1)}r f'
        +\frac{12+\gamma-\gamma^2+\dot\gamma}{r^2}f
        +\frac12[rf'+(1-\gamma)f]\nonumber\\
 &\quad-P_n\left(f'+\frac{2-\gamma}{r}f\right)-U_nf
        +2K_+(V_{1,n}f+V_{2,n}K_-f),\label{us:variable-operator}
\end{align}
where $\dot\gamma=r\gamma'$. The kernels are exactly
\[
 K_+h=m(r)r^{-4}\int_0^rm(s)^{-1}h(s)s^4\,ds,
 \qquad
 K_-h=-m(r)r^2\int_r^\infty m(s)^{-1}h(s)s^{-2}\,ds.
\]

Put $t=\log r$ and $f=r^{1/2}H(t)$, and write
\begin{equation*}
 a=\frac52-\gamma,\quad b=\gamma+\frac12,\quad
 \mathsf A=(D+a)^{-1},\quad \mathsf B=(D-b)^{-1},
 \qquad D=\partial_t.
\end{equation*}
The endpoint conditions are specified by
\begin{equation}\label{us:volterra-kernels}
 \mathsf Ah(t)=\int_{-\infty}^t e^{-\int_s^ta(\tau)d\tau}\,h(s)\,ds,
 \qquad
 \mathsf Bh(t)=-\int_t^\infty e^{-\int_t^sb(\tau)d\tau}\,h(s)\,ds.
\end{equation}
Since $3/2\le a\le5/2$ and $1/2\le b\le3/2$,
\begin{equation}\label{us:volterra-bounds}
 \|\mathsf A\|_{2\to2}\le\frac23,\qquad
 \|\mathsf B\|_{2\to2}\le2.
\end{equation}
Define
\begin{equation}\label{us:l2-coefficients}
 p_n=rP_n,\quad h_n=r^2U_n,\quad
 c_{1,n}=r^3V_{1,n}=3p_n-h_n,\quad
 c_{2,n}=r^4V_{2,n}=-r^3U_n'.
\end{equation}
In particular $\dot p_n=h_n-p_n$. Pairing
\eqref{us:variable-operator} with $f$ yields
\begin{align}
 \mathfrak b_n[H]:=\operatorname{Re}\langle\mathcal B_nf,f\rangle_{L^2(dr)}
 &=\|H'\|_2^2+\int W_{n,\gamma}|H|^2\,dt\nonumber\\
 &\quad+2\operatorname{Re}\langle
       \mathsf A(c_{1,n}H+c_{2,n}\mathsf BH),H\rangle
       +\int d_\gamma r^2|H|^2\,dt,\label{us:variable-form}\\
 W_{n,\gamma}&=\frac{45}{4}+2\gamma-\gamma^2-ap_n-h_n
                           +\frac12\dot p_n,
 \qquad d_\gamma=\frac{1-2\gamma}{4}.\nonumber
\end{align}
For clarity, the homogeneous local expression before taking the real
part is
$-D^2+2(\gamma-1)D+45/4+2\gamma-\gamma^2+\dot\gamma
-p_nD-ap_n-h_n$. Integration by parts cancels $\dot\gamma$
and adds $\dot p_n/2$. The kernel formulas give
$K_+(r^{-3}f)=r^{-3/2}\mathsf AH$ and
$K_-f=r^{3/2}\mathsf BH$, which account for both nonlocal terms.
The drift is paired with the measure $r^2dt$, giving $d_\gamma$.

For the universal core, remove the self-similar drift and use the
constant exponent $\gamma=1$. Denote this transformed operator by
$\overline{\mathcal B}^{(1)}$. Put $g=D_{-3}^{-1}f$ and
$g=r^{3/2}k(t)$; then
\begin{equation*}
 H=(D-\tfrac32)k,\qquad
 f'=r^{-1/2}(D+\tfrac12)H.
\end{equation*}
Here $(D_3^{-1})^*=-D_{-3}^{-1}$ in $L^2(dr)$. Thus the Newton
terms in \eqref{us:partial-localization} become local in $g$:
\[
 \operatorname{Re}\langle D_3^{-1}(V_1f),f\rangle
 =\int(\tfrac12V_1'+3V_1/r)|g|^2\,dr,
 \qquad
 \operatorname{Re}\langle D_3^{-1}(V_2g),f\rangle
 =-\int V_2|g|^2\,dr.
\]
Using $P=2y/r$, $Q=2y\beta/r^2$, and
\eqref{us:autonomous}, the exact real-part form is
\begin{align}
 \operatorname{Re}\langle\overline{\mathcal B}^{(1)}f,f\rangle
 &=\|(D+\tfrac12)H\|_2^2
       +\int(12-y\beta-4y)|H|^2\,dt\nonumber\\
 &\quad+\int4y[3-(1+y)\beta]|k|^2\,dt.\label{us:core-l2-form}
\end{align}
The coarse region \eqref{us:coarse-region} implies
\[
 12-y\beta-4y\ge\frac{11}{3},\qquad
 4y[3-(1+y)\beta]\ge\frac{12}{5}y^2(2y-3)\ge-\frac{12}{5}.
\]
Also $\|H\|_2^2=\|k'\|_2^2+\frac94\|k\|_2^2$. Therefore
\begin{equation}\label{us:core-l2-coercivity}
 \operatorname{Re}\langle\overline{\mathcal B}^{(1)}f,f\rangle
 \ge\|H'\|_2^2+\left(\frac14+\frac{11}{3}-\frac{16}{15}\right)\|H\|_2^2
 \ge\|H'\|_2^2+\frac{13}{5}\|H\|_2^2.
\end{equation}
All integrations are justified first for compactly supported $f$;
the tail of $D_{-3}^{-1}f$ has the decay required for these boundary
terms.

For $U_*=2r^{-2}$ and $P_*=2r^{-1}$, the coefficients in
\eqref{us:l2-coefficients} are $p_*=h_*=2$ and $c_{1,*}=c_{2,*}=4$.
The homogeneous part of \eqref{us:variable-form} is
\begin{equation}\label{us:singular-form}
 Q_\gamma[H]=\|H'\|_2^2+\int(2+5b-b^2)|H|^2\,dt
       +8\operatorname{Re}\langle\mathsf AH,H\rangle
       +8\operatorname{Re}\langle\mathsf A\mathsf BH,H\rangle.
\end{equation}
The relation $a+b=3$ permits an exact localization of this Newton
term. Let $Z=\mathsf A\mathsf BH$ and $Y=(D-b)Z$. Because
$(D+a)(D-b)=(D-b)(D+a)$, one has
\begin{equation}\label{us:singular-localization}
 H=(D+a)Y,\quad Y=\mathsf AH,\quad
 \operatorname{Re}\langle\mathsf AH,H\rangle=\int a|Y|^2\,dt,
 \quad
 \operatorname{Re}\langle Z,H\rangle
     =-\|Z'\|_2^2-\int ab|Z|^2\,dt.
\end{equation}
The uniqueness of the bounded inverses in
\eqref{us:volterra-kernels} justifies the first two identities.
The others follow by integration by parts; the derivatives of $a$
and $b$ cancel because $a'=-b'$.

For the fixed weight \eqref{us:fixed-weight},
\begin{equation}\label{us:b-derivative}
 b'=-2(\tfrac32-b)(b-\tfrac12),\qquad a=3-b\ge b>0,
 \qquad a'\ge0.
\end{equation}
Expanding $Y=Z'-bZ$ gives
\[
 \int\frac ab|Y|^2\,dt
 =\int(|Z'|^2+ab|Z|^2)\,dt
   +\int\left(\frac ab-1\right)|Z'|^2\,dt+\int a'|Z|^2\,dt.
\]
The last two terms are nonnegative. Set $c(b)=2+5b-b^2$.
Equations \eqref{us:singular-form}--\eqref{us:singular-localization}
then yield
\begin{align*}
 Q_\gamma[H]&\ge\|H'\|_2^2+\int c|Y'|^2\,dt+\int\{ca^2-(ca)'+8a(1-1/b)\}|Y|^2\,dt,\\
 \|H\|_2^2&=\|Y'\|_2^2+\int(a^2-a')|Y|^2\,dt.
\end{align*}
Since $c\ge17/4$, it remains to compare the zeroth-order
coefficients. Substitution of \eqref{us:b-derivative} gives the
exact identity
\begin{align*}
 ca^2-(ca)'+8a(1-1/b)-\frac{21}{20}(a^2-a')
 &=\frac{2b-1}{40b}\,
 (960+1141b-2178b^2+1030b^3-140b^4).
\end{align*}
On $1/2\le b\le3/2$, the polynomial in parentheses has Bernstein
coefficients
\begin{equation*}
 1106,\quad 8179/8,\quad 3193/4,\quad 4969/8,\quad 1077/2
\end{equation*}
in the variable $b-1/2\in[0,1]$. Thus
\begin{equation*}
 Q_\gamma[H]\ge\|H'\|_2^2+\frac{21}{20}\|H\|_2^2.
\end{equation*}
Finally,
\[
 d_\gamma-\frac14=-\frac1{2(1+r^2)},\qquad
 \int(d_\gamma-\tfrac14)r^2|H|^2\,dt\ge-\frac12\|H\|_2^2.
\]
The full singular form therefore satisfies
\begin{equation}\label{us:outer-l2-coercivity}
 \mathfrak b_*[H]\ge\frac14\|f\|_{L^2(dr)}^2
                        +\|H'\|_2^2+\frac{11}{20}\|H\|_2^2.
\end{equation}

Lemma~\ref{lem:quantitative-coefficient-transfer} supplies a uniform
error bound for the following core limits; in particular, on each
fixed rescaled interval,
\begin{equation*}
 (p_n,h_n,c_{1,n},c_{2,n})(\mu_ns)
 \longrightarrow(sP,s^2Q,3sP-s^2Q,-s^3Q'),
\end{equation*}
and, by the explicit overlap estimate
\eqref{eq:quantitative-overlap-transfer},
\begin{equation}\label{us:l2-outer-limits}
 \lim_{R\to\infty}\limsup_{n\to\infty}\sup_{r\ge R\mu_n}
 (|p_n-2|+|h_n-2|+|c_{1,n}-4|+|c_{2,n}-4|)=0.
\end{equation}
All these coefficients are globally bounded, uniformly for large $n$.

\begin{lemma}[Regional estimates, including the Newton tail]\label{us:regional}
For the fixed weight \eqref{us:fixed-weight}, the following hold.
For every fixed $S$, all sufficiently large $n$ satisfy
\begin{equation}\label{us:pure-inner}
 \mathfrak b_n[H]\ge\frac14\|f\|_{L^2(dr)}^2+\|H\|_2^2
 \quad\text{if }\operatorname{supp}f\subset(0,S\mu_n].
\end{equation}
There is $R_*>0$ such that, for every fixed $R\ge R_*$ and all
sufficiently large $n$,
\begin{equation}\label{us:pure-outer}
 \mathfrak b_n[H]\ge\frac14\|f\|_{L^2(dr)}^2+\frac{11}{40}\|H\|_2^2
 \quad\text{if }\operatorname{supp}f\subset[R\mu_n,\infty).
\end{equation}
\end{lemma}

\begin{proof}
On the inner support $r\le S\mu_n$,
$|\gamma-1|+|\dot\gamma|\le CS^2\mu_n^2$.
The local homogeneous coefficients in \eqref{us:variable-form}
converge uniformly to those of \eqref{us:core-l2-form}.
The same is true of the nonlocal quadratic terms. To justify this
last assertion precisely, let $T_S=\log(S\mu_n)$.
For a pairing against $H$ supported in $t\le T_S$, the left kernel
of $\mathsf A$ uses only $s\le t\le T_S$. In the second Newton
term, $\mathsf BH(s)$ then uses only points between $s$ and $T_S$.
Thus every kernel path in the pairing remains in the same inner
region. The difference of two exponential kernels is bounded by
the coefficient difference times the path length and an
exponentially decaying kernel. Young's inequality, or the
resolvent identities on the truncated half-line, therefore bounds
the total error by $o_{n,S}(1)\|H\|_2^2$.

Using \eqref{us:core-l2-coercivity} and $d_\gamma\ge-1/4$, we obtain
\[
 \mathfrak b_n[H]-\frac14\|f\|_{L^2(dr)}^2
 \ge\left(\frac{13}{5}-o_{n,S}(1)-\frac12S^2\mu_n^2\right)\|H\|_2^2.
\]
This proves \eqref{us:pure-inner}. The comparison above concerns
the restricted pairings determined by the inner support.

For the outer comparison, the weight and the two Volterra operators
are exactly those of the singular form. Only the profile
coefficients differ. Put $t_0=\log(R\mu_n)$ and
$t_1=t_0-\frac12\log R=\log(\sqrt R\mu_n)$.
If $H$ vanishes below $t_0$, then $Z=\mathsf BH$ need not vanish
there. Formula \eqref{us:volterra-kernels}, with $b_0=1/2$, gives
\[
 |Z(t_0)|\le(2b_0)^{-1/2}\|H\|_2,
 \qquad |Z(t)|\le e^{-b_0(t_0-t)}|Z(t_0)|\quad(t<t_0),
\]
and hence
\begin{equation}\label{us:newton-tail}
 \|\mathbf1_{(-\infty,t_1)}Z\|_2\le C R^{-1/4}\|H\|_2.
\end{equation}
Let $\epsilon_{R,n}$ be the supremum of the coefficient differences
in \eqref{us:l2-outer-limits} on $r\ge\sqrt R\mu_n$.
The quantitative overlap estimate \eqref{eq:quantitative-overlap-transfer}
gives, for fixed $R$ and all sufficiently large $n$,
\[
 \epsilon_{R,n}\le C R^{-1/4}
       +C\mu_n^{1/2}r_0^{3/2}+C(r_0)|\varepsilon_n|.
\]
The local error is at most $C\epsilon_{R,n}\|H\|_2^2$ because
$\dot p_n=h_n-p_n$. The first Newton error contains
$(c_{1,n}-4)H$ and obeys the same bound. For the second, split
$(c_{2,n}-4)Z$ at $t_1$. On its right use $\epsilon_{R,n}$; on
its left use the uniform global coefficient bound and
\eqref{us:newton-tail}. It follows from
\eqref{us:volterra-bounds} that
\begin{equation*}
 |\mathfrak b_n[H]-\mathfrak b_*[H]|
 \le C(\epsilon_{R,n}+R^{-1/4})\|H\|_2^2.
\end{equation*}
First take $R$ large and then $n$ large so that this coefficient is
at most $11/40$. Equation \eqref{us:outer-l2-coercivity} proves
\eqref{us:pure-outer}.
\end{proof}

\begin{lemma}[A quadratic partition and its commutators]\label{us:ims}
Let $\chi,\eta$ be real smooth functions of $t$ with
$\chi^2+\eta^2=1$ and
$\|\chi'\|_\infty+\|\eta'\|_\infty\le L$. Then
\begin{equation}\label{us:ims-inequality}
 \mathfrak b_n[H]\ge\mathfrak b_n[\chi H]+\mathfrak b_n[\eta H]
                                  -(L^2+CL)\|H\|_2^2,
\end{equation}
where $C$ is independent of sufficiently large $n$.
\end{lemma}

\begin{proof}
All multiplication terms in \eqref{us:variable-form} add exactly.
The derivative term obeys
\[
 \|(\chi H)'\|_2^2+\|(\eta H)'\|_2^2
 =\|H'\|_2^2+\int(|\chi'|^2+|\eta'|^2)|H|^2\,dt.
\]
For $j=\chi,\eta$, the first-order inverses satisfy
\[
 [j,\mathsf A]=\mathsf A j'\mathsf A,
 \qquad [j,\mathsf B]=\mathsf B j'\mathsf B.
\]
Hence each commutator with
$\mathsf A c_{1,n}$ or $\mathsf A c_{2,n}\mathsf B$ has norm
at most $C\|j'\|_\infty$, by \eqref{us:volterra-bounds} and the
uniform coefficient bounds. The identity
$\sum_{j=\chi,\eta}jTj-T=\sum_{j=\chi,\eta}j[T,j]$
then bounds the nonlocal localization error by $CL\|H\|_2^2$.
\end{proof}

\begin{proposition}[The global degree-two form]\label{us:global-l2-form}
For all sufficiently large $n$,
\begin{equation}\label{us:global-degree-two-bound}
 \mathfrak b_n[H]\ge\frac14\|f\|_{L^2(dr)}^2
                         +\frac{11}{80}\|H\|_2^2,
 \qquad f=r^{1/2}H(\log r),
\end{equation}
initially on smooth compactly supported functions and then on the
energy domain
\begin{equation}\label{us:energy-domain}
 \mathcal Q=\{f:H=r^{-1/2}f\in H^1(\mathbb R,dt),\ rH\in L^2(dt)\}.
\end{equation}
\end{proposition}

\begin{proof}
Choose a smooth angle function $\theta$ which increases from $0$ to
$\pi/2$ on an interval of length $M$, with
$\|\theta'\|_\infty\le C/M$. Set $\chi=\cos\theta$ and
$\eta=\sin\theta$, translating the interval to start at
$\log(R\mu_n)$. Then
\[
 \operatorname{supp}\chi\subset\{r\le S\mu_n\},\quad
 \operatorname{supp}\eta\subset\{r\ge R\mu_n\},\quad S=Re^M.
\]
Choose $M$ first so that the loss in \eqref{us:ims-inequality} is
at most $11/80$. Next fix $R$ as in Lemma~\ref{us:regional}, and
finally take $n$ large enough for both regional estimates, with
this now fixed $S$. Their smaller homogeneous constant is $11/40$.
Adding them, using $\chi^2+\eta^2=1$, and subtracting the
localization loss gives \eqref{us:global-degree-two-bound}.

All coefficients of the homogeneous multiplication terms are
bounded, the Volterra operators are bounded on $L^2(dt)$, and
$|d_\gamma|\le1/4$. Thus the real form in
\eqref{us:variable-form} is continuous for the norm
$\|H\|_{H^1(dt)}+\|rH\|_{L^2(dt)}$.
Cutting off in $t$ and then mollifying shows that compactly
supported smooth functions are dense in this norm. The inequality
therefore extends to \eqref{us:energy-domain}.
\end{proof}

The transformation $mD_4^{-1}$ maps the original density space
directly into the energy domain. This observation makes a sharp
pointwise expansion of a density eigenfunction at infinity
unnecessary for spectral exclusion.

\begin{lemma}[Bounded transfer into the energy domain]\label{us:l2-domain-transfer}
For every $u\in L^2(r^2\,dr)$, let
$g=D_4^{-1}u$ and $f=mg$. Then $f\in\mathcal Q$, and
\begin{equation}\label{us:l2-transfer-bound}
 \|f\|_{L^2(dr)}+\|f/r\|_{L^2(dr)}+\|f'\|_{L^2(dr)}
 \le C\|u\|_{L^2(r^2\,dr)}.
\end{equation}
If $u\ne0$ and $L_{n,2}u=zu$ in the closed operator domain,
then $f\ne0$ and
\begin{equation}\label{us:l2-eigenform}
 \operatorname{Re}z\,\|f\|_{L^2(dr)}^2=\mathfrak b_n[H].
\end{equation}
\end{lemma}

\begin{proof}
By Lemma~\ref{us:hardy},
$\|g\|_{L^2(dr)}\le(2/7)\|u\|_{L^2(r^2dr)}$.
The weight satisfies $m\le1$, $m/r\le1$, $m\le r$, and
$|m'|\le1$. Since $g'=u-4g/r$,
\[
 f'=m'g+mu-4mg/r,
\]
which proves \eqref{us:l2-transfer-bound}. Moreover,
\begin{equation*}
 \|H\|_{L^2(dt)}=\|f/r\|_{L^2(dr)},\quad
 \|H'\|_{L^2(dt)}=\|f'-f/(2r)\|_{L^2(dr)},\quad
 \|rH\|_{L^2(dt)}=\|f\|_{L^2(dr)}.
\end{equation*}
Thus $f\in\mathcal Q$. The map is injective because $m>0$ and
$D_4g=u$.

For an eigenfunction, the partial-localization identity holds in
 distributions and gives $\mathcal B_nf=zf$ on $(0,\infty)$.
 This follows either by the first-order identities used in
 \eqref{us:partial-localization}, or by local regularization of the
 closed-domain equation; the Newton inverse is always the one in
 \eqref{us:newton}. The transformed equation is first tested against
 compactly supported functions and then, by local $H^1$ approximation,
 against $\rho_N^2\overline f$, where
 $\rho_N=1$ on $[-N,N]$, vanishes outside $[-2N,2N]$, and
 $|\rho_N'|\le C/N$. Here the cutoff is taken in $t=\log r$.
 
 To make the passage to the limit explicit, set
 $H_N=\rho_NH$. Integration by parts in the local part of
 \eqref{us:variable-form} gives the quadratic form with $H_N$,
 modulo errors bounded by
 \[
 \frac{C}{N}\bigl(\|H\|_{H^1(dt)}^2+\|rH\|_{L^2(dt)}^2\bigr).
 \]
 The term involving the self-similar drift has this bound because
 $rH\in L^2(dt)$; no pointwise boundary value at infinity is used.
 For the Newton terms, the resolvent commutator identities
 \[
 [\rho_N,\mathsf A]=\mathsf A\rho_N'\mathsf A,
 \qquad [\rho_N,\mathsf B]=\mathsf B\rho_N'\mathsf B
 \]
 and \eqref{us:volterra-bounds} yield
 $\|[\rho_N,\mathsf A]\|_{2\to2}
 +\|[\rho_N,\mathsf B]\|_{2\to2}\le C/N$.
 The coefficients $c_{1,n},c_{2,n}$ are bounded; therefore the
 commutators of $\rho_N$ with
 $\mathsf A c_{1,n}$ and $\mathsf A c_{2,n}\mathsf B$
 have the same bound. Since $H_N\to H$ in the energy-domain
 norm \eqref{us:energy-domain}, all local and nonlocal form terms
 converge. Finally,
 $\langle\mathcal B_nf,\rho_N^2 f\rangle
 =z\|\rho_N f\|_{L^2(dr)}^2\to
 z\|f\|_{L^2(dr)}^2$.
 Taking real parts gives \eqref{us:l2-eigenform}.
\end{proof}

\begin{proposition}[Exclusion in degree two]\label{us:l2-result}
For all sufficiently large $n$,
\begin{equation}\label{us:l2-spectrum}
 \sigma(L_{n,2})\cap\{\operatorname{Re}z<1/4\}=\varnothing.
\end{equation}
\end{proposition}

\begin{proof}
If $L_{n,2}u=zu$, $u\ne0$, and $\operatorname{Re}z<1/4$, then
Lemma~\ref{us:l2-domain-transfer} and
Proposition~\ref{us:global-l2-form} give
\[
 \operatorname{Re}z\,\|f\|_{L^2(dr)}^2
 \ge\frac14\|f\|_{L^2(dr)}^2+\frac{11}{80}\|H\|_2^2,
\]
with $f\ne0$, a contradiction. Lemma~\ref{us:fixed-mode-fredholm} upgrades
the point-spectrum exclusion to \eqref{us:l2-spectrum}.
\end{proof}

\subsection{The full unstable spectrum}
\label{us:assembly}

\begin{proof}[Proof of Theorem~\ref{thm:main}]
All lower bounds on $n$ in the preceding arguments are independent
of the angular degree: the only infinite collection, $l\ge3$,
is covered by the uniform estimate in
Proposition~\ref{us:high-result}. Fix $n$ above their maximum.

By Lemma~\ref{us:fredholm}, every spectral point of
$ L_n$ in $\operatorname{Re}z<1/4$ is an eigenvalue.
If $ L_nf=zf$ and $f\ne0$, some spherical harmonic
coefficient $f_{l,m}$ is nonzero and satisfies
$L_{n,l}f_{l,m}=zf_{l,m}$ in the corresponding closed domain.
Propositions~\ref{us:high-result} and \ref{us:l2-result} exclude
all $l\ge2$. Proposition~\ref{us:translation-result} leaves only
$-1/2$ in degree one below $\delta_1$, while
Proposition~\ref{us:radial-result} identifies all nonpositive
radial eigenvalues.

Let $\lambda_{+,n}>0$ be the first positive eigenvalue of the
self-adjoint mass operator. It exists because that operator has
compact resolvent, is bounded below, and has no eigenvalue at zero.
The transfer in Lemma~\ref{us:radial-transfer} excludes density
spectrum in $0\le\operatorname{Re}z<\min\{\lambda_{+,n},1/4\}$.
Consequently one may take
\begin{equation*}
 \delta_n^{\rm sp}=\frac12\min\{\lambda_{+,n},\delta_1,1/4\}>0.
\end{equation*}
This proves the asserted description of the spectrum in a right
half-neighborhood of the imaginary axis, including the exclusion of neutral
and nonreal unstable spectrum.

It remains to check the full-space algebraic multiplicities.
The spherical harmonic projections commute with the operator and
therefore preserve generalized eigenspaces. At each $-\mu_{j,n}$
and at $-1$, the eigenfunction is radial and the radial algebraic
multiplicity is one. Any putative generalized vector in another
angular degree would have a nonzero eigenvector at the bottom of
its projected chain, contradicting the exclusions already proved.
Thus these eigenvalues are algebraically simple in the full space.
At $-1/2$, every fixed degree-one component has kernel
$\operatorname{span}\{U_n'\}$ and no Jordan vector. There are exactly
three angular components; the corresponding functions are
$U_n'(r)x_k/r=\partial_kU_n$, $1\le k\le3$. They are linearly
independent. All other angular degrees have empty kernel at this
eigenvalue and hence have no generalized vectors there. Therefore
$-1/2$ is semisimple of multiplicity three.

The negative eigenvalues are isolated and finite in number. If
$\Gamma$ is a positively oriented contour enclosing precisely
these eigenvalues, define
\begin{equation*}
 P_{u,n}=\frac1{2\pi i}\int_\Gamma(\zeta I- L_n)^{-1}\,d\zeta.
\end{equation*}
Its range is the direct sum of the generalized eigenspaces, which
have just been identified and contain no nontrivial Jordan chains.
This identifies the negative Riesz range and gives
$\operatorname{rank}P_{u,n}=N_n+4$.
The enumeration $N_n=n+N_0$ follows from
Proposition~\ref{us:radial-result}.
\end{proof}

The theorem therefore separates the negative spectrum into the four symmetry
directions and $N_n$ genuinely unstable radial directions.  We next turn to
the adjoint modes and to the stable linearized flow.

\label{sec:adjoint-stable}
\subsection{Adjoint discrete modes and biorthogonal normalization}
\label{sec:adjointmodes}

We construct the adjoint functions needed for the later geometric
decomposition and modulation equations in the Section \ref{nl:sec:nonlinear-dynamics}.  
All profiles and discrete modes selected here may be taken real.  

\emph{Formal adjoint and angular Fourier-mode reduction.} We have the following lemma.

\begin{lemma}[Formal adjoint in the unweighted pairing]\label{lem:formaladjoint}
On $C_c^\infty(\R^3)$, the formal adjoint of $ L_n$ is
\begin{equation}
  L_n^*g
 =-\Delta g-\frac12y\cdot\nabla g-\frac12g-U_ng
 +\nabla\Delta^{-1}U_n\cdot\nabla g
+\Delta^{-1}\nabla\cdot(g\nabla U_n).                    \label{eq:formaladjoint}
\end{equation}
If $g(y)=g_l(r)Y_{l,m}(\omega)$ and $J_l=l(l+1)$, then
\begin{align}
 \mathbf L_{n,l}^*g_l
={}&-g_l''-\frac2r g_l'+\frac{J_l}{r^2}g_l
 -\frac r2g_l'-\frac12g_l-U_ng_l+P_ng_l'+\Delta_l^{-1}\left[
\frac1{r^2}\partial_r(r^2U_n'g_l)\right].               \label{eq:adjointangular}
\end{align}
\end{lemma}

\begin{proof}
Let $b_n=\nabla\Delta^{-1}U_n$.  First,
\[
 \left[\frac12(2+y\cdot\nabla)\right]^*
 =1-\frac12\nabla\cdot(y\,\cdot)
 =1-\frac32-\frac12y\cdot\nabla
 =-\frac12-\frac12y\cdot\nabla.
\]
Second,
\[
 (-b_n\cdot\nabla)^*g
 =\nabla\cdot(b_ng)
 =b_n\cdot\nabla g+(\nabla\cdot b_n)g
 =b_n\cdot\nabla g+U_ng.
\]
This $+U_ng$ combines with $-2U_ng$ to give $-U_ng$.  For the final nonlocal
term, let $w=\Delta^{-1}f$.  Symmetry of $\Delta^{-1}$ and integration by
parts give
\begin{align*}
 \left\langle-\nabla U_n\cdot\nabla\Delta^{-1}f,g\right\rangle
 =-\int_{\R^3}g\nabla U_n\cdot\nabla w=\int_{\R^3}w\,\nabla\cdot(g\nabla U_n)=\left\langle f,\Delta^{-1}\nabla\cdot(g\nabla U_n)\right\rangle.
\end{align*}
This proves \eqref{eq:formaladjoint}.

Finally, $\nabla U_n=U_n'(r)y/r$, and the angular Fourier basis function is independent
of $r$, so
\[
 \nabla\cdot\left(g_l(r)Y_{l,m}(\omega)U_n'(r)\frac yr\right)
 =\frac1{r^2}\partial_r(r^2U_n'g_l)Y_{l,m}.
\]
Using \eqref{eq:Deltal} proves \eqref{eq:adjointangular}.
\end{proof}

\begin{lemma}[Leading outer terms of the profile and self-adjoint weight]\label{lem:adjointprofiletail}
There are constants $a_n>0$ and $\mathfrak m_{\infty,n}>0$ such that
\begin{align}
 \Phi_n(r)&=a_nr^{-2}+O(r^{-4}),&
 P_n(r)&=2a_nr^{-1}+O(r^{-3}),\notag\\
 U_n(r)&=2a_nr^{-2}+O(r^{-4}),&
 U_n'(r)&=-4a_nr^{-3}+O(r^{-5}),                           \label{eq:adjointprofiletail}\\
 \rho_n(r)
 &=\mathfrak m_{\infty,n}r^{4+2a_n}e^{-r^2/4}
   \bigl(1+O(r^{-2})\bigr).                                \label{eq:adjointweighttail}
\end{align}
Moreover, $a_n=1+O(\varepsilon_n)$, so $a_n$ stays uniformly away from zero
for large $n$.  The remainders may be differentiated any fixed finite number
of times.
\end{lemma}

\begin{proof}
The outer matching formula and $u_1=r^{-2}-2r^{-4}+O(r^{-6})$ give
$\Phi_n=a_nr^{-2}+O(r^{-4})$, where $a_n=1+O(\varepsilon_n)$.
The expansion for $P_n$ follows immediately from $P_n=2r\Phi_n$.  Substitution
in $U_n=6\Phi_n+2r\Phi_n'$ gives
\[
 U_n=6a_nr^{-2}+2r(-2a_nr^{-3})+O(r^{-4})
 =2a_nr^{-2}+O(r^{-4}).
\]
Differentiation gives the expansion of $U_n'$.  Finally,
\[
 \frac{\rho_n'}{\rho_n}
 =\frac4r+2r\Phi_n-\frac r2
 =\frac{4+2a_n}{r}-\frac r2+O(r^{-3}).
\]
Integrating from a fixed large radius to $r$ gives
\[
 \log \rho_n
 =(4+2a_n)\log r-\frac{r^2}{4}
 +\log\mathfrak m_{\infty,n}+O(r^{-2}),
\]
Exponentiation proves \eqref{eq:adjointweighttail}; derivative expansions
follow from the same weighted outer estimate in Lemma~\ref{lem:outer-derivatives}.
\end{proof}

\emph{Finite-dimensional spectral duality and the biorthogonal family.}
We keep the reduced-mass operator $\mathcal A_n$, the radial density operator
$L_{n,0}$, and the transform $\mathcal T$ introduced above.  The intertwining
relation will be used to transfer the radial mass eigenfunctions to density
variables.  For each isolated spectral point, dual bases are then chosen in
the finite-dimensional ranges of the Riesz projection and its adjoint.  The
explicit integral formulas and endpoint asymptotics below depend on the
particular stationary state $U_n$.

\begin{lemma}[Radial density--mass intertwining]\label{lem:radialdensitymassidentity}
On radial functions regular at the origin and decaying at infinity,
\begin{equation}\label{eq:radialdensitymassidentity}
 L_{n,0}\mathcal T=\mathcal T\mathcal A_n.
\end{equation}
\end{lemma}

\begin{proof}
Let $f=\mathcal Th=6h+2rh'$.  The radial Newton formula gives
\[
 \partial_r\Delta_0^{-1}f
 =\frac1{r^2}\int_0^r(6h+2sh')s^2\dd s
 =\frac2{r^2}[s^3h(s)]_0^r=2rh.                            \label{eq:radialnewtonT}
\]
Write
\[
 A=\frac r2-\frac4r-P_n,\qquad
 B=1-U_n-\frac{3P_n}{r}.
\]
Then $\mathcal A_nh=-h''+Ah'+Bh$.  First compute
\[
 f=\mathcal Th=6h+2rh',\qquad
 f'=8h'+2rh'',\qquad f''=10h''+2rh'''.
\]
Substitute these identities and
$\partial_r\Delta_0^{-1}f=2rh$ into the radial density operator
\[
 L_{n,0}f
 =-f''-\frac2r f'+f+\frac r2f'
  -2U_nf-P_nf'-U_n'\partial_r\Delta_0^{-1}f.
\]
Collecting the coefficients of $h'''$, $h''$, $h'$, and $h$ gives
\begin{align*}
 L_{n,0}\mathcal Th
={}&-2rh'''+(r^2-14-2rP_n)h''\\
 &+\left(6r-\frac{16}{r}-4rU_n-8P_n\right)h'\\
 &+(6-12U_n-2rU_n')h.
\end{align*}
On the other hand, if
$q=\mathcal A_nh=-h''+Ah'+Bh$, then
\[
 q'=-h'''+Ah''+(A'+B)h'+B'h.
\]
Consequently, $\mathcal Tq=6q+2rq'$ gives
\begin{align*}
 \mathcal T \mathcal A_nh
={}&-2rh'''+(-6+2rA)h''\\
 &+\{6A+2r(A'+B)\}h'
 +(6B+2rB')h.
\end{align*}
The second-order coefficients in the first lines agree.  Using
$P_n'=U_n-2P_n/r$, the coefficient of $h'$ becomes
\begin{align*}
 6A+2r(A'+B)
 &=6r-\frac{16}{r}-12P_n-2rP_n'-2rU_n\\
 &=6r-\frac{16}{r}-4rU_n-8P_n,
\end{align*}
and the coefficient of $h$ becomes
\begin{align*}
 6B+2rB'
 &=6-6U_n-\frac{12P_n}{r}-2rU_n'-6P_n'\\
 &=6-12U_n-2rU_n'.
\end{align*}
Every derivative-order coefficient agrees, proving
\eqref{eq:radialdensitymassidentity}.
\end{proof}

\begin{lemma}[Existence of adjoint discrete modes and nondegenerate pairing]\label{lem:adjointdualexistence}
Let $-\nu$ be one of the semisimple discrete eigenvalues obtained above.  Then
\[
 \dim\ker( L_n+\nu)
 =\dim\ker( L_n^*+\nu),
\]
and the $L^2$ pairing between the two eigenspaces is nondegenerate.  Hence one
may choose real adjoint eigenfunctions
\[
 \zeta_0,\quad \zeta_k^{\rm tr}\ (1\le k\le3),\quad
 \zeta_j^{\rm u}\ (1\le j\le N_n),
\]
such that
\begin{equation}\label{eq:biorthogonal}
\begin{aligned}
 \langle\Lambda U_n,\zeta_0\rangle=1,\ 
\langle\partial_kU_n,\zeta_\ell^{\rm tr}\rangle=\delta_{k\ell},
\ 
\langle\phi_{i,n},\zeta_j^{\rm u}\rangle=\delta_{ij},
\end{aligned}
\end{equation}
all other cross pairings being zero.  
\end{lemma}

\begin{proof}
The proof is given in Appendix~\ref{app:proof-adjointdualexistence}.
\end{proof}
\emph{Explicit integral formula for radial adjoint modes.}
We start from the mass eigenvalue equation
$\mathcal A_nh_j=-\nu_jh_j$ in $L^2(\rho_n\,\dd r)$.
After defining \eqref{eq:radialadjointformula}, we differentiate it to obtain
\eqref{eq:radialadjointderivative} and then integrate by parts against
$\mathcal Th$.  Self-adjointness reduces the result to
$\langle h,h_j\rangle_{L^2(\rho_n\,\dd r)}$, thereby verifying both the adjoint
eigenvalue equation and the normalization.  The endpoint asymptotics are
obtained by substituting the expansions of $h_j$ and $\rho_n$ at the two
endpoints into the same integral formula.

Index the scaling mode and radial unstable modes together.  Let
\[
 \nu_0=1,\qquad h_0=\Lambda\Phi_n,
\]
and, for $1\le j\le N_n$, let
\[
 h_j=\mathcal T^{-1}\phi_{j,n},\qquad \nu_j=\mu_{j,n}.
\]
By the intertwining relation,
$\mathcal A_nh_j=-\nu_jh_j$.  Set
\[
 \kappa_j=\int_0^\infty h_j(r)^2\rho_n(r)\dd r>0.
\]

\begin{lemma}[Explicit formula for radial adjoint functions]\label{lem:radialadjointformula}
Define
\begin{equation}
 Z_j(r)=\frac1{8\pi\kappa_j}
 \int_r^\infty s^{-3}\rho_n(s)h_j(s)\dd s,
 \qquad 0\le j\le N_n.                                    \label{eq:radialadjointformula}
\end{equation}
Then
\[
  L_n^*Z_j=-\nu_jZ_j,\qquad
 \langle\mathcal Th_i,Z_j\rangle=\delta_{ij}.
\]
In particular, $\zeta_0=Z_0$ and $\zeta_j^{\rm u}=Z_j$ give all radial
adjoint modes in \eqref{eq:biorthogonal}.
\end{lemma}

\begin{proof}
By \eqref{eq:radialadjointformula},
\begin{equation}
 Z_j'(r)=-\frac{\rho_n(r)h_j(r)}
 {8\pi\kappa_jr^3}.                                       \label{eq:radialadjointderivative}
\end{equation}
For any regular decaying radial $h$, integration by parts and the endpoint
decay give
\begin{align*}
 \langle\mathcal Th,Z_j\rangle_{L^2(\R^3)}
 &=4\pi\int_0^\infty(6h+2rh')Z_jr^2\dd r=8\pi\int_0^\infty(r^3h)'Z_j\dd r\\
 &=-8\pi\int_0^\infty hZ_j'r^3\dd r=\kappa_j^{-1}\int_0^\infty hh_j\rho_n\dd r.
\end{align*}
Since $\mathcal A_n$ is self-adjoint in $L^2(\rho_n\dd r)$, the functions
$h_i,h_j$ corresponding to distinct eigenvalues are orthogonal.  Thus the last
expression equals $\delta_{ij}$ when $h=h_i$.

To identify the operator adjoint, take any radial
$f\in C_c^\infty(\R^3)$ and put $h=\mathcal T^{-1}f$.
Then $h$ is regular at zero and equals $Cr^{-3}$ beyond the support
of $f$, so $h\in D(\mathcal A_n)$; all boundary terms below vanish.
The intertwining identity and the pairing formula give
\begin{align*}
 \langle L_{n,0}f,Z_j\rangle_{L^2(\R^3)}
 &=\langle\mathcal T\mathcal A_nh,Z_j\rangle_{L^2(\R^3)}\\
 &=\kappa_j^{-1}
   \langle\mathcal A_nh,h_j\rangle_{L^2(\rho_n\,\dd r)}
 =-\nu_j\kappa_j^{-1}
   \langle h,h_j\rangle_{L^2(\rho_n\,\dd r)}\\
 &=-\nu_j\langle f,Z_j\rangle_{L^2(\R^3)}.
\end{align*}
For a general $f\in C_c^\infty(\R^3)$, let $\Pi_0f$ be its spherical
average.  Rotational invariance gives $\Pi_0L_n=L_n\Pi_0$, and the
radiality of $Z_j$ gives
\[
 \langle L_nf,Z_j\rangle
 =\langle L_n\Pi_0f,Z_j\rangle
 =-\nu_j\langle\Pi_0f,Z_j\rangle
 =-\nu_j\langle f,Z_j\rangle.
\]
The graph-core argument in Appendix~\ref{app:proof-fixedmodefredholm}
shows that $C_c^\infty(\R^3)$ is a core for $L_n$.  The endpoint
estimates in Lemma~\ref{lem:radialadjointasymptotics} show that
$Z_j\in L^2(\R^3)$.  The displayed identity therefore extends by
graph-norm convergence to every $f\in D(L_n)$.  By the definition
of the adjoint, $Z_j\in D(L_n^*)$ and
$L_n^*Z_j=-\nu_jZ_j$.
\end{proof}

\begin{lemma}[Endpoint asymptotics of radial adjoint modes]\label{lem:radialadjointasymptotics}
For each $0\le j\le N_n$, there are constants $A_{j,n}\in\R$,
$b_{j,n}\ne0$, and $c_{j,n}\ne0$ such that
\begin{align}
 Z_j(r)
 &=A_{j,n}-\frac{h_j(0)}{16\pi\kappa_j}r^2+O(r^4),
 &&r\to0,                                                  \label{eq:radialadjointorigin}\\
 Z_j(r)
 &=c_{j,n}r^{2a_n-2-2\nu_j}e^{-r^2/4}
   \bigl(1+O(r^{-2})\bigr),
 &&r\to\infty.                                             \label{eq:radialadjointinfinity}
\end{align}
More precisely, for each fixed integer $q\ge0$,
\begin{equation}
 \partial_r^qZ_j(r)
 =c_{j,n}\left(-\frac12\right)^q
 r^{2a_n-2-2\nu_j+q}e^{-r^2/4}
 \bigl(1+O(r^{-2})\bigr).                                  \label{eq:radialadjointderivatives}
\end{equation}
\end{lemma}

\begin{proof}
The regular branch at the origin satisfies $h_j(r)=h_j(0)+O(r^2)$, while
$\rho_n(r)=r^4(1+O(r^2))$.  Substitution in
\eqref{eq:radialadjointderivative} gives
\[
 Z_j'(r)=-\frac{h_j(0)}{8\pi\kappa_j}r+O(r^3).
\]
Integration from $0$ to $r$ proves \eqref{eq:radialadjointorigin}.

At infinity, $\mathcal A_nh_j=-\nu_jh_j$ and
\eqref{eq:adjointprofiletail} give
\[
 -h_j''+\left(\frac r2-\frac{4+2a_n}{r}+O(r^{-3})\right)h_j'
 +\left(1+\nu_j-\frac{8a_n}{r^2}+O(r^{-4})\right)h_j=0.
\]
The leading first-order balance is
\[
 \frac r2h_j'+(1+\nu_j)h_j=0,
\]
so the square-integrable branch begins with $r^{-2(1+\nu_j)}$.  Set
$h_j=r^{-2(1+\nu_j)}v_j$ and substitute this expression into the equation.
The resulting Volterra equation for $(v_j,rv_j')$ has a kernel of size
$O(r^{-2})$; the other independent branch contains $e^{r^2/4}$ and is not in
$L^2(\rho_n\dd r)$.  Therefore
\begin{equation}
 h_j(r)=b_{j,n}r^{-2(1+\nu_j)}
 \bigl(1+O(r^{-2})\bigr),\qquad b_{j,n}\ne0,                \label{eq:massmodeinfinity}
\end{equation}
and the expansion may be differentiated term by term.

Substitute \eqref{eq:adjointweighttail} and
\eqref{eq:massmodeinfinity} into \eqref{eq:radialadjointformula}.  Its
integrand is
\[
 \mathfrak m_{\infty,n}b_{j,n}
 s^{2a_n-1-2\nu_j}e^{-s^2/4}
 \bigl(1+O(s^{-2})\bigr).
\]
One integration by parts gives
\[
 \int_r^\infty s^\beta e^{-s^2/4}\dd s
 =2r^{\beta-1}e^{-r^2/4}\bigl(1+O(r^{-2})\bigr).
\]
Hence \eqref{eq:radialadjointinfinity} holds, with
\[
 c_{j,n}
 =\frac{\mathfrak m_{\infty,n}b_{j,n}}{4\pi\kappa_j}\ne0.
\]
When differentiating this expansion, the leading contribution each time comes
from $\partial_re^{-r^2/4}=-(r/2)e^{-r^2/4}$.  This proves
\eqref{eq:radialadjointderivatives}.
\end{proof}

{We next establish the Newton dipole tail of the translation adjoint modes.}
\begin{lemma}[Translation adjoint functions at the origin and infinity]\label{lem:translationadjointasymptotics}
After the normalization \eqref{eq:biorthogonal}, there is a smooth radial
function $g_{\rm tr}$ such that
\[
 \zeta_k^{\rm tr}(y)=g_{\rm tr}(r)\omega_k,\qquad
 \omega_k=\frac{y_k}{r}.
\]
There is $\beta_n\in\R$ such that
\begin{align}
 \zeta_k^{\rm tr}(y)
 &=\beta_n y_k+O(|y|^3), &&y\to0,                           \label{eq:translationorigin}\\
 \zeta_k^{\rm tr}(y)
 &=-\frac1{4\pi}\frac{y_k}{|y|^3}
   +O(|y|^{-4}), &&|y|\to\infty.                            \label{eq:translationinfinity}
\end{align}
For every multi-index $\alpha$,
\begin{equation}
 \partial^\alpha\zeta_k^{\rm tr}(y)
 =-\frac1{4\pi}\partial^\alpha\left(\frac{y_k}{|y|^3}\right)
 +O(|y|^{-4-|\alpha|}),\qquad |y|\to\infty.                \label{eq:translationderivatives}
\end{equation}
\end{lemma}

\begin{proof}
Set $l=1$ and the eigenvalue to $-1/2$ in
\eqref{eq:adjointangular}.  If
\[
 H_{\rm tr}(r)=\frac1{r^2}\partial_r(r^2U_n'g_{\rm tr}),\qquad
 W_{\rm tr}=\Delta_1^{-1}H_{\rm tr},
\]
then the radial equation is
\begin{equation}
 -g_{\rm tr}''-\frac2r g_{\rm tr}'+\frac2{r^2}g_{\rm tr}
 -\frac r2g_{\rm tr}'-U_ng_{\rm tr}+P_ng_{\rm tr}'
 +W_{\rm tr}=0.                                            \label{eq:translationadjointeq}
\end{equation}
The indicial exponents of the Laplace principal part at the origin are
$1,-2$.  The closed operator domain excludes $r^{-2}$, and hence
$g_{\rm tr}(r)=\beta_nr+O(r^3)$.  This proves
\eqref{eq:translationorigin}; repeated differentiation of the equation gives
the full odd Taylor expansion.

We now compute the leading term at infinity.  The $l=1$ Newton formula is
\begin{equation}
 W_{\rm tr}(r)=-\frac13\left[
 r^{-2}\int_0^rH_{\rm tr}(s)s^3\dd s
 +r\int_r^\infty H_{\rm tr}(s)\dd s\right].                \label{eq:translationnewton}
\end{equation}
We derive the decay directly from the Newton formula, without a
rescaled elliptic estimate whose drift constant depends on the radius.
Write $g=g_{\rm tr}$, $W=W_{\rm tr}$ and set
\[
 I(r)=\int_0^r s^2U_n'(s)g(s)\,ds,\qquad
 J(r)=\int_r^\infty s^{-1}U_n'(s)g(s)\,ds.
\]
Integration by parts in \eqref{eq:translationnewton} cancels its two
local terms and gives
\begin{equation}\label{eq:translation-newton-integrated}
 W(r)=\frac{I(r)}{3r^2}-\frac{2r}{3}J(r).
\end{equation}
This identity is valid for $g\in L^2(r^2\,dr)$: first use compactly
supported smooth approximations, integrate the distributional derivative
in $H_{\rm tr}$, and pass to the limit by Cauchy--Schwarz. In particular,
no pointwise decay of $g'$ has been assumed.
Since $U_n'=O(r^{-3})$, $I_\infty:=I(\infty)$ converges and
\[
 |I_\infty-I(r)|\le C_nr^{-3/2}\|g\|_{L^2((r,\infty),s^2ds)},
 \qquad
 |J(r)|\le C_nr^{-9/2}\|g\|_{L^2((r,\infty),s^2ds)}.
\]
Consequently,
\begin{equation}\label{eq:translation-W-initial}
 W(r)=c_{\rm tr}r^{-2}+O(r^{-7/2}),\qquad c_{\rm tr}=I_\infty/3.
\end{equation}

Rewrite \eqref{eq:translationadjointeq} as
\[
 g''+a(r)g'=b(r)g+W,
 \quad a(r)=\frac r2+\frac2r-P_n(r),\quad b(r)=\frac2{r^2}-U_n(r).
\]
Choose a fixed $R$ so large that $a(r)\ge r/4$ for $r\ge R$ and put
$m(r)=\exp(\int_R^r a(t)\,dt)$. The integrating-factor formula is
\begin{equation}\label{eq:translation-integrating-factor}
 g'(r)=\frac{m(R)g'(R)}{m(r)}
       +\int_R^r\frac{m(t)}{m(r)}\{b(t)g(t)+W(t)\}\,dt.
\end{equation}
Here $m(t)/m(r)\le\exp(-(r^2-t^2)/8)$ for $R\le t\le r$.
The contribution of $[R,r/2]$ is exponentially small. On $[r/2,r]$,
Cauchy--Schwarz and $b(t)=O(t^{-2})$ show that the term containing $g$
is bounded by $C_nr^{-7/2}\|g\|_{L^2(r^2dr)}$, whereas
\eqref{eq:translation-W-initial} bounds the other term by $C_nr^{-3}$.
Thus $g'=O(r^{-3})$. It follows that $g$ has a finite limit at infinity;
membership in $L^2(r^2dr)$ forces that limit to be zero, and hence
$g=O(r^{-2})$.

Returning to \eqref{eq:translation-newton-integrated} now gives
$W=c_{\rm tr}r^{-2}+O(r^{-4})$. Since $bg=O(r^{-4})$ and
$a(r)=r/2+O(r^{-1})$, one integration by parts in
\eqref{eq:translation-integrating-factor} yields
\[
 \frac1{m(r)}\int_R^r m(t)t^{-2}\,dt
 =2r^{-3}+O(r^{-5}).
\]
Therefore
\begin{equation}\label{eq:translationroughasympt}
 g'(r)=2c_{\rm tr}r^{-3}+O(r^{-5}),\qquad
 g(r)=-c_{\rm tr}r^{-2}+O(r^{-4}).
\end{equation}
All higher derivative bounds follow from the same integral equation.
More explicitly, differentiating \eqref{eq:translation-newton-integrated}
expresses $W^{(j)}$ through $I,J$ and derivatives of $U_n'g$ of order
at most $j-1$. Induction in the differentiated equation
$g''+ag'=bg+W$, using the above kernel bound at each step, gives
\[
 \partial_r^j\bigl(g+c_{\rm tr}r^{-2}\bigr)=O(r^{-4-j})
 \quad(j\ge0).
\]
The profile derivatives required in this induction follow from the
exterior ODE and its expansion in Lemma~\ref{lem:outer-derivatives}.

Finally, the biorthogonal normalization fixes the coefficient:
\[
 1=\langle\partial_kU_n,\zeta_k^{\rm tr}\rangle
   =\frac{4\pi}{3}\int_0^\infty U_n'(r)g(r)r^2\,dr
   =\frac{4\pi}{3}I_\infty.
\]
Thus $c_{\rm tr}=1/(4\pi)$. Equivalently,
\begin{equation*}
 \mathcal M_{\rm tr}:=\int_0^\infty H_{\rm tr}(r)r^3\,dr
 =-I_\infty=-\frac3{4\pi},
\end{equation*}
and
\begin{equation*}
 W_{\rm tr}(r)=\frac1{4\pi}r^{-2}+O(r^{-4}).
\end{equation*}
Combining these identities with \eqref{eq:translationroughasympt}
and the derivative bounds proves \eqref{eq:translationinfinity} and
\eqref{eq:translationderivatives}.
\end{proof}

\begin{proposition}[Unified asymptotics and biorthogonal normalization of adjoint modes]
\label{prop:adjointmodeasymptotics}
For every sufficiently large matched profile, there is a unique real adjoint
family normalized by \eqref{eq:biorthogonal},
\[
 \zeta_0,\qquad \zeta_k^{\rm tr}\ (1\le k\le3),\qquad
 \zeta_j^{\rm u}\ (1\le j\le N_n),
\]
satisfying
\[
  L_n^*\zeta_0=-\zeta_0,\qquad
  L_n^*\zeta_k^{\rm tr}=-\frac12\zeta_k^{\rm tr},\qquad
  L_n^*\zeta_j^{\rm u}=-\mu_{j,n}\zeta_j^{\rm u}.
\]
Let $\nu=1$ for $\zeta_0$ and $\nu=\mu_{j,n}$ for
$\zeta_j^{\rm u}$.  Each radial adjoint mode has the even expansion
\[
 \zeta(y)=A+B|y|^2+O(|y|^4),
\]
and, for $|\alpha|\le4$,
\[
 \partial^\alpha\bigl(\zeta-A-B|y|^2\bigr)
 =O(|y|^{4-|\alpha|}).
\]
At infinity,
\[
 \partial_r^q\zeta(r)
 =c\left(-\frac12\right)^q
 r^{2a_n-2-2\nu+q}e^{-r^2/4}
 \bigl(1+O(r^{-2})\bigr),\qquad q\ge0,
\]
where $c\ne0$.  At the origin and infinity, respectively, the translation
adjoint modes satisfy
\[
 \zeta_k^{\rm tr}(y)=\beta_ny_k+O(|y|^3),
\]
and, for $|\alpha|\le3$,
\[
 \partial^\alpha\bigl(\zeta_k^{\rm tr}-\beta_ny_k\bigr)
 =O(|y|^{3-|\alpha|}),\ \ 
 \partial^\alpha\zeta_k^{\rm tr}(y)
 =-\frac1{4\pi}\partial^\alpha\left(\frac{y_k}{|y|^3}\right)
 +O(|y|^{-4-|\alpha|}).
\]
In particular, all these adjoint functions belong to $L^2(\R^3)$.  The radial
adjoint modes have Gaussian decay, whereas the translation adjoint modes have
the algebraic $|y|^{-2}$ tail determined by the Newton dipole moment.
\end{proposition}

\begin{proof}
Existence, unique normalization, and cross-orthogonality follow from
Lemma~\ref{lem:adjointdualexistence}.  The radial formulas and asymptotics
come from Lemmas~\ref{lem:radialadjointformula} and
\ref{lem:radialadjointasymptotics}; translation modes are covered by
Lemma~\ref{lem:translationadjointasymptotics}.  Gaussian functions plainly
belong to $L^2(\R^3)$.  The square of $|y|^{-2}$ times the three-dimensional
volume element is $r^{-2}\dd r$, integrable at infinity; the regular origin
expansions give local integrability.
\end{proof}

\begin{lemma}[Adjoint testing of the nonlinear drift]\label{lem:adjoint-drift-testing}
Let $\zeta$ be any of the normalized adjoint modes in
Proposition~\ref{prop:adjointmodeasymptotics}.  Define
$T=\nabla\Delta^{-1}$ and $\Lambda=2+y\cdot\nabla$.
Then
\begin{equation}\label{eq:adjoint-test-spaces}
 \zeta\in L^2\cap L^3,\qquad
 \nabla\zeta\in L^2\cap L^3,\qquad
 y\cdot\nabla\zeta\in L^2.
\end{equation}
For $f,g\in L^2(\mathbb R^3)$, the nonlinear flux pairing
\begin{equation}\label{eq:adjoint-flux-pairing}
 \mathscr N_\zeta(f,g)
 :=-\int_{\mathbb R^3}f\,Tg\cdot\nabla\zeta\,dy
\end{equation}
is absolutely convergent and satisfies
\begin{equation}\label{eq:adjoint-flux-estimate}
 |\mathscr N_\zeta(f,g)|
 \le C_\zeta\|f\|_2\|g\|_2.
\end{equation}
It agrees with
$\langle\nabla\cdot(fTg),\zeta\rangle$ when the latter
is defined by smooth approximation.  The dual identities
\begin{equation}\label{eq:adjoint-unbounded-generator-pairing}
 \langle\Lambda f,\zeta\rangle
 =-\langle f,\zeta+y\cdot\nabla\zeta\rangle,
 \qquad
 \langle\partial_k f,\zeta\rangle
 =-\langle f,\partial_k\zeta\rangle
\end{equation}
hold for $f\in\mathcal Y$; their right-hand sides extend
continuously to $f\in L^2$.
Moreover, the spectral projections $P_{{\rm u},n}$ and
$P_{{\rm s},n}$ preserve $\mathcal Y$ and are bounded on that space.
\end{lemma}

\begin{proof}
The Gaussian estimates for the radial modes and the $r^{-2}$
translation tail, together with their differentiated asymptotics,
prove \eqref{eq:adjoint-test-spaces}.  The
Hardy--Littlewood--Sobolev estimate
$\|Tg\|_6\le C\|g\|_2$ and H\"older's inequality give
\[
 \int|f|\,|Tg|\,|\nabla\zeta|
 \le\|f\|_2\|Tg\|_6\|\nabla\zeta\|_3,
\]
proving \eqref{eq:adjoint-flux-estimate}.  To justify the weak
integration by parts, insert a radial cutoff $\chi_R$ equal to one
on $|y|\le R$, with $|\nabla\chi_R|\le C/R$.
The additional term is bounded by
\[
 \left|\int f\,Tg\cdot\nabla\chi_R\,\zeta\right|
 \le \frac{C}{R}\|fTg\|_{3/2}\|\zeta\|_3
 \longrightarrow0.
\]
Approximation of $f,g$ in $L^2$ then gives the asserted pairing.
The generator identities follow first for compactly supported smooth
$f$ and then by density in the graph norm of $\mathcal Y$.
For the projections, the right eigenfunctions
$\Lambda U_n$, $\partial_kU_n$ and $\phi_{j,n}$ belong to
$\mathcal Y$: this follows from their regularity at zero and the
profile and radial eigenfunction expansions at infinity.
The explicit finite-rank formula \eqref{eq:unstableprojection}
therefore gives
$\|P_{{\rm u},n}f\|_{\mathcal Y}\le C_n\|f\|_2$
for $f\in\mathcal Y$.  Consequently, $P_{{\rm s},n}=I-P_{{\rm u},n}$ satisfies
$\|P_{{\rm s},n}f\|_{\mathcal Y}\le(1+C_n)\|f\|_{\mathcal Y}$.
\end{proof}

\subsection{Stable projection, semigroup decay, and an equivalent \texorpdfstring{$L^2$}{L2} energy}
\label{sec:stablelinear}

After the scaling mode, the three translation modes, and the $N_n$ genuine
radial unstable modes have been removed, we seek estimates in the ordinary
space $L^2(\R^3)$.  A spectral gap alone does not imply the numerical-range
inequality
\[
 \operatorname{Re}\langle L_nf,f\rangle\ge c\|f\|_2^2,
\]
nor does it by itself identify the growth bound of an arbitrary strongly
continuous semigroup.  We therefore prove that $- L_n$ is a compact
bounded perturbation of a strictly dissipative generator.  Quasi-compact
semigroup theory and the spectral classification above then yield exponential
decay on the full stable subspace.  This is the route used in
\cite[Propositions~2.3 and 2.7]{LiZhouKSNS2025} and
\cite[arXiv v2, Proposition~2.16]{LiZhou2026}.  We finally construct a quadratic energy
equivalent to the ordinary $L^2$ norm.

Before invoking the semigroup decomposition, we also close an operator-theoretic
point left open by the quadratic-form sections.  Those sections exclude point
spectrum for the radial coefficient operators $L_{n,l}$, whereas the compact
perturbation and essential-spectrum arguments below are first carried out for
the full-space operator $ L_n$.  We therefore identify the closed
realizations mode by mode and use Fredholm index zero to pass from exclusion of
eigenvalues to exclusion of the full spectrum.

Use the normalization in Proposition~\ref{prop:adjointmodeasymptotics} and set
\[
 e_0=\Lambda U_n,\qquad e_k^{\rm tr}=\partial_kU_n,\qquad
 e_j^{\rm u}=\phi_{j,n}.
\]
Define
\begin{align}
 P_{{\rm u},n}f={}&\langle f,\zeta_0\rangle e_0
 +\sum_{k=1}^3\langle f,\zeta_k^{\rm tr}\rangle e_k^{\rm tr}
 +\sum_{j=1}^{N_n}\langle f,\zeta_j^{\rm u}\rangle e_j^{\rm u},
                                                               \label{eq:unstableprojection}\\
 P_{{\rm s},n}={}&I-P_{{\rm u},n}.\notag
\end{align}

\begin{lemma}[Exact orthogonality conditions for the stable space]
\label{lem:stablespaceprojection}
$P_{{\rm u},n}$ is a bounded finite-rank projection on $L^2(\R^3)$ and
commutes with $ L_n$ on $D( L_n)$.  Moreover,
\[
 X_{{\rm s},n}:=\operatorname{Ran}P_{{\rm s},n}
 =\left\{f\in L^2(\mathbb R^3):
 \begin{array}{l}
 \langle f,\zeta_0\rangle=0,\\
 \langle f,\zeta_k^{\rm tr}\rangle=0\quad(1\le k\le3),\\
 \langle f,\zeta_j^{\rm u}\rangle=0\quad(1\le j\le N_n)
 \end{array}\right\}.
\]
Thus, for the non-self-adjoint operator $ L_n$, removing the unstable
modes means orthogonality to the adjoint eigenfunctions, not to the right
eigenfunctions.
\end{lemma}

\begin{proof}
The proof is given in Appendix~\ref{app:proof-stablespaceprojection}.
\end{proof}

\emph{Closed realization in \texorpdfstring{$L^2$}{L2} and the stable semigroup.}
Recall that
\begin{equation*}
 \mathbf L_0=-\Delta+\frac12(2+y\cdot\nabla)                 
\end{equation*}
and write $ L_n=\mathbf L_0+\mathbf K_n$, where
\begin{equation*}
 \mathbf K_nf=-2U_nf-b_n\cdot\nabla f
 -\nabla U_n\cdot\nabla\Delta^{-1}f,\qquad
 b_n=\nabla\Delta^{-1}U_n.                                  
\end{equation*}

\begin{lemma}[Semigroup generation and compact decomposition]
\label{lem:compactgenerator}
Under the closed realization of Lemma~\ref{us:fredholm}, fix a sufficiently
large $n$. Then:
\begin{enumerate}
\item $\mathbf K_n:D(\mathbf L_0)\to L^2(\R^3)$ is compact;
\item $- L_n$ generates a strongly continuous semigroup;
\item there are a maximal dissipative operator $A_{0,n}$ and a compact
      operator $K_{0,n}$ such that
      \begin{equation}
       - L_n=A_{0,n}-\frac1{16}+K_{0,n}.              \label{eq:maxdissdecomp}
      \end{equation}
\end{enumerate}
\end{lemma}

\begin{proof}
The proof is given in Appendix~\ref{app:proof-compactgenerator}.
\end{proof}

\begin{lemma}[Exact essential spectrum and the free-semigroup threshold]
\label{lem:exactessentialthreshold}
With the Fredholm definition of essential spectrum,
\begin{equation}
 \sigma_{\rm ess}(\mathbf L_0)=\sigma_{\rm ess}( L_n)
 =\{z\in\mathbb C:\operatorname{Re}z\ge1/4\}.               \label{eq:exactessentialspectrum}
\end{equation}
If $S_0(s)=e^{-s\mathbf L_0}$ and $S_n(s)=e^{-s L_n}$, then
\begin{equation}
 \|S_0(s)\|_{2\to2}=\|S_0(s)\|_{\rm ess}=e^{-s/4},
 \qquad \|S_n(s)\|_{\rm ess}=e^{-s/4}.                     \label{eq:exactessentialnorm}
\end{equation}
Thus the essential growth bound of $- L_n$ is exactly $-1/4$; the
$-1/16$ in \eqref{eq:maxdissdecomp} is only a convenient nonsharp number.
The corresponding half-plane theorem for more general Ornstein--Uhlenbeck
operators on $L^p(\R^d)$ is proved in \cite{Metafune2001}; the calculation
below verifies the precise coefficients and sign convention used here.
\end{lemma}

\begin{proof}
The proof is given in Appendix~\ref{app:proof-exactessentialthreshold}.
\end{proof}

\emph{Mode-by-mode spectral closure.}
The closed fixed-mode realizations and their Fredholm properties were established
in Lemma~\ref{us:fixed-mode-fredholm}.  We now combine those facts with the
mode-by-mode point-spectrum analysis obtained above.

\begin{proposition}[The full spectrum and Riesz decomposition in the $l=1$ mode]
\label{prop:l1fullclosure}
Let $\delta_1$ be the constant in Proposition~\ref{us:translation-result}.  For every
sufficiently large $n$,
\begin{equation}
 \sigma(L_{n,1})\cap\{z\in\mathbb C:\operatorname{Re}z<\delta_1\}
 =\{-1/2\}.                                                \label{eq:l1fullspectrum}
\end{equation}
Moreover,
\[
 \ker(L_{n,1}+1/2)=\operatorname{span}\{U_n'\},
\]
and $-1/2$ is an algebraically simple eigenvalue.  If $P_{n,1}$ is the Riesz
projection of $L_{n,1}$ at $-1/2$, then
\begin{equation}
 \sigma\!\left(L_{n,1}|_{\operatorname{Ran}(I-P_{n,1})}\right)
 \subset\{z\in\mathbb C:\operatorname{Re}z\ge\delta_1\}.  \label{eq:l1rieszgap}
\end{equation}
On the complete $l=1$ mode space $\mathcal H_1$,
\begin{equation}
 \ker( L_n+1/2)\cap\mathcal H_1
 =\operatorname{span}\{\partial_{y_1}U_n,\partial_{y_2}U_n,
                         \partial_{y_3}U_n\}.               \label{eq:fulltranslationkernel}
\end{equation}
Consequently, the geometric and algebraic multiplicities of $-1/2$ on
$\mathcal H_1$ both equal $3$; equivalently, this eigenvalue is semisimple on
the complete $l=1$ mode space.
\end{proposition}

\begin{proof}
Since $\delta_1<1/4$, combining the point-spectrum identity
\eqref{us:l1-spectrum} of Proposition~\ref{us:translation-result} with the Fredholm closure
\eqref{eq:fixedmodespectrumpoint} gives \eqref{eq:l1fullspectrum}.  In
particular, $-1/2$ is isolated; Lemma~\ref{us:fixed-mode-fredholm} also shows
that it has finite algebraic multiplicity.  Proposition~\ref{us:translation-result}
proves that its eigenspace is one-dimensional and that there is no nontrivial
Jordan chain.  Hence it is algebraically simple.

The Riesz projection gives the invariant direct-sum decomposition
\[
 L^2((0,\infty),r^2\dd r)
 =\operatorname{Ran}P_{n,1}\oplus\operatorname{Ran}(I-P_{n,1}),
 \qquad \operatorname{Ran}P_{n,1}=\operatorname{span}\{U_n'\}.
\]
The spectrum of a closed operator under a Riesz decomposition is the union of
the spectra of the two restrictions, and
$L_{n,1}|_{\operatorname{Ran}P_{n,1}}=-1/2$.  Removing $-1/2$ from
\eqref{eq:l1fullspectrum} therefore leaves no spectral point with
$\operatorname{Re}z<\delta_1$ on the complementary subspace.  This is
\eqref{eq:l1rieszgap}.

Finally, apply Lemma~\ref{us:fixed-mode-fredholm}.  The three fixed components
$m=-1,0,1$ are all unitarily equivalent to the same $L_{n,1}$.  In the
angular space $\mathcal Y_1$, the functions $y_1/r,y_2/r,y_3/r$ form a basis,
and
\[
 \partial_{y_j}U_n(r)=U_n'(r)\frac{y_j}{r},\qquad1\le j\le3.
\]
Thus the one-dimensional kernels in the three fixed modes combine to give
\eqref{eq:fulltranslationkernel}.  None of the three blocks has a
nontrivial Jordan chain, and neither does their direct sum; hence both
multiplicities equal $3$.
\end{proof}

\begin{proposition}[Full spectral closure in every \texorpdfstring{$l\ge2$}{l>=2} mode]
\label{prop:fixedmodefullclosure}
For all sufficiently large $n$ and every $l\ge2$,
\begin{equation}
 \sigma(L_{n,l})\cap\{z:\operatorname{Re}z<1/4\}=\varnothing,  \label{eq:fixedmodefullgap}
\end{equation}
and
\begin{equation}
 \frac14+i\mathbb R\subset\sigma(L_{n,l}).                \label{eq:fixedmodeboundaryline}
\end{equation}
Thus $1/4$ is the exact spectral left boundary in every fixed $l\ge2$ mode.
If
\begin{equation*}
 \mathscr X_{\ge2}:=
 \overline{\bigoplus_{l\ge2}\ \bigoplus_{m=-l}^l\mathscr X_{l,m}},
\end{equation*}
then
\begin{equation}
 \sigma( L_n|_{\mathscr X_{\ge2}})
 \cap\{z:\operatorname{Re}z<1/4\}=\varnothing.           \label{eq:fullnonradialgap}
\end{equation}
\end{proposition}

\begin{proof}
Proposition~\ref{us:l2-result} excludes point spectrum for $l=2$, and
Proposition~\ref{us:high-result} does so for every $l\ge3$.  Combining these facts
with \eqref{eq:fixedmodespectrumpoint} proves
\eqref{eq:fixedmodefullgap}.  In the Weyl sequence from Lemma
\ref{lem:exactessentialthreshold}, fix $h(\omega)=Y_{l,m}(\omega)$; this proves
\eqref{eq:fixedmodeboundaryline}.

It remains to justify the infinite direct sum.  Let
$\Pi_{[2,N]}=\sum_{l=2}^N\sum_{m=-l}^l\Pi_{l,m}$.  The commutation relations
proved above show that, for
$F\in D( L_n)\cap\mathscr X_{\ge2}$,
\[
 \Pi_{[2,N]}F\longrightarrow F,
 \qquad
  L_n\Pi_{[2,N]}F
 =\Pi_{[2,N]} L_nF\longrightarrow L_nF
\]
in $L^2$.  Consequently,
$ L_n|_{\mathscr X_{\ge2}}$ is the Hilbert direct sum of the closed
fixed-mode operators, including their graph domains; this is more than a
formal Fourier decomposition.

Fix $\operatorname{Re}z<1/4$.  For $l\ge3$,
coercivity \eqref{us:high-coercivity}, Cauchy--Schwarz, and surjectivity give
\begin{equation*}
 \|(L_{n,l}-z)^{-1}\|_{2\to2}
 \le\frac1{1/4-\operatorname{Re}z},\qquad l\ge3.           
\end{equation*}
For $l=2$ there are only finitely many values of $m$, all unitarily equivalent
to $L_{n,2}$, so those resolvent norms are finite.  Hence the block
resolvents are uniformly bounded for all $l\ge2$.  Given
$G=\sum_{l,m}G_{l,m}\in\mathscr X_{\ge2}$, set
\[
 F_{l,m}=(L_{n,l}-z)^{-1}G_{l,m},
 \qquad F=\sum_{l,m}F_{l,m}.
\]
The uniform resolvent bound first gives $F\in L^2$.  Since
$L_{n,l}F_{l,m}=zF_{l,m}+G_{l,m}$, it also gives
$\sum_{l,m}\|L_{n,l}F_{l,m}\|_2^2<\infty$.  Thus $F$ belongs to the graph
direct-sum domain and $( L_n-z)F=G$.  This proves
\eqref{eq:fullnonradialgap}.
\end{proof}

\begin{proposition}[Stable semigroup and an equivalent modified energy]
\label{prop:stablelinearenergy}\label{thm:stable-semigroup}
For every fixed sufficiently large $n$, there are $M_n\ge1$ and $\omega_n>0$
such that
\begin{equation}
 \|e^{-s L_n}P_{{\rm s},n}f\|_2
 \le M_ne^{-\omega_ns}\|P_{{\rm s},n}f\|_2,\qquad s\ge0.  \label{eq:stablesemigroup}
\end{equation}
On $X_{{\rm s},n}$ define the bounded positive self-adjoint operator
\begin{equation}
 Q_nf=\int_0^\infty
 \bigl(e^{-s L_n}|_{X_{{\rm s},n}}\bigr)^*
 e^{-s L_n}f\dd s.                                  \label{eq:LyapunovQ}
\end{equation}
If $a_n\ge0$ satisfies
\begin{equation*}
 \operatorname{Re}\langle L_nf,f\rangle
 \ge-a_n\|f\|_2^2,\qquad f\in D( L_n),              
\end{equation*}
and
\begin{equation}
 \alpha_n=2a_n+1,\qquad B_n=I+\alpha_nQ_n,\qquad
 \mathcal E_n[f]=\langle B_nf,f\rangle,                    \label{eq:modifiedenergydefinition}
\end{equation}
then
\begin{equation}
 \|f\|_2^2\le\mathcal E_n[f]
 \le\left(1+\frac{\alpha_nM_n^2}{2\omega_n}\right)\|f\|_2^2, \label{eq:energyequivalence}
\end{equation}
and, for some $\gamma_n>0$,
\begin{equation*}
 2\operatorname{Re}\langle B_n L_nf,f\rangle
 \ge\gamma_n\mathcal E_n[f],
 \qquad f\in D( L_n)\cap X_{{\rm s},n}.             
\end{equation*}
Consequently, if $\partial_sf+ L_nf=0$ and
$f(0)\in X_{{\rm s},n}$, then
\begin{equation*}
 \mathcal E_n[f(s)]\le e^{-\gamma_ns}\mathcal E_n[f(0)].  
\end{equation*}
\end{proposition}

\begin{proof}
The proof is given in Appendix~\ref{app:proof-stablelinearenergy}.
\end{proof}

The preceding proposition gives one positive decay rate.  To determine the
exact logarithmic decay exponent, we next use the essential spectrum
\eqref{eq:exactessentialspectrum} and the definition of the stable spectral
bound \eqref{eq:mainsharpgap} to identify the low-mode discrete spectrum that
may lie below $1/4$.  The standard passage from the spectral endpoint to the
fixed-time semigroup spectral radius, and then to a nearly critical
Lyapunov energy, is proved in Appendix~\ref{app:proof-sharpspectraldecay}.

\begin{proposition}[Exact stable spectral endpoint and logarithmic semigroup decay]
\label{prop:sharpspectraldecay}
Let
\[
 A_{{\rm s},n}= L_n|_{X_{{\rm s},n}},\qquad
 S_{{\rm s},n}(s)=e^{-sA_{{\rm s},n}}.
\]
Define
\begin{equation}\label{eq:mainsharpgap}
 \gamma_n^{\rm sp}:=\inf\{\operatorname{Re}z:
 z\in\sigma( L_n|_{X_{{\rm s},n}})\}.
\end{equation}
Then:
\begin{enumerate}
\item $0<\gamma_n^{\rm sp}\le1/4$ and
      \begin{equation}
       \gamma_n^{\rm sp}=\min\{1/4,\gamma_{0,n}^{\rm st},
       \gamma_{1,n}^{\rm st}\}.                            \label{eq:sharpgaplowmodes}
      \end{equation}
      Here, for $l=0,1$, $\gamma_{l,n}^{\rm st}$ is the infimum of the real
      parts of the stable discrete spectrum of $ L_{n,l}$ after the
      known negative eigenspaces in that mode have been removed.  If no such
      discrete spectrum lies in $0<\operatorname{Re}z<1/4$, this quantity is
      defined to be $+\infty$.
\item The logarithmic decay exponent is
      \begin{equation}\label{eq:mainsharprate}
       \lim_{s\to\infty}-\frac1s\log
       \|S_{{\rm s},n}(s)\|_{2\to2}=\gamma_n^{\rm sp},
      \end{equation}
      and, for every $0<\varepsilon<\gamma_n^{\rm sp}$,
      \begin{equation}\label{eq:almostsharprate}
       \|S_{{\rm s},n}(s)\|_{2\to2}
       \le C_{n,\varepsilon}e^{-(\gamma_n^{\rm sp}-\varepsilon)s}.
      \end{equation}
\item For every $0<\eta<\gamma_n^{\rm sp}$, there is an equivalent quadratic
      energy $\mathcal E_{n,\eta}[f]=\langle B_{n,\eta}f,f\rangle$ such that
      \begin{equation}
       2\operatorname{Re}\langle B_{n,\eta}A_{{\rm s},n}f,f\rangle
       \ge2\eta\mathcal E_{n,\eta}[f],                      \label{eq:nearsharpenergycoercive}
      \end{equation}
      and
      \begin{equation*}
       \mathcal E_{n,\eta}[S_{{\rm s},n}(s)f]
       \le e^{-2\eta s}\mathcal E_{n,\eta}[f].             
      \end{equation*}
      Thus the exponential decay rate in the modified energy can be chosen
      arbitrarily close to $\gamma_n^{\rm sp}$.
\end{enumerate}
\end{proposition}

\begin{proof}
We first identify the modes that can determine the exact left endpoint of
the stable spectrum.  Proposition~\ref{prop:fixedmodefullclosure} gives a
unitary equivalence, on the closed domains, between
$ L_n|_{\mathscr X_{l,m}}$ and $L_{n,l}$.  Its Fredholm-index-zero
argument proves that every $l\ge2$ mode has no spectrum in
$\operatorname{Re}z<1/4$, while its fixed-mode Weyl sequence proves that the
left spectral boundary of each such mode is exactly $1/4$.  Only the modes
$l=0,1$ can therefore move the stable spectral bound below $1/4$, which
proves \eqref{eq:sharpgaplowmodes}.

The low-spectrum results above show that, after the finite-dimensional
unstable eigenspaces have been removed, zero is separated from the stable
spectrum.  Hence $\gamma_n^{\rm sp}>0$.  On the other hand, the essential
spectrum begins at $1/4$, so $\gamma_n^{\rm sp}\le1/4$.

The remaining passage from the spectral boundary to the exact logarithmic
semigroup exponent, together with the construction of the nearly critical
Lyapunov energy, uses only quasi-compactness, the spectral-radius formula,
and a weighted orbit integral.  The complete calculation is given in
Appendix~\ref{app:proof-sharpspectraldecay}.
\end{proof}

 \section{Nonlinear stability}
\label{nl:sec:nonlinear-dynamics}

We prove Theorem \ref{thm:nonradial-stability} in this section.  Firstly,
we give the following local well-posedness result for \eqref{eq:KS}.
\begin{proposition}[Local well-posedness and propagation of $\mathcal Y$]
\label{prop:local-wp}
For every $u_0\in H^2(\mathbb R^3)$ there is a unique maximal solution
\[
 u\in C([0,T_{\max});H^2)\cap C^1([0,T_{\max});L^2)
       \cap L^2_{\rm loc}([0,T_{\max});H^3).
\]
It is smooth for positive times. If $T_{\max}<\infty$, then
$\limsup_{t\to T_{\max}}\|u(t)\|_{H^2}=\infty$.
If, in addition, $u_0\in\mathcal Y$, then
\begin{equation}\label{eq:Y-propagation}
 u\in C([0,T_{\max});\mathcal Y),\qquad
 \Lambda u\in L^2_{\rm loc}([0,T_{\max});H^1).
\end{equation}
\end{proposition}

\begin{proof}
Write $T=\nabla\Delta^{-1}$ and
$\mathcal B(f,g)=\nabla\cdot(fTg)$.  In dimension three,
\[
 \|Tg\|_6\lesssim\|g\|_2,
 \qquad\|Tg\|_\infty\lesssim\|g\|_{H^1},
 \qquad\|\nabla Tg\|_\infty\lesssim\|g\|_{H^2}.
\]
These estimates, the Calder\'on--Zygmund bounds for derivatives of
$T$, and the product rule imply
\begin{equation}\label{eq:local-bilinear-H2}
 \|\mathcal B(f,g)\|_{H^1}
 \le C\|f\|_{H^2}\|g\|_{H^2}.
\end{equation}
For example, the highest product derivative is bounded using
$\|\nabla^2 f\|_2\|Tg\|_\infty$,
$\|\nabla f\|_2\|\nabla Tg\|_\infty$, and
$\|f\|_6\|\nabla^2Tg\|_3$.
Combining \eqref{eq:local-bilinear-H2} with
$\|e^{t\Delta}\|_{H^1\to H^2}\lesssim 1+t^{-1/2}$
gives a contraction on $C([0,\tau];H^2)$ for $\tau$ depending on an
$H^2$ bound on the initial datum.  The same estimate for differences
makes the local flow locally Lipschitz in $H^2$ on common finite
existence intervals.  The usual parabolic energy estimate and the
identity $u_t=\Delta u+\mathcal B(u,u)$ then give
$u\in L^2_{\rm loc}H^3$ and $u_t\in C_{\rm loc}L^2$.
Continuation and positive-time smoothing follow by repeating the local
argument.

We next prove the stronger weighted-continuity statement.  Approximate $u_0$ in $\mathcal Y$ by smooth,
compactly supported data and initially work with the corresponding smooth
solutions.  Write $Z=\Lambda u$.  The
commutator identities $[\Lambda,\Delta]=-2\Delta$ and
$[\Lambda,T]=T$ give
\[
 Z_t=\Delta Z+\nabla\cdot\bigl(ZTu+uTZ\bigr)-2u_t.
\]
Since $\nabla\cdot Tu=u$, integration by parts and the
Hardy--Littlewood--Sobolev estimate $\|TZ\|_6\lesssim\|Z\|_2$ yield
\begin{align*}
 \left\langle\nabla\cdot(ZTu),Z\right\rangle
    &=\frac12\int_{\mathbb R^3}u|Z|^2,\\
 \left|\left\langle\nabla\cdot(uTZ),Z\right\rangle\right|
    &\le\|u\|_3\|TZ\|_6\|\nabla Z\|_2
     \lesssim\|u\|_{H^2}\|Z\|_2\|\nabla Z\|_2.
\end{align*}
Consequently, Young's inequality gives, on every compact subinterval
of $[0,T_{\max})$,
\begin{equation}\label{eq:local-Y-energy}
 \frac12\frac{d}{dt}\|Z\|_2^2+\frac12\|\nabla Z\|_2^2
 \le C\bigl(1+\|u\|_{H^2}^2\bigr)\|Z\|_2^2
      +C\|u_t\|_2^2.
\end{equation}
The coefficient and forcing on the right are locally integrable in
time by the $H^2$ local theory.  Gronwall's inequality bounds $Z$
in $L_t^\infty L_x^2\cap L_t^2H_x^1$ on compact time intervals.
Moreover, $Tu\in L_t^\infty L_x^\infty$,
$u\in L_t^\infty L_x^3$, and $TZ\in L_t^\infty L_x^6$, so
\[
 ZTu+uTZ\in L_t^2L_x^2,
 \qquad Z_t\in L_t^2H_x^{-1}.
\]
The evolution-triple continuity theorem applied to
$H^1\hookrightarrow L^2\hookrightarrow H^{-1}$ therefore gives
$Z\in C_tL_x^2$, not merely weak continuity.

Here is also the quantitative dependence needed at a bootstrap exit.
Let $u_i$ ($i=1,2$) be two solutions on a common compact time interval
$[0,\tau]$, with $\sup_i\|u_i\|_{C_tH^2}\le R$, and set
\[
 w=u_1-u_2,\quad Z_i=\Lambda u_i,\quad W=Z_1-Z_2.
\]
Subtracting their weighted equations gives the exact identity
\begin{equation}\label{eq:local-Y-difference}
 W_t=\Delta W+\nabla\cdot
 \bigl(WTu_1+Z_2Tw+wTZ_1+u_2TW\bigr)-2w_t.
\end{equation}
The estimates
$\|Tw\|_\infty\lesssim\|w\|_{H^1}$,
$\|TZ_i\|_6\lesssim\|Z_i\|_2$,
and $\|u_2\|_3\lesssim R$ allow us to test
\eqref{eq:local-Y-difference} against $W$.  The term $WTu_1$
is integrated by parts as above, and Young's inequality yields
\begin{equation}\label{eq:local-Y-difference-energy}
 \frac{d}{dt}\|W\|_2^2+\|\nabla W\|_2^2
 \le C_R\|W\|_2^2
   +C_R\bigl(1+\|Z_1\|_2^2+\|Z_2\|_2^2\bigr)
       \|w\|_{H^2}^2.
\end{equation}
Here $\|w_t\|_2\le C_R\|w\|_{H^2}$ follows directly from the
original difference equation and \eqref{eq:local-bilinear-H2}.
The local $H^2$ flow estimate and \eqref{eq:local-Y-energy} give a
uniform bound on $\|Z_i\|_{C_tL^2}$ on bounded subsets of
$\mathcal Y$ for which the solutions have the stated $H^2$ bound.
Gronwall's inequality applied to \eqref{eq:local-Y-difference-energy}
therefore gives, locally uniformly on such subsets,
\begin{equation}\label{eq:local-Y-lipschitz}
 \sup_{0\le t\le\tau}\|u_1(t)-u_2(t)\|_{\mathcal Y}
 \le C_{R,\tau,\mathcal Y}\|u_1(0)-u_2(0)\|_{\mathcal Y}.
\end{equation}
The argument applies first to smooth data; the uniform estimates,
density of $C_c^\infty$ in $\mathcal Y$, and completeness of that
space extend both strong continuity and \eqref{eq:local-Y-lipschitz}
to general data in $\mathcal Y$.  This proves
\eqref{eq:Y-propagation} and the asserted dependence.
\end{proof}

\begin{lemma}[Translations and dilations on the graph space]
\label{lem:Y-similarity-action}
For $\nu>0$ and $z\in\mathbb R^3$, put
$\mathscr S_{\nu,z}f(y)=\nu^2f(\nu y+z)$.
The map $(\nu,z,f)\mapsto\mathscr S_{\nu,z}f$ is jointly continuous
from $(0,\infty)\times\mathbb R^3\times\mathcal Y$ into $\mathcal Y$.
The operators $\mathscr S_{\nu,z}$ are locally uniformly bounded on
$\mathcal Y$ in $(\nu,z)$.  In particular,
\begin{equation}\label{eq:Y-similarity-generator}
 \Lambda\mathscr S_{\nu,z}f
 =\mathscr S_{\nu,z}(\Lambda f-z\cdot\nabla f),
 \qquad
 \nabla\mathscr S_{\nu,z}f=\nu\mathscr S_{\nu,z}(\nabla f).
\end{equation}
Moreover, $C_c^\infty(\mathbb R^3)$ is dense in $\mathcal Y$.
\end{lemma}

\begin{proof}
The identities \eqref{eq:Y-similarity-generator} hold first for smooth
functions, and by distributional approximation for every $f\in\mathcal Y$.
They give the local uniform boundedness.  Strong continuity of
translations and dilations on $H^2$ and $L^2$, applied to $f$,
$\Lambda f$, and $\nabla f$, then gives joint continuity in
$\mathcal Y$.
For density, truncate $f$ by $\chi_R(y)=\chi(y/R)$: the identity
\[
 \Lambda(\chi_Rf)=\chi_R\Lambda f+(y\cdot\nabla\chi_R)f
\]
and boundedness of $y\cdot\nabla\chi_R$ imply convergence in the
$\mathcal Y$ graph norm.  For a compactly supported truncated $f$,
mollify with a standard approximate identity $\rho_\eta$.
The commutator
\[
 [\Lambda,\rho_\eta*]f
 =\bigl(3\rho_\eta+y\cdot\nabla\rho_\eta\bigr)*f
\]
has an $L^1$ kernel of integral zero and uniformly bounded $L^1$ norm;
it converges strongly to zero on $L^2$.  Thus mollification also
converges in $\mathcal Y$, proving density.
\end{proof}

\begin{remark}[Compatibility with the closed linearized generator]
The domain in \eqref{us:domain} is exactly
$D(L_n)=\mathcal Y$ as a set.  The free graph estimate
\eqref{us:free-graph} and the infinitesimal $L_0$-boundedness of
$K_n=L_n-L_0$ proved in the linear analysis show that
\[
 C_n^{-1}\|f\|_{\mathcal Y}
 \le\|f\|_2+\|L_nf\|_2
 \le C_n\|f\|_{\mathcal Y},\qquad f\in\mathcal Y.
\]
Thus Proposition~\ref{prop:local-wp} propagates the graph domain of
the closed generator, and the continuity established there is strong
graph-norm continuity.  No differentiability of the translation or
dilation action in the $\mathcal Y$ norm is needed for the
finite-dimensional geometric decomposition.
\end{remark}

\subsection{Geometric decomposition and renormalized flow}

Throughout this section, \(n\) is fixed. Let $U_n$ be the corresponding radial self-similar profile which satisfies \eqref{eq:steadystateU}.
For simplicity, we write $\phi_j$, $\zeta_{0}$, $
\zeta_{k}^{\rm tr}$,
$\zeta_{j}^{\rm u}$, $\mathcal X_{\delta}$, $\mu_j$, $N$
instead of $\phi_{j,n}$, $\zeta_{0,n}$, $
\zeta_{k,n}^{\rm tr}$,
$\zeta_{j,n}^{\rm u}$, $\mathcal X_{n,\delta}$, $\mu_{j,n}$, $N_n$ in the following. Since all functions under consideration below are real-valued, we omit
$\operatorname{Re}$ in the $L^2$ scalar products below. 

Recall that
\[
L_n(\Lambda U_n)=-\Lambda U_n,
\qquad
L_n(\partial_k U_n)=-\frac12\partial_k U_n,
\qquad 1\le k\le3.
\]
Moreover, according to
Theorem~\ref{thm:main},
apart from the scaling mode,
there also exist exactly $N$ unstable radial eigenfunctions
$\phi_{j}$ $(1\le j\le N)$, satisfying
\[
L_n\phi_{j}=-\mu_j\phi_{j},
\qquad
\mu_j>1.
\]

Since $L_n$ is not a self-adjoint operator, by \eqref{eq:biorthogonal}, we
know that the corresponding eigenfunctions of the adjoint operator $\zeta_{0}$, $
\zeta_{k}^{\rm tr}$,
$\zeta_{j}^{\rm u}$
satisfying
\[
L_n^*\zeta_{0}
=
-\zeta_{0},
\ 
L_n^*\zeta_{k}^{\rm tr}
=
-\frac12\zeta_{k}^{\rm tr},
\ \ 
L_n^*\zeta_{j}^{\rm u}
=
-\mu_j\zeta_{j}^{\rm u},
\ \ 1\le k\le3,\ 
1\le j\le N.
\]
We normalize them so that
\begin{equation}\label{dual-normalization-mod}
\begin{aligned}
\langle \Lambda U_n,\zeta_{0}\rangle =1,\ 
\langle \partial_kU_n,\zeta_{l}^{\rm tr}\rangle
=\delta_{k\ell},\ 
\langle \phi_{i},\zeta_{j}^{\rm u}\rangle
=\delta_{ij},\ \ 1\le k,\ell\le3,\ 1\le i,j\le N
\end{aligned}
\end{equation}
and all the remaining cross pairings vanish.

We define the $H^2$ tube around the rescaled and translated
version of $U_n$ by
\begin{equation*}
\mathcal X_{\delta}
=
\left\{
u(x)
=
\frac1{\mu^2}
\left(U_n+\bar u\right)
\left(\frac{x-x'}{\mu}\right);
\ 
\mu>0,\ x'\in\mathbb R^3,\ 
\|\bar u\|_{H^2}<\delta
\right\}.
\end{equation*}
Inspired by \cite{Collot-Pierre-2019, CollotZhang2024}, we give the following decomposition of the solution in $\mathcal X_{\delta}$.
\begin{lemma}[Geometrical decomposition]\label{lemma-geometric-Un}
There exist $\delta>0$ and $C>0$ such that every
$u\in\mathcal X_{\delta}$ admits a locally unique decomposition
\begin{equation}\label{geometric-decomposition-Un}
u(x)
=
\frac1{\lambda^2}
\left(
U_n
+\sum_{j=1}^N a_j\phi_{j}
+\varepsilon
\right)
\left(\frac{x-\bar x}{\lambda}\right),
\end{equation}
where
$\lambda>0,\ 
\bar x\in\mathbb R^3,\ 
a_j\in\mathbb R,$
and the remainder $\varepsilon$ satisfies
\begin{equation}\label{orthogonality-Un}
\begin{aligned}
\langle\varepsilon,\zeta_{0}\rangle=0,\ 
\langle\varepsilon,\zeta_{k}^{\rm tr}\rangle=0,\ 
\langle\varepsilon,\zeta_{j}^{\rm u}\rangle=0, \  1\le k\le3,
\ 1\le j\le N.
\end{aligned}
\end{equation}
Moreover, the maps
\[
\bar u\mapsto \bar\lambda(\bar u),\qquad
\bar u\mapsto \widetilde x(\bar u),\qquad
\bar u\mapsto a_j(\bar u),
\]
are $C^1$ and well-defined in a  small neighborhood of the origin in
$H^2(\mathbb R^3)$, and satisfy
\[
\lambda=\mu\bar\lambda(\bar u),\qquad
\bar x=x'+\mu\widetilde x(\bar u).
\]
 In particular,
\[
|\bar\lambda(\bar u)-1|
+|\widetilde x(\bar u)|
+\sum_{j=1}^N|a_j(\bar u)|
\le C\|\bar u\|_{H^2}.
\]
\end{lemma}

\begin{proof}
For
\[
\nu>0,\qquad
z\in\mathbb R^3,\qquad
b=(b_1,\ldots,b_N)\in\mathbb R^N,
\]
we define
\begin{equation}\label{Fn-geometric}
F_n(\bar u,\nu,z,b)(y)
=
\nu^2
\left(U_n+\bar u\right)(\nu y+z)
-U_n(y)
-\sum_{j=1}^N b_j\phi_j(y).
\end{equation}
We next define the finite-dimensional map
\begin{equation*}
\begin{aligned}
G_n(\bar u,\nu,z,b)
=
\Big(
&\langle F_n,\zeta_{0}\rangle,
\langle F_n,\zeta_{1}^{\rm tr}\rangle,
\langle F_n,\zeta_{2}^{\rm tr}\rangle,
\langle F_n,\zeta_{3}^{\rm tr}\rangle,
\langle F_n,\zeta_{1}^{\rm u}\rangle,
\ldots,
\langle F_n,\zeta_{N}^{\rm u}\rangle
\Big).
\end{aligned}
\end{equation*}
Clearly,
$G_n(0,1,0,0)=0.$
We now compute the differential of $F_n$ at the point $(0,1,0,0)$.
Since
\[
\left.
\partial_\nu
\big[
\nu^2U_n(\nu y)
\big]
\right|_{\nu=1}
=
2U_n+y\cdot\nabla U_n
=
\Lambda U_n,
\]
we have
\[
\left.\partial_\nu F_n\right|_{(0,1,0,0)}
=
\Lambda U_n.
\]
Similarly,
\[
\left.\partial_{z_k}F_n\right|_{(0,1,0,0)}
=
\partial_kU_n,
\qquad
1\le k\le3,
\]
and
\[
\left.\partial_{b_j}F_n\right|_{(0,1,0,0)}
=
-\phi_j,
\qquad
1\le j\le N.
\]
Consequently, by the duality relations \eqref{dual-normalization-mod}, we have
\[
\left.
D_{(\nu,z,b)}G_n
\right|_{(0,1,0,0)}
=
\begin{pmatrix}
1 & 0 & 0\\
0 & I_3 & 0\\
0 & 0 & -I_N
\end{pmatrix}.
\]
Here $I_N$ is the identity matrix of order $N$, hence the above matrix is
invertible.

For every adjoint mode $\zeta$, a change of variables gives
\[
\left\langle
\nu^2\bar u(\nu\,\cdot+z),\zeta
\right\rangle
=
\nu^{-1}
\int_{\mathbb R^3}
\bar u(x)
\zeta\left(\frac{x-z}{\nu}\right)\,dx.
\]
By Proposition~\ref{prop:adjointmodeasymptotics},
\[
\zeta,\quad \nabla\zeta,\quad
y\cdot\nabla\zeta
\in L^2(\mathbb R^3).
\]
Consequently, the derivatives with respect to $\nu$ and $z$ fall on
the adjoint mode rather than on $\bar u$, and the resulting terms are
bounded by $C\|\bar u\|_{L^2}$, hence by
$C\|\bar u\|_{H^2}$.
In fact, because the derivatives of the perturbation term can be
transferred to the adjoint test functions, $G_n$ is $C^1$ even when
its perturbation variable is endowed with the $L^2$ topology.  This
stronger observation justifies differentiating the orthogonality
conditions along the $C_t^1L^2$ local flow, without claiming that the
translation or dilation action is differentiable as an operator on
$H^2$.

By the implicit function theorem, there exist unique $C^1$ maps
\[
\bar u\mapsto
\bar\lambda(\bar u),\qquad
\bar u\mapsto\widetilde x(\bar u),\qquad
\bar u\mapsto a_j(\bar u),
\]
defined for $\|\bar u\|_{H^2}$ sufficiently small, such that
\[
 G_n\bigl(
 \bar u,\bar\lambda(\bar u),
 \widetilde x(\bar u),a(\bar u)
 \bigr)=0.
\]
The uniqueness in the implicit function theorem also shows that any
second tuple $(\nu,z,b)$ in the same neighborhood satisfying
$G_n(\bar u,\nu,z,b)=0$ coincides with
$(\bar\lambda,\widetilde x,a)$. 
Moreover,
$
\bar\lambda(0)=1,\ 
\widetilde x(0)=0,\ 
a_j(0)=0.$
Since these maps are $C^1$, decreasing the neighborhood if necessary gives
\begin{equation}\label{modulation-parameter-bound}
|\bar\lambda(\bar u)-1|
+
|\widetilde x(\bar u)|
+
\sum_{j=1}^N|a_j(\bar u)|
\le
C\|\bar u\|_{H^2}.
\end{equation}

Setting
\begin{equation*}
\varepsilon(y)
=
F_n\big(
\bar u,\bar\lambda,\widetilde x,a
\big)(y),
\end{equation*}
we obtain exactly the orthogonality conditions
\eqref{orthogonality-Un}.  Equation \eqref{Fn-geometric} also gives
\begin{equation}\label{relative-decomposition-Un}
\left(U_n+\bar u\right)
(\bar\lambda y+\widetilde x)
=
\frac1{\bar\lambda^2}
\left(
U_n+\sum_{j=1}^Na_j\phi_{j}
+\varepsilon
\right)(y).
\end{equation}

Now let $u\in\mathcal X_{\delta}$ be written as
\[
u(x)
=
\frac1{\mu^2}
\left(U_n+\bar u\right)
\left(\frac{x-x'}{\mu}\right).
\]
Using \eqref{relative-decomposition-Un}, we find
\[
u(x)
=
\frac1{(\mu\bar\lambda)^2}
\left(
U_n+\sum_{j=1}^Na_j\phi_{j}
+\varepsilon
\right)
\left(
\frac{x-x'-\mu\widetilde x}
{\mu\bar\lambda}
\right).
\]
Therefore, defining
\[
\lambda=\mu\bar\lambda,
\qquad
\bar x=x'+\mu\widetilde x,
\]
gives the decomposition \eqref{geometric-decomposition-Un}.

Using the regularity and decay properties of $U_n$ and $\phi_j$,
together with \eqref{modulation-parameter-bound}, we further obtain
\[
\|\varepsilon\|_{H^2}
+
\left|\frac{\lambda}{\mu}-1\right|
+
\sum_{j=1}^N|a_j|
+
\frac{|\bar x-x'|}{\mu}
\le
C\|\bar u\|_{H^2}.
\]
This completes the proof.
\end{proof}

\subsection{Bootstrap setting and modulation estimate}\label{Nonl}

Consider initial data of the form
\begin{equation}\label{initial-data}
u_0=\frac{1}{\lambda_0^2}\left(U_n+v_0\right)\left(\frac{x}{\lambda_0}\right),
\end{equation}
where
\begin{equation*}
v_0=\sum_{j=1}^N a_{j,0}\phi_{j}
+\varepsilon_0
\end{equation*}
where $\varepsilon_0$ satisfies
\begin{equation}\label{initial-stable-orthogonality}
\begin{aligned}
\langle\varepsilon_0,\zeta_{0}\rangle=0,\ 
\langle\varepsilon_0,\zeta_{k}^{\rm tr}\rangle=0,\ 
\langle\varepsilon_0,\zeta_{j}^{\rm u}\rangle=0, \  1\le k\le3,
\ 1\le j\le N.
\end{aligned}
\end{equation}
Let $s_0\gg1$ and $0<\mu,K_0\ll1$ be fixed below, and assume
\begin{equation}\label{initial-setting-KS}
\lambda_0=\lambda(s_0)=e^{-\frac{s_0}{2}},\ ||\Lambda \varepsilon_0||_2\le K_0e^{-\mu s_0},\ 
\left\|\varepsilon_0\right\|_{H^2} \le K_0 e^{-\mu s_0}.
\end{equation}
and
\begin{equation}\label{initial-unstable-model-KS}
    \sum_{j=1}^{N}|a_{j,0}|^2\le e^{-2\mu s_0}.
\end{equation}
\textbf{Renormalized flow.}
As long as the solution $u(t)$ starting from \eqref{initial-data} belongs to $\mathcal X_{\delta}$, then we apply  Lemma \ref{lemma-geometric-Un} to deduce that it can be written in the form
\begin{equation}\label{re}
u(t, x)=\frac{1}{\lambda(t)^2}\left(U_n+v\right)(s, y),\ y=\frac{x-x(t)}{\lambda(t)},
\end{equation}
where \begin{equation}\label{dec}
v=\varepsilon+\psi,\ \ \psi=\sum_{j=1}^{N}a_j\phi_j\end{equation}
and
$\varepsilon$  satisfying
\begin{equation}\label{Ortho}
\begin{aligned}
\langle\varepsilon,\zeta_{0}\rangle=0,\ 
\langle\varepsilon,\zeta_{k}^{\rm tr}\rangle=0,\ 
\langle\varepsilon,\zeta_{j}^{\rm u}\rangle=0, \  1\le k\le3,
\ 1\le j\le N.
\end{aligned}
\end{equation}
and $s$ is the  renormalized time defined by
\begin{equation}\label{renormalized-time}
   s(t):=\int_0^t\frac{d\tau}{\lambda^2(\tau)}+s_0. 
\end{equation}
The proof of Lemma~\ref{lemma-geometric-Un} shows that the
finite-dimensional orthogonality map is $C^1$ in its $L^2$
perturbation variable, even though dilation is not differentiable
on general $H^2$ or $\mathcal Y$ data.  Since
$u\in C_t^1L_x^2$ by Proposition~\ref{prop:local-wp},
the finite-dimensional implicit function theorem therefore gives
differentiable modulation parameters $\lambda(t)$, $x(t)$, and
$a_j(t)$ whenever the solution remains in the tube
$\mathcal X_\delta$.  The graph-space continuity in
Lemma~\ref{lem:Y-similarity-action} is what is used separately for
bootstrap continuity.

Inserting \eqref{re} into the equation \eqref{eq:KS}, one has the renormalized remainder equation
\begin{equation}\label{renormalized equation}
\partial_{s} \varepsilon+L_n \varepsilon=F+\operatorname{Mod},
\end{equation}
with the modulation term
\begin{equation*}
\operatorname{Mod}=\sum_{j=1}^N[\mu_ja_j-(a_j)_s]\phi_j+\left(\frac{\lambda_s}{\lambda}+\frac{1}{2}\right) \left(\Lambda U_n+\Lambda\psi\right)+\frac{x_s}{\lambda}\cdot(\nabla U_n+\nabla\psi),
\end{equation*}
and the force terms are decomposed as follows:
\begin{equation}\label{force1}
 F=\widetilde L(\varepsilon)+NL+UN,
 \qquad
 NL=\varepsilon^2+\nabla\varepsilon\cdot\nabla\Delta^{-1}\varepsilon,
 \qquad
 UN=\psi^2+\nabla\psi\cdot\nabla\Delta^{-1}\psi.
\end{equation}
\begin{equation}\label{force}
\widetilde L(\varepsilon)=\left(\frac{\lambda_s}{\lambda}+\frac{1}{2}\right) \Lambda\varepsilon+\frac{x_s}{\lambda}\cdot\nabla\varepsilon+\nabla\varepsilon\cdot\nabla\Delta^{-1}\psi+\nabla\psi\cdot\nabla\Delta^{-1}\varepsilon+2\psi\varepsilon.
\end{equation}

We claim the following bootstrap proposition.
\begin{proposition}[Nonlinear bootstrap]\label{bootstrap}
For each sufficiently large fixed $n$, there exist constants
$0<\mu,K_0\ll1$, $K\gg1$, and $K'\gg1$ with the following property.
For every sufficiently large
$s_0\ge s_{0,*}(n,\mu,K_0,K,K' )$ (with $K'$ fixed as below) and every $\varepsilon_0$ satisfying the
orthogonality conditions in \eqref{initial-stable-orthogonality} and the two norm bounds in
\eqref{initial-setting-KS}, there exists
$(a_{1,0},\ldots,a_{N,0})$ satisfying
\eqref{initial-unstable-model-KS} such that the solution with initial datum
\eqref{initial-data} admits the decomposition \eqref{re} for all
$s\ge s_0$ and satisfies the following estimates:
\begin{enumerate}
\item[(i)] The scaling parameter satisfies
\begin{equation}\label{1.7}
0<\lambda(s)<e^{-\mu s};
\end{equation}
\item[(ii)] The unstable coefficients satisfy
\begin{equation}\label{chca}
\sum_{j=1}^{N}\left|a_{j}\right|^2 \le e^{-2\mu s};
\end{equation}
\item[(iii)] The stable remainder satisfies
\begin{equation}\label{1.9}
\left\|\varepsilon\right\|_{H^2} < K e^{-\mu s};
\end{equation}
\item[(iv)] The scaling derivative satisfies
\begin{equation}\label{L2Lambda}
\left\|\Lambda\varepsilon\right\|_{L^2} < K' e^{-\mu s};
\end{equation}
\end{enumerate}
\end{proposition}

The remainder of this section proves Proposition \ref{bootstrap}; the final two subsection then deduces the nonlinear theorem from it.
We next define the exit time
\begin{equation*}
s^*=\sup\{s\ge s_0: \text{the bounds}\ \eqref{1.7}-\eqref{L2Lambda}\ \text{hold on}\ [s_0,s) \},
\end{equation*}
and we assume, by contradiction, that
\[
s^*<+\infty.
\]
From now on we work on $[s_0,s^*]$, where
\eqref{1.7}--\eqref{L2Lambda} hold.

\medskip
\noindent
\textbf{Choice of the bootstrap constants; no circular dependence.}
We specify a single selection order that will also be used for the
Lipschitz graph.  First fix the large profile index \(n\), its
spectral projections, and the constants in the linear Lyapunov and
G\aa rding estimates.  Fix the small Young-inequality coefficients
needed to absorb the gradient terms.  This yields positive numbers
\(\kappa_{H,n}\), \(\kappa_{\Lambda,n}\), and
\(\kappa_{{\rm diff},n}\) for the \(H^2\), scaling-derivative,
and difference energies, respectively.  Their construction uses the
fixed linear operators, not the eventual decay exponent \(\mu\),
the bootstrap amplitudes, or \(s_0\).  Choose once and for all
\begin{equation}\label{eq:bootstrap-mu-uniform}
 0<\mu<\min\left\{\frac12,\frac{\gamma_n}{4},
   \frac{\kappa_{H,n}}{2},
   \frac{\kappa_{\Lambda,n}}{2},
   \frac{\kappa_{{\rm diff},n}}{2}\right\}.
\end{equation}
Fix \(K_0>0\) small and then \(K\ge1\) so large that the initial
terms in the \(H^2\) improvement (and the corresponding \(L^2\)
improvement) are at most \(K/4\).  The preliminary smallness
condition \(K^2e^{-\mu s_0}\le K_0\) will be enforced below.

The scaling-derivative estimate is \emph{independent of a prior
bound involving \(K'\)}: after the preceding smallness condition,
its proof gives \eqref{eq:Lambda-uniform-prechoice} below with a
constant independent of \(s_0\).  We therefore fix a finite
\(K'\), depending only on \(n,\mu,K_0\) and the already fixed
linear/Young constants, so that its improvement holds for every
sufficiently large \(s_0\).  In particular, \(K'\) is fixed
\emph{before} the final choice of \(s_0\), not afterward.

Finally choose \(s_0\ge s_{0,*}(n,\mu,K_0,K,K')\) large enough
for all the estimates in this section and for the difference
argument in Proposition~\ref{prop:Lipschitz-dependence}.  Concretely,
we require the finitely many inequalities
\begin{equation}\label{eq:bootstrap-constant-ledger}
\begin{aligned}
 &K^2e^{-\mu s_0}\le K_0,\qquad
 2e^{-(1/2-\mu)s_0}<\tfrac12,\qquad
 C_nK^2e^{-2\mu s_0}<\log2,\\
 &C_n(K+1)e^{-\mu s_0}<\delta_{\rm geo},\qquad
 C_n(K^2+1)e^{-\mu s_0}<\tfrac14(\mu+\mu_*),\\
 &C_{n,K,K'}e^{-\mu s_0}<\eta_n,
\end{aligned}
\end{equation}
where \(\delta_{\rm geo}\) is the fixed geometric-decomposition
radius, \(\mu_*=\min_j\mu_j>0\), and \(\eta_n>0\) is chosen
small enough for the modulation matrix and the coercive-dissipation
absorptions, including those in the two-solution argument.  The
constants \(C_n,C_{n,K,K'}\) are fixed \emph{before} \(s_0\)
is increased.  All conditions then hold simultaneously, and
further enlarging \(s_0\) only improves them.  This order removes
the potential circularity between selecting \(K'\) and requiring
smallness of \(K'e^{-\mu s_0}\).

\begin{lemma}[Quadratic bounds for the modulation equations]\label{modulation}
If $s_0$ is sufficiently large, then
\begin{equation}\label{semin}
\left|\frac{\lambda_s}{\lambda}+\frac12\right|
+\sum_{j=1}^N\left|(a_j)_s-\mu_ja_j\right|
+\frac{|x_s|}{\lambda}
\lesssim\|\varepsilon\|_2^2+\sum_{j=1}^N|a_j|^2,
\qquad s\in[s_0,s^*].
\end{equation}
\end{lemma}
\begin{proof}
We use the conservative structure of the quadratic term to obtain
one finite-dimensional system, rather than estimating the individual
modulation equations separately.  Put
\[
 b=\frac{\lambda_s}{\lambda}+\frac12,
 \qquad \beta=\frac{x_s}{\lambda},
 \qquad q_j=(a_j)_s-\mu_ja_j,
 \qquad v=\varepsilon+\psi.
\]
Since $\nabla\cdot(Tv)=v$, the sum of the three quadratic terms
in \eqref{force1}--\eqref{force} is exactly
\begin{equation}\label{eq:nonlinear-flux-single}
 v^2+\nabla v\cdot Tv
 =\nabla\cdot(vTv).
\end{equation}
Consequently the remainder equation \eqref{renormalized equation}
can be written, without omitted linear terms, as
\begin{equation}\label{eq:modulation-one-system}
 \partial_s\varepsilon+L_n\varepsilon
 =b\Lambda(U_n+v)
   +\beta\cdot\nabla(U_n+v)
   -\sum_{j=1}^Nq_j\phi_j
   +\nabla\cdot(vTv).
\end{equation}

Let $\zeta_m$ range over $\zeta_0$, the three translation adjoints,
and the $N$ radial unstable adjoints.  Since
$L_n^*\zeta_m=-\nu_m\zeta_m$, the orthogonality conditions
\eqref{Ortho} imply
$\langle\partial_s\varepsilon+L_n\varepsilon,\zeta_m\rangle=0$.
The pairing is legitimate by the closed adjoint-domain result in
Proposition~\ref{prop:adjointmodeasymptotics}; at positive times one
may first use smooth solutions and then pass to the initial time by
continuity.  Testing \eqref{eq:modulation-one-system} therefore yields
\begin{equation}\label{eq:modulation-finite-matrix}
 \bigl(M_0+E(v)\bigr)
 \begin{pmatrix}b\\\beta\\q\end{pmatrix}
 =-\begin{pmatrix}
 \mathscr N_{\zeta_0}(v,v)\\
 (\mathscr N_{\zeta_k^{\rm tr}}(v,v))_{k=1}^3\\
 (\mathscr N_{\zeta_j^{\rm u}}(v,v))_{j=1}^N
 \end{pmatrix},
 \qquad M_0=\operatorname{diag}(1,I_3,-I_N).
\end{equation}
Here the first four columns of $E(v)$ have entries
$\langle\Lambda v,\zeta_m\rangle$ and
$\langle\partial_kv,\zeta_m\rangle$, respectively, and the
remaining columns vanish.  The diagonal matrix $M_0$ is exactly
that obtained from the biorthogonality relations
\eqref{dual-normalization-mod}.

Lemma~\ref{lem:adjoint-drift-testing} transfers the derivatives to
the fixed adjoint modes, giving
\begin{equation}\label{eq:modulation-matrix-small}
 \|E(v)\|\le C_n\|v\|_2,\qquad
 |(\mathscr N_{\zeta_m}(v,v))_m|
 \le C_n\|v\|_2^2.
\end{equation}
Since $n$ is fixed, all norms on the finite-dimensional coefficient
space are equivalent.  For $\|v\|_2\le(2C_n)^{-1}$,
$M_0+E(v)$ is invertible with a uniformly bounded inverse
(by the Neumann series after factoring out $M_0$).
The bootstrap assumptions \eqref{chca} and \eqref{1.9} give
$\|v\|_2\le C_n(K+1)e^{-\mu s}$, so this smallness holds on
$[s_0,s^*]$ when $s_0$ is chosen sufficiently large.
It follows that
\[
 |b|+|\beta|+\sum_{j=1}^N|q_j|
 \le C_n\|v\|_2^2
 \le C_n\left(\|\varepsilon\|_2^2
             +\sum_{j=1}^N|a_j|^2\right),
\]
which is \eqref{semin}.
\end{proof}

\subsection{Energy estimates}

For simplicity, we write
\[
b:=\frac{\lambda_s}{\lambda}+\frac12,
    \qquad
    \beta:=\frac{x_s}{\lambda},\qquad |a|^2:=\sum_{j=1}^N|a_j|^2,
\]
in the following.  We first estimate the $L^2$ norm of $\varepsilon$ on
$[s_0,s^*]$.

We first record the compatibility between the Lyapunov energy of
Proposition~\ref{prop:stablelinearenergy} and the infinitesimal scaling and
translation operators.  This is needed because $B_n$ is not a multiplication
operator, so the required estimate does not follow merely from the boundedness
of $B_n$ on $L^2$.

\begin{lemma}[Lyapunov energy under projected scaling and translation]
\label{lem:projected-generator-energy}
Let
\[
 \Gamma_0=\Lambda,
 \qquad \Gamma_k=\partial_k\quad(1\le k\le3),
 \qquad
 D_{j,n}=P_{{\rm s},n}\Gamma_j|_{X_{{\rm s},n}}.
\]
Let $B_n$ and $\mathcal E_n[f]=\langle B_nf,f\rangle$ be defined by
\eqref{eq:LyapunovQ}--\eqref{eq:modifiedenergydefinition}.  There is a
constant $C_n$ such that, for $0\le j\le3$ and every $f$ in the domain of
$D_{j,n}$,
\begin{equation}\label{eq:projected-generator-form}
 \left|\operatorname{Re}\langle B_nD_{j,n}f,f\rangle\right|
 \le C_n\mathcal E_n[f].
\end{equation}
\end{lemma}

\begin{proof}
Write
\[
 S_n(s)=e^{-sL_n},
 \qquad
 S_{{\rm s},n}(s)=S_n(s)|_{X_{{\rm s},n}}.
\]
By Proposition~\ref{prop:stablelinearenergy}, after decreasing the decay
exponent if necessary, there are $C_n>0$ and $\omega_n>0$ such that
\begin{equation}\label{eq:stable-semigroup-used-nonlinear}
 \|S_{{\rm s},n}(s)\|_{2\to2}\le C_ne^{-\omega_ns},
 \qquad s\ge0.
\end{equation}

We first prove a short-time commutator bound.  Put
\[
 L_0=-\Delta+\frac12\Lambda,
 \qquad
 K_n=L_n-L_0.
\]
The free semigroup is
\begin{equation}\label{eq:free-OU-semigroup-nonlinear}
 e^{-sL_0}f(y)
 =e^{-s}\bigl[e^{(1-e^{-s})\Delta}f\bigr](e^{-s/2}y).
\end{equation}
Since dilations commute with $\Lambda$ and
\[
 [e^{a\Delta},\Lambda]=2a\Delta e^{a\Delta},
\]
the heat-semigroup estimate
$\|a\Delta e^{a\Delta}\|_{2\to2}\le C$ gives
\begin{equation}\label{eq:free-scaling-commutator-nonlinear}
 \|[e^{-sL_0},\Lambda]\|_{2\to2}\le C,
 \qquad 0<s\le1.
\end{equation}
In fact this $O(1)$ small-time bound is sharp, not an
$O(s^{1/2})$ bound. To see this, write $a=1-e^{-s}$,
$\delta=e^{-s/2}$, and $D_\delta f(y)=f(\delta y)$. Then
\[
 [S_0(s),\Lambda]
 =2ae^{-s}D_\delta\Delta e^{a\Delta},
 \qquad
 \|D_\delta\|_{2\to2}=\delta^{-3/2}.
\]
Since $\|a\Delta e^{a\Delta}\|_{2\to2}=e^{-1}$,
this yields the exact identity
\begin{equation}\label{eq:exact-free-scaling-commutator}
{\|[S_0(s),\Lambda]\|_{2\to2}
 =2e^{-1}e^{-s/4}},\qquad s>0.
\end{equation}
In particular, the commutator is uniformly bounded but does not
tend to zero in operator norm as $s\to0$.
Similarly, either differentiating \eqref{eq:free-OU-semigroup-nonlinear} or
using $[L_0,\partial_k]=-\frac12\partial_k$ gives
\begin{equation}\label{eq:free-translation-commutator-nonlinear}
 \|[e^{-sL_0},\partial_k]\|_{2\to2}\le Cs^{1/2},
 \qquad 0<s\le1.
\end{equation}

Let $b_n=\nabla\Delta^{-1}U_n$ and $R=\nabla\Delta^{-1}$.  Direct use of
the product rule, together with
$R(\Lambda f)=(1+y\cdot\nabla)Rf$, gives
\begin{align*}
 [K_n,\Lambda]f
 ={}&2(y\cdot\nabla U_n)f
 -(b_n-y\cdot\nabla b_n)\cdot\nabla f +(\nabla U_n+y\cdot\nabla\nabla U_n)\cdot Rf,            \\ 
 [K_n,\partial_k]f
 ={}&2(\partial_kU_n)f
 +(\partial_kb_n)\cdot\nabla f
 +(\partial_k\nabla U_n)\cdot Rf.                         
\end{align*}
The profile estimates proved above, the Sobolev inequality
$\|Rf\|_6\lesssim\|f\|_2$, and H\"older's inequality imply
\begin{equation}\label{eq:K-commutator-H1-nonlinear}
 \|K_nf\|_2+\|[K_n,\Gamma_j]f\|_2
 \le C_n\|f\|_{H^1}.
\end{equation}
Let $S_0(s)=e^{-sL_0}$.  The two Duhamel formulas
\[
 S_n(s)=S_0(s)-\int_0^sS_0(s-\tau)K_nS_n(\tau)\,\dd\tau
       =S_0(s)-\int_0^sS_n(s-\tau)K_nS_0(\tau)\,\dd\tau
\]
together with \eqref{eq:K-commutator-H1-nonlinear} first give
\[
 \|S_n(s)\|_{2\to2}
 +s^{1/2}\|S_n(s)\|_{2\to H^1}\le C_n,
 \qquad 0<s\le1.
\]
Indeed, successive substitution produces only convolutions of
$s^{-1/2}$; the first nontrivial convolution is
\[
 \int_0^s(s-\tau)^{-1/2}\tau^{-1/2}\,d\tau=\pi,
\]
and all higher convolutions are bounded on $0<s\le1$ by the corresponding
beta integrals.

For completeness, put
$C_j(s)=[S_n(s),\Gamma_j]$ and
$C_j^0(s)=[S_0(s),\Gamma_j]$.  Commuting $\Gamma_j$ through the second
Duhamel formula gives, initially on $C_c^\infty$,
\begin{align*}
 C_j(s)=C_j^0(s)-\int_0^s\bigl[{}
 &C_j(s-\tau)K_nS_0(\tau)
 +S_n(s-\tau)[K_n,\Gamma_j]S_0(\tau)\\
 &+S_n(s-\tau)K_nC_j^0(\tau)\bigr],\dd\tau.
\end{align*}
Besides \eqref{eq:free-scaling-commutator-nonlinear} and
\eqref{eq:free-translation-commutator-nonlinear}, the explicit free formula
also gives
\[
 \|S_0(\tau)\|_{2\to H^1}
 +\|C_j^0(\tau)\|_{2\to H^1}\le C\tau^{-1/2},\qquad0<\tau\le1.
\]
Thus the preceding identity implies
\[
 \|C_j(s)\|_{2\to2}
 \le C_n+C_n\int_0^s\tau^{-1/2}
       \bigl(1+\|C_j(s-\tau)\|_{2\to2}\bigr)\,\dd\tau.
\]
Iteration and the same beta-integral calculation yield
\begin{equation}\label{eq:full-short-commutator-nonlinear}
 \|[S_n(s),\Gamma_j]\|_{2\to2}\le C_n,
 \qquad 0<s\le1.
\end{equation}

The explicit finite-rank formula \eqref{eq:unstableprojection} shows that
\begin{equation*}
 [P_{{\rm s},n},\Gamma_j]\in\mathcal B(L^2).
\end{equation*}
To see this directly, move $\Gamma_j$ from $f$ onto each adjoint mode in
$\langle\Gamma_jf,\zeta\rangle$.  The required functions
$\Gamma_j^*\zeta$ belong to $L^2$ by
Proposition~\ref{prop:adjointmodeasymptotics}, while the right eigenfunctions
and their images under $\Gamma_j$ belong to $L^2$ by the profile and elliptic
estimates above. In particular, $P_{{\rm u},n}\Gamma_j$ extends from
$D(\Gamma_j)$ to a bounded operator on $L^2$. This also proves that
$D_{j,n}=P_{{\rm s},n}\Gamma_j|_{X_{{\rm s},n}}$ is closed:
if $f_m\to f$ and $D_{j,n}f_m\to h$ in the stable subspace, then
$\Gamma_jf_m=D_{j,n}f_m+P_{{\rm u},n}\Gamma_jf_m$ converges in $L^2$;
closedness of $\Gamma_j$ gives $f\in D(D_{j,n})$ and $D_{j,n}f=h$.
Since $S_n(s)$ commutes with $P_{{\rm s},n}$, for
$f\in X_{{\rm s},n}$ we have
\begin{equation*}
 [S_{{\rm s},n}(s),D_{j,n}]f
 =P_{{\rm s},n}[S_n(s),\Gamma_j]f.
\end{equation*}
Thus \eqref{eq:full-short-commutator-nonlinear} controls this commutator for
$0<s\le1$. We explain why these initially formal identities are
legitimate for the unbounded $\Gamma_j$. Let
$D(\Gamma_j)=\{f\in L^2:\Gamma_jf\in L^2\}$ with the maximal
closed distributional realization. Smooth compactly supported functions
form a graph core: for $\partial_j$ this is standard, and for
$\Lambda$ one first truncates at large radius and then mollifies,
using the bounded dilation--mollifier commutator. The bounded extension
of $[S_n(s),\Gamma_j]$, first obtained on that core, gives
\[
 \Gamma_jS_n(s)f
 =S_n(s)\Gamma_jf-[S_n(s),\Gamma_j]f,
 \qquad f\in D(\Gamma_j).
\]
Indeed, approximate $f$ in the graph norm; the right-hand side
converges in $L^2$, and closedness of $\Gamma_j$ proves both domain
invariance and the identity. The finite-rank projection
$P_{{\rm s},n}$ preserves $D(\Gamma_j)$ because its right modes
belong to this domain. The same argument, applied on the invariant
stable subspace, therefore establishes the projected identity on
$D(D_{j,n})$. In particular, all the following commutators involving
$D_{j,n}$ are bounded extensions of genuine graph-domain identities.

If $C_j(s)=[S_{{\rm s},n}(s),D_{j,n}]$, the semigroup identity gives
\[
 C_j(s+t)=S_{{\rm s},n}(s)C_j(t)+C_j(s)S_{{\rm s},n}(t).
\]
Iterating this identity over unit intervals and using
\eqref{eq:stable-semigroup-used-nonlinear}, we obtain
\begin{equation}\label{eq:stable-long-commutator-nonlinear}
 \|C_j(s)\|_{2\to2}
 \le C_n(1+s)e^{-\omega_ns},
 \qquad s>0,
\end{equation}
after decreasing $\omega_n$ once more if necessary.

We can now estimate the Lyapunov quadratic form directly, without
commuting $D_{j,n}$ through $Q_n$ or identifying the domain of its
Hilbert-space adjoint.  For $g\in D(D_{j,n})\subset X_{{\rm s},n}$,
\[
 D_{j,n}g=\Gamma_jg-P_{{\rm u},n}\Gamma_jg.
\]
By integration by parts, the unprojected generators satisfy
\[
 2\operatorname{Re}\langle\Lambda g,g\rangle=\|g\|_2^2,
 \qquad
 2\operatorname{Re}\langle\partial_k g,g\rangle=0.
\]
Moreover, the explicit finite-rank projection and the adjoint-mode
regularity give
\[
 P_{{\rm u},n}\Gamma_jg
 =\sum_{\ell}e_{\ell,n}\langle g,\Gamma_j^*\zeta_{\ell,n}\rangle,
 \qquad
 \|P_{{\rm u},n}\Gamma_jg\|_2\le C_n\|g\|_2,
\]
where the sum runs over the scaling, translation and radial unstable
modes. Consequently,
\begin{equation}\label{eq:projected-generator-diagonal-bound}
 \left|2\operatorname{Re}\langle D_{j,n}g,g\rangle\right|
 \le C_n\|g\|_2^2.
\end{equation}

Recall that
\[
 Q_n=\int_0^\infty S_{{\rm s},n}(s)^\dagger S_{{\rm s},n}(s)\,ds,
 \qquad B_n=I+\alpha_nQ_n.
\]
Here $\dagger$ is simply the adjoint of the \emph{bounded} stable
semigroup operator. The graph-domain invariance already proved and the
bounded commutator
$C_j(s)=S_{{\rm s},n}(s)D_{j,n}-D_{j,n}S_{{\rm s},n}(s)$
show that, for each $f\in D(D_{j,n})$,
\[
 S_{{\rm s},n}(s)D_{j,n}f
 =D_{j,n}S_{{\rm s},n}(s)f+C_j(s)f.
\]
Thus, without applying $D_{j,n}$ to the orbit integral itself,
\begin{align*}
 2\operatorname{Re}\langle Q_nD_{j,n}f,f\rangle
 &=2\int_0^\infty
   \operatorname{Re}\langle
        S_{{\rm s},n}(s)D_{j,n}f,
        S_{{\rm s},n}(s)f\rangle\,ds\\
 &=2\int_0^\infty
   \operatorname{Re}\langle
        D_{j,n}S_{{\rm s},n}(s)f,
        S_{{\rm s},n}(s)f\rangle\,ds\\
 &\quad+2\int_0^\infty
   \operatorname{Re}\langle C_j(s)f,S_{{\rm s},n}(s)f\rangle\,ds.
\end{align*}
The diagonal bound \eqref{eq:projected-generator-diagonal-bound},
the stable semigroup estimate
\eqref{eq:stable-semigroup-used-nonlinear}, and the commutator estimate
\eqref{eq:stable-long-commutator-nonlinear} imply
\[
 \left|2\operatorname{Re}\langle Q_nD_{j,n}f,f\rangle\right|
 \le C_n\!\int_0^\infty
    \left(e^{-2\omega_ns}+(1+s)e^{-2\omega_ns}\right)ds\,\|f\|_2^2
 \le C_n\|f\|_2^2.
\]
Combining this with
\eqref{eq:projected-generator-diagonal-bound} gives
\[
 \left|2\operatorname{Re}\langle B_nD_{j,n}f,f\rangle\right|
 \le C_n\|f\|_2^2
 \le C_n\mathcal E_n[f],
\]
which proves \eqref{eq:projected-generator-form}.
\end{proof}

\begin{lemma}[Decay of the stable remainder in $L^2$]\label{lem:ordinary-L2}
Assume that
$0<\mu<\frac{\gamma_n}{4}$.
Then there exists a constant $C^1_n>0$ such that
\begin{equation}\label{eq:ordinary-L2-decay}
\|\varepsilon(s)\|_{L^2}
\le
C^1_n
\left(
K_0+K^2e^{-\mu s_0}
\right)e^{-\mu s},
\qquad
s\in[s_0,s^*].
\end{equation}
\end{lemma}

\begin{proof}
Since the orthogonality conditions are preserved along the flow,
we have
\[
\varepsilon(s)\in X_{{\rm s},n},
\qquad s\in[s_0,s^*].
\]
We first recall the projected generators
\[
D_{0,n}:=P_{{\rm s},n}\Lambda|_{X_{{\rm s},n}},
\qquad
D_{k,n}:=P_{{\rm s},n}\partial_k|_{X_{{\rm s},n}},
\quad 1\le k\le3.
\]
Projecting the renormalized equation \eqref{renormalized equation} onto $X_{{\rm s},n}$, we obtain
\begin{equation}\label{stable-projected-equation}
\partial_s\varepsilon+L_n\varepsilon
=
bD_{0,n}\varepsilon
+
\sum_{k=1}^3\beta_kD_{k,n}\varepsilon
+
G,
\end{equation}
where
\begin{equation*}
G
:=
P_{{\rm s},n}
\left[
G_\psi(\varepsilon)+NL+UN+\operatorname{Mod}
\right],
\end{equation*}
with
\[
G_\psi(\varepsilon)
=
\nabla\varepsilon\cdot\nabla\Delta^{-1}\psi
+
\nabla\psi\cdot\nabla\Delta^{-1}\varepsilon
+
2\psi\varepsilon.
\]

Since $P_{{\rm s},n}$ is bounded on $L^2$, the finite-dimensionality
of the unstable space, Sobolev embedding, and the
Hardy--Littlewood--Sobolev inequality yield
\[
\|G_\psi(\varepsilon)\|_{L^2}
\lesssim
|a|\,\|\varepsilon\|_{H^2}.
\]
Moreover,
\[
\|NL\|_{L^2}
\lesssim
\|\varepsilon\|_{H^2}^2,
\qquad
\|UN\|_{L^2}
\lesssim
|a|^2.
\]
By the modulation estimate,
\[
\|\operatorname{Mod}\|_{L^2}
\lesssim
\|\varepsilon\|_{L^2}^2+|a|^2.
\]
Consequently,
\begin{equation}\label{G-estimate}
\begin{aligned}
\|G\|_{L^2}
&\lesssim
\|\varepsilon\|_{H^2}^2
+
|a|\,\|\varepsilon\|_{H^2}
+
|a|^2\lesssim
\|\varepsilon\|_{H^2}^2+|a|^2\lesssim
K^2e^{-2\mu s}.
\end{aligned}
\end{equation}

Differentiating
$\mathcal E_n[\varepsilon]
=
\langle B_n\varepsilon,\varepsilon\rangle$
along the flow and using \eqref{stable-projected-equation}, we obtain
\begin{align}
\frac12\frac{d}{ds}\mathcal E_n[\varepsilon]
={}&
-\langle B_nL_n\varepsilon,\varepsilon\rangle
+b\langle B_nD_{0,n}\varepsilon,\varepsilon\rangle
+
\sum_{k=1}^3
\beta_k
\langle B_nD_{k,n}\varepsilon,\varepsilon\rangle
+
\langle B_nG,\varepsilon\rangle .
\label{modified-energy-new}
\end{align}

By Proposition~\ref{prop:stablelinearenergy},
\begin{equation*}
\langle B_nL_n\varepsilon,\varepsilon\rangle
\ge
\frac{\gamma_n}{2}\mathcal E_n[\varepsilon].
\end{equation*}
On the other hand, Lemma~\ref{lem:projected-generator-energy} gives
\begin{equation*}
\begin{aligned}
&
\left|
b\langle B_nD_{0,n}\varepsilon,\varepsilon\rangle
+
\sum_{k=1}^3
\beta_k
\langle B_nD_{k,n}\varepsilon,\varepsilon\rangle
\right|
\le
C_n(|b|+|\beta|)
\mathcal E_n[\varepsilon].
\end{aligned}
\end{equation*}

By the modulation estimate and the bootstrap assumptions,
\[
|b|+|\beta|
\lesssim
\|\varepsilon\|_{L^2}^2+|a|^2
\lesssim
K^2e^{-2\mu s}.
\]
Taking $s_0$ sufficiently large gives
\begin{equation*}
C_n(|b|+|\beta|)
\mathcal E_n[\varepsilon]
\le
\frac{\gamma_n}{8}\mathcal E_n[\varepsilon].
\end{equation*}

Since $B_n$ is bounded on $L^2$ and
$\mathcal E_n$ is equivalent to the $L^2$ norm, by
\eqref{G-estimate} and Young's inequality,
\begin{align}
|\langle B_nG,\varepsilon\rangle|
\lesssim
\|G\|_{L^2}\mathcal E_n[\varepsilon]^{1/2}
\le
\frac{\gamma_n}{8}\mathcal E_n[\varepsilon]
+
C_n\|G\|_{L^2}^2
\le
\frac{\gamma_n}{8}\mathcal E_n[\varepsilon]
+
C_nK^4e^{-4\mu s}.
\label{forcing-new}
\end{align}

Combining
\eqref{modified-energy-new}--\eqref{forcing-new}, we obtain
\begin{equation*}
\frac{d}{ds}\mathcal E_n[\varepsilon]
+
\frac{\gamma_n}{2}\mathcal E_n[\varepsilon]
\lesssim
K^4e^{-4\mu s}.
\end{equation*}
Multiplying the above inequality
 by $e^{\frac{\gamma_n }{2}s}$
and integrating from $s_0$ to $s$, we obtain
\begin{align*}
\mathcal E_n[\varepsilon(s)]
\lesssim&
e^{-\frac{\gamma_n}{2}(s-s_0)}
\mathcal E_n[\varepsilon(s_0)]
+
K^4
\int_{s_0}^s
e^{-\frac{\gamma_n}{2}(s-\tau)}
e^{-4\mu\tau}\,d\tau.
\end{align*}
Since
$0<\mu<\frac{\gamma_n}{4}$,
we have
$\frac{\gamma_n}{2}-2\mu>0.$
By \eqref{initial-setting-KS} and the equivalence of
$\mathcal E_n$ and the $L^2$ norm, we have
$\mathcal E_n[\varepsilon(s_0)]
\lesssim
K_0^2e^{-2\mu s_0}$.
Hence
\begin{equation}\label{F1}
 e^{-\frac{\gamma_n}{2}(s-s_0)}
\mathcal E_n[\varepsilon(s_0)]
\lesssim
K_0^2e^{-2\mu s}.   
\end{equation}
Moreover, since $\tau\geq s_0$, we have
$e^{-4\mu\tau}
\leq
e^{-2\mu s_0}e^{-2\mu\tau}$,
and therefore
\begin{equation}\label{F2}
\int_{s_0}^s
e^{-\frac{\gamma_n}{2}(s-\tau)}
e^{-4\mu\tau}\,d\tau
\leq
e^{-2\mu s_0}e^{-2\mu s}
\int_{s_0}^s
e^{-\left(\frac{\gamma_n}{2}-2\mu\right)(s-\tau)}
\,d\tau
\lesssim
e^{-2\mu s_0}e^{-2\mu s}.
\end{equation}
Combining \eqref{F1} and \eqref{F2} gives
\[\mathcal E_n[\varepsilon(s)]
\lesssim
\left(
K_0^2+K^4e^{-2\mu s_0}
\right)e^{-2\mu s}.\]
Finally, by the equivalence of
$\mathcal E_n^{\frac12}$ and the $L^2$ norm, there exists a constant $C^1_n>0$ such that 
\begin{equation}\label{L2}
    \|\varepsilon(s)\|_{L^2}
    \le C^1_n\left(
K_0+K^2e^{-\mu s_0}
\right)
    e^{-\mu s}.
\end{equation}
This completes the proof.
\end{proof}

\begin{lemma}[Affine-drift energy identity at the $H^2$ level]
\label{lem:affine-drift-energy}
Let $I=[\sigma,\tau]$ be a compact interval and suppose
$w\in C(I;L^2(\mathbb R^3))\cap L^2(I;H^1(\mathbb R^3))$
satisfies, in distributions,
\begin{equation}\label{eq:affine-drift-weak}
 \partial_s w-\Delta w+\alpha(s)y\cdot\nabla w+c(s)w
 =G_0+\nabla\cdot G_1,
\end{equation}
where $\alpha,c\in L^\infty(I)$ and
$G_0,G_1\in L^2(I;L^2)$.  Then, for $t\in I$,
\begin{equation}\label{eq:affine-drift-energy}
\begin{aligned}
 \frac12\|w(t)\|_2^2
 &+\int_\sigma^t\|\nabla w(s)\|_2^2\,ds
 +\int_\sigma^t\bigl(c(s)-\tfrac32\alpha(s)\bigr)
                          \|w(s)\|_2^2\,ds\\
 &=\frac12\|w(\sigma)\|_2^2
   +\operatorname{Re}\int_\sigma^t
   \bigl(\langle G_0,w\rangle-\langle G_1,\nabla w\rangle\bigr)\,ds.
\end{aligned}
\end{equation}
In particular, the drift pairing is a quadratic-form identity:
no assumption $y\cdot\nabla w\in L^2$ is required.
\end{lemma}

\begin{proof}
Choose a radial cutoff $\chi_R=\chi(y/R)$, equal to one on
$|y|\le R$ and supported on $|y|\le2R$.  Since the coefficients of
\eqref{eq:affine-drift-weak} are bounded on the support of $\chi_R$,
local Steklov regularization permits testing the distributional equation
with $\chi_R^2w$.  The drift contribution is
\[
 \operatorname{Re}\langle y\cdot\nabla w,\chi_R^2w\rangle
 =-\frac32\int\chi_R^2|w|^2
   -\frac12\int(y\cdot\nabla\chi_R^2)|w|^2.
\]
The last integral tends to zero in $L^1(I)$ because
$\sup_R\|y\cdot\nabla\chi_R^2\|_\infty<\infty$ and its support
escapes to infinity.  The diffusion cutoff terms are controlled by
$CR^{-1}\|w\|_2\|\nabla w\|_2+CR^{-2}\|w\|_2^2$; the $G_1$ cutoff
term is bounded by $CR^{-1}\|G_1\|_2\|w\|_2$.  All tend to zero after
integration over $I$.  The remaining terms converge by dominated
convergence and Cauchy--Schwarz, proving \eqref{eq:affine-drift-energy}.
\end{proof}

\begin{remark}[Nonlocal products and approximation in the energy topology]
\label{rem:H2-nonlocal-approx}
With $T=\nabla\Delta^{-1}$ and
$\mathcal N(f)=f^2+\nabla f\cdot Tf=\nabla\cdot(fTf)$,
\eqref{eq:local-bilinear-H2} implies
\begin{equation}\label{eq:nonlocal-H1-Lipschitz}
 \|\mathcal N(f)-\mathcal N(g)\|_{H^1}
 \le C(\|f\|_{H^2}+\|g\|_{H^2})\|f-g\|_{H^2}.
\end{equation}
For the fixed finite-dimensional unstable component $\psi$,
$G_\psi(f)=\nabla f\cdot T\psi+\nabla\psi\cdot Tf+2\psi f$
obeys
\begin{equation}\label{eq:nonlocal-mixed-H1-Lipschitz}
 \|G_\psi(f)-G_\psi(g)\|_{H^1}
 \le C_n|a|\,\|f-g\|_{H^2}.
\end{equation}
Thus smooth $H^2$ approximations of the remainder give
$H^1$-convergent nonlinear sources, and their Laplacians converge in
$H^{-1}$, not necessarily in $L^2$.  In the $\Delta\varepsilon$ energy
identity they are paired with $w\in H^1$ via
$\langle\Delta F,w\rangle_{H^{-1},H^1}
=-\langle\nabla F,\nabla w\rangle$.  This avoids asking for one
unavailable extra derivative of $\varepsilon$.
\end{remark}

\begin{lemma}[Decay of the stable remainder in $H^2$]\label{lem:H2-estimate}
There exists $\kappa_n>0$ such that, if
$0<\mu<
    \min\left\{
    \frac{\gamma_n}{4},
    \frac{\kappa_n}{2}
    \right\}$, 
then, after first choosing $K$ sufficiently large and then $s_0$
sufficiently large, one has
\begin{equation}\label{H^2-es}
    \|\varepsilon(s)\|_{H^2}
    \leq
    \frac{K}{2}e^{-\mu s},\qquad
s\in[s_0,s^*].
\end{equation}
\end{lemma}

\begin{proof}
The computations below are quadratic-form estimates.  On each compact
renormalized-time interval Proposition~\ref{prop:local-wp} (after the
smooth change of time and space variables) gives
$\varepsilon\in C_tH^2\cap L_t^2H^3$ and
$\Lambda\varepsilon\in C_tL^2\cap L_t^2H^1$.
The nonlinear terms have precisely the $H^1$ regularity needed for the
$H^{-1}$--$H^1$ pairing, as recorded in
Remark~\ref{rem:H2-nonlocal-approx}.  We invoke
Lemma~\ref{lem:affine-drift-energy} for the affine scaling drift;
in particular we do \emph{not} assume that
$\Delta\Lambda\varepsilon\in L^2$.
Let
   $ L_n=L^0+K_n$
where
\[
   L^0\varepsilon=-\Delta\varepsilon+\frac12 \Lambda\varepsilon,\]\[ K_n\varepsilon
    =
    -2U_n\varepsilon
    -
   b_n\cdot\nabla\varepsilon
    -
    \nabla U_n\cdot\nabla\Delta^{-1}\varepsilon,
    \ \ 
     b_n:=\nabla\Delta^{-1}U_n.
\]

Set
\[
    w:=\Delta\varepsilon.
\]
Applying $\Delta$ to \eqref{renormalized equation} in distributions
and using Lemma~\ref{lem:affine-drift-energy} to justify the resulting
energy pairing, we obtain
\begin{align*}
\frac12\frac{d}{ds}\|w\|_{L^2}^2
+
\langle\Delta L^0\varepsilon,w\rangle
={}&
-\langle\Delta K_n\varepsilon,w\rangle
+
\langle\Delta\widetilde L(\varepsilon),w\rangle
\\
&+
\langle\Delta NL,w\rangle
+
\langle\Delta UN,w\rangle
+
\langle\Delta\operatorname{Mod},w\rangle.
\end{align*}
For clarity, the weak equation required by
Lemma~\ref{lem:affine-drift-energy} has exactly the form
\eqref{eq:affine-drift-weak} with
\[
 \alpha=\tfrac12-b,\qquad c=2-4b,\qquad
 G_0=-\Delta K_n\varepsilon+\beta\cdot\nabla w,\qquad
 G_1=\nabla\bigl(G_\psi(\varepsilon)+NL+UN+\operatorname{Mod}\bigr).
\]
Here $\Delta K_n\varepsilon\in L_t^2L_x^2$ follows from the
product expansions below, $\nabla w\in L_t^2L_x^2$ from parabolic
regularity, and $G_1\in L_t^2L_x^2$ from
\eqref{eq:local-bilinear-H2} and the finite-dimensional estimates.
The drift coefficient in the integrated energy is exactly
\[
 c-\tfrac32\alpha=\tfrac54-\tfrac52 b.
\]
The pointwise identities below should therefore be understood first
for smooth solutions and, for general data, through this integrated
quadratic-form identity.

We first consider the term $\langle\Delta L^0\varepsilon,w\rangle$. By the commutator relation
\begin{equation}\label{commutator}
[\Delta,y\cdot\nabla]=2\Delta, 
\end{equation}
we have
\[
\Delta L^0\varepsilon
=
-\Delta w+2w+\frac12y\cdot\nabla w.
\]
Consequently,
\begin{align}
\langle\Delta L^0\varepsilon,w\rangle
&=
\|\nabla w\|_{L^2}^2
+
2\|w\|_{L^2}^2
+
\frac12\langle y\cdot\nabla w,w\rangle
=
\|\nabla w\|_{L^2}^2
+
\frac54\|w\|_{L^2}^2,
\label{eq:free-H2-coercivity}
\end{align}
where we used
   $ \langle y\cdot\nabla w,w\rangle
    =
    -\frac32\|w\|_{L^2}^2$.
    
We next estimate the term $\langle\Delta K_n\varepsilon,w\rangle$. We claim that for every
$\delta>0$,
\begin{equation}\label{eq:Kn-H2}
\left|
\langle\Delta K_n\varepsilon,w\rangle
\right|
\leq
\delta\|\nabla w\|_{L^2}^2
+
C_{\delta}\|\varepsilon\|_{L^2}^2.
\end{equation}
For the first term,
\[
\Delta(U_n\varepsilon)
=
U_nw
+
2\nabla U_n\cdot\nabla\varepsilon
+\varepsilon
\Delta U_n,
\]
and hence
\begin{equation}\label{Tir1}
\left|
\langle\Delta(U_n\varepsilon),w\rangle
\right|
\lesssim
\|w\|_{L^2}^2
+
\|\nabla\varepsilon\|_{L^2}^2
+
\|\varepsilon\|_{L^2}^2.
\end{equation}
For the transport term,
\[
\Delta(b_n\cdot\nabla\varepsilon)
=
 b_n\cdot\nabla w
+
2\nabla b_n:\nabla^2\varepsilon
+
(\Delta b_n)\cdot\nabla\varepsilon,
\]
where ``$:$'' denotes the Frobenius inner product, i.e.,
\[
A:B=\sum_{i,j=1}^3 A_{ij}B_{ij},
\qquad
\nabla b_n:\nabla^2\varepsilon
=
\sum_{i,j=1}^3
\partial_i b_{n,j}\,\partial_{ij}\varepsilon.
\]
By integrating by parts, the highest-order term satisfies
\[
\langle b_n\cdot\nabla w,w\rangle
=
-\frac12
\int_{\mathbb R^3}
(\nabla\cdot b_n)|w|^2
=
-\frac12
\int_{\mathbb R^3}
U_n|w|^2,
\]
In addition,
\[
\left|\left\langle\nabla b_n:\nabla^2\varepsilon,w\right\rangle\right|\lesssim||\nabla b_n||_\infty||\nabla^2\varepsilon||_{L^2}||w||_{L^2}=||\nabla b_n||_\infty||w||_{L^2}^2\lesssim||w||_{L^2}^2,
\]
where $\|\nabla^2\varepsilon\|_2=\|\Delta\varepsilon\|_2$ follows from
Plancherel's theorem.  The remaining term is controlled by
$\Delta b_n=\nabla U_n$ and the boundedness of $\nabla U_n$.
Consequently,
\begin{equation}\label{Tir2}
\left|\langle
\Delta(b_n\cdot\nabla\varepsilon),w
\rangle
\right|
\lesssim
\|w\|_{L^2}^2
+
\|\nabla\varepsilon\|_{L^2}^2.
\end{equation}

Setting
\[
T:=\nabla\Delta^{-1},
\]
we have
\begin{align*}
\Delta(\nabla U_n\cdot T\varepsilon)
={}&
\nabla\Delta U_n\cdot T\varepsilon
+
2\nabla^2U_n:\nabla T\varepsilon
+
\nabla U_n\cdot\nabla\varepsilon.
\end{align*}
By the Hardy--Littlewood--Sobolev inequality and the
$L^2$ boundedness of the Riesz transforms,
\[
    \|T\varepsilon\|_{L^6}
    \lesssim
    \|\varepsilon\|_{L^2},
    \qquad
    \|\nabla T\varepsilon\|_{L^2}
    \lesssim
    \|\varepsilon\|_{L^2}.
\]
Using the smoothness and decay of $U_n$, we infer
\begin{equation}\label{Tir3}
\left|
\left\langle
\Delta(\nabla U_n\cdot T\varepsilon),w
\right\rangle
\right|
\lesssim
\left(
\|\varepsilon\|_{L^2}
+
\|\nabla\varepsilon\|_{L^2}
\right)
\|w\|_{L^2}.
\end{equation}
Combining the estimates \eqref{Tir1}, \eqref{Tir2} and \eqref{Tir3} yields
\[
\left|
\langle\Delta K_n\varepsilon,w\rangle
\right|
\lesssim
\|w\|_{L^2}^2
+
\|\nabla\varepsilon\|_{L^2}^2
+
\|\varepsilon\|_{L^2}^2.
\]
Using the interpolation inequality
\begin{equation*}
    \|w\|_{L^2}^2
    +
    \|\nabla\varepsilon\|_{L^2}^2
    \leq
    \delta\|\nabla w\|_{L^2}^2
    +
    C_\delta\|\varepsilon\|_{L^2}^2,
\end{equation*}
we obtain \eqref{eq:Kn-H2}.

We now estimate the modulation terms contained in
$\widetilde L(\varepsilon)$. Since
    $\Lambda=2+y\cdot\nabla$, by \eqref{commutator},
we have
    $\Delta\Lambda\varepsilon
    =
    4w+y\cdot\nabla w$.
The term $y\cdot\nabla w$ need not lie in $L^2$; the following
calculation is an identity of quadratic forms justified by
Lemma~\ref{lem:affine-drift-energy}, not a separate $L^2$ pairing.
Therefore,
\begin{equation}\label{eq:scaling-H2N}
\langle\Delta(b\Lambda\varepsilon),w\rangle
=
b\left(
4\|w\|_{L^2}^2
+
\langle y\cdot\nabla w,w\rangle
\right)
=
\frac52 b\|w\|_{L^2}^2.
\end{equation}
Similarly,
$\Delta(\beta\cdot\nabla\varepsilon)
=
\beta\cdot\nabla w$,
and $\nabla_y\cdot\beta=0$; hence
\begin{equation}\label{eq:translation-H2}
    \langle
    \Delta(\beta\cdot\nabla\varepsilon),w
    \rangle
    =
    0.
\end{equation}

To estimate the mixed term, recall that $T=\nabla\Delta^{-1}$ and that the
finite-dimensionality of the unstable eigenspace gives
\[
 \|\psi\|_{W^{2,\infty}}+\|T\psi\|_{W^{1,\infty}}
 +\|\nabla^2\psi\|_{L^3}\le C_n|a|.
\]
The product rule, $\|T\varepsilon\|_6\lesssim\|\varepsilon\|_2$, and
$\|\nabla T\varepsilon\|_2\lesssim\|\varepsilon\|_2$ then give
\begin{align*}
 \|\nabla\varepsilon\cdot T\psi\|_{H^1}
 &\le C_n|a|\bigl(\|\nabla\varepsilon\|_2
                   +\|\nabla^2\varepsilon\|_2\bigr),\\
 \|\nabla\psi\cdot T\varepsilon\|_{H^1}
 &\le C_n|a|\bigl(\|T\varepsilon\|_6
                   +\|\nabla T\varepsilon\|_2\bigr),\\
 \|2\psi\varepsilon\|_{H^1}
 &\le C_n|a|\|\varepsilon\|_{H^1}.
\end{align*}
Consequently,
\[
 \|G_\psi(\varepsilon)\|_{H^1}
 \le C_n|a|\|\varepsilon\|_{H^2}.
\]
Consequently, by integrating by parts,
\begin{equation}\label{eq:Gpsi-energy}
\begin{aligned}
 |\langle\Delta G_\psi(\varepsilon),w\rangle|
 =|\langle\nabla G_\psi(\varepsilon),\nabla w\rangle|\le\delta\|\nabla w\|_2^2+C_\delta|a|^2\|\varepsilon\|_{H^2}^2.
\end{aligned}
\end{equation}
Combining \eqref{eq:scaling-H2N}, \eqref{eq:translation-H2}, and
\eqref{eq:Gpsi-energy} gives
\begin{equation}\label{Linear-estimate}
\left|\langle\Delta\widetilde L(\varepsilon),w\rangle\right|\le \frac52 |b|\|w\|_{L^2}^2+\delta\|\nabla w\|_{L^2}^2
+
C_{\delta}
|a|^2
\|\varepsilon\|_{H^2}^2
\end{equation}

We next consider the quadratic nonlinearity
    $NL
    =
    \varepsilon^2
    +
    \nabla\varepsilon\cdot
    \nabla\Delta^{-1}\varepsilon.$
We claim that
\begin{equation}\label{eq:NL-H1}
    \|NL\|_{H^1}
    \lesssim
    \|\varepsilon\|_{H^2}^2.
\end{equation}
Indeed,
\[
    \|\varepsilon^2\|_{H^1}
    \lesssim
    \|\varepsilon\|_{L^\infty}
    \|\varepsilon\|_{H^1}
    \lesssim
    \|\varepsilon\|_{H^2}^2.
\]
Moreover, we have
\[
    \|T\varepsilon\|_{L^\infty}
    \lesssim
    \|\varepsilon\|_{H^1},\  \|\nabla T\varepsilon\|_{L^3}
    \lesssim
    \|\varepsilon\|_{H^1}.
\]
The Sobolev inequality therefore gives
\begin{equation*}
\begin{aligned}
\|
\nabla\varepsilon\cdot T\varepsilon
\|_{H^1}&\le ||\nabla\varepsilon\cdot T\varepsilon
\|_{L^2}+||\nabla(\nabla\varepsilon\cdot T\varepsilon)
\|_{L^2}\\
&\lesssim{}
\|\nabla\varepsilon\|_{L^2}
\|T\varepsilon\|_{L^\infty}
+
\|\nabla^2\varepsilon\|_{L^2}
\|T\varepsilon\|_{L^\infty}
+
\|\nabla\varepsilon\|_{L^6}
\|\nabla T\varepsilon\|_{L^3}
\lesssim
\|\varepsilon\|_{H^2}^2.
\end{aligned}
\end{equation*}
This proves \eqref{eq:NL-H1}. Hence
\begin{equation}\label{eq:NL-H2-energy}
\left|
\langle\Delta NL,w\rangle
\right|
=
\left|
\langle\nabla NL,\nabla w\rangle
\right|\le
\delta\|\nabla w\|_{L^2}^2
+
C_\delta
\|\varepsilon\|_{H^2}^4.
\end{equation}

Since
    $\|UN\|_{H^1}
    \lesssim
    |a|^2$,
\begin{equation}\label{eq:UN-H2-energy}
\left|
\langle\Delta UN,w\rangle
\right|
\leq
\delta\|\nabla w\|_{L^2}^2
+
C_{\delta}|a|^4.
\end{equation}

By the modulation estimate \eqref{semin},
we obtain
\begin{equation*}
    \|\operatorname{Mod}\|_{H^1}
    \lesssim
    \|\varepsilon\|_{L^2}^2
    +
    |a|^2.
\end{equation*}
 Hence
\begin{equation}\label{eq:Mod-H2-energy}
\left|
\langle
\Delta\operatorname{Mod},w
\rangle
\right|
=
\left|
\langle
\nabla\operatorname{Mod},\nabla w
\rangle
\right|
\leq
\delta\|\nabla w\|_{L^2}^2
+
C_{\delta}
\left(
\|\varepsilon\|_{L^2}^2
+
|a|^2
\right)^2.
\end{equation}

By the modulation estimate \eqref{semin} and the bootstrap assumptions \eqref{chca} and \eqref{1.9}, we have
$|b|
\lesssim
K^2e^{-2\mu s}$.
Choosing $s_0$ sufficiently large makes the contribution in
\eqref{eq:scaling-H2N} small enough to be absorbed into the positive term
$\frac54\|w\|_{L^2}^2$ in \eqref{eq:free-H2-coercivity}.
Then
combining \eqref{eq:free-H2-coercivity},
\eqref{eq:Kn-H2}, \eqref{Linear-estimate}, \eqref{eq:NL-H2-energy}, \eqref{eq:UN-H2-energy} and \eqref{eq:Mod-H2-energy},
and then choosing $\delta>0$ sufficiently small, we obtain
for some $\kappa_n>0$,
\begin{align*}
\frac{d}{ds}\|w\|_{L^2}^2
+
\kappa_n\|w\|_{L^2}^2
+
\kappa_n\|\nabla w\|_{L^2}^2
&\lesssim
\|\varepsilon\|_{L^2}^2
+
|a|^2\|\varepsilon\|_{H^2}^2
+
\|\varepsilon\|_{H^2}^4
+
|a|^4
+
\left(
\|\varepsilon\|_{L^2}^2
+
|a|^2
\right)^2\\
&\lesssim
\|\varepsilon\|_{L^2}^2
+
|a|^2\|\varepsilon\|_{H^2}^2
+
\|\varepsilon\|_{H^2}^4
+
|a|^4.
\end{align*}
At this stage $K$ and $K_0$ have been fixed. Increase $s_0$, if needed,
so that
\begin{equation}\label{eq:H2-smallness-choice}
 K^2e^{-\mu s_0}\le K_0.
\end{equation}
The $L^2$ estimate \eqref{eq:ordinary-L2-decay} first gives
$\|\varepsilon\|_2^2\lesssim
(K_0^2+K^4e^{-2\mu s_0})e^{-2\mu s}$; the second summand is absorbed
using \eqref{eq:H2-smallness-choice}. Therefore the bootstrap bounds yield
\begin{equation}\label{eq:H2-differential-simple}
\begin{aligned}
\frac{d}{ds}\|w\|_{L^2}^2
+
\kappa_n\|w\|_{L^2}^2
&\lesssim
K_0^2e^{-2\mu s}
+K^2e^{-4\mu s}+
K^4e^{-4\mu s}+e^{-4\mu s}\\
&\lesssim K_0^2e^{-2\mu s}
+
K^4e^{-4\mu s}.
\end{aligned}
\end{equation}

Since $2\mu<\kappa_n$, multiplying
\eqref{eq:H2-differential-simple} by $e^{\kappa_n s}$
and integrating from $s_0$ to $s$, we obtain
\begin{align*}
\|w(s)\|_{L^2}^2
\lesssim
e^{-\kappa_n(s-s_0)}
\|w(s_0)\|_{L^2}^2
+
K_0^2
\int_{s_0}^s
e^{-\kappa_n(s-\tau)}
e^{-2\mu\tau}\,d\tau+
K^4
\int_{s_0}^s
e^{-\kappa_n(s-\tau)}
e^{-4\mu\tau}\,d\tau.
\end{align*}
By the initial estimate,
    $\|w(s_0)\|_{L^2}^2
    \leq
    K_0^2e^{-2\mu s_0}$,
we have
\[
e^{-\kappa_n(s-s_0)}
\|w(s_0)\|_{L^2}^2
\lesssim
K_0^2e^{-2\mu s}.
\]
Moreover,
$\int_{s_0}^s
e^{-\kappa_n(s-\tau)}
e^{-2\mu\tau}\,d\tau
\lesssim
e^{-2\mu s}$.
Since 
\[
e^{-4\mu\tau}
\leq
e^{-2\mu s_0}e^{-2\mu\tau},\ \ \tau\geq s_0,
\]
and hence
\[
\int_{s_0}^s
e^{-\kappa_n(s-\tau)}
e^{-4\mu\tau}\,d\tau
\lesssim
e^{-2\mu s_0}e^{-2\mu s}.
\]
Consequently,
\begin{equation*}
\|w(s)\|_{L^2}^2=\|\Delta\varepsilon(s)\|_{L^2}^2
    \lesssim
    \left(
    K_0^2
    +
    K^4e^{-2\mu s_0}
    \right)
    e^{-2\mu s}.
\end{equation*}
Taking square roots,
\[
    \|\Delta\varepsilon(s)\|_{L^2}
    \lesssim
    \left(
    K_0
    +
    K^2e^{-\mu s_0}
    \right)e^{-\mu s}.
\]
Finally, by the standard elliptic estimate on $\mathbb R^3$,
\[
    \|\varepsilon\|_{H^2}
    \lesssim
    \|\varepsilon\|_{L^2}
    +
    \|\Delta\varepsilon\|_{L^2},
\]
and the $L^2$ estimate \eqref{eq:ordinary-L2-decay}, we conclude that there exists a  constant $C'_n>0$ such that
\begin{equation}\label{eq:H2-final}
    \|\varepsilon(s)\|_{H^2}
    \leq
    C'_n
    \left(
    K_0
    +
    K^2e^{-\mu s_0}
    \right)e^{-\mu s}.
\end{equation}
As arranged in the preceding parameter selection, $K$ is fixed so that
   $ C'_nK_0\leq\frac{K}{4}$,
and $s_0$ is then taken sufficiently large so that
   $ C'_nK^2e^{-\mu s_0}
    \leq
    \frac{K}{4}$.
It follows that
\[
    \|\varepsilon(s)\|_{H^2}
    \leq
    \frac{K}{2}e^{-\mu s}.
\]
This closes the $H^2$ bootstrap estimate.
\end{proof}

\begin{lemma}[Finite-dimensional G\aa rding inequality]
\label{lem:finite-dimensional-Garding}
There exist $q_1,\ldots,q_{N_1}\in C_c^\infty(\mathbb R^3)$ and constants
$c_n,C_n>0$ such that every $f\in H^1(\mathbb R^3)$ satisfies
\begin{align}
 \operatorname{Re}\langle L_nf,f\rangle
 &\ge \frac18\|f\|_2^2
 -C_n\sum_{j=1}^{N_1}|\langle f,q_j\rangle|^2,
\label{eq:finite-Garding}\\
 \operatorname{Re}\langle L_nf,f\rangle
 &\ge c_n\|\nabla f\|_2^2+c_n\|f\|_2^2
 -C_n\sum_{j=1}^{N_1}|\langle f,q_j\rangle|^2.
           \label{eq:strong-Garding}
\end{align}
where the numerical range on the left is understood in the quadratic-form
sense when $f\notin D(L_n)$.
\end{lemma}

\begin{proof}
The integration-by-parts identity \eqref{eq:L2numericalform} reads
\begin{equation}\label{eq:Garding-form-decomposition}
 \operatorname{Re}\langle L_nf,f\rangle
 =\|\nabla f\|_2^2+\frac14\|f\|_2^2+\mathfrak k_n[f],
\end{equation}
where
\[
 \mathfrak k_n[f]
 =-\frac32\int_{\mathbb R^3}U_n|f|^2
 -\operatorname{Re}\int_{\mathbb R^3}
 (\nabla U_n\cdot\nabla\Delta^{-1}f)\overline f.
\]
The form $\mathfrak k_n$ is compact on $H^1(\mathbb R^3)$.  We verify the
only point that is not immediate.  Choose a cutoff $\chi_R$ supported in
$|y|\le2R$ and equal to one in $|y|\le R$.  On the support of $\chi_R$,
Rellich's theorem gives compactness of $H^1\hookrightarrow L^2$, while
$\nabla\Delta^{-1}:L^2\to L^6$ and
$\nabla(\nabla\Delta^{-1}):L^2\to L^2$ imply local compactness of
$\nabla\Delta^{-1}f$ in every $L^p$, $2\le p<6$.  On $|y|\ge R$, the
decay of $U_n$ and $\nabla U_n$ gives
\[
 \|(1-\chi_R)U_n\|_\infty
 +\|(1-\chi_R)\nabla U_n\|_3\longrightarrow0.
\]
H\"older's inequality therefore makes the outer part of
$\mathfrak k_n$ uniformly small on the unit ball of $H^1$.  Combining the
inner compactness and the outer smallness proves the claim.

By polarization, the real part of $\mathfrak k_n$ is represented by a
compact self-adjoint operator on $H^1$.  Approximate that operator in norm
by a finite-rank operator whose range is contained in
$C_c^\infty(\mathbb R^3)$.  Since
\[
 \langle f,h\rangle_{H^1}
 =\langle f,h-\Delta h\rangle_{L^2},
 \qquad h\in C_c^\infty(\mathbb R^3),
\]
the resulting finite number of functionals can be written as
$f\mapsto\langle f,q_j\rangle$ with $q_j\in C_c^\infty$.  Thus, for every
$\eta>0$,
\begin{equation}\label{eq:compact-form-finite-rank}
 |\mathfrak k_n[f]|
 \le \eta\|f\|_{H^1}^2
 +C_{n,\eta}\sum_{j=1}^{N_1}|\langle f,q_j\rangle|^2.
\end{equation}
Taking $\eta=1/8$ in
\eqref{eq:Garding-form-decomposition}--\eqref{eq:compact-form-finite-rank}
gives
\[
 \operatorname{Re}\langle L_nf,f\rangle
 \ge\frac78\|\nabla f\|_2^2+\frac18\|f\|_2^2
 -C_n\sum_{j=1}^{N_1}|\langle f,q_j\rangle|^2,
\]
which proves \eqref{eq:finite-Garding}.

For the second estimate, \eqref{eq:L2numericalform},
$\|\nabla\Delta^{-1}f\|_6\lesssim\|f\|_2$, and H\"older's inequality give
\begin{equation}\label{eq:rough-Garding}
 \operatorname{Re}\langle L_nf,f\rangle
 \ge\|\nabla f\|_2^2-C_n\|f\|_2^2.
\end{equation}
Take $(1-\theta)$ times \eqref{eq:finite-Garding} plus $\theta$ times
\eqref{eq:rough-Garding}, where $\theta>0$ is chosen so small that
$(1-\theta)/8-\theta C_n>0$.  The resulting inequality is
\eqref{eq:strong-Garding}.
\end{proof}

We next control the scaling derivative of the remainder.

\begin{lemma}[Decay of the scaling derivative]\label{lem:Lambda-epsilon}
There exists a constant $K'\gg1$ such that, for $s_0$ sufficiently large,
\begin{equation}\label{Lambda-es}
\|\Lambda\varepsilon(s)\|_{L^2}
    \le
    \frac{K'}{2}e^{-\mu s},\qquad s\in[s_0,s^*].
\end{equation}
\end{lemma}

\begin{proof}
As in the preceding energy estimate, we first argue for a smooth solution on
$[s_0+\eta,s^*]$.  The bounds below are independent of $\eta>0$; parabolic
regularization and continuity in $\mathcal Y$ then allow $\eta\to0$.
For simplicity, we write the renormalized equation \eqref{renormalized equation} in the form
\begin{equation}\label{eq:eps-for-Lambda}
    \partial_s\varepsilon+L_n\varepsilon
    =
    b\Lambda\varepsilon
    +
    \beta\cdot\nabla\varepsilon
    +
    \mathcal R,
\end{equation}
where
\[
   \mathcal R
    := G_\psi(\varepsilon)+NL+UN+\operatorname{Mod},
\]
with
\[
G_\psi(\varepsilon)
=
\nabla\varepsilon\cdot\nabla\Delta^{-1}\psi
+
\nabla\psi\cdot\nabla\Delta^{-1}\varepsilon
+
2\psi\varepsilon.
\]
 Set
\[
    Z:=\Lambda\varepsilon.
\]
Applying $\Lambda$ to \eqref{eq:eps-for-Lambda}, we use
    $[\Lambda,\nabla]=-\nabla$
to obtain
\[
    \Lambda(\beta\cdot\nabla\varepsilon)
    =
    \beta\cdot\nabla Z
    -
    \beta\cdot\nabla\varepsilon.
\]
Since $b$ depends only on $s$, we have
    $\Lambda(b\Lambda\varepsilon)=b\Lambda Z$.
Moreover,
\[
    \Lambda L_n\varepsilon
    =
    L_nZ+[\Lambda,L_n]\varepsilon.
\]
Hence $Z$ satisfies
\begin{equation}\label{eq:Z-equation}
    \partial_sZ
    +
    L_nZ
    -
    b\Lambda Z
    -
    \beta\cdot\nabla Z
    =
    \Lambda\mathcal R
    -
    \beta\cdot\nabla\varepsilon
    -
    [\Lambda,L_n]\varepsilon.
\end{equation}

We use Lemma~\ref{lem:finite-dimensional-Garding}; its estimate does not
require $Z$ to belong to the stable subspace $X_{{\rm s},n}$.

Since
$    \Lambda^*
    =
    -1-y\cdot\nabla$
in $\mathbb R^3$, we have
    $\langle Z,q_\ell\rangle
    =
\langle\varepsilon,\Lambda^*q_\ell\rangle$.
Since $q_\ell$ is smooth and compactly supported, we obtain
\[
    |\langle Z,q_\ell\rangle|=|\langle\varepsilon,\Lambda^*q_\ell\rangle|
    \lesssim
    \|\varepsilon\|_{L^2}.
\]
Thus \eqref{eq:strong-Garding} gives
\begin{equation}\label{eq:Z-coercivity}
    \langle L_nZ,Z\rangle
    \geq
    c_n\|\nabla Z\|_{L^2}^2
    +
    c_n\|Z\|_{L^2}^2
    -
    \tilde{C}_n\|\varepsilon\|_{L^2}^2.
\end{equation}

We next compute the commutator $[\Lambda,L_n]$. Recall that
\[
    L_n
    =
    L^0+K_n,
    \qquad
    L^0
    =
    -\Delta+1+\frac12y\cdot\nabla,
\]
and
\[
    K_nf
    =
    -2U_nf
    -
     b_n\cdot\nabla f
    -
    \nabla U_n\cdot Tf,\qquad b_n=TU_n,
    \qquad
    T:=\nabla\Delta^{-1}.
\]
Since
 \[
    [\Lambda,L^0]=2\Delta,
    \ 
    [\Lambda,T]=T,
\]
 a direct computation gives
\begin{equation*}
[\Lambda,K_n]\varepsilon
=
-2(y\cdot\nabla U_n)\varepsilon
-
T(y\cdot\nabla U_n)\cdot\nabla\varepsilon
-
\nabla(y\cdot\nabla U_n)\cdot T\varepsilon.
\end{equation*}
Consequently, since $U_n$ is a smooth function, and by $ \|T\varepsilon\|_{L^6}
    \lesssim
    \|\varepsilon\|_{L^2}$, we have
\begin{equation}\label{eq:commutator-estimate}
    \|[\Lambda,L_n]\varepsilon\|_{L^2}
    \lesssim
    \|\varepsilon\|_{H^2}.
\end{equation}

We next test \eqref{eq:Z-equation} with $Z$.  The unbounded
scaling drift is treated by the same spatial-cutoff and time-regularization
argument as in Lemma~\ref{lem:affine-drift-energy}: here
$Z\in C_tL^2\cap L_t^2H^1$, the profile coefficients are bounded,
and the displayed expansions of $\Lambda\mathcal R$ are in
$L_t^2L^2$.  Thus the identity below is a legitimate quadratic-form
identity; it does not presuppose $y\cdot\nabla Z\in L^2$.
Since
\[
    \langle\Lambda Z,Z\rangle
    =
    2\|Z\|_{L^2}^2
    +
    \langle y\cdot\nabla Z,Z\rangle
    =
    \frac12\|Z\|_{L^2}^2,
\]
and
   $ \langle\beta\cdot\nabla Z,Z\rangle=0$,
we obtain
\begin{align*}
\frac12\frac{d}{ds}\|Z\|_{L^2}^2
+
\langle L_nZ,Z\rangle
-
\frac b2\|Z\|_{L^2}^2
={}&
\langle\Lambda\mathcal R,Z\rangle
-
\langle\beta\cdot\nabla\varepsilon,Z\rangle
-
\langle[\Lambda,L_n]\varepsilon,Z\rangle.
\end{align*}

By the modulation estimate \eqref{semin}, taking $s_0$ sufficiently large
makes the term
    $\frac{|b|}{2}\|Z\|_{L^2}^2$
small enough to be absorbed into the positive $c_n\|Z\|_{L^2}^2$ term in
\eqref{eq:Z-coercivity}.

Moreover, by \eqref{eq:commutator-estimate}, Cauchy-Schwarz inequality and Young's inequality,
\begin{align*}
\left|
\langle[\Lambda,L_n]\varepsilon,Z\rangle
\right|
\lesssim
\|\varepsilon\|_{H^2}\|Z\|_{L^2}
\leq
\frac{c_n}{16}\|Z\|_{L^2}^2
+
C_n\|\varepsilon\|_{H^2}^2.
\end{align*}
Similarly,
\begin{equation*}
\left|
\langle\beta\cdot\nabla\varepsilon,Z\rangle
\right|
\leq
\frac{c_n}{16}\|Z\|_{L^2}^2
+
C_n|\beta|^2\|\varepsilon\|_{H^2}^2.
\end{equation*}

We next estimate $|\langle\Lambda\mathcal R,Z\rangle|$.
For the quadratic nonlinearity
   $ NL
    =
    \varepsilon^2
    +
    \nabla\varepsilon\cdot T\varepsilon$,
we first note that
\[
    \Lambda(\varepsilon^2)
    =
    2\varepsilon Z-2\varepsilon^2.
\]
Furthermore, using
   $[\Lambda,\nabla]=-\nabla$,
    $[\Lambda,T]=T$,
we obtain
\begin{equation*}
\begin{aligned}
\Lambda (NL)
={}&
2\varepsilon Z
-
2\varepsilon^2
+
\nabla Z\cdot T\varepsilon
+
\nabla\varepsilon\cdot TZ
-
2\nabla\varepsilon\cdot T\varepsilon.
\end{aligned}
\end{equation*}
H\"older's inequality gives
\begin{equation*}
\begin{aligned}
\left|
\langle\Lambda NL,Z\rangle
\right|
\lesssim{}&
\|\varepsilon\|_{L^\infty}\|Z\|_{L^2}^2
+
\|\varepsilon\|_{L^\infty}
\|\varepsilon\|_{L^2}\|Z\|_{L^2}
+
\|T\varepsilon\|_{L^\infty}
\|\nabla Z\|_{L^2}\|Z\|_{L^2}
\\
&+
\|\nabla\varepsilon\|_{L^3}
\|TZ\|_{L^6}\|Z\|_{L^2}
+
\|\nabla\varepsilon\|_{L^3}
\|T\varepsilon\|_{L^6}\|Z\|_{L^2}.
\end{aligned}
\end{equation*}
Since 
\[
    \langle\Lambda \varepsilon,\varepsilon\rangle
    =
    2\|\varepsilon\|_{L^2}^2
    +
    \langle y\cdot\nabla \varepsilon,\varepsilon\rangle
    =
    \frac12\|\varepsilon\|_{L^2}^2,
\]
the Cauchy--Schwarz inequality gives
$
\|\varepsilon\|_{L^2}^2\le 2||\Lambda \varepsilon||_2 ||\varepsilon||_{L^2},
$
and hence
\begin{equation}\label{equi}
\|\varepsilon\|_{L^2}\le 2||Z||_2
\end{equation}
In $\mathbb{R}^3$, using
\[
    \|\varepsilon\|_{L^\infty}
    +
    \|\nabla\varepsilon\|_{L^3}
    \lesssim
    \|\varepsilon\|_{H^2},
\]
\[
    \|T\varepsilon\|_{L^\infty}
    \lesssim
    \|\varepsilon\|_{H^1}\le\|\varepsilon\|_{H^2},
    \qquad
    \|Tf\|_{L^6}
    \lesssim
    \|f\|_{L^2},
\]
then combining \eqref{equi}, we  obtain
\begin{equation*}
\begin{aligned}
\left|
\langle\Lambda NL,Z\rangle
\right|
\lesssim{}&
(\|\varepsilon\|_{L^\infty}+\|\nabla\varepsilon\|_{L^3})\|Z\|_{L^2}^2
+
\|T\varepsilon\|_{L^\infty}
\|\nabla Z\|_{L^2}\|Z\|_{L^2}\\
&\lesssim 
(\|\varepsilon\|_{L^\infty}+\|\nabla\varepsilon\|_{L^3})\|Z\|_{L^2}^2+\delta\|\nabla Z\|_{L^2}^2+C_\delta\|T\varepsilon\|^2_{L^\infty}\|Z\|_{L^2}^2\\
&\lesssim \delta\|\nabla Z\|_{L^2}^2+\bar{C}_\delta(\|\varepsilon\|_{H^2}+\|\varepsilon\|_{H^2}^2)||Z||_{L^2}^2.
\end{aligned}
\end{equation*}

For the mixed term $G_\psi(\varepsilon)$, set $W=\Lambda\psi$.  The product
rule and the identities $[\Lambda,\nabla]=-\nabla$ and
$[\Lambda,T]=T$ give
\begin{align*}
 \Lambda G_\psi(\varepsilon)
 ={}&2Z\psi+2\varepsilon W-4\varepsilon\psi
 +\nabla Z\cdot T\psi+\nabla\varepsilon\cdot TW
 -2\nabla\varepsilon\cdot T\psi\\
 &+\nabla W\cdot T\varepsilon+\nabla\psi\cdot TZ
 -2\nabla\psi\cdot T\varepsilon.
\end{align*}
Since $\psi$ and $W$ belong to a fixed finite-dimensional space of smooth
decaying functions,
\[
 \|\psi\|_{W^{1,3}\cap L^\infty}
 +\|T\psi\|_{L^\infty}
 +\|W\|_{W^{1,3}\cap L^\infty}
 +\|TW\|_{L^\infty}
 \le C_n|a|.
\]
Pairing the displayed expansion with $Z$, using
$\|TZ\|_6\lesssim\|Z\|_2$, $\|T\varepsilon\|_6\lesssim
\|\varepsilon\|_2$, and \eqref{equi}, and applying Young's inequality to
the terms containing $\nabla Z$, yields
\begin{align*}
\left|
\langle\Lambda G_\psi(\varepsilon),Z\rangle
\right|
\leq{}&
\delta\|\nabla Z\|_{L^2}^2
+\delta\| Z\|_{L^2}^2+
C_{\delta}
\left(
|a|+|a|^2
\right)
\|Z\|_{L^2}^2
+
C_{\delta}
|a|^2
\|\varepsilon\|_{H^2}^2.
\end{align*}

Since $UN$ is quadratic in $\psi$ and the unstable space is
finite-dimensional, we have
    $\|\Lambda UN\|_{L^2}
    \lesssim
    |a|^2$.
Hence
\begin{equation*}
\left|
\langle\Lambda UN,Z\rangle
\right|
\leq
\frac{c_n}{16}\|Z\|_{L^2}^2
+
C_n|a|^4.
\end{equation*}

Finally, by the modulation estimate and the smoothness and decay of the
profile and the discrete eigenfunctions,
\[
    \|\Lambda\operatorname{Mod}\|_{L^2}
    \lesssim
    \left(
    \|\varepsilon\|_{L^2}^2
    +
    |a|^2
    \right)
    (1+|a|).
\]
Thus, for $s_0$ sufficiently large,
\begin{equation}\label{eq:Lambda-Mod-energy}
\left|
\langle
\Lambda\operatorname{Mod},Z
\rangle
\right|
\leq
\frac{c_n}{16}\|Z\|_{L^2}^2
+
C_n
\left(
\|\varepsilon\|_{L^2}^2
+
|a|^2
\right)^2.
\end{equation}

Combining \eqref{eq:Z-coercivity}--\eqref{eq:Lambda-Mod-energy},
choosing $\delta>0$ sufficiently small, and then taking $s_0$
sufficiently large so that
    $\|\varepsilon\|_{H^2}
    +
    |a|
    +
    |b|
    \ll1$,
we obtain
\begin{align*}
\frac{d}{ds}\|Z\|_{L^2}^2
+
\kappa_n\|Z\|_{L^2}^2
+
\kappa_n\|\nabla Z\|_{L^2}^2
\lesssim{}&
\|\varepsilon\|_{H^2}^2
+
\|\varepsilon\|_{H^2}^4
\\
&+
|a|^2\|\varepsilon\|_{H^2}^2
+
|a|^4
+
\left(
\|\varepsilon\|_{L^2}^2
+
|a|^2
\right)^2,
\end{align*}
where $\kappa_n>0$ is the minimum of the positive constants obtained in
Lemma~\ref{lem:finite-dimensional-Garding} and Lemma~\ref{lem:H2-estimate}.

Using \eqref{L2}, \eqref{eq:H2-final}, and the bootstrap bound
    $|a(s)|
    \leq
    e^{-\mu s}$, and enlarging $s_0$ so that
$\|\varepsilon(s)\|_{H^2}\le1$ on the bootstrap interval, we have
$\|\varepsilon\|_2^4\le\|\varepsilon\|_{H^2}^4
\le\|\varepsilon\|_{H^2}^2$.  Hence
\begin{equation*}
\frac{d}{ds}\|Z\|_{L^2}^2
+
\kappa_n\|Z\|_{L^2}^2
\le \left(1+e^{-2\mu s}\right)
\left[C_n'\left(
K_0
+
K^2e^{-\mu s_0}
\right)\right]^2e^{-2\mu s}
+
e^{-4\mu s}.
\end{equation*}

Since $2\mu<\kappa_n$, multiplication by
$e^{\kappa_n s}$ and integration from $s_0$ to $s$ yield
\begin{equation}\label{sca-en}
\begin{aligned}
\|Z(s)\|_{L^2}^2
\lesssim{}&
e^{-\kappa_n(s-s_0)}
\|Z(s_0)\|_{L^2}^2
\\
&+
\left[C_n'\left(
K_0
+
K^2e^{-\mu s_0}
\right)\right]^2
\left[\int_{s_0}^s
e^{-\kappa_n(s-\tau)}
e^{-2\mu\tau}\,d\tau+\int_{s_0}^s
e^{-\kappa_n(s-\tau)}
e^{-4\mu\tau}\,d\tau\right]
\\
&+
\int_{s_0}^s
e^{-\kappa_n(s-\tau)}
e^{-4\mu\tau}\,d\tau.
\end{aligned}
\end{equation}
Using \eqref{initial-setting-KS} and $2\mu<\kappa_n$, we obtain
\begin{equation}\label{sca-en1}
e^{-\kappa_n(s-s_0)}
\|Z(s_0)\|_{L^2}^2\le K_0^2e^{-2\mu s}.
\end{equation}
In addition, we have
\begin{equation*}
\int_{s_0}^s
e^{-\kappa_n(s-\tau)}
e^{-2\mu\tau}\,d\tau=e^{-2\mu s}\int_{s_0}^se^{(2\mu-\kappa_n)(s-\tau)}\,d\tau\le \frac{e^{-2\mu s}}{\kappa_n-2\mu}.
\end{equation*}
We also have
\begin{equation}\label{sca-en3}
\int_{s_0}^s
e^{-\kappa_n(s-\tau)}
e^{-4\mu\tau}\,d\tau\le e^{-2\mu s_0}\int_{s_0}^s
e^{-\kappa_n(s-\tau)}
e^{-2\mu\tau}\,d\tau\le
\frac{e^{-2\mu (s_0+s)}}{\kappa_n-2\mu}.
\end{equation}
Substituting \eqref{sca-en1}--\eqref{sca-en3} into \eqref{sca-en} gives
\[
\|Z(s)\|_{L^2}^2\lesssim K_0^2e^{-2\mu s}+\left[\left[C_n'\left(
K_0
+
K^2e^{-\mu s_0}
\right)\right]^2+1\right]{e^{-2\mu s}}
.\]
The implicit constants in the preceding inequalities depend only on
\(n\), the fixed \(\mu\), and the previously chosen Young constants;
none depends on \(s_0\), \(K'\), or the particular solution.  Since
we enforce \(K^2e^{-\mu s_0}\le K_0\) and may also require
\(e^{-\mu s_0}\le1\), it follows that there is a constant
\(C_{\Lambda,n,\mu}>0\), independent of \(s_0\), such that
\begin{equation}\label{eq:Lambda-uniform-prechoice}
 \|\Lambda\varepsilon(s)\|_2
 \le C_{\Lambda,n,\mu}(1+K_0)e^{-\mu s},
 \qquad s\in[s_0,s^*].
\end{equation}
We can therefore choose, once and for all \emph{before} \(s_0\),
\[
 K'\ge \max\{4K_0,\,4C_{\Lambda,n,\mu}(1+K_0),\,2\}.
\]
This gives \(\|\Lambda\varepsilon(s)\|_2\le
(K'/4)e^{-\mu s}\), and hence the strict improvement
\(\|\Lambda\varepsilon(s)\|_2\le(K'/2)e^{-\mu s}\).
The fixed \(K'\) can now be included in the final lower bound on
\(s_0\), as needed for continuation and Lipschitz dependence.
This completes the proof.
\end{proof}

\begin{lemma}[Continuation at a finite bootstrap exit and uniform outward crossing]
\label{lem:bootstrap-exit-continuation}
Fix $n$, $\mu$, $K_0$, and $K$ as above, and take $s_0$ sufficiently
large that $C_n(K+1)e^{-\mu s_0}$ lies inside the neighborhood of
Lemma~\ref{lemma-geometric-Un}.  Here the constant $K'$ has already been fixed; the final threshold
for $s_0$ is allowed to depend on it, as specified in
\eqref{eq:bootstrap-constant-ledger}.
Since $n$ is sufficiently large, $N=N_n\ge1$.  For any initial
unstable vector in the closed ball $|a(s_0)|\le e^{-\mu s_0}$,
the following assertions hold.
\begin{enumerate}
\item If the bootstrap interval has a finite endpoint $s^*$, the
physical solution, its $\mathcal Y$ norm, and its geometrically
normalized remainder all extend continuously a little beyond $s^*$.
Thus a finite $s^*$ cannot be caused by loss of local existence or by
failure of the geometric decomposition.  If the other three
bootstrap inequalities are strict at $s^*$, then necessarily
$\Theta(s^*)=1$, where $\Theta(s)=e^{2\mu s}|a(s)|^2$.
\item Put $\mu_*:=\min_{1\le j\le N}\mu_j>0$.  On the bootstrap
interval, whenever $1/2\le\Theta(s)\le1$, one has
\begin{equation}\label{eq:uniform-unstable-exit-speed}
 \Theta'(s)\ge c_{\rm out}:=\frac12(\mu+\mu_*)>0,
\end{equation}
provided $s_0$ is sufficiently large depending only on the already
fixed constants.  At a finite exit this inequality persists on a
short interval of the continued solution around $s^*$.
\end{enumerate}
\end{lemma}
\begin{proof}
The modulation estimate \eqref{semin} and the bootstrap hypotheses give
\[
 |b(s)|+|\beta(s)|+|R(s)|
 \le C_n(K^2+1)e^{-2\mu s},\qquad
 b=\frac{\lambda_s}{\lambda}+\frac12,\quad
 \beta=\frac{x_s}{\lambda},\quad
 R_j=(a_j)_s-\mu_ja_j.
\]
Using $\lambda(s_0)=e^{-s_0/2}$, integration gives
\[
 \left|\log\bigl(\lambda(s)e^{s/2}\bigr)\right|
 \le \frac{C_n(K^2+1)}{2\mu}e^{-2\mu s_0}.
\]
In particular, after increasing $s_0$,
\begin{equation}\label{eq:finite-exit-lambda-control}
 \frac12e^{-s/2}\le\lambda(s)\le2e^{-s/2},\qquad
 |x_s(s)|\le C_n(K^2+1)e^{-(2\mu+1/2)s}.
\end{equation}
If $s^*<\infty$, then $\lambda$ is bounded away from zero on
$[s_0,s^*]$, $x(s)$ has a finite limit, and the corresponding physical
time $t(s)=\int_{s_0}^s\lambda(\tau)^2\,d\tau$ has a finite endpoint.
Furthermore,
\[
 \|v(s)\|_{H^2}\le C_n(K+1)e^{-\mu s_0}
\]
keeps the solution in the geometric tube.  The scaling law gives
\[
 \|u(t(s))\|_{H^2}\le C_n
 \bigl(\lambda^{-1/2}+\lambda^{-3/2}+\lambda^{-5/2}\bigr)
 \bigl(1+\|v(s)\|_{H^2}\bigr),
\]
which is bounded up to every finite $s^*$.  The $\mathcal Y$ bound
follows in the same way from $\|\Lambda\varepsilon\|_2<
K'e^{-\mu s}$, the smoothness of the discrete modes, and the
bounded affine change of variables; it need not be small.
Proposition~\ref{prop:local-wp} thus extends the physical solution
past $t(s^*)$.  Local uniqueness in Lemma~\ref{lemma-geometric-Un},
together with strong continuity on $\mathcal Y$, extends the same
modulation branch and hence all bootstrap quantities past $s^*$.
The final assertion in part~(1) follows by continuity.

For part~(2), differentiation gives the exact identity
\[
 \frac12\Theta'=e^{2\mu s}
 \sum_{j=1}^N\bigl((\mu+\mu_j)a_j^2+a_jR_j\bigr).
\]
The same estimate holds at $s=s_0$ for boundary initial data:
although their bootstrap interval may reduce to the single point
$s_0$, the initial bounds and the finite-dimensional modulation
system \eqref{eq:modulation-finite-matrix} are valid there.
The preceding bound on $R$ therefore yields
\begin{equation}\label{eq:unstable-outflow-barrier}
 \Theta'(s)\ge
 2(\mu+\mu_*)\Theta(s)
 -C_n(K^2+1)e^{-\mu s}\sqrt{\Theta(s)}.
\end{equation}
For $\Theta\in[1/2,1]$, the first term on the right is at least
$\mu+\mu_*$ and the error is at most
$C_n(K^2+1)e^{-\mu s_0}$.  Choose $s_0$ so that this last quantity
is no larger than $(\mu+\mu_*)/4$.  We therefore obtain the
strictly stronger bound $\Theta'\ge 3(\mu+\mu_*)/4$ while
$1/2\le\Theta\le1$, which implies
\eqref{eq:uniform-unstable-exit-speed}.  Local $C^1$-in-time
modulation and continuous dependence on the initial data extend the
strict derivative inequality across a finite crossing, after
possibly shortening the continuation interval.
\end{proof}

\subsection{Completion of the proof}
We are now in position to conclude the proof of Proposition \ref{bootstrap}.
\begin{proof}[Proof of Proposition \ref{bootstrap}]

   \emph{$\mathbf{Step\ 1}$. Improved scaling control.}
    By \eqref{chca}, \eqref{1.9}, and \eqref{semin}, we have
\begin{equation}\label{fky}
\left|
\frac{\lambda_s}{\lambda}
+\frac12
\right|
\lesssim
K^2e^{-2\mu s},
\qquad s\in[s_0,s^*].
\end{equation}
Integrating \eqref{fky} from $s_0$ to $s$, we obtain
\[
\left|
\log\frac{\lambda(s)}{\lambda(s_0)}
+\frac{s-s_0}{2}
\right|
\lesssim
K^2\int_{s_0}^{s}e^{-2\mu\tau}\,d\tau
\leq
\frac{K^2}{2\mu}e^{-2\mu s_0}.
\]
Hence,
\[
\lambda(s)
=
\lambda(s_0)e^{-\frac{s-s_0}{2}}
\exp\left(
O\left(K^2e^{-2\mu s_0}\right)
\right).
\]
Using $
\lambda(s_0)=e^{-s_0/2}$,
we obtain
\begin{equation}\label{lambda-asymptotic}
\lambda(s)
=
e^{-s/2}
\exp\left(
O\left(K^2e^{-2\mu s_0}\right)
\right)
.
\end{equation}
Since $K$ and $\mu$ are fixed before $s_0$ is chosen, by taking
$s_0$ sufficiently large we may ensure that
\[
C K^2e^{-2\mu s_0}\leq \log 2.
\]
It follows from \eqref{lambda-asymptotic} that
\begin{equation}\label{control-sca}
    \frac12e^{-s/2}
\leq
\lambda(s)
\leq
2e^{-s/2},
\qquad s\in[s_0,s^*].
\end{equation}

Since $\mu<1/2$, by further enlarging $s_0$ if necessary so that
$
2e^{-(\frac12-\mu)s}<2e^{-(\frac12-\mu)s_0}<\frac12,$
then
\begin{equation}\label{poaf}
0<\lambda(s)<\frac12e^{-\mu s},
\qquad s\in[s_0,s^*].
\end{equation}
Thus the scaling bootstrap bound is strictly improved.

\emph{$\mathbf{Step\ 2}$. The outgoing property.}
Recall the exit time
\[
s^*
=
\sup\left\{
s\geq s_0:
\eqref{1.7}-\eqref{L2Lambda}
\text{ hold on }[s_0,s)
\right\}.
\]
Suppose that $s^*<+\infty$.
By the strict improvements \eqref{H^2-es}, \eqref{Lambda-es}, and
\eqref{poaf}, the $H^2$, scaling-derivative, and scaling bootstrap bounds remain strict
at $s=s^*$, provided $s_0$ is sufficiently large. Hence, by
Lemma~\ref{lem:bootstrap-exit-continuation}, finite-time loss of the
solution or of the geometric gauge is excluded.  Therefore, by
continuity, the solution can exit the bootstrap regime only through
the unstable boundary, and
\begin{equation}\label{s*}
\sum_{j=1}^N |a_j(s^*)|^2
=
e^{-2\mu s^*}.
\end{equation}

Set
\[
 \Theta(s):=e^{2\mu s}\sum_{j=1}^N|a_j(s)|^2.
\]
Lemma~\ref{lem:bootstrap-exit-continuation} justifies continuation
past any finite $s^*$ and, in view of \eqref{s*}, supplies the
uniformly transverse crossing estimate
\begin{equation}\label{fas}
 \Theta'(s^*)\ge c_{\rm out}>0.
\end{equation}
The strict positivity, not merely the sign of the linear unstable
eigenvalues, is the ingredient used below to establish continuity of
the first-exit map.

\emph{$\mathbf{Step\ 3}$. The Brouwer argument.}
Let
\[
\overline{\mathcal B}
:=
\left\{
 b=(b_1,\ldots,b_N)\in\mathbb R^{N}:
| b|\leq1
\right\}
\]
be the closed unit ball in $\mathbb R^{N}$, and let
$\partial\mathcal B$ denote its boundary.
For $b\in\overline{\mathcal B}$, we choose the initial
unstable parameters according to
\[
a_j(s_0)=e^{-\mu s_0}b_j,
\qquad 1\leq j\leq N,
\]
and denote by $s^*( b)$ the corresponding exit time.

We argue by contradiction and assume that
\[
s^*( b)<+\infty,
\qquad
\text{for every }  b\in\overline{\mathcal B}.
\]
By the strict improvements of all the other bootstrap estimates in Proposition \ref{bootstrap},
the solution can exit the bootstrap regime only through the unstable
boundary. Hence
\[
e^{2\mu s^*( b)}
\sum_{j=1}^N
|a_j(s^*( b))|^2
=1.
\]
We may therefore define the exit map
$\Phi:\overline{\mathcal{B}}\longrightarrow\partial \mathcal{B}$
by
\[
\Phi( b):=
\left(
e^{\mu s^*( b)}
a_j(s^*( b),b)
\right)_{1\leq j\leq N}.
\]

We next prove continuity of the exit map.  Define
\[
\theta(s, b)
:=
e^{2\mu s}\sum_{j=1}^N |a_j(s, b)|^2.
\]
Fix $b_0\in\overline{\mathcal B}$ and put $\sigma_0=s^*(b_0)$.  Then
$\theta(\sigma_0,b_0)=1$ and, by \eqref{fas},
\begin{equation}\label{eq:transversal-exit}
 \partial_s\theta(\sigma_0,b_0)>0.
\end{equation}
Proposition~\ref{prop:local-wp} continues the solution in $\mathcal Y$
for a short time past $\sigma_0$; only a bootstrap inequality has ceased
to hold.  The locally Lipschitz estimate
\eqref{eq:local-Y-lipschitz} and the joint strong continuity of
translations and dilations on $\mathcal Y$ from
Lemma~\ref{lem:Y-similarity-action} give continuous dependence of the
geometric remainder in $\mathcal Y$ on compact physical-time
intervals.  The implicit-function parameters are continuous there as
well.  While the solution remains in the tube, $\lambda$ is bounded
away from zero on each such compact interval, so
$s(t)=s_0+\int_0^t\lambda(\tau)^{-2}\,d\tau$ and its increasing
inverse also vary continuously.  Hence the rescaled bootstrap norms,
including $\|\Lambda\varepsilon\|_2$, depend continuously on $b$
at each common renormalized time.

First suppose that $\sigma_0>s_0$.  Choose
$0<\eta<(\sigma_0-s_0)/2$ sufficiently small.  Transversality gives
\[
 \theta(\sigma_0-\eta,b_0)<1,
 \qquad
 \theta(\sigma_0+\eta,b_0)>1.
\]
On the compact interval $[s_0,\sigma_0-\eta]$, the unstable bound and all
the other bootstrap bounds have a uniform strict margin.  Continuous
dependence on $b$ preserves these margins and the two displayed strict
inequalities for $b$ close to $b_0$.  Therefore
\[
 \sigma_0-\eta<s^*(b)<\sigma_0+\eta.
\]

It remains to treat the endpoint case $\sigma_0=s_0$, which is precisely the
case relevant to $b_0\in\partial\mathcal B$.  By
\eqref{eq:uniform-unstable-exit-speed} and the local continuation
just proved, there are $c_0,\eta_0>0$ such that
\[
 \partial_s\theta(s,b)\ge c_0
\]
for $s_0\le s\le s_0+\eta_0$ and $b$ sufficiently close to $b_0$, as long
as $1/2\le\theta(s,b)\le1$.  Here one restricts to
$|b|^2>1/2$ and uses the strict improvement of the other bootstrap
bounds to ensure that the relevant first crossing is governed by
$\theta$ alone.  Since
$\theta(s_0,b)=|b|^2$, the first time at which $\theta$ reaches one obeys
\[
 0\le s^*(b)-s_0\le c_0^{-1}(1-|b|^2),
\]
after shrinking the neighborhood if necessary; the other bootstrap bounds
remain strict on this short interval.  Hence $s^*(b)\to s_0$ as
$b\to b_0$ within $\overline{\mathcal B}$.

The two cases prove that $b\mapsto s^*(b)$ is continuous.  Since the local
flow and the modulation parameters depend continuously on $(s,b)$, the map
$\Phi$ is continuous on $\overline{\mathcal B}$.

If $b\in\partial\mathcal B$, then
$e^{2\mu s_0}
\sum_{j=1}^N|a_j(s_0)|^2=1$.
By the uniformly outgoing property \eqref{fas}, there exists $\eta>0$ such that
$e^{2\mu s}
\sum_{j=1}^N|a_j(s, b)|^2>1$
for every $s\in(s_0,s_0+\eta]$.
Hence $s^*( b)=s_0$, and consequently
$\Phi( b)
=
\left(e^{\mu s_0}a_j(s_0)\right)_{1\leq j\leq N}
=
 b.$
Thus
\[
\Phi|_{\partial\mathcal B}
=
\operatorname{Id}_{\partial\mathcal B}.
\]

Therefore, $\Phi$ is a continuous retraction of the closed ball
$\overline{\mathcal B}$ onto its boundary $\partial\mathcal B$,
which is impossible by the Brouwer no-retraction theorem.
This contradiction shows that there exists
$ b_0\in\overline{\mathcal B}$ such that
\[
s^*( b_0)=+\infty.
\]
For the corresponding choice of the initial unstable parameters,
all the bootstrap estimates hold for every $s\geq s_0$.
This completes the proof of Proposition~\ref{bootstrap}.
\end{proof}

We next prove the Lipschitz dependence of the
unstable parameters selected in Proposition~\ref{bootstrap}.
This also yields the uniqueness of the unstable parameters
corresponding to a given stable perturbation.

\begin{proposition}[Lipschitz dependence]
\label{prop:Lipschitz-dependence}
Fix the index $n$, the bootstrap constants, and one sufficiently large
initial renormalized time $s_0$ for which Proposition~\ref{bootstrap}
applies.
Both initial data are prescribed with the same normalization
$\lambda^{(i)}(s_0)=\lambda_0=e^{-s_0/2}$ and
$x^{(i)}(s_0)=0$, $i=1,2$.
Let $\varepsilon_0^{(1)}$ and $\varepsilon_0^{(2)}$ be two
admissible initial remainders satisfying the assumptions of
Proposition~\ref{bootstrap} at this fixed $s_0$. Let
\[
a_0^{(i)}
=
\bigl(a_{1,0}^{(i)},\ldots,a_{N,0}^{(i)}\bigr),
\qquad i=1,2,
\]
be two choices of the unstable initial parameters such that the
corresponding solutions remain in the bootstrap regime for all
$s\ge s_0$. Then
\begin{equation*}
\bigl|a_0^{(1)}-a_0^{(2)}\bigr|
\lesssim
\bigl\|
\varepsilon_0^{(1)}-\varepsilon_0^{(2)}
\bigr\|_{H^2}.
\end{equation*}
In particular, the unstable parameters are uniquely determined
by the stable initial perturbation and depend Lipschitz
continuously on it.
\end{proposition}

\begin{proof}
For $i=1,2$, let
\[
\varepsilon^{(i)},\qquad
\lambda^{(i)},\qquad
x^{(i)},\qquad
a^{(i)}=(a_1^{(i)},\ldots,a_N^{(i)})
\]
be the corresponding modulation parameters and remainders.
For the two physical solutions the maps
$t\mapsto s^{(i)}(t)=s_0+\int_0^t(\lambda^{(i)}(\tau))^{-2}\,d\tau$
need not coincide.  All differences below are instead taken at the
\emph{same renormalized time} $s$ in the respective, autonomous
similarity equations.  This comparison is valid because both
renormalized evolutions start at the same $s_0$ with the same initial
scale and center; it does not assert that their physical times
$t^{(1)}(s)$ and $t^{(2)}(s)$ agree. Set
\[
\psi^{(i)}
=
\sum_{j=1}^{N}a_j^{(i)}\phi_j,\ \ v^{(i)}=\varepsilon^{(i)}+\psi^{(i)}.
\]
We introduce
\[
b_i
=
\frac{\lambda_s^{(i)}}{\lambda^{(i)}}+\frac12,
\qquad
\beta_i
=
\frac{x_s^{(i)}}{\lambda^{(i)}},
\]
and define the differences
\[
\delta\varepsilon
=
\varepsilon^{(1)}-\varepsilon^{(2)},
\qquad
\delta a_j
=
a_j^{(1)}-a_j^{(2)},
\qquad
\delta\psi
=
\psi^{(1)}-\psi^{(2)},\qquad \delta v=v^{(1)}-v^{(2)},
\]
and
\[
\delta b=b_1-b_2,
\qquad
\delta\beta=\beta_1-\beta_2.
\]

Since both solutions remain in the bootstrap regime, there holds
\begin{equation*}
\|\varepsilon^{(i)}(s)\|_{H^2}
+
|a^{(i)}(s)|
+
|b_i(s)|
+
|\beta_i(s)|
\lesssim
e^{-\mu s},
\qquad i=1,2.
\end{equation*}

\emph{$\mathbf{Step\ 1}$. Difference modulation estimates.}
For $i=1,2$, put
\[
 z_i=(b_i,\beta_i,q^{(i)})^T,
 \qquad q_j^{(i)}=(a_j^{(i)})_s-\mu_ja_j^{(i)},
\]
and write $\delta z=z_1-z_2$, so that
$\delta q_j=(\delta a_j)_s-\mu_j\delta a_j$.
The exact single-solution identity
\eqref{eq:modulation-one-system} gives
\begin{equation}\label{eq:difference-equation}
\begin{aligned}
 \partial_s\delta\varepsilon+L_n\delta\varepsilon
 ={}&\delta b\,\Lambda U_n
      +\delta\beta\cdot\nabla U_n
      -\sum_{j=1}^N\delta q_j\phi_j
      +b_1\Lambda\delta v+\beta_1\cdot\nabla\delta v\\
 &+\delta b\,\Lambda v^{(2)}
      +\delta\beta\cdot\nabla v^{(2)}
      +\delta\mathcal N,\\
 \delta\mathcal N
 :={}&\nabla\cdot\bigl(\delta v\,T v^{(1)}
                   +v^{(2)}T\delta v\bigr),
 \qquad T=\nabla\Delta^{-1}.
\end{aligned}
\end{equation}
Since each remainder is orthogonal to the same fixed adjoint modes,
\begin{equation}\label{eq:difference-orthogonality}
 \langle\delta\varepsilon,\zeta_m\rangle=0
 \quad\text{for }\zeta_m\in
 \{\zeta_0,\zeta_k^{\rm tr},\zeta_j^{\rm u}\}.
\end{equation}
These identities are taken in the respective similarity coordinates,
with the adjoint modes independent of $s$.

We derive the estimates simultaneously from the finite-dimensional
matrix equation \eqref{eq:modulation-finite-matrix}.
Let
\[
 \mathscr F(v)=\bigl(\mathscr N_{\zeta_m}(v,v)\bigr)_m,
 \qquad M(v)=M_0+E(v),
 \quad M_0=\operatorname{diag}(1,I_3,-I_N).
\]
Then $M(v^{(i)})z_i=-\mathscr F(v^{(i)})$; subtracting
these two \emph{exact} equations yields
\begin{equation}\label{eq:difference-matrix-identity}
 M(v^{(1)})\delta z
 =-\bigl(\mathscr F(v^{(1)})-\mathscr F(v^{(2)})\bigr)
  -E(\delta v)z_2.
\end{equation}
Notice that $E$ is linear in its argument and the nonlinear
functional is quadratic, but not generally symmetric.  More
precisely, by the flux identity and Lemma~\ref{lem:adjoint-drift-testing},
\[
 \mathscr F_m(v^{(1)})-\mathscr F_m(v^{(2)})
 =\mathscr N_{\zeta_m}(\delta v,v^{(1)})
   +\mathscr N_{\zeta_m}(v^{(2)},\delta v),
\]
so
\begin{equation}\label{eq:difference-matrix-L2-bounds}
\begin{aligned}
 \|E(\delta v)\|&\le C_n\|\delta v\|_2,\\
 |\mathscr F(v^{(1)})-\mathscr F(v^{(2)})|
 &\le C_n(\|v^{(1)}\|_2+\|v^{(2)}\|_2)
 \|\delta v\|_2.
\end{aligned}
\end{equation}
The two trapped solutions satisfy
\[
 \|v^{(i)}(s)\|_2\le C_ne^{-\mu s},\qquad
 |z_i(s)|\le C_ne^{-2\mu s},\qquad i=1,2,
\]
by the bootstrap bounds and \eqref{semin}.
The inverses $M(v^{(i)}(s))^{-1}$ are uniformly bounded on
$[s_0,\infty)$, by the same smallness choice used to prove
Lemma~\ref{modulation}.  We therefore deduce from
\eqref{eq:difference-matrix-identity}--\eqref{eq:difference-matrix-L2-bounds}
\begin{equation}\label{eq:difference-modulation-final}
 |\delta b|+|\delta\beta|
 +\sum_{j=1}^N| (\delta a_j)_s-\mu_j\delta a_j|
 \le C_ne^{-\mu s}\bigl(\|\delta\varepsilon\|_2+|\delta a|\bigr),
\end{equation}
where we used the finite-dimensional norm equivalence
\begin{equation}\label{eq:delta-v-L2}
 \|\delta v\|_2
 \le\|\delta\varepsilon\|_2+C_n|\delta a|.
\end{equation}
In particular, the modulation difference estimate has no derivative
loss and requires no smallness absorption involving
$\delta b$ or $\delta\beta$.
For later use the single-solution modulation estimate also gives
\begin{equation}\label{mo-es1}
 |b_1|+|\beta_1|\le C_ne^{-2\mu s}.
\end{equation}

\emph{$\mathbf{Step\ 2}$. Difference estimate on the stable component.}
Recall that
$\delta\varepsilon
=
\varepsilon^{(1)}-\varepsilon^{(2)}$
satisfies the same orthogonality conditions as
$\varepsilon^{(1)}$ and $\varepsilon^{(2)}$. Hence
\[
\delta\varepsilon(s)\in X_{s,n},
\qquad s\ge s_0.
\]
Let $P_{s,n}$ denote the projection onto $X_{s,n}$ and recall
\[
D_{0,n}=P_{s,n}\Lambda|_{X_{s,n}},
\qquad
D_{k,n}=P_{s,n}\partial_k|_{X_{s,n}},
\qquad 1\le k\le3.
\]

After projecting the difference equation onto $X_{s,n}$, the
pure scaling, translation and unstable directions disappear, and
we obtain
\begin{equation}\label{eq:projected-difference}
\partial_s\delta\varepsilon
+
L_n\delta\varepsilon
=
b_1D_{0,n}\delta\varepsilon
+
\sum_{k=1}^3\beta_{1,k}D_{k,n}\delta\varepsilon
+
P_{s,n}\mathcal G_\delta,
\end{equation}
where
\begin{equation*}
\begin{aligned}
\mathcal G_\delta
={}&
\delta b\,\Lambda v^{(2)}
+
\delta\beta\cdot\nabla v^{(2)}
+
b_1\Lambda\delta\psi
+
\beta_1\cdot\nabla\delta\psi
+
\mathcal N(v^{(1)})
-
\mathcal N(v^{(2)}),
\end{aligned}
\end{equation*}
with
\[
\mathcal N(v)
=
v^2+\nabla v\cdot\nabla\Delta^{-1}v.
\]

We first estimate $\mathcal G_\delta$. By the bootstrap bounds,
the finite dimensionality of the unstable space and the scaling
derivative estimate,
\begin{equation}\label{eq:individual-strong-bounds}
\|v^{(i)}\|_{H^2}
+
\|\Lambda v^{(i)}\|_{L^2}
+
|a^{(i)}|
\lesssim
e^{-\mu s},
\qquad i=1,2.
\end{equation}
Moreover,
\[
\|\delta\psi\|_{H^2}
+
\|\Lambda\delta\psi\|_{L^2}
\lesssim
|\delta a|.
\]

We next estimate the nonlinear difference. The exact conservative
form in \eqref{eq:difference-equation} gives, for the strong
solutions considered here,
\begin{align*}
\mathcal N(v^{(1)})-\mathcal N(v^{(2)})
={}&
(v^{(1)}+v^{(2)})\delta v
+
\nabla\delta v\cdot Tv^{(1)}
+
\nabla v^{(2)}\cdot T\delta v.
\end{align*}

Using
\[
H^2(\mathbb R^3)\hookrightarrow L^\infty(\mathbb R^3),
\qquad
\|Tf\|_{L^\infty}\lesssim\|f\|_{H^1},
\qquad
\|Tf\|_{L^6}\lesssim\|f\|_{L^2},
\]
together with the bootstrap bounds and \eqref{eq:delta-v-L2}, we obtain
\begin{equation}\label{eq:nonlinear-difference-L2}
\begin{aligned}
\|
\mathcal N(v^{(1)})
-
\mathcal N(v^{(2)})
\|_{L^2}
\lesssim
e^{-\mu s}
\left(
\|\delta\varepsilon\|_{H^1}
+
|\delta a|
\right).
\end{aligned}
\end{equation}

Combining
\eqref{eq:individual-strong-bounds},
\eqref{eq:difference-modulation-final} and
\eqref{eq:nonlinear-difference-L2}, we obtain
\begin{equation}\label{eq:G-delta-estimate}
\|\mathcal G_\delta\|_{L^2}
\lesssim
e^{-\mu s}
\left(
\|\delta\varepsilon\|_{H^1}
+
|\delta a|
\right).
\end{equation}

Recall that
\[
\mathcal E_n[f]
=
\langle B_nf,f\rangle
\]
be the modified energy on $X_{s,n}$, there holds
\begin{equation*}
c_n\|f\|_{L^2}^2
\le
\mathcal E_n[f]
\le
C_n\|f\|_{L^2}^2.
\end{equation*}
Let
\begin{equation}\label{eq:difference-dissipation}
\mathcal D_n[f]:=\operatorname{Re}\langle B_nL_nf,f\rangle
\ge
\frac{\gamma_n}{2}\mathcal E_n[f].
\end{equation}
Moreover, the dissipation controls the full $H^1$ norm. Indeed, by
\eqref{eq:L2numericalform}, \eqref{eq:anchoice},
\eqref{eq:Lyapunovidentity}, and
$B_n=I+\alpha_nQ_n$ with $\alpha_n=2a_n+1$, we have
\begin{align*}
\mathcal D_n[f]
&=\operatorname{Re}\langle L_nf,f\rangle
  +\alpha_n\operatorname{Re}\langle Q_nL_nf,f\rangle\ge \|\nabla f\|_2^2
 +\left(\frac14-a_n+\frac{\alpha_n}{2}\right)\|f\|_2^2=\|\nabla f\|_2^2+\frac34\|f\|_2^2.
\end{align*}
Consequently,
\begin{equation}\label{eq:D-controls-H1}
\|f\|_{H^1}^2
\lesssim
\mathcal D_n[f].
\end{equation}

Set \[f:=\delta \varepsilon.\] The energy identity below is
first justified for smooth approximations, then passed to
$f\in D(L_n)\cap X_{s,n}$ by the graph-domain approximation and
local regularity already used in the single-solution estimates.
Taking the real part of the $B_n$-energy of
\eqref{eq:projected-difference} gives
\begin{align*}
\frac12\frac{d}{ds}\mathcal E_n[f]
+\mathcal D_n[f]
={}&
b_1\langle B_nD_{0,n}f,f\rangle
+
\sum_{k=1}^3
\beta_{1,k}\langle B_nD_{k,n}f,f\rangle
+
\langle B_nP_{s,n}\mathcal G_\delta,f\rangle.
\end{align*}
By 
the boundedness of $B_n$ and $P_{s,n}$ on $L^2$, and \eqref{eq:projected-generator-form},
\eqref{eq:difference-dissipation},
we obtain
\begin{equation}\label{eq:difference-energy-before}
\frac12\frac{d}{ds}\mathcal E_n[f]
+
\mathcal D_n[f]
\le
C_n(|b_1|+|\beta_1|)
\mathcal D_n[f]
+
C_n\|\mathcal G_\delta\|_{L^2}
\mathcal E_n[f]^{1/2}.
\end{equation}

From \eqref{mo-es1}
and, by \eqref{eq:G-delta-estimate} and
\eqref{eq:D-controls-H1},
\[
\|\mathcal G_\delta\|_{L^2}
\lesssim
e^{-\mu s}
\left(
\mathcal D_n[f]^{1/2}
+
|\delta a|
\right),
\]
Consequently,
\begin{align*}
C_n\|\mathcal G_\delta\|_{L^2}
\mathcal E_n[f]^{1/2}
\lesssim
e^{-\mu s}
\left(
\mathcal D_n[f]^{1/2}
+
|\delta a|
\right)
\mathcal E_n[f]^{1/2}
\lesssim
e^{-\mu s}
\mathcal D_n[f]
+
e^{-\mu s}
|\delta a|\,
\mathcal D_n[f]^{1/2}.
\end{align*}
Hence, for any $\eta>0$, Young's inequality yields
\begin{equation}\label{eq:G-delta-energy}
\begin{aligned}
C_n\|\mathcal G_\delta\|_{L^2}
\mathcal E_n[f]^{1/2}
\le
C_n e^{-\mu s}
\mathcal D_n[f]
+
\eta\mathcal D_n[f]
+
C_{n,\eta}e^{-2\mu s}
|\delta a|^2.
\end{aligned}
\end{equation}

We first choose $\eta>0$ sufficiently small and then take
$s_0$ sufficiently large so that
\[
C_n e^{-\mu s_0}
+
\eta
\ll 1.
\]
Since $s\ge s_0$, the terms involving
$\mathcal D_n[f]$ on the right-hand side of
\eqref{eq:difference-energy-before} and
\eqref{eq:G-delta-energy} can therefore be absorbed into the
left-hand side. Thus there exists $c_n>0$ such that
\begin{equation*}
\frac{d}{ds}
\mathcal E_n[f](s)
+
c_n\mathcal D_n[f](s)
\lesssim
e^{-2\mu s}
|\delta a|^2.
\end{equation*}

Finally, using \eqref{eq:difference-dissipation},
we obtain, for some $\kappa_n>0$,
\begin{equation}\label{eq:difference-energy-final}
\frac{d}{ds}
\mathcal E_n[f](s)
+
\kappa_n
\mathcal E_n[f](s)
\lesssim
e^{-2\mu s}
|\delta a|^2.
\end{equation}

\emph{$\mathbf{Step\ 3}$. Backward control of the unstable modes.}
The coercivity constant $\kappa_n$ in
\eqref{eq:difference-energy-final} is fixed by the linear energy
and the Young-inequality coefficients.  In the parameter order
\eqref{eq:bootstrap-mu-uniform}, we included this dissipation rate
in $\kappa_{{\rm diff},n}$ and imposed $2\mu<\kappa_{{\rm diff},n}$
\emph{before} fixing the bootstrap amplitudes or $s_0$.
No change of $\mu$ is made at this stage.  The constants in
\eqref{eq:individual-strong-bounds} can depend on the already fixed
$K'$; the last inequality of \eqref{eq:bootstrap-constant-ledger}
ensures that their contribution is small enough to absorb.

For each $1\le j\le N$, define
\begin{equation*}
R_j(s):=(\delta a_j)_s-\mu_j\delta a_j.
\end{equation*}
The difference modulation estimate
\eqref{eq:difference-modulation-final} gives
\begin{equation*}
|R_j(s)|
\lesssim
e^{-\mu s}
\left(
\|\delta\varepsilon(s)\|_{L^2}
+
|\delta a(s)|
\right).
\end{equation*}

Set
\begin{equation*}
M
:=
\sup_{s\ge s_0}e^{\mu s}|\delta a(s)|,
\qquad
\widetilde M
:=
\sup_{s\ge s_0}
e^{2\mu s}
\mathcal E_n[\delta\varepsilon(s)].
\end{equation*}
The individual bootstrap estimates for the two trapped solutions,
together with
$\mathcal E_n[\delta\varepsilon]
\asymp
\|\delta\varepsilon\|_{L^2}^2$,
ensure that $M$ and $\widetilde M$ are finite. Moreover,
\[
|\delta a(s)|\le Me^{-\mu s},
\qquad
\|\delta\varepsilon(s)\|_{L^2}
\lesssim
e^{-\mu s}\sqrt{\widetilde M}.
\]
Hence
\begin{equation}\label{eq:Rj-M-bound}
|R_j(s)|
\lesssim
e^{-2\mu s}
\left(
M+\sqrt{\widetilde M}
\right).
\end{equation}
The eigenvalues $\mu_j$ are strictly positive (indeed these are
the genuine unstable modes of $-L_n$), so the following backward
integrals converge. Since
\begin{equation}\label{changes}
\partial_s\bigl(e^{-\mu_js}\delta a_j(s)\bigr)
=
e^{-\mu_js}R_j(s),
\end{equation}
and both solutions remain trapped, we have
$|\delta a_j(s)|\lesssim e^{-\mu s}$.
Hence $\lim_{s\to+\infty}e^{-\mu_js}\delta a_j(s)=0$.
Integrating \eqref{changes} from $s$ to $+\infty$,
\begin{equation}\label{eq:backward-representation-aj}
\delta a_j(s)
=
-
e^{\mu_js}
\int_s^\infty e^{-\mu_j\tau}R_j(\tau)\,d\tau.
\end{equation}
At $s=s_0$ the unstable coefficient satisfies the
cancellation condition
\begin{equation*}
e^{-\mu_js_0}\delta a_j(s_0)
+
\int_{s_0}^\infty
e^{-\mu_j\tau}R_j(\tau)\,d\tau
=
0.
\end{equation*}
Combining \eqref{eq:Rj-M-bound} and
\eqref{eq:backward-representation-aj}, we obtain
\begin{align*}
|\delta a_j(s)|
\lesssim
e^{\mu_js}
\int_s^\infty
e^{-(\mu_j+2\mu)\tau}\,d\tau
\left(
M+\sqrt{\widetilde M}
\right)
\lesssim
e^{-2\mu s}
\left(
M+\sqrt{\widetilde M}
\right).
\end{align*}
Summing over $1\le j\le N$ yields
$|\delta a(s)|
\lesssim
e^{-2\mu s}
\left(
M+\sqrt{\widetilde M}
\right)$.

Multiplying by $e^{\mu s}$ and taking the supremum over
$s\ge s_0$, we find
\begin{equation*}
M
\lesssim
e^{-\mu s_0}
\left(
M+\sqrt{\widetilde M}
\right).
\end{equation*}
For $s_0=s_0(n)$ sufficiently large, the term involving $M$ on
the right-hand side can be absorbed. Thus
\begin{equation}\label{eq:M-Mtilde-relation}
M
\lesssim
e^{-\mu s_0}\sqrt{\widetilde M}.
\end{equation}
Evaluating the preceding estimate for $\delta a(s)$ at $s=s_0$
and using \eqref{eq:M-Mtilde-relation}, we obtain
\begin{equation}\label{eq:delta-a-s0-preliminary}
   |\delta a(s_0)|
\lesssim
e^{-2\mu s_0}\sqrt{\widetilde M}. 
\end{equation}

\emph{$\mathbf{Step\ 4}$. Closing the difference estimates.}
For brevity, we write
\[
E_\delta(s)
:=
\mathcal E_n[\delta\varepsilon(s)].
\]
The difference energy estimate
\eqref{eq:difference-energy-final}, together with
$|\delta a(s)|\le Me^{-\mu s}$, yields
\begin{equation*}
\frac{d}{ds}E_\delta(s)
+
\kappa_nE_\delta(s)
\lesssim
M^2e^{-4\mu s}.
\end{equation*}
By Gronwall's inequality,
\begin{equation*}
E_\delta(s)
\lesssim
e^{-\kappa_n(s-s_0)}E_\delta(s_0)
+
M^2
\int_{s_0}^{s}
e^{-\kappa_n(s-\tau)}
e^{-4\mu\tau}\,d\tau.
\end{equation*}
Since $\tau\ge s_0$ and $2\mu<\kappa_n$,
\begin{align*}
\int_{s_0}^{s}
e^{-\kappa_n(s-\tau)}
e^{-4\mu\tau}\,d\tau
&\le
e^{-2\mu s_0}
\int_{s_0}^{s}
e^{-\kappa_n(s-\tau)}
e^{-2\mu\tau}\,d\tau\lesssim
e^{-2\mu s_0}e^{-2\mu s}.
\end{align*}
Consequently,
\begin{equation*}
E_\delta(s)
\lesssim
e^{-\kappa_n(s-s_0)}E_\delta(s_0)
+
e^{-2\mu s_0}M^2e^{-2\mu s}.
\end{equation*}

Multiplying by $e^{2\mu s}$ and taking the supremum over
$s\ge s_0$, we obtain
\begin{equation}\label{eq:Mtilde-pre-close}
\widetilde M
\lesssim
e^{2\mu s_0}E_\delta(s_0)
+
e^{-2\mu s_0}M^2.
\end{equation}
On the other hand, \eqref{eq:M-Mtilde-relation} implies
\[
M^2
\lesssim
e^{-2\mu s_0}\widetilde M.
\]
Substituting this estimate into
\eqref{eq:Mtilde-pre-close} gives
\begin{equation*}
\widetilde M
\lesssim
e^{2\mu s_0}E_\delta(s_0)
+
e^{-4\mu s_0}\widetilde M.
\end{equation*}
Increasing $s_0(n)$ if necessary, we absorb the last term and
conclude that
\begin{equation}\label{eq:Mtilde-final}
\widetilde M
\lesssim
e^{2\mu s_0}E_\delta(s_0).
\end{equation}
Together with \eqref{eq:M-Mtilde-relation}, this also gives
\begin{equation*}
M
\lesssim
E_\delta(s_0)^{1/2}.
\end{equation*}

Finally, combining \eqref{eq:delta-a-s0-preliminary} with
\eqref{eq:Mtilde-final}, we obtain
\begin{align*}
|\delta a(s_0)|
&\lesssim
e^{-2\mu s_0}\sqrt{\widetilde M}
\lesssim
e^{-\mu s_0}E_\delta(s_0)^{1/2}.
\end{align*}
By 
$E_\delta(s_0)
\asymp
\|\delta\varepsilon(s_0)\|_{L^2}^2$,
we arrive at
\[
|\delta a(s_0)|
\lesssim
e^{-\mu s_0}
\|\delta\varepsilon(s_0)\|_{L^2}.
\]
Since $s_0$ is fixed, this proves that the graph map determining
the unstable coefficients from the stable component is Lipschitz:
\[
|\delta a(s_0)|
\lesssim
\|\delta\varepsilon(s_0)\|_{L^2}.
\]
\end{proof}

We now prove the nonlinear theorem.  The point of fixing the initial
renormalized time below is that Proposition~\ref{prop:Lipschitz-dependence}
compares trapped solutions only when they are written with the same value of
$s_0$.

\begin{proof}[Proof of Theorem~\ref{thm:nonradial-stability}]
Fix a sufficiently large $n$.  Throughout Section~4 we use the abbreviation
$N=N_n$.  Fix the constants $\mu,K_0,K,K'$ in the order
\eqref{eq:bootstrap-mu-uniform}--\eqref{eq:bootstrap-constant-ledger}.
We now choose a number
\[
 s_{0,n}\gg1
\]
large enough so that every lower bound on $s_0$ required in
Proposition~\ref{bootstrap} and in
Proposition~\ref{prop:Lipschitz-dependence} is satisfied.  This choice depends only on the fixed profile $U_n$ and the
constants $\mu,K_0,K,K'$ selected before $s_0$; in particular, it is
independent of the stable initial perturbation.  Set
\begin{equation*}
 s_0:=s_{0,n},\qquad
 \lambda_0:=e^{-s_0/2},\qquad
 \delta_n:=\frac{K_0}{2}e^{-\mu s_0}.
\end{equation*}
After decreasing $\delta_n$ further if necessary, the geometric decomposition
of Lemma~\ref{lemma-geometric-Un} is valid on the corresponding
neighborhood.

Let $v_0\in V_n$ satisfy $\|v_0\|_{\mathcal Y}<\delta_n$ and put
$\varepsilon_0=v_0$.  Since
\[
 \|v_0\|_{H^2}+\|\Lambda v_0\|_2
 =\|v_0\|_{\mathcal Y}<\delta_n
 <K_0e^{-\mu s_0},
\]
the initial bounds in Proposition~\ref{bootstrap} hold with this same fixed
$s_0$.  Hence the Brouwer argument yields at least one
vector
\[
 a_0=(a_{1,0},\ldots,a_{N,0}),\qquad |a_0|\le e^{-\mu s_0},
\]
for which the corresponding solution remains in the bootstrap regime for all
$s\ge s_0$.

If $a_0$ and $\widetilde a_0$ were two such
choices for the same $v_0$, Proposition~\ref{prop:Lipschitz-dependence}
with $\varepsilon_0^{(1)}=\varepsilon_0^{(2)}=v_0$
shows that
$|a_0-\widetilde a_0|=0$. Therefore, this vector is unique. We next define
\begin{equation*}
 c_n(v_0):=a_0.
\end{equation*}
For $v_0=0$, the exact self-similar trajectory corresponds to $a_0=0$ and
remains trapped, the uniqueness therefore gives $c_n(0)=0$.

Let now $v_0^{(1)},v_0^{(2)}\in V_n\cap B_{\mathcal Y}(0,\delta_n)$.  Because
both data are evolved from the same $s_0=s_{0,n}$, the same
$\lambda_0=e^{-s_0/2}$ and the same initial center $x(0)=0$,
Proposition~\ref{prop:Lipschitz-dependence} applies
directly and yields
\[
 \bigl|c_n(v_0^{(1)})-c_n(v_0^{(2)})\bigr|
 \le C_n\|v_0^{(1)}-v_0^{(2)}\|_{H^2}
 \le C_n\|v_0^{(1)}-v_0^{(2)}\|_{\mathcal Y}.
\]
Thus $c_n$ is a well-defined Lipschitz map on the whole fixed ball
$V_n\cap B_{\mathcal Y}(0,\delta_n)$.  This is precisely where the
uniform choice of $s_0$ is needed.

Proposition~\ref{bootstrap} is written for the initial datum
\[
 \widehat u_0(x)=\lambda_0^{-2}
 \left(U_n+v_0+\sum_{j=1}^Nc_j(v_0)\phi_j\right)(x/\lambda_0).
\]
If $\widehat u$ denotes the corresponding solution, define
\[
 u(t,x)=\lambda_0^2\widehat u(\lambda_0^2t,\lambda_0x).
\]
The scaling invariance of \eqref{eq:KS} shows that $u$ is a solution and that
\[
 u(0,x)=U_n(x)+v_0(x)+\sum_{j=1}^Nc_j(v_0)\phi_j(x),
\]
which is \eqref{eq:corrected-initial-data}.  If $\widehat u$ blows up at
$\widehat T$, then $u$ blows up at $T=\widehat T/\lambda_0^2$; type~I and
convergence in the rescaled variables are unchanged by this fixed scaling.
It is therefore enough to prove the asserted conclusions for $\widehat u$.
The
corresponding solution of
\eqref{eq:KS}
admits the decomposition \eqref{re} and satisfies
\eqref{1.7}--\eqref{1.9}
for all $s\ge s_0$.\\
\emph{$\mathbf{Step\ 1}$. Finite time blow-up and blow-up point.}
Equation \eqref{renormalized-time} gives $s_t=\lambda^{-2}(t)$ and hence
\[
 T=\int_{s_0}^{\infty}\lambda^2(s)\,ds
 <\int_{s_0}^{\infty}e^{-2\mu s}\,ds<\infty.
\]
Moreover, \eqref{control-sca} gives
\[
 \frac14e^{-s}\le T-t=\int_s^\infty\lambda^2(\sigma)\,d\sigma
 \le4e^{-s}.
\]
Thus $\lambda(t)\to0$ as $t\to T$.  Since
$\lambda_t=\lambda_s/\lambda^2$, the modulation estimate \eqref{semin}
implies
\[
 \left|\lambda\lambda_t+\frac12\right|
 =\left|\frac{\lambda_s}{\lambda}+\frac12\right|
 \lesssim e^{-2\mu s}\lesssim(T-t)^{2\mu}.
\]
Equivalently,
\[
 \left|\frac{d}{dt}\lambda^2+1\right|
 \lesssim(T-t)^{2\mu}.
\]
Integrating from $t$ to $T$ yields
\[
 \lambda^2(t)=T-t+O\bigl((T-t)^{1+2\mu}\bigr),
\]
and therefore
\begin{equation}\label{time0s}
\lambda(t)=\sqrt{T-t}\bigl(1+O((T-t)^{2\mu})\bigr),
\qquad t\to T.
\end{equation}
Finally, \eqref{semin} and \eqref{poaf} imply
\[
 \int_0^T|x_t|\,dt
 =\int_{s_0}^\infty|x_s|\,ds
 \lesssim\int_{s_0}^\infty e^{-3\mu s}\,ds<\infty.
\]
Thus
the second limit in \eqref{eq:nonlinear-convergence}
is proved.

\emph{$\mathbf{Step\ 2}$. Convergence at the self-similar scale.}
By \eqref{dec}, we have
 $v=\varepsilon+\sum_{j=1}^{N}a_j\phi_j$.
The finite-dimensional norm equivalence, \eqref{chca}, and \eqref{1.9} give
\[
 \|v(s)\|_{H^2}\lesssim e^{-\mu s}\lesssim (T-t)^{\mu}.
\]
Therefore,
\begin{equation}\label{v-as}
\lim_{t\to T}||v(t)||_{H^2}=0.
\end{equation}

For $\rho>0$, define the dilation operator
\[
\mathcal D_\rho f(y)
:=
\rho^2 f(\rho y).
\]
Set
$\rho(t)
:=
\frac{\sqrt{T-t}}{\lambda(t)}$.
By \eqref{time0s}, we have $\lim_{t\to T}\rho(t)=1$.

Recall that
\[
u(t,x)
=
\frac{1}{\lambda(t)^2}
\left[
U_n+v(s)
\right]
\left(
\frac{x-x(t)}{\lambda(t)}
\right).
\]
By
$y=\frac{x-x(t)}{\sqrt{T-t}}$,
we obtain
\[
u(t,x)
=
\frac{1}{T-t}
\left[
\mathcal D_{\rho(t)}U_n(y)
+
\mathcal D_{\rho(t)}v(s,y)
\right].
\]
Hence, defining
\[
\widetilde u(t)
:=
\mathcal D_{\rho(t)}U_n-U_n
+
\mathcal D_{\rho(t)}v(s),
\]
we have
\[
u(t,x)
=
\frac{1}{T-t}
\left[
U_n\left(
\frac{x-x(t)}{\sqrt{T-t}}
\right)
+
\widetilde u\left(
t,\frac{x-x(t)}{\sqrt{T-t}}
\right)
\right].
\]

It remains to prove that
\[
\lim_{t\to T}\|\widetilde u(t)\|_{H^2}= 0.
\]
Since the dilation map
$\rho\mapsto \mathcal D_\rho f$
is strongly continuous on $H^2(\mathbb R^3)$ for every
$f\in H^2(\mathbb R^3)$, and since $\rho(t)\to1$, we have
\begin{equation}\label{dil-map}
   \lim_{t\to T}\|\mathcal D_{\rho(t)}U_n-U_n\|_{H^2}
= 0. 
\end{equation}
Moreover, for $\rho$ in any compact subset of $(0,\infty)$, the dilation
operators are uniformly bounded on $H^2(\mathbb R^3)$. In particular,
for $t$ sufficiently close to $T$, we have
$\frac12\le \rho(t)\le2$,
and hence
\begin{equation}\label{maD}
    \|\mathcal D_{\rho(t)}f\|_{H^2}
\lesssim
\|f\|_{H^2}.
\end{equation}
Consequently, by \eqref{v-as}, \eqref{dil-map} and \eqref{maD}
\[
\|\widetilde u(t)\|_{H^2}
\le
\|\mathcal D_{\rho(t)}U_n-U_n\|_{H^2}
+
\|\mathcal D_{\rho(t)}v(s)\|_{H^2}
\to0
\qquad\text{as } t\to T.
\]
This proves
the first limit in \eqref{eq:nonlinear-convergence}.

From $H^2(\mathbb R^3)\hookrightarrow L^\infty(\mathbb R^3)$ and
$U_n(0)>0$, we have
\[
 u(t,x(t))
 =\frac1{T-t}\bigl(U_n(0)+\widetilde u(t,0)\bigr)
 =\frac{U_n(0)+o(1)}{T-t}\longrightarrow+\infty.
\]
The convergence of the centers can be sharpened using the modulation
estimate.  From \eqref{semin} and \eqref{time0s}, together with
$T-t\asymp e^{-s}$, we have
\[
 |x_s(s)|\lesssim\lambda(s)e^{-2\mu s}
             \lesssim e^{-(\frac12+2\mu)s}.
\]
Consequently,
\begin{equation}\label{eq:center-sharp-convergence}
 |x_*-x(s)|
 \le\int_s^\infty|x_\sigma(\sigma)|\,d\sigma
 \lesssim e^{-(\frac12+2\mu)s}
 \lesssim (T-t)^{\frac12+2\mu}.
\end{equation}
In particular,
\[
 \frac{x_*-x(t)}{\sqrt{T-t}}\longrightarrow0.
\]
Using the continuity of $U_n$ at zero and
$\|\widetilde u(t)\|_{L^\infty}\to0$, we may now evaluate the
self-similar decomposition at the \emph{fixed} spatial point $x_*$:
\[
 u(t,x_*)
 =\frac1{T-t}\left[
 U_n\left(\frac{x_*-x(t)}{\sqrt{T-t}}\right)
 +\widetilde u\left(t,\frac{x_*-x(t)}{\sqrt{T-t}}\right)
 \right]
 =\frac{U_n(0)+o(1)}{T-t}\longrightarrow+\infty.
\]
Thus $x_*$ is indeed a blow-up point.  Moreover,
\[
 (T-t)\|u(t)\|_{L^\infty}
 \le \|U_n\|_{L^\infty}+\|\widetilde u(t)\|_{L^\infty}
 =\|U_n\|_{L^\infty}+o(1).
\]
Thus the blow-up rate is of type~I.
This  completes the proof.
\end{proof}

\subsection{Removing the geometric gauge: a full-space Lipschitz graph}
\label{subsec:full-geometry-chart}

\begin{proof}[Proof of Corollary~\ref{cor:full-geometry-stable}]
We keep $n$ fixed and suppress its index.  Let $G$ denote the
four-dimensional geometric span, $E=\operatorname{span}\{\phi_j\}_{j=1}^N$,
and $V=V_n$.  Define bounded finite-rank projections on both $L^2$
and $\mathcal Y$ by
\begin{align*}
 Q_{\rm g}f&=\langle f,\zeta_0\rangle\Lambda U_n
        +\sum_{k=1}^3\langle f,\zeta_k^{\rm tr}\rangle\partial_kU_n,\\
 Q_{\rm r}f&=\sum_{j=1}^N\langle f,\zeta_j^{\rm u}\rangle\phi_j,
 \qquad Q_{\rm s}=I-Q_{\rm g}-Q_{\rm r}.
\end{align*}
The biorthogonality relations give
\begin{equation}\label{eq:full-geometry-splitting}
 \mathcal Y=V\oplus G\oplus E,
 \qquad Q_{\rm s}\mathcal Y=V,\quad Q_{\rm g}\mathcal Y=G,
 \quad Q_{\rm r}\mathcal Y=E.
\end{equation}
We write $\Phi(a)=\sum_{j=1}^N a_j\phi_j$ for $a\in\mathbb R^N$.

\smallskip\noindent
\emph{Step 1: A geometric gauge which is $C^1$ in $L^2$.}
For $\theta=(\sigma,z)\in\mathbb R\times\mathbb R^3$ put
\[
 \mathscr G_\theta f(y)=e^{2\sigma}f(e^\sigma y+z).
\]
For a geometric adjoint mode $\zeta$, changing variables gives
\begin{equation}\label{eq:full-geometry-dual-action}
 \langle\mathscr G_\theta f,\zeta\rangle
 =e^{-\sigma}\int_{\mathbb R^3}
       f(x)\,\zeta(e^{-\sigma}(x-z))\,dx.
\end{equation}
Proposition~\ref{prop:adjointmodeasymptotics} ensures
$\zeta,\nabla\zeta,y\cdot\nabla\zeta\in L^2$.
Hence the right side of \eqref{eq:full-geometry-dual-action}
is $C^1$ in $(\theta,f)\in\mathbb R^4\times L^2$, including
when $f$ itself has no derivatives.  Define
\[
 \mathcal H(\theta,w)
 :=\bigl(\langle\mathscr G_\theta(U_n+w)-U_n,\zeta_0\rangle,
      (\langle\mathscr G_\theta(U_n+w)-U_n,
                         \zeta_k^{\rm tr}\rangle)_{k=1}^3\bigr).
\]
At $(0,0)$ we have $\mathcal H=0$ and, by the dual normalization,
\[
 D_\theta\mathcal H(0,0)=I_4,
 \qquad D_w\mathcal H(0,0)h
       =\bigl(\langle h,\zeta_0\rangle,
               (\langle h,\zeta_k^{\rm tr}\rangle)_{k=1}^3\bigr).
\]
The finite-dimensional implicit function theorem therefore gives a
$C^1$ map $\theta=\theta(w)$ on an $L^2$ neighborhood of $0$, with
$\theta(0)=0$, $|\theta(w)|\le C\|w\|_2$, and
\begin{equation}\label{eq:full-geometry-theta-derivative}
 D\theta(0)h
 =-\bigl(\langle h,\zeta_0\rangle,
          (\langle h,\zeta_k^{\rm tr}\rangle)_{k=1}^3\bigr).
\end{equation}
Set
\begin{equation}\label{eq:full-geometry-normalized-data}
 h(w):=\mathscr G_{\theta(w)}(U_n+w)-U_n.
\end{equation}
Then $Q_{\rm g}h(w)=0$.

By Lemma~\ref{lem:Y-similarity-action}, the group acts
jointly continuously and locally boundedly on $\mathcal Y$.
The right geometric modes belong to $\mathcal Y$, hence
$\|\mathscr G_\theta U_n-U_n\|_{\mathcal Y}\le C|\theta|$
for small $\theta$.  Thus
\begin{equation}\label{eq:full-geometry-h-small}
 \|h(w)\|_{\mathcal Y}\le C\|w\|_{\mathcal Y},
 \qquad \|w\|_{\mathcal Y}\text{ sufficiently small}.
\end{equation}
Notice that this estimate is based on the smooth 
\emph{reference profile} $U_n$ and uniform boundedness of the
action; it does not assert a parameter-Lipschitz bound for the action
on an arbitrary element of $\mathcal Y$.

We also need differentiability in the weaker $L^2$ target space.
For $f\in\mathcal Y$ the genuine $L^2$ derivatives are
\begin{align}
 \partial_\sigma\mathscr G_{\sigma,z}f
 &=\mathscr G_{\sigma,z}(\Lambda f-z\cdot\nabla f),
 &\partial_{z_k}\mathscr G_{\sigma,z}f
 &=\mathscr G_{\sigma,z}(\partial_k f).
 \label{eq:full-geometry-action-weak-C1}
\end{align}
The right sides depend continuously on $(\sigma,z,f)$ with
$f\in\mathcal Y$, while
$\|\mathscr G_\theta-\mathscr G_{\theta'}\|_{\mathcal Y\to L^2}
\le C|\theta-\theta'|$ locally.
Consequently $h:\mathcal Y\to L^2$ is $C^1$, with
\begin{equation}\label{eq:full-geometry-h-derivative}
 Dh(0)f=f-Q_{\rm g}f.
\end{equation}
This is the weaker differentiability required in what follows;
no $C^1$ assertion with target $\mathcal Y$ is used.  In fact,
if $f\in C_c^0\cap H^2$ but $\partial_1f\notin H^2$, then
$f\in\mathcal Y$, whereas its translation orbit cannot be
Lipschitz in $\mathcal Y$: otherwise the difference quotients
would be bounded in $H^2$ and their distributional limit
$\partial_1f$ would belong to $H^2$.

\smallskip\noindent
\emph{Step 2: A scalar Lipschitz equation for the nongeometric
unstable coefficients.}
For $w$ in a sufficiently small $\mathcal Y$ ball define
\[
 v(w)=Q_{\rm s}h(w)\in V,
 \qquad b_j(w)=\langle h(w),\zeta_j^{\rm u}\rangle,
 \qquad b(w)=(b_j(w))_{j=1}^N.
\]
By \eqref{eq:full-geometry-h-small}, $v(w)$ lies in the domain of
the fixed-gauge graph map $c_n$.  Its difference estimate established
in Proposition~\ref{prop:Lipschitz-dependence} and the proof of
Theorem~\ref{thm:nonradial-stability} is, in fact,
\begin{equation}\label{eq:full-geometry-L2-graph}
 |c_n(v_1)-c_n(v_2)|\le L_n\|v_1-v_2\|_2
\end{equation}
for all admissible $v_1,v_2\in V$ in that ball.
Define the finite-dimensional constraint
\begin{equation}\label{eq:full-geometry-constraint}
 \mathcal F(w):=b(w)-c_n(v(w))\in\mathbb R^N.
\end{equation}
The map $h:\mathcal Y\to L^2$ is locally Lipschitz, so the
boundedness of $Q_{\rm s},Q_{\rm r}$ on $L^2$ and
\eqref{eq:full-geometry-L2-graph} imply that
$\mathcal F:\mathcal Y\to\mathbb R^N$ is locally Lipschitz.
Moreover, by construction,
\begin{equation}\label{eq:full-geometry-orbit-equivalence}
 \mathcal F(w)=0
 \quad\Longleftrightarrow\quad
 \mathscr G_{\theta(w)}(U_n+w)
   =U_n+v(w)+\Phi(c_n(v(w))).
\end{equation}
Thus the zero set is exactly the symmetry-saturated fixed-gauge
stable data, within a sufficiently small neighborhood of $U_n$.

\smallskip\noindent
\emph{Step 3: The zero set is a Lipschitz graph over $V\oplus G$.}
Let $Z=V\oplus G=\ker Q_{\rm r}$.  By
\eqref{eq:full-geometry-h-derivative} and the projection
relations, for every $a\in\mathbb R^N$,
\begin{equation}\label{eq:full-geometry-transversality}
 Db(0)\Phi(a)=a,
 \qquad Dv(0)\Phi(a)=0\quad\text{in }L^2.
\end{equation}
Since $h:\mathcal Y\to L^2$ is $C^1$, for sufficiently small
$z\in Z$ and $a,a'\in\mathbb R^N$, the mean-value formula yields
\begin{equation}\label{eq:full-geometry-small-directional}
 \begin{aligned}
 &\big|b(z+\Phi(a))-b(z+\Phi(a'))-(a-a')\big|\\
 &\qquad+\|v(z+\Phi(a))-v(z+\Phi(a'))\|_2
 \le \omega_n(\rho)|a-a'|,
 \end{aligned}
\end{equation}
provided all arguments are in a $\mathcal Y$ ball of radius $\rho$,
where $\omega_n(\rho)\to0$ as $\rho\to0$.
The normalization of the finitely many $\phi_j$ is absorbed in
$\omega_n$.  Combining \eqref{eq:full-geometry-small-directional}
with \eqref{eq:full-geometry-L2-graph}, we obtain
\begin{equation}\label{eq:full-geometry-contraction}
 \big|\mathcal F(z+\Phi(a))-\mathcal F(z+\Phi(a'))
                   -(a-a')\big|
 \le (1+L_n)\omega_n(\rho)|a-a'|.
\end{equation}
Shrink $\rho$ so that the coefficient is at most $1/2$.
For fixed small $z\in Z$, the map
\[
 a\longmapsto a-\mathcal F(z+\Phi(a))
\]
is therefore a contraction on a sufficiently small, fixed
coefficient ball: indeed,
$|\mathcal F(z)|\le C\|z\|_{\mathcal Y}$.
For $\|z\|_{\mathcal Y}\le r$, choose the coefficient ball
$|a|\le 2Cr$ and then take $r>0$ sufficiently small that
$\|z+\Phi(a)\|_{\mathcal Y}\le\rho$ on this ball.
The contraction estimate yields
$|a-\mathcal F(z+\Phi(a))|\le Cr+|a|/2\le2Cr$;
thus it maps the ball into itself.  Denote its unique fixed point
by $a=\widehat c_n(z)$.  The local Lipschitz bound
for $\mathcal F$ and the contraction constant $1/2$ give
\[
 |\widehat c_n(z)-\widehat c_n(z')|
 \le C_n\|z-z'\|_{\mathcal Y},
 \qquad \widehat c_n(0)=0.
\]
After restricting the neighborhood if needed,
\begin{equation}\label{eq:full-geometry-final-graph}
 \begin{aligned}
 &\bigl\{w=z+\Phi(a):
       \|z\|_{\mathcal Y}<r,\ |a|\le2Cr,\ \mathcal F(w)=0\bigr\}\\
 &\qquad=\bigl\{z+\Phi(\widehat c_n(z)):
                 z\in Z,\ \|z\|_{\mathcal Y}<r\bigr\}.
 \end{aligned}
\end{equation}
The graph lies inside the coefficient ball after reducing $r$ if
necessary; thus this identity is an equality in an explicit local
product chart near $0$.
This is a Lipschitz graph of codimension $N$ in the entire space
$\mathcal Y$.  If $z=v\in V$ is sufficiently small, the fixed-gauge
data $U_n+v+\Phi(c_n(v))$ already have vanishing geometric moments,
so $\theta(v+\Phi(c_n(v)))=0$ and
\eqref{eq:full-geometry-orbit-equivalence} holds.  Uniqueness in
the contraction proves $\widehat c_n(v)=c_n(v)$.

Finally, if $u_0=U_n+w$ belongs to the graph,
\eqref{eq:full-geometry-orbit-equivalence} reduces its initial datum
by a positive dilation and a translation to the fixed-gauge datum
of Theorem~\ref{thm:nonradial-stability}.  The scaling and
translation invariance of the Keller--Segel flow, already stated
after \eqref{eq:KS}, carry the type~I blow-up and rescaled $H^2$
convergence back to $u_0$.  This proves all assertions.
\end{proof}

\appendix
\section{Auxiliary linear estimates}
\label{app:auxiliary}

This appendix supplies the auxiliary arguments used in the spectral
analysis. Throughout, $\mathcal A_n$ and $\rho_n$ denote the reduced-mass
operator and its weight, as in \eqref{local-operator} and
\eqref{us:mass-weight}; $L_n$ acts on the unweighted density space.

\subsection{Quantitative bounds for the universal core}
Introduce
\[
 y=s^2\Q(s),\qquad \alpha=-\frac{s\Q'(s)}{\Q(s)},\qquad
 \zeta=\frac\alpha y,\qquad \beta=3-\alpha,\qquad t=\log s.
\]
A dot in this subsection denotes differentiation in $t$.

\begin{lemma}[Core phase-plane bounds]\label{us:phase-bounds}
The functions $\Q,P,Q$ are positive, and
$-\Q'(s)\ge cs/(1+s)^4$ for $s>0$. Their phase-plane orbit satisfies
\begin{equation}\label{us:coarse-region}
 0<y<\frac53,\qquad 0<\beta<3-\frac65y,\qquad
 \zeta>\frac65+\frac y{20}.
\end{equation}
If $t_M$ is the first maximum of $y$, then
\[
 y(t_M)<\frac{21}{16},\qquad
 \frac7{10}<y(t)<\frac43,\quad
 \frac32<\alpha(t)<\frac{59}{25}\quad(t\ge t_M).
\]
The coefficient
\begin{equation}\label{us:core-potential-coefficient}
 C_0(y,\alpha)=2+y\left(19-\alpha-\frac{48}{\alpha}\right)
 +y^2\left(1-\frac{24}{\alpha}+\frac{72}{\alpha^2}\right)
\end{equation}
satisfies $C_0(y(s),\alpha(s))\ge1/20$ for all $s>0$.
\end{lemma}

\begin{proof}
The identities $Q'=-PQ$ and $P'+2P/s=Q$, together with $Q(0)=1$,
give positivity of $Q,P$ and $\Q=P/(2s)$. Direct calculation gives
\begin{equation}\label{us:autonomous}
 \dot y=y(2-\alpha)=y(\beta-1),\quad
 \dot\alpha=(3-\alpha)(2y-\alpha),\quad
 \dot\beta=\beta(3-\beta-2y),\quad
 \dot\zeta=6-5\zeta+2y\zeta(\zeta-1).
\end{equation}
Here $\beta=Q/(2\Q)>0$, and the origin expansion gives
$\zeta=6/5+(12/175)y+O(y^2)$.
Let $F(y,\zeta)=y(2-y\zeta)$, $G(y,\zeta)=6-5\zeta+2y\zeta(\zeta-1)$,
and $\ell(y)=6/5+y/20$. On this boundary,
\[
 G(y,\ell)-\ell'F(y,\ell)=\frac{y(26+y)(2+3y)}{400}>0.
\]
The first-contact argument gives $\zeta>\ell(y)$, hence $\alpha>0$
and $\Q'<0$. The endpoint expansions and compactness on intermediate
intervals give the stated lower bound for $-\Q'$.

For $\varphi(v)=v-1-\log v$, one has
\[
 H(y,\beta)=2\varphi(y)+\varphi(\beta),\qquad
 \dot H=-(\beta-1)^2.
\]
The nonzero oscillatory tail of $\Q$ ensures a first maximum of $y$.
Before it, $\beta>1$; at it, $\beta=1$, $\alpha=2$, and $y_M>1$.
The straight barrier implies $y_M<5/3$. Afterwards the Lyapunov
inequality gives $y\le y_M$ whenever $y\ge1$. This proves
\eqref{us:coarse-region}.

For a sharper bound use the polynomial barrier
\[
 p(y)=\frac65+\frac{6857y+3047y^2+1568y^3+884y^4+529y^5
                 +331y^6+213y^7+141y^8}{100000}.
\]
Writing $p=\sum c_jy^j$, with $c_j=0$ for $j>8$, the coefficient
of $y^k$ in $G(y,p)-p'F(y,p)$ is
\[
 r_k=\sum_{i+j=k-1}(2+j)c_ic_j-2c_{k-1}-(5+2k)c_k.
\]
The constant term vanishes; for $1\le k\le8$ the numerators over
$10^{10}$ are
\[
 100000,\ 500000,\ 695347,\ 572953,\ 1491044,\ 373556,\ 1802130,\ 404062.
\]
For $k=9$, the terms involving $c_0c_8$ minus $2c_8$ are positive
because $12c_0-2>0$; for $10\le k\le17$ all terms are nonnegative.
Thus $\zeta>p(y)$ by first contact and $12/175>6857/100000$.
In particular,
\[
 p(6/5)-\frac{176}{125}=\frac{6410173}{19531250000}>0.
\]
The equation for $y$ gives
\[
 E=\tfrac12\dot y^2+V(y),\quad V(y)=\tfrac23y^3-y^2,
 \qquad \dot E=-(2y-1)\dot y^2.
\]
If $y_M\ge6/5$, then on the increasing branch at $y=6/5$,
\[
 E\le\frac12\left[\frac65\left(2-\frac65\frac{176}{125}\right)\right]^2
       +V(6/5)=-\frac{2135052}{9765625}<-\frac{109}{500}.
\]
Since $V$ increases for $y>1$ and
$V(21/16)=-441/2048>-109/500$, energy monotonicity gives
$y_M<21/16$; the case $y_M<6/5$ is immediate.
After that maximum $H\le2\varphi(y_M)<2\varphi(21/16)<1/12$.
The inequalities
\[
 2\varphi(7/10)\ge9/100>1/12,\quad
 \varphi(16/25)>1/12,\quad \varphi(3/2)\ge89/960>1/12
\]
give $y>7/10$ and $16/25<\beta<3/2$, and hence the stated rectangle.
These comparisons follow from the alternating series for
$\log(1+x)$ and the positive series for $-\log(1-x)$.

Before the first maximum, while $y\le1$, the upper barrier
$v(y)=6/5+7y/100+y^2/10$ satisfies
\[
 v'F(y,v)-G(y,v)=\frac y{10000}
 (100+6200y-5347y^2-490y^3-400y^4)>0.
\]
The polynomial in parentheses has Bernstein coefficients
$100,1650,13853/6,1954,63$ on $[0,1]$, so the sign is exact.
It follows that $6/5<\zeta<v(y)\le137/100$. In these variables,
\begin{equation}\label{us:c-zeta}
 C_0(y,y\zeta)=2+\frac{24(3-2\zeta)}{\zeta^2}
              +y\left(19-\frac{24}{\zeta}\right)+y^2(1-\zeta).
\end{equation}
The successive terms give $C_0\ge2-y-37y^2/100\ge63/100$.
On the rest of the increasing branch, $y\ge1$ and $\alpha\le2$;
\[
 \partial_\alpha C_0
 =\frac y{\alpha^3}\{-\alpha^3+48\alpha+y(24\alpha-144)\}<0.
\]
Thus $C_0(y,\alpha)\ge C_0(y,2)=2-7y+7y^2\ge2$.

On the rectangle after the first maximum the coefficient of $y^2$
is positive, and
\[
 \partial_y C_0(7/10,\alpha)
 =\frac{-5\alpha^3+102\alpha^2-408\alpha+504}{5\alpha^2}>0.
\]
The numerator decreases on $[3/2,59/25]$ and has endpoint value
$135931/3125>0$. Hence it suffices to consider $y=7/10$, where
\[
 C_0(7/10,\alpha)-\frac1{20}
 =\frac{1764-2268\alpha+787\alpha^2-35\alpha^3}{50\alpha^2}.
\]
The numerator has the following positive Bernstein coefficients:
\[
\begin{array}{c|rrrr}
{}[3/2,9/5]&117/8&3/10&243/50&684/25\\
{}[9/5,21/10]&684/25&2493/50&903/10&29547/200\\
{}[21/10,59/25]&29547/200&24689/125&975203/3750&1046082/3125
\end{array}
\]
Here the Bernstein coefficients $b_j$ mean
$P(a+(b-a)\theta)=\sum_{j=0}^m b_j\binom mj\theta^j(1-\theta)^{m-j}$.
Their positivity proves the remaining bound.
\end{proof}

\subsection{Spectral duality, graph cores, and semigroup generation}

\subsubsection{Adjoint modes and projections}
\label{app:proof-adjointdualexistence}
\begin{proof}[Proof of Lemma~\ref{lem:adjointdualexistence}]
Let $\Pi_\nu=(2\pi i)^{-1}\int_{\Gamma_\nu}(z-L_n)^{-1}\,dz$
be the Riesz projection at the real isolated eigenvalue $-\nu$.
Its adjoint is the corresponding projection for $L_n^*$ and has the
same finite rank. If $f\in\operatorname{Ran}\Pi_\nu$ is orthogonal
to $\operatorname{Ran}\Pi_\nu^*$, then
$\langle f,g\rangle=\langle f,\Pi_\nu^*g\rangle=0$ for every $g$.
Thus the pairing is nondegenerate. Semisimplicity identifies these
ranges with the eigenspaces. Choosing dual bases proves the normalization.
Different eigenvalues have zero cross pairings, as do different angular
sectors. Rotational invariance identifies the translation adjoint space
with the degree-one representation, giving a common radial coefficient
for its three coordinate directions.
\end{proof}

\label{app:proof-stablespaceprojection}
\begin{proof}[Proof of Lemma~\ref{lem:stablespaceprojection}]
Each term in \eqref{eq:unstableprojection} is bounded and rank one.
Biorthogonality gives $P_{{\rm u},n}^2=P_{{\rm u},n}$ and the stated
description of its kernel. For a paired eigenfunction and adjoint
eigenfunction $e,\zeta$ at $-\nu$,
\[
 \langle L_nf,\zeta\rangle=-\nu\langle f,\zeta\rangle,
 \qquad L_n(\langle f,\zeta\rangle e)=-\nu\langle f,\zeta\rangle e.
\]
Summation proves commutation on $\mathscr D$.
\end{proof}

\subsubsection{A graph core and fixed angular modes}
\label{app:proof-fixedmodefredholm}
\begin{proof}[Proof of Lemma~\ref{us:fixed-mode-fredholm}]
First, $C_c^\infty(\mathbb R^3)$ is a core for $\mathscr D$.
For $f\in\mathscr D$, multiplication by $\chi(x/R)$ converges to $f$
both in $H^2$ and in the norm of $x\cdot\nabla f$; the additional term
$(x\cdot\nabla\chi(x/R))f$ tends to zero by the $L^2$ tail of $f$.
For compactly supported $f$, mollification converges in the same norm.
Indeed the commutator between $x\cdot\nabla$ and a mollifier is convolution
with $\operatorname{div}(x\eta_\epsilon)$; its integral is zero and its
$L^1$ norm is bounded, so it converges strongly to zero on $L^2$.
The graph-norm equivalence in \eqref{us:free-graph} gives the assertion
for $\mathbf L_0$, and relative bound zero gives it for $L_n$.

The angular projection $\Pi_{l,m}$ commutes with $\mathbf L_0$ and $L_n$.
Projecting the above core approximations therefore gives a graph core
in each fixed mode, consisting of smooth compactly supported functions
regular at the origin. The spherical harmonic identities on this core
give \eqref{eq:fixedmodeidentification} and
\eqref{eq:fixedmodefreeidentification}; graph closure proves the domain
identities. Normalization of $Y_{l,m}$ proves unitarity.

By the free resolvent result established in the proof of
Lemma~\ref{us:fredholm}, $\mathbf L_0-z$ is invertible for
$\operatorname{Re}z<1/4$.
Since $\Pi_{l,m}$ preserves $\mathscr D$ and commutes with
$\mathbf L_0$, it also commutes with $(\mathbf L_0-z)^{-1}$.
Consequently, $L_{0,l}-z$ is invertible, with
$(L_{0,l}-z)^{-1}
 =\mathcal U_{l,m}^{-1}
 \left((\mathbf L_0-z)^{-1}|_{\mathscr X_{l,m}}\right)
 \mathcal U_{l,m}$.
Restricting the compact resolvent perturbation
in Lemma~\ref{us:fredholm} gives
\[
 L_{n,l}-z=[I+(L_{n,l}-L_{0,l})(L_{0,l}-z)^{-1}](L_{0,l}-z),
\]
where $(L_{n,l}-L_{0,l})(L_{0,l}-z)^{-1}$ is compact. The index is zero by Atkinson's
theorem. A trivial kernel then implies surjectivity and a bounded inverse,
proving \eqref{eq:fixedmodespectrumpoint}.
\end{proof}

\subsubsection{Generation and finite-rank separation}
\label{app:proof-compactgenerator}
\begin{proof}[Proof of Lemma~\ref{lem:compactgenerator}]
Relative compactness was proved in Lemma~\ref{us:fredholm}.
Integration by parts, using $\nabla\cdot b_n=U_n$, gives
\begin{equation}\label{eq:L2numericalform}
 \operatorname{Re}\langle L_nf,f\rangle
 =\|\nabla f\|_2^2+\frac14\|f\|_2^2
 -\frac32\int U_n|f|^2
 -\operatorname{Re}\langle\nabla U_n\cdot\nabla\Delta^{-1}f,f\rangle.
\end{equation}
The last two forms are compact relative to $H^1$: for multiplication by
$U_n$ use local Rellich compactness and $U_n\to0$ at infinity; the
Newton term is a compact $L^2$ operator, as shown in
Lemma~\ref{us:fredholm}. Approximate the compact form on $H^1$ by
finite-rank forms. Its finite collection of $H^{-1}$ functionals can in
turn be approximated by pairings with $C_c^\infty$ functions. Absorbing
the approximation errors into $\|\nabla f\|_2^2+\frac14\|f\|_2^2$
gives an $L^2$-orthonormal finite family $q_j\in C_c^\infty$ and $C_n$ with
\[
 \operatorname{Re}\langle L_nf,f\rangle
 \ge\frac18\|f\|_2^2-C_n\sum_j|\langle f,q_j\rangle|^2.
\]
Set $K_{0,n}=C_n\sum_j\langle\cdot,q_j\rangle q_j$ and
$A_{0,n}=-L_n+1/16-K_{0,n}$. Then $A_{0,n}$ is dissipative.
The relative-bound-zero estimate in Lemma~\ref{us:fredholm} and the
free resolvent show that $\lambda-A_{0,n}$ is surjective for large real
$\lambda$ by a Neumann series. Hence $A_{0,n}$ is maximal dissipative.
Lumer--Phillips and bounded perturbation prove generation and
\eqref{eq:maxdissdecomp}.
\end{proof}

\subsection{The free essential spectrum and radial-variable transfer}

\subsubsection{Essential spectrum}
\label{app:proof-exactessentialthreshold}
\begin{proof}[Proof of Lemma~\ref{lem:exactessentialthreshold}]
Fourier transformation gives
\[
 \widehat{S_0(s)f}(\xi)
 =e^{s/2}e^{-(e^s-1)|\xi|^2}\widehat f(e^{s/2}\xi),\qquad
 \|S_0(s)f\|_2^2
 =e^{-s/2}\int e^{-2(1-e^{-s})|\eta|^2}|\widehat f(\eta)|^2\,d\eta.
\]
Fourier supports in disjoint annuli shrinking to zero give a weakly null
orthonormal sequence realizing the norm $e^{-s/4}$. This also proves
the same lower bound for the essential norm.

Under the unitary change $\xi=e^\rho\theta$,
$v(\rho,\theta)=e^{3\rho/2}\widehat f(e^\rho\theta)$, the generator becomes
\[
 B=-\tfrac12\partial_\rho+\tfrac14+e^{2\rho}
 \quad\hbox{on }L^2(\mathbb R\times\mathbb S^2).
\]
Its resolvent exists for $\operatorname{Re}z<1/4$ by the free semigroup.
For $\operatorname{Re}z>1/4$,
\[
 w_z(\rho,\theta)=\exp\{2(\overline z-1/4)\rho-e^{2\rho}\}h(\theta)
 \in\ker(B^*-\overline z)
\]
for every $h\in L^2(\mathbb S^2)$; the cokernel is infinite dimensional.
At $z=1/4+i\tau$, normalized disjoint cutoffs
\[
 v_j=c_j\chi((\rho+R_j)/L_j)e^{-2i\tau\rho}h(\theta),
 \qquad R_j\gg L_j\longrightarrow\infty,
\]
satisfy $v_j\rightharpoonup0$ and $\|(B-z)v_j\|_2\to0$.
This proves the free Fredholm essential spectrum, including its boundary.
Choosing $h=Y_{l,m}$ also gives a boundary sequence in each fixed mode.
We do not infer a fixed-mode essential half-plane from the infinite
angular multiplicity argument.

Relative compactness transfers the full-space essential spectrum to
$L_n$. Moreover $\mathbf K_nS_0(\tau)$ is compact for $\tau>0$ and
\[
 \|\mathbf K_nS_0(\tau)\|_{2\to2}\le C_{n,s}(1+\tau^{-1/2})
 \quad(0<\tau\le s).
\]
This follows from heat smoothing, local Rellich compactness, coefficient
decay and the $L^2\to L^6$ Newton estimate. Duhamel's formula gives
\[
 S_n(s)-S_0(s)=-\int_0^s S_n(s-\tau)\mathbf K_nS_0(\tau)\,d\tau.
\]
On $[\epsilon,s]$ the integrand is norm continuous and compact (strong
semigroup continuity is uniform on compact sets); the displayed bound
allows $\epsilon\to0$ in operator norm. The difference is compact.
Thus the essential norms coincide, proving \eqref{eq:exactessentialnorm}.
\end{proof}

\subsubsection{Transfer between density and reduced mass}
\label{app:proof-radialunweightedspectrum}
\begin{proof}[Proof of Lemma~\ref{us:radial-transfer}]
We work in the radial subspace, with density measure $r^2\,\dd r$.
The key point is that $\mathcal T$ is not a boundedly invertible
similarity between the two Hilbert spaces.  We only use it on eigenvectors
and generalized eigenvectors, where the endpoint conditions can be checked
explicitly.

\smallskip
\noindent\emph{From density eigenvectors to mass eigenvectors.}
Let $f$ be an eigenvector or a generalized eigenvector of $L_{n,0}$,
and set
\[
 h=\mathcal T^{-1}f
 =\frac1{2r^3}\int_0^r f(s)s^2\,\dd s.
\]
Local elliptic regularity for the equation $(L_{n,0}-z)f=g$ (with $g$
an eigenvector or a preceding generalized eigenvector) gives smoothness
at the origin.  In particular,
$h(r)=h(0)+O(r^2)$ and $h'(r)=O(r)$ as $r\to0^+$.  For $r\ge1$,
Cauchy--Schwarz and differentiation of the defining integral give
\[
 |h(r)|\le Cr^{-3/2}\|f\|_{L^2(r^2\dd r)},\qquad
 h'(r)=\frac{f(r)}{2r}-\frac{3h(r)}r.
\]
Since $\rho_n(r)\asymp r^{4+2a_n}e^{-r^2/4}$ at infinity,
these estimates imply
\[
 \int_1^\infty \rho_n |h|^2\,\dd r
 \le C_n\|f\|_{L^2(r^2\dd r)}^2
 \int_1^\infty r^{1+2a_n}e^{-r^2/4}\,\dd r<\infty,
\]
and, using $|h'|^2\le C(|f|^2/r^2+|h|^2/r^2)$,
\[
 \int_1^\infty\rho_n|h'|^2\,\dd r
 \le C_n\int_1^\infty r^{2+2a_n}e^{-r^2/4}|f|^2\,\dd r
 +C_n\int_1^\infty\rho_n|h|^2\,\dd r<\infty.
\]
At zero, the regular expansion $h=h(0)+O(r^2)$ and
$\rho_n(r)\asymp r^4$ give the same integrability. Thus $h$
belongs to the quadratic-form domain of $\mathcal A_n$.
To treat Jordan chains without assuming a bounded similarity,
write $(L_{n,0}-z)f_k=f_{k-1}$ with $f_{-1}=0$ and
$h_k=\mathcal T^{-1}f_k$. Starting at $k=0$ and proceeding by induction,
the distributional intertwining identity gives
\[
 \mathcal T\bigl((\mathcal A_n-z)h_k-h_{k-1}\bigr)=0,
 \qquad h_{-1}=0.
\]
The expression in parentheses is well defined distributionally on
$(0,\infty)$; the regularity at the origin shown above and the smooth
coefficients of $\mathcal A_n$ make it locally bounded there.  Since
$\ker\mathcal T=\operatorname{span}\{r^{-3}\}$ on $(0,\infty)$,
the expression in parentheses is $C_kr^{-3}$.  Local boundedness
forces $C_k=0$, whence
\[
 (\mathcal A_n-z)h_k=h_{k-1}.
\]
Inductively $h_{k-1}\in L^2(\rho_n\,\dd r)$; the estimates above
also give $h_k\in L^2(\rho_n\,\dd r)$. Therefore
$\mathcal A_nh_k=zh_k+h_{k-1}$ lies in that space in the
distributional sense. Both endpoints are limit-point for the scalar
Sturm--Liouville expression, so its maximal $L^2(\rho_n\,\dd r)$
realization is the self-adjoint operator $\mathcal A_n$; thus
$h_k\in D(\mathcal A_n)$. This proves domain transfer at every
level of the Jordan chain, rather than merely at the eigenvector level.

Self-adjointness now forces every density eigenvalue in this half-plane
to be real and rules out every nontrivial Jordan chain.  Since the
Sturm--Liouville eigenvalues of $\mathcal A_n$ are simple and
$\mathcal T^{-1}$ is injective on the radial density space, the density
eigenspaces have dimension at most one.

\smallskip
\noindent\emph{From mass eigenvectors to density eigenvectors.}
Let
$\mathcal A_nh=\lambda h$ with $h\in D(\mathcal A_n)$ and
$\lambda<\tfrac14$.  At zero the limit-point condition selects the
regular branch $h=h(0)+O(r^2)$.  At infinity, Lemma~\ref{lem:adjointprofiletail}
gives
\[
 \begin{aligned}
 -h''+\left(\frac r2-\frac{4+2a_n}{r}+O(r^{-3})\right)h'+\left(1-\lambda-\frac{8a_n}{r^2}+O(r^{-4})\right)h=0.
 \end{aligned}
\]
Put $p=2\lambda-2$ and write $h=r^pv$.  The constant terms cancel,
and the equation for $v$ has the form
\[
 -v''+a_{\lambda,n}(r)v'+d_{\lambda,n}(r)v=0,
 \qquad
 a_{\lambda,n}(r)=\frac r2+O_{\lambda,n}(r^{-1}),\quad
 d_{\lambda,n}(r)=O_{\lambda,n}(r^{-2}).
\]
The Volterra equation obtained by integrating against
$\exp(-\int_R^r a_{\lambda,n})$ is a contraction on bounded functions
for $R$ large.  It gives the unique non-growing branch
\[
 v(r)=c\bigl(1+O(r^{-2})\bigr),\qquad
 v'(r)=O(r^{-3}),\qquad v''(r)=O(r^{-4}),
 \qquad c\ne0.
\]
The other branch grows like $e^{r^2/4}$ times a power and is excluded
by $h\in L^2(\rho_n\,\dd r)$.  Consequently,
\[
 \mathcal Th
 =2r^p\bigl((2\lambda+1)v+rv'\bigr).
\]
If $\lambda\ne-\tfrac12$, this is
$2(2\lambda+1)c r^p(1+O(r^{-2}))$; at the exceptional value
$\lambda=-\tfrac12$, the leading term cancels and
$\mathcal Th=O(r^{p-2})$.  In either case,
\[
 |\partial_r^k(\mathcal Th)(r)|\le C_{n,\lambda}r^{p-k},
 \qquad k=0,1,2,\quad r\ge R.
\]
Since $p=2\lambda-2<-\tfrac32$, this proves
$\mathcal Th\in H^2(\mathbb R^3)$ and
$r\partial_r(\mathcal Th)\in L^2(r^2\,\dd r)$.
Regularity at zero gives the same local membership, so
$\mathcal Th\in\mathscr D$.  The intertwining identity then yields
\[
 L_{n,0}(\mathcal Th)=\mathcal T\mathcal A_nh
 =\lambda\mathcal Th.
\]
Moreover $\mathcal Th\ne0$: otherwise $h=Cr^{-3}$, contradicting
regularity at zero.  (For $\lambda\ge\tfrac14$, the coefficient
$2(2\lambda+1)c$ is nonzero and the same asymptotic shows that
$\mathcal Th\notin L^2(r^2\,\dd r)$, so the threshold is sharp.)

We have therefore obtained a bijection between the eigenspaces in
$\operatorname{Re}z<\tfrac14$.  The first part maps any generalized
eigenvector of $L_{n,0}$ to a generalized eigenvector of the
self-adjoint operator $\mathcal A_n$; hence no nontrivial Jordan chain
can occur for $L_{n,0}$.  Thus the correspondence also preserves
algebraic multiplicity.  No bounded similarity of the two operators
is asserted or used.
\end{proof}

\subsection{Stable semigroups and modified energies}
\subsubsection{Decay and an equivalent energy}
\label{app:proof-stablelinearenergy}
\begin{proof}[Proof of Proposition~\ref{prop:stablelinearenergy}]
Lemma~\ref{lem:exactessentialthreshold} gives essential growth bound
$-1/4$, and the preceding spectral results give a number $\delta_n^{\rm lin}>0$
such that
$\sigma(L_n|_{X_{{\rm s},n}})\subset\{\operatorname{Re}z\ge\delta_n^{\rm lin}\}$.
Fix $s_1>0$ and put $S=S_{{\rm s},n}(s_1)$.
Its essential spectral radius is $e^{-s_1/4}<1$.
Every spectral point of $S$ outside this essential spectral disk is an
isolated eigenvalue of finite algebraic multiplicity. On its
finite-dimensional Riesz space the semigroup is a continuous matrix
semigroup, and therefore that eigenvalue of $S$ is $e^{-s_1z}$
for a generator eigenvalue $z$ of $L_n|_{X_{{\rm s},n}}$.
Consequently every such point has modulus at most
$e^{-s_1\delta_n^{\rm lin}}$, and
\[
 r(S)\le \max\{e^{-s_1/4},e^{-s_1\delta_n^{\rm lin}}\}<1.
\]
This argument uses quasi-compactness outside the essential disk,
not an unrestricted spectral mapping theorem for arbitrary
$C_0$-semigroups. The spectral-radius formula now gives exponential
decay at times $ks_1$; boundedness on $[0,s_1]$ gives
\eqref{eq:stablesemigroup} for all $s\ge0$.

Equation~\eqref{eq:L2numericalform} and the Newton estimate allow the choice
\begin{equation}\label{eq:anchoice}
 a_n=\frac32\|U_n\|_\infty+C\|\nabla U_n\|_3,
 \qquad
 \operatorname{Re}\langle L_nf,f\rangle
 \ge\|\nabla f\|_2^2+(1/4-a_n)\|f\|_2^2.
\end{equation}
The orbit integral defining $Q_n$ converges in operator norm and
$0\le Q_n\le M_n^2/(2\omega_n)$. For $f\in\mathscr D\cap X_{{\rm s},n}$,
integration of the orbit derivative gives
\begin{equation}\label{eq:Lyapunovidentity}
 2\operatorname{Re}\langle Q_nL_nf,f\rangle
 =-\int_0^\infty\frac d{ds}\|e^{-sL_n}f\|_2^2\,ds=\|f\|_2^2.
\end{equation}
Thus $B_n=I+(2a_n+1)Q_n$ satisfies
\[
 2\operatorname{Re}\langle B_nL_nf,f\rangle\ge\|f\|_2^2
 \ge\left(1+\frac{(2a_n+1)M_n^2}{2\omega_n}\right)^{-1}\mathcal E_n[f].
\]
This proves norm equivalence and dissipation with the stated factor of
two. Differentiating the energy and then using density gives the decay
for all stable initial data.
\end{proof}

\subsubsection{The sharp logarithmic rate}
\label{app:proof-sharpspectraldecay}
\begin{proof}[Completion of the proof of Proposition~\ref{prop:sharpspectraldecay}]
Removing the finite-dimensional unstable space preserves the essential
spectral radius, although it need not preserve the exact essential norm
under a nonorthogonal projection. The semigroup property gives
\[
 r_{\rm ess}(S_{{\rm s},n}(s_1))
 =r_{\rm ess}(S_n(s_1))
 =\lim_{k\to\infty}\|S_n(ks_1)\|_{\rm ess}^{1/k}=e^{-s_1/4}.
\]
The finite-dimensional argument above and spectral inclusion then yield
$r(S_{{\rm s},n}(s_1))=e^{-s_1\gamma_n^{\rm sp}}$.
The spectral-radius formula proves the logarithmic limit along multiples
of $s_1$ and every almost-sharp upper bound. Boundedness on $[0,s_1]$
extends the upper bound to all times; the same radius identity at each
$s>0$ gives $\|S_{{\rm s},n}(s)\|\ge e^{-s\gamma_n^{\rm sp}}$.
This proves \eqref{eq:mainsharprate} and \eqref{eq:almostsharprate}.

For $0<\eta<\gamma_n^{\rm sp}$ choose
$0<\epsilon<\gamma_n^{\rm sp}-\eta$ and set
\[
 Q_{n,\eta}=\int_0^\infty e^{2\eta s}S_{{\rm s},n}(s)^*S_{{\rm s},n}(s)\,ds,
 \quad c_{n,\eta}=2a_n+2\eta+1,\quad B_{n,\eta}=I+c_{n,\eta}Q_{n,\eta}.
\]
The almost-sharp bound makes the integral convergent, and
\[
 0\le Q_{n,\eta}\le\frac{C_{n,\epsilon}^2}
 {2(\gamma_n^{\rm sp}-\epsilon-\eta)}I.
\]
The identity term supplies the lower norm bound. Orbit integration gives
$2\operatorname{Re}\langle Q_{n,\eta}A_{{\rm s},n}f,f\rangle
=\|f\|_2^2+2\eta\langle Q_{n,\eta}f,f\rangle$ and consequently
\[
 2\operatorname{Re}\langle B_{n,\eta}(A_{{\rm s},n}-\eta)f,f\rangle
 =2\operatorname{Re}\langle(A_{{\rm s},n}-\eta)f,f\rangle
   +c_{n,\eta}\|f\|_2^2\ge\|f\|_2^2.
\]
This proves \eqref{eq:nearsharpenergycoercive} and the energy decay.
The equivalence constants may deteriorate as $\eta\to\gamma_n^{\rm sp}$.
\end{proof}

\clearpage
\section*{Acknowledgments}
The authors thank Charles Collot for his helpful comments and for suggesting that they study the stability of these self-similar profiles.
T. Li is partially supported by the National Natural Science Foundation of China under Grant~12421001.
K. Zhang is partially supported by the National Natural Science Foundation of China under Grant~ 12601366.

\section*{Declaration of generative AI use}
During revision, the authors used ChatGPT (OpenAI) to edit the English and to
help develop and draft parts of the Brouwer exit argument and the proof restoring
the translation and scaling parameters. The authors take full
responsibility for the manuscript.


\begin{thebibliography}{99}

\bibitem{Bai-Zhou-2025}
X.~Bai and M.~Zhou,
Exact blow-up profiles for the parabolic-elliptic Keller--Segel system in dimensions $N\ge3$,
\emph{Math. Ann.} \textbf{392} (2025), 313--337.

\bibitem{Biler2020}
P.~Biler,
\emph{Singularities of Solutions to Chemotaxis Systems},
De Gruyter Series in Mathematics and Life Sciences, vol.~6, De Gruyter, Berlin, 2020.

\bibitem{Blanchet-2012}
A.~Blanchet, E.~A.~Carlen and J.~A.~Carrillo,
Functional inequalities, thick tails and asymptotics for the critical mass Patlak--Keller--Segel model,
\emph{J. Funct. Anal.} \textbf{262} (2012), no.~5, 2142--2230.

\bibitem{Blanchet-2006}
A.~Blanchet, J.~Dolbeault and B.~Perthame,
Two-dimensional Keller--Segel model: optimal critical mass and qualitative properties of the solutions,
\emph{Electron. J. Differential Equations} (2006), no.~44, 32~pp..

\bibitem{BrennerEtAl1999}
M.~P.~Brenner, P.~Constantin, L.~P.~Kadanoff, A.~Schenkel and S.~C.~Venkataramani,
Diffusion, attraction and collapse,
\emph{Nonlinearity} \textbf{12} (1999), no.~4, 1071--1098.


\bibitem{Buseghin-Collot-202602}
F.~Buseghin and C.~Collot,
Classification of the dynamics of radial solutions to the 2D parabolic-elliptic Keller--Segel system,
arXiv:2606.20870 (2026).


\bibitem{Buseghin-Collot-202601}
F.~Buseghin and C.~Collot,
Determination of the long-time dynamics for the 2D Keller--Segel equation at critical mass,
arXiv:2606.20866 (2026).


\bibitem{Buseghin-Davila-2026}
F.~Buseghin, J.~D\'avila, M.~del Pino and M.~Musso,
Existence of finite time blow-up in Keller--Segel system,
\emph{Ann. PDE} \textbf{12} (2026), article~17.


\bibitem{Buseghin-Davila-Pino-Musso-2026}
F.~Buseghin, J.~D\'avila, M.~del Pino and M.~Musso,
Finite-time blow-up for the three dimensional axially symmetric Keller--Segel system,
\emph{J. Funct. Anal.} \textbf{290} (2026), no.~12, article~111443.

\bibitem{Chabi-Souplet-2026}
L.~D.~Chabi and P.~Souplet,
Sharp macroscopic blow-up behavior for the parabolic-elliptic Keller--Segel system in dimensions $n\ge3$,
arXiv:2601.14469 (2026).

\bibitem{Childress}
S.~Childress,
Chemotactic collapse in two dimensions,
In \emph{Modelling of Patterns in Space and Time} (Heidelberg, 1983), Lecture Notes in Biomathematics, vol.~55, Springer, Berlin, 1984, pp.~61--66.

\bibitem{Childress-Jerome}
S.~Childress and J.~K.~Percus,
Nonlinear aspects of chemotaxis,
\emph{Math. Biosci.} \textbf{56} (1981), nos.~3--4, 217--237.

\bibitem{Collot-2023}
C.~Collot, T.~Ghoul and N.~Masmoudi,
Collapsing-ring blow-up solutions for the Keller--Segel system in three dimensions and higher,
\emph{J. Funct. Anal.} \textbf{285} (2023), no.~7, article~110065.

\bibitem{Collot-2022}
C.~Collot, T.~Ghoul, N.~Masmoudi and V.~T.~Nguyen,
Refined description and stability for singular solutions of the 2D Keller--Segel system,
\emph{Comm. Pure Appl. Math.} \textbf{75} (2022), no.~7, 1419--1516.

\bibitem{Collot-Ghoul-2022}
C.~Collot, T.~Ghoul, N.~Masmoudi and V.~T.~Nguyen,
Spectral analysis for singularity formation of the two dimensional Keller--Segel system,
\emph{Ann. PDE} \textbf{8} (2022), article~5.

\bibitem{Collot-Ghoul-Masmoudi-Nguyen-2024}
C.~Collot, T.~Ghoul, N.~Masmoudi and V.~T.~Nguyen,
Singularity formed by the collision of two collapsing solitons in interaction for the 2D Keller--Segel system,
arXiv:2409.05363 (2024).


\bibitem{Collot-Pierre-2019}
C.~Collot, P.~Rapha\"el and J.~Szeftel,
On the stability of type I blow up for the energy super critical heat equation,
\emph{Mem. Amer. Math. Soc.} \textbf{260} (2019), no.~1255, v+97~pp..


\bibitem{CollotZhang2024}
C.~Collot and K.~Zhang,
On the stability of type I self-similar blowups for the Keller--Segel system in three dimensions and higher,
arXiv:2406.11358v2 (2024).


\bibitem{Davila-2024}
{J.~D\'avila, M. del Pino, J. Dolbeault, M. Musso, and J. Wei, Existence and stability of infinite time blow-up in the Keller-Segel system. Arch. Ration. Mech. Anal. (2024) 248(4): 61.}


\bibitem{Diaz-1998}
J.~I.~D\'iaz, T.~Nagai and J.-M.~Rakotoson,
Symmetrization techniques on unbounded domains: application to a chemotaxis system on $\mathbb R^N$,
\emph{J. Differential Equations} \textbf{145} (1998), no.~1, 156--183.

\bibitem{Giga-Mizoguchi-Senba-2011}
Y.~Giga, N.~Mizoguchi and T.~Senba,
Asymptotic behavior of type I blow-up solutions to a parabolic-elliptic system of drift-diffusion type,
\emph{Arch. Ration. Mech. Anal.} \textbf{201} (2011), 549--573.

\bibitem{GlogicSchoerkhuber2024}
I.~Glogi\'c and B.~Sch\"orkhuber,
Stable singularity formation for the Keller--Segel system in three dimensions,
\emph{Arch. Ration. Mech. Anal.} \textbf{248} (2024), article~4.


\bibitem{Herrero-1997}
M.~A.~Herrero, E.~Medina and J.~J.~L.~Vel\'azquez,
Finite-time aggregation into a single point in a reaction-diffusion system,
\emph{Nonlinearity} \textbf{10} (1997), no.~6, 1739--1754.

\bibitem{Herrero-1998}
M.~A.~Herrero, E.~Medina and J.~J.~L.~Vel\'azquez,
Self-similar blow-up for a reaction-diffusion system,
\emph{J. Comput. Appl. Math.} \textbf{97} (1998), nos.~1--2, 99--119.

\bibitem{Horstmann}
D.~Horstmann,
From 1970 until present: the Keller--Segel model in chemotaxis and its consequences. I,
\emph{Jahresber. Dtsch. Math.-Ver.} \textbf{105} (2003), no.~3, 103--165.

\bibitem{Hou-Nguyen-Song-2025}
T.~Y.~Hou, V.~T.~Nguyen and P.~Song,
Axisymmetric type II blowup solutions to the three dimensional Keller--Segel system,
arXiv:2502.19775 (2025).




\bibitem{Hou-Qin-Wang-2026}
T.~Y.~Hou, X.~Qin, P.~Song and Z.~Wang,
On the stability of type II blowup for the Keller--Segel system in high dimensions,
arXiv:2609.16467 (2026).


\bibitem{KellerSegel1970}
E.~F.~Keller and L.~A.~Segel,
Initiation of slime mold aggregation viewed as an instability,
\emph{J. Theoret. Biol.} \textbf{26} (1970), no.~3, 399--415.



\bibitem{LiWeiZhang2020}
T.~Li, D.~Wei and Z.~Zhang,
Pseudospectral and spectral bounds for the Oseen vortices operator,
\emph{Ann. Sci. \'Ec. Norm. Sup\'er.} (4) \textbf{53} (2020), no.~4, 993--1035.


\bibitem{LiZhouKSNS2025}
Z.~Li and T.~Zhou,
Finite-time blowup for Keller--Segel--Navier--Stokes system in three dimensions,
\emph{Comm. Math. Phys.} \textbf{406} (2025), no.~8, article~186.


\bibitem{LiZhou2026}
Z.~Li and T.~Zhou,
Nonradial stability of self-similar blowup to the Keller--Segel equation in three dimensions,
\emph{Ann. PDE} \textbf{12} (2026), article~29. 

\bibitem{Liu-wang-Zhou-2025}
J.~Liu, Y.~Wang and T.~Zhou,
Finite time blowup for Keller--Segel equation with logistic damping in three dimensions,
arXiv:2504.12231 (2025).




\bibitem{Metafune2001}
G.~Metafune,
$L^p$-spectrum of Ornstein--Uhlenbeck operators,
\emph{Ann. Scuola Norm. Sup. Pisa Cl. Sci.} (4) \textbf{30} (2001), no.~1, 97--124.


\bibitem{Mizoguchi-2022}
N.~Mizoguchi,
Refined asymptotic behavior of blow-up solutions to a simplified chemotaxis system,
\emph{Comm. Pure Appl. Math.} \textbf{75} (2022), no.~8, 1870--1886.

\bibitem{Mizo-Senba}
N.~Mizoguchi and T.~Senba,
Type II blow-up solutions to a parabolic-elliptic system,
\emph{Adv. Math. Sci. Appl.} \textbf{17} (2007), 505--545.

\bibitem{Mizoguchi-2011}
N.~Mizoguchi and T.~Senba,
A sufficient condition for type I blow-up in a parabolic-elliptic system,
\emph{J. Differential Equations} \textbf{250} (2011), no.~1, 182--203.

\bibitem{Nagai-1995}
T.~Nagai,
Blow-up of radially symmetric solutions to a chemotaxis system,
\emph{Adv. Math. Sci. Appl.} \textbf{5} (1995), 581--601.

\bibitem{Naito-2008}
Y.~Naito and T.~Suzuki,
Self-similarity in chemotaxis systems,
\emph{Colloq. Math.} \textbf{111} (2008), no.~1, 11--34.

\bibitem{Nguyen-Zaag-2023}
V.~T.~Nguyen, N.~Nouaili and H.~Zaag,
Construction of type I-log blow-up for the Keller--Segel system in dimensions 3 and 4,
\emph{Ann. PDE} \textbf{11} (2025), article~12.

\bibitem{NWZ2026}
V.~T.~Nguyen, Z.-A.~Wang and K.~Zhang,
Infinitely many self-similar blow-up profiles for the Keller--Segel system in dimensions 3 to 9,
\emph{J. Differential Equations} \textbf{458} (2026), article~114033.


\bibitem{DLMF}
NIST Digital Library of Mathematical Functions,
Chapters 5 and 13,
National Institute of Standards and Technology.


\bibitem{Raphael-2014}
P.~Rapha\"el and R.~Schweyer,
On the stability of critical chemotactic aggregation,
\emph{Math. Ann.} \textbf{359} (2014), nos.~1--2, 267--377.

\bibitem{Souplet-Winkler}
P.~Souplet and M.~Winkler,
Blow-up profiles for the parabolic-elliptic Keller--Segel system in dimensions $n\ge3$,
\emph{Comm. Math. Phys.} \textbf{367} (2019), 665--681.

\bibitem{Suzuki-2011}
T.~Suzuki and T.~Senba,
\emph{Applied Analysis: Mathematical Methods in Natural Science}, 2nd ed.,
Imperial College Press, London, 2011.

\bibitem{Wolansky}
G.~Wolansky,
On steady distributions of self-attracting clusters under friction and fluctuations,
\emph{Arch. Ration. Mech. Anal.} \textbf{119} (1992), 355--391.

\bibitem{Zettl2005}
A.~Zettl,
\emph{Sturm--Liouville Theory},
Mathematical Surveys and Monographs, vol.~121, American Mathematical Society, Providence, RI, 2005.

\end{thebibliography}
\end{document}